\documentclass[12pt,leqno]{article}

\usepackage{amsmath}
\usepackage{amssymb}
\usepackage{amsxtra}
\usepackage{amscd}
\usepackage{amsthm}
\usepackage[all]{xy}
\usepackage[mathscr]{eucal}
\usepackage{hyperref}
\usepackage{stmaryrd}
\usepackage{newtxtext}

\usepackage{enumitem}
\setlist[1]{itemsep=0ex, topsep=1ex}
\setlist[2]{nosep}

\theoremstyle{plain}

\newtheorem{sbthm}[subsubsection]{Theorem}
\newtheorem{sbprop}[subsubsection]{Proposition}
\newtheorem{sbcor}[subsubsection]{Corollary}
\newtheorem{sblem}[subsubsection]{Lemma}

\theoremstyle{definition}

\newtheorem{sbrem}[subsubsection]{Remark}

\newtheorem{sbpara}[subsubsection]{}
\newtheorem{example}[subsubsection]{Example}

\numberwithin{equation}{section}

\newenvironment{pf}{\proof[\proofname]}{\endproof}

\DeclareMathOperator{\PGL}{PGL}
\DeclareMathOperator{\GL}{GL}
\DeclareMathOperator{\SL}{SL}

\DeclareMathOperator{\Hom}{Hom}
\DeclareMathOperator{\pole}{pole}

\DeclareMathOperator{\Spf}{Spf}
\DeclareMathOperator{\Spa}{Spa}
\DeclareMathOperator{\faces}{face}
\DeclareMathOperator{\class}{class}

\newcommand{\BT}{\mathrm{BT}}
\newcommand{\AP}{\mathrm{AP}}

\newcommand{\smatrix}[1]{\left(\begin{smallmatrix}#1\end{smallmatrix}\right)}
\newcommand{\Pmatrix}[1]{\begin{pmatrix}#1\end{pmatrix}}

\newcommand{\ps}[1]{\llbracket #1 \rrbracket}
\newcommand{\ls}[1]{(\!(#1)\!)}

\newcommand{\mr}{\mathrm}

\newcommand{\textsp}[1]{\text{ #1 }}

\begin{document}

\title{Toroidal compactifications of the moduli spaces of Drinfeld modules}

\author{Takako Fukaya, Kazuya Kato, Romyar Sharifi}
%\subjclass{Primary 14M25; Secondary 14F20}

\maketitle

\newcommand\Cal{\mathcal}
\newcommand\define{\newcommand}

\define\gp{\mathrm{gp}}%
\define\fs{\mathrm{fs}}%
\define\an{\mathrm{an}}%
\define\mult{\mathrm{mult}}%
\define\Coker{\mathrm{Coker}\,}%
\define\Ext{\mathrm{Ext}\,}%
\define\rank{\mathrm{rank}\,}%
\define\gr{\mathrm{gr}}%
\define\cHom{\Cal{Hom}}
\define\cExt{\Cal Ext\,}%

\define\cA{\Cal A}
\define\cB{\Cal B}
\define\cC{\Cal C}
\define\cD{\Cal D}
\define\cI{\Cal I}
\define\cO{\Cal O}
\define\cS{\Cal S}
\define\cT{\Cal T}
\define\cM{\Cal M}
\define\cG{\Cal G}
\define\cH{\Cal H}
\define\cR{\Cal R}
\define\cE{\Cal E}
\define\cF{\Cal F}
\define\cN{\Cal N}
\define\cU{\Cal U}
\define\cV{\Cal V}
\define\cY{\Cal Y}
\define\fF{\frak F}
\define\fM{\frak M}
\define\Dc{\check{D}}
\define\Ec{\check{E}}

\newcommand{\N}{{\mathbb{N}}}
\newcommand{\Q}{{\mathbb{Q}}}
\newcommand{\Z}{{\mathbb{Z}}}
\newcommand{\R}{{\mathbb{R}}}
\newcommand{\C}{{\mathbb{C}}}
\newcommand{\bN}{{\mathbb{N}}}
\newcommand{\bQ}{{\mathbb{Q}}}
\newcommand{\bF}{{\mathbb{F}}}
\newcommand{\bZ}{{\mathbb{Z}}}
\newcommand{\bP}{{\mathbb{P}}}
\newcommand{\bR}{{\mathbb{R}}}
\newcommand{\bC}{{\mathbb{C}}}
\newcommand{\bbQ}{{\bar \mathbb{Q}}}
\newcommand{\ol}[1]{\overline{#1}}
\newcommand{\too}{\longrightarrow}
\newcommand{\respect}{\rightsquigarrow}
\newcommand{\compatible}{\leftrightsquigarrow}
\newcommand{\upc}[1]{\overset {\lower 0.3ex \hbox{${\;}_{\circ}$}}{#1}}
\newcommand{\Gmlog}{\bG_{m, \log}}
\newcommand{\Gm}{\bG_m}
\newcommand{\Ga}{\Gamma}
\newcommand{\ep}{\varepsilon}
\newcommand{\Spec}{\operatorname{Spec}}
\newcommand{\val}{{\mathrm{val}}} 
\newcommand{\n}{\operatorname{naive}}
\newcommand{\bs}{\backslash}
\newcommand{\Gal}{\operatorname{{Gal}}}
\newcommand{\gal}{{\rm {Gal}}({\bar \Q}/{\Q})}
\newcommand{\galp}{{\rm {Gal}}({\bar \Q}_p/{\Q}_p)}
\newcommand{\gall}{{\rm{Gal}}({\bar \Q}_\ell/\Q_\ell)}
\newcommand{\wep}{W({\bar \Q}_p/\Q_p)}
\newcommand{\wel}{W({\bar \Q}_\ell/\Q_\ell)}
\newcommand{\Ad}{{\rm{Ad}}}
\newcommand{\BS}{{\rm {BS}}}
\newcommand{\even}{\operatorname{even}}
\newcommand{\End}{{\rm {End}}}
\newcommand{\odd}{\operatorname{odd}}
\newcommand{\np}{\text{non-$p$}}
\newcommand{\g}{{\gamma}}
\newcommand{\G}{{\Gamma}}
\newcommand{\Lam}{{\Lambda}}
\newcommand{\La}{{\Lambda}}
\newcommand{\lam}{{\lambda}}
\newcommand{\la}{{\lambda}}
\newcommand{\uL}{{{\hat {L}}^{\rm {ur}}}}
\newcommand{\uQp}{{{\hat \Q}_p}^{\text{ur}}}
\newcommand{\sel}{\operatorname{Sel}}
\newcommand{\dt}{{\rm{Det}}}
\newcommand{\Sig}{\Sigma}
\newcommand{\fil}{{\rm{fil}}}
\newcommand{\spl}{{\rm{spl}}}
\newcommand{\st}{{\rm{st}}}
\newcommand{\Isom}{{\rm {Isom}}}
\newcommand{\Mor}{{\rm {Mor}}}
\newcommand{\bg}{\bar{g}}
\newcommand{\id}{{\rm {id}}}
\newcommand{\cone}{{\rm {cone}}}
\newcommand{\al}{a}
\newcommand{\ChL}{{\cal{C}}(\La)}
\newcommand{\Image}{{\rm {Image}}}
\newcommand{\toric}{{\operatorname{toric}}}
\newcommand{\torus}{{\operatorname{torus}}}
\newcommand{\Qp}{{\mathbb{Q}}_p}
\newcommand{\barQp}{{\mathbb{Q}}_p}
\newcommand{\Qpur}{{\mathbb{Q}}_p^{\rm {ur}}}
\newcommand{\Zp}{{\mathbb{Z}}_p}
\newcommand{\Zl}{{\mathbb{Z}}_l}
\newcommand{\Ql}{{\mathbb{Q}}_l}
\newcommand{\Qlur}{{\mathbb{Q}}_l^{\rm {ur}}}
\newcommand{\F}{{\mathbb{F}}}
\newcommand{\eps}{{\epsilon}}
\newcommand{\add}{{{\rm add}}}
\newcommand{\epsLa}{{\epsilon}_{\La}}
\newcommand{\epsLaVxi}{{\epsilon}_{\La}(V, \xi)}
\newcommand{\epsOLaVxi}{{\epsilon}_{0,\La}(V, \xi)}
\newcommand{\Qplin}{{\mathbb{Q}}_p(\mu_{l^{\infty}})}
\newcommand{\otimesQplin}{\otimes_{\Qp}{\mathbb{Q}}_p(\mu_{l^{\infty}})}
\newcommand{\galFl}{{\rm{Gal}}({\bar {\Bbb F}}_\ell/{\Bbb F}_\ell)}
\newcommand{\gallur}{{\rm{Gal}}({\bar \Q}_\ell/\Q_\ell^{\rm {ur}})}
\newcommand{\galFF}{{\rm {Gal}}(F_{\infty}/F)}
\newcommand{\galFv}{{\rm {Gal}}(\bar{F}_v/F_v)}
\newcommand{\galF}{{\rm {Gal}}(\bar{F}/F)}
\newcommand{\epsVxi}{{\epsilon}(V, \xi)}
\newcommand{\epsOVxi}{{\epsilon}_0(V, \xi)}
\newcommand{\plim}{\lim_
{\scriptstyle 
\longleftarrow \atop \scriptstyle n}}
\newcommand{\sig}{{\sigma}}
\newcommand{\ga}{{\gamma}}
\newcommand{\del}{{\delta}}
\newcommand{\Vss}{V^{\rm {ss}}}
\newcommand{\Bst}{B_{\rm {st}}}
\newcommand{\Dpst}{D_{\rm {pst}}}
\newcommand{\Dcrys}{D_{\rm {crys}}}
\newcommand{\DdR}{D_{\rm {dR}}}
\newcommand{\Fin}{F_{\infty}}
\newcommand{\Kla}{K_{\lambda}}
\newcommand{\Ola}{O_{\lambda}}
\newcommand{\Mla}{M_{\lambda}}
\newcommand{\Det}{{\rm{Det}}}
\newcommand{\Sym}{{\rm{Sym}}}
\newcommand{\LaSa}{{\La_{S^*}}}
\newcommand{\cX}{{\cal {X}}}
\newcommand{\MHG}{{\frak {M}}_H(G)}
\newcommand{\tauMla}{\tau(M_{\lambda})}
\newcommand{\Fvur}{{F_v^{\rm {ur}}}}
\newcommand{\Lie}{{\rm {Lie}}}
\newcommand{\cL}{{\cal {L}}}
\newcommand{\cW}{{\cal {W}}}
\newcommand{\fq}{{\frak {q}}}
\newcommand{\cont}{{\rm {cont}}}
\newcommand{\SC}{{SC}}
\newcommand{\Om}{{\Omega}}
\newcommand{\dR}{{\rm {dR}}}
\newcommand{\crys}{{\rm {crys}}}
\newcommand{\hatSig}{{\hat{\Sigma}}}
\newcommand{\rdet}{{{\rm {det}}}}
\newcommand{\BdR}{{B_{\rm {dR}}}}
\newcommand{\BdRO}{{B^0_{\rm {dR}}}}
\newcommand{\Bcrys}{{B_{\rm {crys}}}}
\newcommand{\Qw}{{\mathbb{Q}}_w}
\newcommand{\barkappa}{{\bar{\kappa}}}
\newcommand{\cP}{{\Cal {P}}}
\newcommand{\cQ}{{\Cal {Q}}}
\newcommand{\cZ}{{\Cal {Z}}}
\newcommand{\oppLa}{{\Lambda^{\circ}}}
\newcommand{\sn}{{{\rm {sn}}}}
\newcommand{\et}{\text{\'et}}
\newcommand{\ZR}{{{\rm {ZR}}}}
\newcommand{\nl}{{{\rm {nl}}}}
\newcommand{\adic}{{{\rm {adic}}}}
\newcommand{\Sa}{{{\rm {Sa}}}}
\newcommand{\pair}{{{\rm {pair}}}}
\newcommand{\sep}{{{\rm{sep}}}}

\begin{abstract}
We construct toroidal compactifications of the moduli spaces of Drinfeld $\F_q[T]$-modules of rank $d$ with level $N$ structure as moduli spaces of log Drinfeld modules of rank $d$ with level $N$ structure. The toroidal compactifications are log regular schemes associated to rational cone decompositions, and there are regular ones among them. To construct these toroidal compactifications, we blow up the Satake compactification of Pink and employ the theory of formal moduli and a process of iterated Tate uniformization.
\end{abstract}

\setcounter{tocdepth}{1}
\tableofcontents

\section{Introduction} 

\subsection{Summary}\label{s1.1}

\begin{sbpara}\label{001} Let $F$ be a function field in one variable over a finite field. Fix a place $\infty$ of $F$, and let $A$ be the subring of $F$ consisting of all elements which are integral outside $\infty$. 

In this paper, 
we construct toroidal compactifications of the moduli spaces of Drinfeld $A$-modules
in the case $A=\F_q[T]$. 
In a planned sequel to this paper, they will be constructed for general $A$ by a method of reduction to the case $A=\F_q[T]$. 
\end{sbpara}

\begin{sbpara}\label{002}

Assume that $A=\F_q[T]$ for some prime power $q$, and let $N$ be an element of $A$ which does not belong to $\F_q$.  

In the case that $N$ has at least two distinct prime divisors (resp., $N$ has only one prime divisor), let $\cM^d_N$ over $A$ (resp.,  over $A[\frac{1}{N}]$) be the moduli space of Drinfeld $A$-modules of rank $d$ with Drinfeld level $N$ structure. Recall that by \cite[Section 5]{D}, the moduli space $\cM^d_N$ is regular, and $A[\frac{1}{N}]\otimes_A \cM^d_N$ is smooth over $A[\frac{1}{N}]$. 

In this paper, we construct toroidal compactifications of $\cM^d_N$ as the moduli spaces of {\it log Drinfeld modules of rank $d$ with level $N$ structure}. 
\end{sbpara}

\begin{sbpara}\label{004}
Similarly to toroidal compactifications of the moduli spaces of abelian varieties as in \cite{AMRT, FC}, our toroidal compactifications are indexed by cone decompositions and form a projective system under subdivisions.  After inverting $N$, they are log smooth over $A[\frac{1}{N}]$, and among them, smooth toroidal compactifications are cofinal.  
In the case that $N$ has at least two prime divisors, toroidal compactifications (without inverting $N$) are log regular (or equivalently, log smooth over $\F_q$ by \ref{logD20}), and among them, regular toroidal compactifications are cofinal. 
\end{sbpara}

\begin{sbpara}\label{003}  The compactification of $\cM^d_N$ of Satake-Baily-Borel type (called a Satake compactification in \cite{P2}) was constructed by Kapranov \cite{K} in the case that $A=\F_q[T]$ and by Pink \cite{P2} in general.

In a short summary \cite{P1}, Pink explained how to construct toroidal compactifications of $\cM^d_N$ quoting  a  work of K. Fujiwara. The details are not yet published. In this paper, we use the ideas of  Pink and Fujiwara.

The  toroidal compactification of the moduli space of Drinfeld $A$-modules with $A=\F_q[T]$ and $N=T$ is contained in the work of Puttick \cite{PA}.

In the case $d=2$, our toroidal compactification coincides with the proper model of the Drinfeld modular curve constructed by Lehmkuhl \cite{TL}.

\end{sbpara}

\subsection{Log Drinfeld modules with level structures}\label{lD0}
Let $F$, $\infty$, and $A$ be as in \ref{001}. Let $p$ be the characteristic of $F$.

\begin{sbpara}\label{logD0}

 Let $S$ be a scheme over $A$.  In this paper, we frequently refer to Drinfeld $A$-modules over $S$ more simply as Drinfeld modules (over $S$). 

Recall that a Drinfeld module over $S$ of rank $d\geq 1$ is a pair $(\cL, \phi)$ where $\cL$ is a line bundle over $S$ and $\phi$ is an action of $A$ on the additive group scheme $\cL$, satisfying a certain condition. A generalized Drinfeld module  over $S$ is such a pair $(\cL, \phi)$  satisfying a slightly weaker condition, found in \ref{gen_Drin_mod}.   

The Satake compactification of $\cM^d_N$ obtained in \cite{K} and \cite{P2} is regarded as the moduli space of generalized Drinfeld modules  (see Section \ref{Satake}). 

For the toroidal compactification, we introduce the notion of a log Drinfeld module of rank $d$ with level $N$ structure below.

\end{sbpara}

\begin{sbpara}\label{logD2} Let $N$ be an element of $A$ which does not belong to the total constant field of $F$. Let $d\geq 1$ be an integer. 

Recall that for a scheme $S$ over $A$ and for a Drinfeld $A$-module $(\cL, \phi)$ over $S$ of rank $d$, a Drinfeld level $N$ structure on $(\cL,\phi)$ is a homomorphism of $A$-modules
 $$
 	\iota \colon (\tfrac{1}{N}A/A)^d \to \cL,
$$
where $A$ acts on $\cL$ via $\phi$, satisfying a certain condition  (\ref{Dlevel}). If $N$ is invertible on $S$, the condition is that $\iota$ induces an isomorphism  
 $$
 	(\tfrac{1}{N}A/A)^d \xrightarrow{\sim} \phi[N]
$$ 
of group schemes, where $\phi[N] =\ker(\phi(N) \colon \cL\to \cL)$.
 
 \end{sbpara}

\begin{sbpara}\label{log1} In this paper, we consider schemes with log structures.  For generalities of log structures, we refer the reader to the paper \cite{KK1} of the second author, the paper \cite{IL} of Illusie, and the book \cite{O} of Ogus. See \ref{sat} of this paper for a short introduction. A scheme endowed with a log structure is called a log scheme. In this paper, we use saturated log structures (\ref{sat}). Saturated log structures relate well with toric geometry and the cone decompositions which we use in our theory of toroidal compactifications.

\end{sbpara}

\begin{sbpara}\label{log2}

A standard example of a saturated log structure which we consider is as follows.

Let $S$ be a normal scheme, and let $U$ be a dense open subset of $S$ with the inclusion morphism $j \colon U\to S$. Then the sheaf 
$$
	M_S :=  \cO_S\cap j_*(\cO_U^\times) =\{f\in \cO_S \mid f\;\text{is invertible on $U$}\} \subset j_*(\cO_U)
$$ 
and the map $M_S \hookrightarrow \cO_S$ form a saturated log structure. We call $M_S$ the log structure on $S$ associated to $U$ (or the log structure on $S$ associated to the closed set $S\setminus U$).

\end{sbpara}

\begin{sbpara}\label{log3}  Our toroidal compactifications are related to toric varieties. Consider the affine toric variety $\Spec(\F_p[P])$ over $\F_p$, for $P$ a finitely generated, torsion-free saturated monoid (fs monoid for short, see \ref{sat}), where $\F_p[P]$ denotes the semigroup ring of $P$ over $\F_p$. It has the standard log structure.

For a scheme $S$ over $A$ endowed with a log structure $M_S$, the following four conditions are equivalent:
\begin{enumerate}
\item[(i)] The scheme $S$ is log smooth over $\F_p$.
\item[(ii)] The scheme $S$ is log regular (\ref{logreg}) and locally of finite type over $A$.
\item[(iii)] \'Etale locally on $S$, there are an fs monoid $P$ and an \'etale morphism $S\to \Spec(\F_p[P])$ such that $M_S$ is the inverse image of the standard log structure on $\Spec(\F_p[P])$.
\item[(iv)] \'Etale locally on $S$, there are an fs monoid $P$ and a smooth morphism $S\to \Spec(\F_p[P])$ such that $M_S$ is the inverse image of the standard log structure on $\Spec(\F_p[P])$.
\end{enumerate}
Under these equivalent conditions,  $S$ is normal, and $M_S=\cO_S\cap j_*(\cO_U^\times)$, where $U$ is the dense open subset of $S$ consisting of all points at which the log structure of $S$ is trivial. 

\end{sbpara}

\begin{sbpara}\label{tL} Let $S$ be a scheme with a saturated log structure $M_S$. For a line bundle  $\cL$ on $S$, we define a sheaf $\overline{\cL}\supset \cL$  on the \'etale site of $S$. Put succinctly, $\overline{\cL}$ is an enlargement of $\cL$ allowing poles which belong to the log structure of $S$. It plays a role here because  torsion points of  generalized Drinfeld modules have poles. 

Let $\cL^\times\subset \cL$ denote the sheaf of bases of $\cL$. Setting
$$
	M_S^{-1}=\{f^{-1} \mid f\in M_S\}\subset M_S^{\gp}=\{fg^{-1} \mid f,g \in M_S\},
$$ 
we let $M_S^{-1} \cL^\times$ denote the quotient $M_S^{-1}\times^{\cO_S^\times}\cL^\times$ of $M_S^{-1}\times \cL^\times$ by the relation $(hu, e)\sim (h, ue)$ for $h\in M_S^{-1}$, $u\in \cO^\times_S$, and $e\in \cL^\times$.
We then define $\overline{\cL}$ as
$$\overline{\cL} = \cL \cup_{\cL^\times} M_S^{-1}\cL^\times,$$
where $\cup_{\cL^\times}$ means the union obtained by identifying $\cL^\times$ viewed as a subsheaf of $\cL$ and of 
$M_S^{-1}\cL^\times$. 
Note that the additive group structure of $\cL$ does not extend to $\overline{\cL}$.

For the trivial line bundle $\cO_S$, we have $\overline{\cO_S} = \cO_S \cup_{\cO_S^\times} M_S^{-1}$, which
can be identified with the sheaf of morphisms to the log scheme $\mathbb{P}^1_{\Z}$, the projective line over $\Z$, with the 
log structure associated to the divisor at infinity. 
 
We define $$\pole \colon \overline{\cL} \to M_S/\cO_S^\times$$ as the map which sends $\cL$ to $1$ and sends $f^{-1}e$  to $f\bmod \cO_S^\times$ for $f\in M_S$ and $e\in \cL^\times$.  
 
  \end{sbpara}
  
  \begin{sbpara}\label{log4} Now taking $S$ to be an $A$-scheme, let $j \colon U \to S$ be an inclusion of a dense open 
subset as in \ref{log2}.  By a 
  \emph{generalized Drinfeld module over $(S,U)$ of rank $d$ with level $N$ structure}, we mean a pair $((\cL, \phi), \iota)$ of
\begin{itemize}
	\item a generalized Drinfeld module $(\cL, \phi)$ over $S$ whose restriction $(\cL,\phi)|_U$ to $U$ is a Drinfeld module of rank $d$ and 
        \item a Drinfeld level $N$ structure $\iota$ on $(\cL, \phi)|_U$.
\end{itemize} 

We can show  (see \ref{logD53}) that for such a generalized Drinfeld module over $(S,U)$, the level $N$ structure $\iota$ on $(\cL,\phi)|_U$ extends uniquely to a map of sheaves
$$
 	\iota \colon (\tfrac{1}{N}A/A)^d\to \overline{\cL}\subset j_*(\cL|_U),
$$ 
where $\overline{\cL}$ is defined by the log structure of $S$ in \ref{log2}.

  \end{sbpara}
 
 \begin{sbpara}\label{logD6}

Let $S$ be a scheme over $A$ with saturated log structure $M_S$. By a
 \emph{log Drinfeld module over $S$} of rank $d$ with level $N$ structure, we mean a pair $((\cL, \phi), \iota)$ of  
a generalized Drinfeld module $(\cL, \phi)$ over $S$ and a map 
$$
	\iota \colon (\tfrac{1}{N}A/A)^d\to \overline{\cL}
$$ 
such that \'etale locally on $S$, there are 
\begin{itemize}
	\item a log scheme $S'$ over $A$ satisfying the equivalent conditions (i)--(iv) in \ref{log3},
	\item a morphism $f \colon S\to S'$ of log schemes over $A$,  
	\item a generalized Drinfeld module  $((\cL', \phi'), \iota')$ over $(S',U')$ of rank $d$ with level $N$ structure in the sense of \ref{log4}, 
	where $U'$ denotes the dense open set of $S'$ at which the log structure of $S'$ is trivial, and 
	\item an isomorphism between $((\cL, \phi), \iota)$ and the pullback of $((\cL', \phi'), \iota')$ 
	by $f$, where $\iota'$ is regarded as a map $(\frac{1}{N}A/A)^d\to \overline{\cL'}$ as in \ref{log4}.
 \end{itemize}
    
 Note that we do not attempt to define the notion of a log Drinfeld module without reference to a level structure.

\end{sbpara}

   \begin{sbrem}\label{logD8} \
\begin{enumerate}
 \item[(1)]   If $S$ is a  scheme over $A$ with trivial log structure, the notions of a log Drinfeld module over $S$ of rank $d$ with level $N$ structure
 and of a Drinfeld module over $S$ of rank $d$ with Drinfeld level $N$ structure are equivalent (see \ref{Deflog1}). 
 \item[(2)] A Drinfeld module is additive and a log structure is multiplicative. Hence it is not very easy to connect these two
different notions to define a log Drinfeld module.  In our formulation, a log Drinfeld module  is like a centaur: it has the body $\overline{\cL}$ whose
half $\cL$  is additive and whose other half $M_S^{-1}\cL^\times$  is of multiplicative nature.  
\end{enumerate}
   \end{sbrem}
   
 We will prove the following two propositions in the remarks following \ref{regdiv} and in \ref{3logD9}, respectively.

 \begin{sbprop}\label{logD7} 
   Let $S$ be as in \ref{logD6}, and let $((\cL, \phi), \iota)$ be a log Drinfeld module over $S$ of rank $d$ with level $N$ structure. 
   Let $a,b\in (\frac{1}{N}A/A)^d$. Then locally on  $S$, we have either 
$$
	\frac{\pole(\iota(a))}{\pole(\iota(b))} \in M_S/\cO_S^\times \quad\text{or}\quad
	\frac{\pole(\iota(b))}{\pole(\iota(a))} \in M_S/\cO_S^\times\quad \text{in}\quad M_S^{\gp}/\cO^\times_S.
	$$
 \end{sbprop}

\begin{sbprop}\label{logD9} 
	Let $S$ and $((\cL, \phi), \iota)$ be as in \ref{logD7}, and suppose that $A=\F_q[T]$.  Then:
	\begin{enumerate}
	\item[(1)] Locally on $S$, there exists an $A/NA$-basis $(e_i)_{0\leq i\leq d-1}$ of $(\frac{1}{N}A/A)^d$ such that 
	for each $i$, we have
	$$
		\frac{\pole(\iota(a))}{\pole(\iota(e_i))} \in M_S/\cO^\times_S \subset M_S^{\gp}/\cO^\times_S
	$$
	for all $a\in (\frac{1}{N}A/A)^d$ with $a \notin \sum_{j=0}^{i-1} (A/NA)e_j$.
	\item[(2)] The values $\pole(\iota(e_i))$ with $0\leq i\leq d-1$ are independent of the choice of basis $(e_i)_i$ in (1), 
	and $\pole(\iota(e_0))=1$ for every choice of $(e_i)_i$.  
	\end{enumerate}
\end{sbprop}

\subsection{Main results}\label{main_result}

\begin{sbpara}

In the rest of this introduction, we suppose that $A=\F_q[T]$. 

We fix $N$ and $d$ as in Section \ref{lD0}.
In the case $N$ has at least two prime divisors (resp.,  $N$ has only one prime divisor), let  $\cC_{\log}$ be the category of schemes over $A$ (resp., $A[\frac{1}{N}]$) with saturated log structures.

\end{sbpara}

\begin{sbpara}\label{toroidal_functor}

Let $${\overline {\fM}}^d_N \colon \cC_{\log} \to (\text{Sets})$$ be the contravariant functor for which ${\overline {\fM}}^d_{N}(S)$ is the set of all isomorphism classes of log Drinfeld modules over $S$ of rank $d$ with level $N$ structure.

\end{sbpara}

\begin{sbpara}

 Let 
 $\cM^d_N$ be the moduli space of Drinfeld modules of \ref{002}.
We endow $\cM^d_N$ with the trivial log structure. 
Viewed as a functor, $\cM^d_N \colon \cC_{\log} \to (\text{Sets})$ sends an object $S$ to the set
$\text{Mor}(S, \cM^d_N)$ which consists of all isomorphism classes of Drinfeld modules over the underlying scheme of $S$ over $A$ of rank $d$ endowed with Drinfeld level $N$ structure.  So by \ref{logD8}(1), we have an embedding of functors  
$\cM^d_N \hookrightarrow \overline{\fM}^d_N$.

\end{sbpara}

\begin{sbpara}\label{Cd} 
Consider the cone
$$
	C_d := \{(s_1,\dots, s_{d-1})\in \R^{d-1} \mid 0\leq s_1\leq \dots \leq s_{d-1}\}.
$$  
For a finite rational cone decomposition $\Sig$ of $C_d$, we define the subfunctor $\overline{\fM}^d_{N, \Sig}$  of $\overline{\mathfrak{M}}^d_N$  which contains $\cM^d_N$ as follows. For a cone $\sig$ in $\R^{d-1}$, let
$$
	\sig^{\vee} = \left\{(b_i)_i\in \Z^{d-1} \mid \sum_{i=1}^{d-1} b_is_i\geq 0 \text{ for all } (s_i)_i\in \sig\right\}.
$$

For an object $S$ of $\cC_{\log}$, the set $\overline{\fM}^d_{N,\Sig}(S)$ consists of all 
$$
	((\cL, \phi), \iota)\in \overline{\fM}^d_N(S)
$$ 
such that locally on $S$, there exist an $A/NA$-basis $(e_i)_{0\leq i\leq d-1}$ of $(\frac{1}{N}A/A)^d$ satisfying (1) of \ref{logD9}  and a cone $\sig\in \Sig$ such that  
$$
	\prod_{i=1}^{d-1} \pole(\iota(e_i))^{b_i}\in M_S/\cO^\times_S \subset M_S^{\gp}/\cO_S^\times
$$
for all $b \in \sig^{\vee}$.

\end{sbpara}

The following theorem is the main result of this paper. 

\begin{sbthm}\label{main} 
\begin{enumerate}
 	\item[(1)] For every finite rational cone decomposition $\Sigma$ of $C_d$, 
	 the   functor ${\overline {\fM}}^d_{N, \Sig}$ is represented by an object $\overline{\cM}^d_{N, \Sig}$ of $\cC_{\log}$. This log scheme $\overline{\cM}^d_{N, \Sig}$ is 
	 log  regular, and  $A[\frac{1}{N}]\otimes_A \overline{\cM}^d_{N, \Sig}$  is log smooth over $A[\frac{1}{N}]$. The 
	 underlying scheme of $\overline{\cM}^d_{N, \Sig}$ is proper over $A$ (resp., $A[\frac{1}{N}]$)  if $N$ has at least two prime divisors (resp.,  only one prime divisor). 
        \item[(2)] If $\Sig'$ is a  finite rational subdivision  of $\Sig$, 
	the morphism  $\overline \cM^d_{N,\Sig'}\to \overline \cM^d_{N,\Sig}$ is proper, birational, and log \'etale.
	\item[(3)] For a given finite rational cone decomposition $\Sig$ of $C_d$, there is a finite 
	rational subdivision $\Sig'$ of $\Sig$ such that the underlying scheme of $\overline \cM^d_{N, \Sig'}$ is regular, the boundary 
	$\overline{\cM}^d_{N, \Sig'}   \setminus \cM^d_N$ is a divisor
	 with normal crossings, the underlying scheme of $A[\frac{1}{N}]\otimes_A \overline{\cM}^d_{N, \Sig'}$ is smooth over $A[\frac{1}{N}]$, and 
	$A[\frac{1}{N}] \otimes_A \overline{\cM}^d_{N, \Sig'}  \setminus A[\frac{1}{N}]\otimes_A \cM^d_N$ is a relative normal crossing divisor 
	over $A[\frac{1}{N}]$. 
	 \end{enumerate}
 \end{sbthm}
  
We call the space $\overline{\cM}^d_{N, \Sig}$ in (1) a toroidal compactification of $\cM^d_N$. The proof of Theorem \ref{main}  is given in Section \ref{toroidal}.
  
Our theory of cone decompositions is explained in Section \ref{ss:cone1} and is based on the decomposition of the Bruhat-Tits building into simplices, which is reviewed in Section \ref{BT}.  
 
Part (3) of the following theorem provides a relationship between the shapes of cones in $\Sig$ and the local shapes of toroidal compactifications.  For $\sig\in \Sig$ and a commutative ring $R$, let $\toric_R(\sig)$ be the affine toric variety $\Spec(R[\sig^{\vee}])$ over $R$ with the standard log structure, where $\sig^{\vee}$ is the torsion free fs monoid defined in \ref{Cd}.

\begin{sbthm}\label{shape} 
There is a special finite rational cone decomposition $\Sig_k$ of $C_d$ for each $k\geq 1$ such that if $N \in A = \F_q[T]$ has degree $k$, then the following hold.
\begin{enumerate}  
	\item[(1)] The subfunctor $\overline{\fM}^d_{N, \Sig_k}$ equals $\overline{\fM}^d_N$.
   	\item[(2)] For any finite rational cone decomposition $\Sig$ of $C_d$, the morphism 
	$\overline{\cM}^d_{N,\Sig*\Sig_k} \to \overline{\cM}^d_{N, \Sig}$ is an isomorphism, where $*$ denotes the join. 
  	\item[(3)] Let $\Sig$ be a finite rational subdivision of $\Sig_k$. Then there are
	\begin{itemize}
		\item $A$-schemes $U_{\sig}$ for $\sig\in \Sig$, 
		\item an \'etale surjective morphism  $\coprod_{\sig\in \Sig} U_{\sig} \to \overline{\cM}^d_{N, \Sig}$ of $A$-schemes, 
		\item smooth morphisms $U_{\sig} \to \toric_{\F_p}(\sig)$ over $\F_p$ 
		such that the inverse image of $\cM^d_N$ in 
		$U_{\sig}$ coincides with the inverse image of the torus part $\mathbb{G}_{m,\F_p}^{d-1}$ of  
		$\toric_{\F_p}(\sig)$, and 
		\item \'etale morphisms $A[\frac{1}{N}]\otimes_A U_{\sig} \to \toric_{A[\frac{1}{N}]}(\sig)$ over $A[\frac{1}{N}]$ such that the 
		inverse image of $A[\frac{1}{N}]\otimes_A \cM^d_N$  in $A[\frac{1}{N}]\otimes_A U_{\sig}$ coincides with the inverse image 
		of the torus part $\mathbb{G}^{d-1}_{m,A[\frac{1}{N}]}$ of $\toric_{A[\frac{1}{N}]}(\sig)$. 
	\end{itemize}
\end{enumerate}

\end{sbthm}

In brief, \ref{shape}(3) says that the log scheme $\overline{\cM}^d_{N,\Sig}$ is \'etale locally isomorphic to $\toric_{\F_p}(\sig)\times \mathbb{G}_m$ ($\sig\in \Sig$), and the log $A[\frac{1}{N}]$-scheme $A[\frac{1}{N}]\otimes_A \overline{\cM}^d_{N,\Sig}$ is \'etale locally isomorphic to $\toric_{A[\frac{1}{N}]}(\sig)$ ($\sig\in \Sig$). 
 
See Proposition \ref{Sigk8} for part (1) of Theorem \ref{shape} and Section \ref{toroidal} for parts (2) and (3), specifically \ref{repnl2} and \ref{repnl3} for (2) and \ref{shapethm3} for (3). 

Part (3) of Theorem \ref{main} follows from part (3) of Theorem \ref{shape}, as explained in \ref{mainthm3}. 
The key point is that, as is well known in toric geometry, there is a finite subdivision $\Sig'$ of $\Sig$ such that for all $\sig\in \Sig'$, the toric variety $\toric_{\F_p}(\sig)$ is smooth over $\F_p$ and the complement of the torus part is a divisor with normal crossings.

The following theorem shows that log Drinfeld modules over an object $S$ of $\cC_{\log}$ are simply understood in the case  $S$ is as in \ref{log2}.  It will be proved in Section \ref{ss:tornl}.

    \begin{sbthm}\label{M=M0} Let $S$ be an object of $\cC_{\log}$ such that the underlying scheme of $S$ is normal and such that the log structure of $S$ is associated to a dense open subset $U$ of $S$ as in \ref{log2}. Then  the following two notions are equivalent:
\begin{enumerate}    
	\item[(i)] a log Drinfeld module $((\cL, \phi), \iota)$ over $S$ of rank $d$ with  level $N$ structure and 
    	\item[(ii)] a generalized Drinfeld module $((\cL, \phi), \iota)$ over $(S, U)$ of rank $d$ with level $N$ structure such that for each $a,b \in (\frac{1}{N}A/A)^d$, locally on $S$, either $\pole(\iota(a))\pole(\iota(b))^{-1}$ or $\pole(\iota(b))\pole(\iota(a))^{-1}$ is an element of $M_S/\cO_S^\times$ inside $M_S^{\gp}/\cO_S^\times$. 
\end{enumerate}
More specifically, we obtain (ii) from (i) by restricting $\iota$ of (i) to $U$, and we obtain (i) from (ii) by taking the unique extension of $\iota$ in (ii)  to a map to $\overline{\cL}$ (see \ref{log4}). 
    \end{sbthm}

\subsection{Comparison with abelian varieties}
\label{how_to}

We explain the similarities and differences between our constructions of toroidal compactifications and the constructions of toroidal compactifications of the moduli spaces of abelian varieties in \cite{AMRT, FC}.  In particular, we describe the Tate uniformizations of Drinfeld modules on adic spaces and a theory of iterated Tate uniformizations on formal schemes.

\begin{sbpara}\label{abdr}  The most difficult part of this paper is to prove that our toroidal compactification is log regular (that is, it has only toric singularities), so is worthy of the name.

Consider the compactification of the moduli space of elliptic curves with level $N$ structure. Roughly speaking, the formal completion of the compactified moduli space at the boundary is isomorphic to $\Spf(\Z[\zeta_N]\ps{q^{1/N}})$, and from this we can deduce that the compactified moduli space is regular and its tensor product with $\Z[\frac{1}{N}]$ is smooth over $\Z[\frac{1}{N}]$. Here, $q$ is the $q$-invariant which appears in the formal moduli theory.   The universal degenerate elliptic curve over  $\Spf(\Z[\zeta_N]\ps{q^{1/N}})$ is expressed as  the quotient of the multiplicative group by a $\Z$-lattice of rank $1$, which is the theory of Tate uniformization in (1) below, and $q$ is the $\Z$-basis of this $\Z$-lattice. 

Similarly, to prove the log regularity and the log smoothness of our toroidal compactification, we
consider the formal completion of the compactification at the boundary, which is a moduli space of degenerations of Drinfeld modules over formal schemes. 

\begin{enumerate}
\item[(1)] In the theory of degeneration of abelian varieties over formal schemes, the analytic presentation $X=Y/\La$  of an abelian variety $X$ with degeneration as a quotient of a semi-abelian variety $Y$ by a $\Z$-lattice $\La$ plays an important role. This construction is due to Tate, Raynaud, Mumford, and Faltings-Chai (see \cite{FC}). This $\La$ is useful in understanding the formal completion of the compactified moduli space. 

\item[(2)] In the theory of degeneration of Drinfeld modules over formal schemes contained in Section \ref{ss:Tate} of this paper, we will obtain the Tate uniformization $X=Y/\La$ of a Drinfeld module with degeneration as a quotient of a Drinfeld module $Y$ by a certain $\La$, generalizing the result  of Drinfeld over complete discrete valuation fields. Though the result of Drinfeld was useful to understand the compactified moduli space of Drinfeld modules of rank $2$, this generalized Tate uniformization is not so useful in general because this $\La$, which is a locally constant sheaf of $A$-lattices on an adic space, has too large of a monodromy action in the case of rank greater than $2$, which is hard to handle. Instead, we employ a theory of iterated Tate uniformizations, which expresses $X$ as 
an iterated quotient of $Y$. That is,  we obtain $X$ from $Y$ via $Y_0=Y$, 
 $Y_{i+1}=Y_i/\La_i$ for $0\leq i\leq n-1$ and $X=Y_n$, where the $Y_i$ are generalized Drinfeld modules and the $\La_i$ are $A$-lattices in $Y_i$ of rank $1$. These $\La_i$ are useful in understanding the compactified moduli space. See Sections \ref{itTa0} and \ref{itTa}. 
\end{enumerate}
 
 The difference between (1) and (2) is further discussed in \ref{4.2Tv2}. 
 \end{sbpara}
 
 \begin{sbpara}
The difference between the above (1) and (2) is related to the following difference in local monodromy over a valuation ring. 

\begin{itemize}
	\item Let $\cV$ be a strictly Henselian valuation ring (which need not be of height one) with field of fractions $K$. For an abelian variety over $K$,  the local monodromy satisfies $(\sig_1-1)(\sig_2-1)=0$ after taking a finite extension of $K$. 
On the other hand, for a Drinfeld module over $K$, the local monodromy satisfies $(\sig_1-1)(\sig_2-1)\dots (\sig_n-1)=0$ for some $n$ after a finite extension of $K$, but we cannot always take $n=2$. This is discussed in Section \ref{monodromy}.  
\end{itemize}

\end{sbpara}

\begin{sbpara}

 In the articles \cite{KKN}, the notion of a log abelian variety is defined, and it is shown that the toroidal compactifications of moduli spaces of abelian varieties  \cite{AMRT, FC} are understood as the moduli spaces of log abelian varieties. This notion is 
similar to the notion of a log Drinfeld module in this paper. For a log scheme $S$, a  log abelian variety $\cA$  over $S$ is contravariant functor from  the category of fs log schemes over $S$ to the category of abelian groups satisfying certain conditions and which has a subgroup functor represented by a semi-abelian scheme $\cG$ over $S$. The inclusion $\cG\subset \cA$ is similar to $\cL\subset \overline{\cL}$. One difference is that $\overline{\cL}$ does not have a group structure, though $\cA$ does.
 
 \end{sbpara}

\subsection{Plan of this paper, acknowledgements}

\begin{sbpara} The plan of this paper is as follows. Section \ref{Drin_mod} contains generalities on Drinfeld modules, generalized Drinfeld modules and log Drinfeld modules.  
In Section \ref{cone_decomp}, we discuss cone decompositions related to Bruhat-Tits buildings and Drinfeld exponential maps. In Section \ref{s:Tate}, we discuss Tate uniformizations, and we discuss formal moduli by using iterated Tate uniformizations.  In Section \ref{toroidal}, we prove our main theorem. 
  
 In Sections \ref{gen_Drin_mod}--\ref{BT} and Sections \ref{ss:constrquot}--\ref{logD_formal}, we consider general $A$. In the rest of the paper (that is, Section \ref{normord}, Section \ref{cone_decomp}, Sections \ref{iterated}--\ref{itTa}, and Section \ref{toroidal}), we restrict to the case $A=\F_q[T]$. 
 \end{sbpara}

\begin{sbpara} We plan to extend the results of Sections \ref{cone_decomp}--\ref{toroidal}  to general $A$ in a sequel (Part II) of this paper.
The reasons why we need two papers, Part I (this paper) and Part II, are as follows. First, in the theory for general $A$, we use reduction to the case  $A=\F_q[T]$ in several key points.  Hence, it is better to treat the case $A=\F_q[T]$ first. Second, the case $A=\F_q[T]$ is especially simple and the ideas can be seen clearly. (For general $A$, the description of the theory becomes very involved in parts, and the ideas are hidden behind complicated definitions.) 
 
 \end{sbpara}

\begin{sbpara}  In this paper, we use ideas of Richard Pink and Kazuhiro Fujiwara described in the paper \cite{P1} of Pink. 
In \cite{P1}, Pink uses torsion points of Drinfeld modules for his theory of toroidal compactifications of moduli spaces of Drinfeld modules. Our use of torsion points in Section \ref{main_result} above is similar to it. 
The importance of the iterated Tate uniformization appears in Pink \cite{P1}  at the place where he introduces  the study of Fujiwara. The log regularity and the log smoothness of the toroidal compactifications are not discussed in \cite{P1}. 

Though details are not written in \cite{P1}, we believe that  the ideas of Pink and Fujiwara in \cite{P1} were enough to have the theory of toroidal compactifications.

\end{sbpara}

\begin{sbpara} The second author wishes to express his thanks to Sampei Usui, Chikara Nakayama, and Takeshi Kajiwara  that the joint works with them on partial toroidal compactifications of period domains and on log abelian varieties (\cite{KU}, \cite{KNU}, \cite{KKN}) were very helpful to find the right ways in this work. He also thanks Kazuhiro Fujiwara for advice.

The work of the first two authors was supported in part by the National
Science Foundation under Grant Nos.~DMS-1303421 and DMS-1601861.  The work of the third author was supported in part by the National Science Foundation under Grant Nos.~DMS-1801963 and DMS-2101889.

\end{sbpara}

\section{Generalized and log Drinfeld modules} \label{Drin_mod}

Let $F$ be a function field in one variable with total constant field $\F_q$. Fix a place $\infty$ of $F$. Let $A$ be the ring of all elements of $F$ which are integral outside $\infty$. Let $F_{\infty}$ be the completion of $F$ at $\infty$, and let $|\;\,|$ be the standard absolute value on $F_{\infty}$. 

\subsection{Basic things} \label{gen_Drin_mod}

We review the definitions of a Drinfeld module \cite{D} and of a generalized Drinfeld module \cite{P2}. 

Let $S$ denote a scheme over $A$.

\begin{sbpara}\label{110}

A \emph{generalized Drinfeld module} over $S$  is a pair $(\cL, \phi)$, where $\cL$ is a line bundle over $S$ and $\phi$ is an $\F_q$-linear action of the ring $A$ on the (additive) group scheme $\cL$ over $S$ satisfying the following conditions. Locally on $S$, take a basis $\delta$ of $\cL$, and write 
$$
	\phi(a)(z\delta)= \sum_{n=0}^{\infty} c(\delta,a,n)z^n\delta
$$ 
($a\in A$,  $z \in \cO_S$,  $c(\delta,a, n) \in \cO_S$ independent of $z$ and zero for sufficiently large $n$). Then
\begin{enumerate}
\item[(i)] $c(\delta,a,1)=a$ for all $a \in A$, and
\item[(ii)] for each $s\in S$, there are $a \in A$ and $n \ge 2$ such that $c(\delta,a,n) \in \cO_{S,s}^{\times}$.
\end{enumerate}

A generalized Drinfeld module $(\cL,\phi)$ will frequently be written simply as $\phi$.  Its local coefficients $c(\delta, a, n)$ for a given $\delta$ will also be written $c(a,n)$.  By definition, we have 
$$
	c(u\delta, a,n)= c(\delta, a, n)u^{n-1}
$$ 
for $u \in \cO_S^{\times}$.  If $\cO_S\cong \cL$, we will often identify $\cL$ with $\cO_S$ by a choice of $\delta$ and omit the notation for $\delta$ in the formula for $\phi$.

A homomorphism of generalized Drinfeld modules is a homomorphism of line bundles which is compatible with the given actions of $A$. 

\end{sbpara}

\begin{sbpara}\label{111} 

For a generalized Drinfeld module $\phi$ over $S$, we have $\phi(a)(z)=az$ for all $a \in \F_q$.  From this, it follows that $c(a, n)=0$ for all $a \in A$ unless $n$ is a power of $q$. That is, locally we have
$$
	\phi(a)(z) = \sum_{i\geq 0}^{\infty} a_i\tau^i (z)
$$ 
for some $a_i \in \cO_S$, where $\tau$ is the homomorphism $z\mapsto z^q$. 

\end{sbpara} 

\begin{sbpara}\label{112} 

For a generalized Drinfeld module $\phi$ over $S$ and $s\in S$, there is an integer $r(s)\geq 1$ such that for each $a\in A\setminus \{0\}$, the coefficient $c(a, |a|^{r(s)})$ is invertible at $s$ and $c(a,n)$ is not invertible at $s$ if $n>|a|^{r(s)}$.  This is quickly reduced to the case that $S=\Spec(k)$ for a field $k$ over $A$ and \cite[Prop.~2.1]{D}. 
That $r(s)$ is an integer is \cite[Cor. to Prop.~2.2]{D} (see also \cite[Prop.~4.5.3]{Go}).

\end{sbpara}

\begin{sbpara} A generalized Drinfeld module over $S$ is called a \emph{Drinfeld module} over $S$ of rank $d$ if it satisfies
\begin{enumerate}
\item[(i)] $r(s)=d$ for all $s\in S$, where $r(s)$ is as in \ref{112}, and
\item[(ii)] $c(a,n)=0$ if $a\in A$ and $n>|a|^d$. 
\end{enumerate}

\end{sbpara}

\begin{sbrem} The definition of an elliptic module over $S$ of rank $d$ and the definition of a standard elliptic module over $S$ of rank $d$ are given in (B) of \cite[Section 5]{D}. 
Actually, the above definition of a Drinfeld module over $S$ of rank $d$ is the same as the definition of a standard elliptic module over $S$ of rank $d$. As is explained in \cite[Section 5]{D}, the category of Drinfeld modules (i.e., standard elliptic modules) over $S$ of rank $d$ is equivalent to the category of elliptic modules over $S$ of rank $d$. 

\end{sbrem}

\begin{sbpara}
In the case that $S$ is an integral scheme, by a generalized Drinfeld module over $S$ of \emph{generic rank} $d$ we mean a generalized Drinfeld module over $S$ whose restriction to the generic point of $S$ is a Drinfeld module of rank $d$. 

Note that for a generalized Drinfeld module $\phi$ over $S$ of generic rank $d$, we have $r(s)\leq d$ for every $s\in S$, and $\phi$ is a Drinfeld module if and only if $r(s)=d$ for all $s\in S$.

\end{sbpara}

\subsection{Stable reduction theorem} \label{stable_red}

The following proposition and its corollary are  generalizations of \cite[Prop.~7.1]{D}. The proofs are essentially the same as those given in  \cite[Section 7]{D}. 

Let $S$ be an integral scheme over $A$, and let $K$ be its function field.  We require the following lemma.

\begin{sblem}\label{stlem} Suppose that $S$ is normal.  Let $(\cL, \phi)$ and $(\cL', \phi')$ be generalized Drinfeld modules over $S$, and let $f \colon (\cL, \phi)_K \xrightarrow{\sim} (\cL', \phi')_K$ be an isomorphism  between their pullbacks to $\Spec K$. Then $f$ induces an isomorphism  $(\cL, \phi) \xrightarrow{\sim} (\cL', \phi')$. 

\end{sblem}

\begin{pf} 
	See \cite[Prop.~3.7]{P2}, in which more generally, homomorphisms (not just isomorphisms) are treated.  
\end{pf}

\begin{sbprop}\label{stable}
	Let $U$ be a dense open subset of $S$, and let $\phi$ be a generalized Drinfeld module over $U$. There exist a proper birational 
	morphism $T\to S$ and a finite separable extension $K'$ of $K$ such that if $T'$ denotes the integral closure of $T$ in $K'$ and 
	$U'$ denotes the inverse image of $U$ in $T'$, then the pullback of $\phi$ to $U'$ extends uniquely to a generalized Drinfeld module 
	over $T'$. 
\end{sbprop}

\begin{proof} 
Take a finite set of $\F_q$-algebra generators $y_i$ of $A$ with $1 \le i \le k$. Fix a linear basis element $\delta$ of the Drinfeld module over $K$, and define $c(y_i, n)\in K$ using $\delta$. Let $Y$ be the finite set of all pairs $(i, n)$ of integers with $1 \le i \le k$ and $n \ge 2$ such that $c(y_i, n) \neq 0$. Let $K'$ be the finite extension of $K$ obtained by adjoining all $(n-1)$th roots of $c(y_i, n)$ for all $(i,n)\in Y$. Since for every $(i,n)\in Y$, the integer $n$ is a power of $q$ by \ref{111}, the field $K'$ is separable over $K$. 

Take an integer $m\geq 1$ which satisfies $\frac{m}{n-1} \in \Z$ for all $(i, n)\in Y$, and consider the fractional $\cO_S$-ideal $I$ on $S$  generated by the $c(y_i, n)^{m/(n-1)}$. Let $T\to S$ be the blow-up of $S$ by $I$, and let $T'$ be the integral closure of $T$ in $K'$. For $(i,n) \in Y$, let $U(i, n)$ be the open part of $T$ on which $c(y_i, n)^{m/(n-1)}$ divides $c(y_j, h)^{m/(h-1)}$ for all $(j, h) \in Y$. Then the $U(i, n)$ form an open covering of $T$. Let $U'(i,n)$ be the inverse image of $U(i,n)$ in $T'$.

Let $\cL'$ be the line bundle on $T'$ for which $\delta_{i,n} := c(y_i, n)^{-1/(n-1)}\delta$ is a basis on $U'(i,n)$. 
For $x\in \mathcal{O}_{U'(i,n)}$, we set
$$
	\phi'(y_j)(x \delta_{i,n})= y_jx\delta_{i,n} + \sum_h x^h \left(\frac{c(y_j, h)^{1/(h-1)}}{c(y_i, n)^{1/(n-1)}}\right)^{h-1} \delta_{i,n},
$$
where $h$ ranges over all integers such that $(j, h)\in Y$.  Since 
$$
	\frac{c(y_j, h)^{1/(h-1)}}{c(y_i, n)^{1/(n-1)}} \in \cO_{U'(i,n)}
$$ 
by definition, $(\cL', \phi')$ is a generalized Drinfeld module over $T'$.  Over $K'$, this Drinfeld module
$\phi'$ is the pullback of $\phi$ on $K$.  Since $T'$ is normal, Lemma \ref{stlem} tells us that $\phi'$ is the unique extension to $T'$ of the pullback
of $\phi$ to $U'$.
\end{proof}

Taking $S$ to be the spectrum of a valuation ring in Proposition \ref{stable}, we have the
following corollary (as the valuative criterion for properness forces $T = S$).

\begin{sbcor}\label{2.2.2}  Let $\cV$ be a valuation ring with field of fractions $K$, and let $\phi$ be a Drinfeld module over $K$. Then there is a finite separable extension $K'$ of $K$ such that the pullback of $\phi$ to $K'$  comes from a generalized Drinfeld module over the integral closure of $\cV$ in $K'$. 

\end{sbcor}

\subsection{Generalized Drinfeld modules over complete rings} \label{compl_ring}

We give complements to \cite[Prop.~7.2]{D}, which characterizes generalized Drinfeld modules over a complete discrete valuation ring in terms of pairs consisting of a Drinfeld module and a lattice. Proposition \ref{val2} generalizes this result to allow a complete valuation ring of height one. This generalization becomes important in Section \ref{monodromy}. Proposition \ref{prop51} generalizes a part of the result even further, to commutative rings which are $I$-adically complete for an ideal $I$.  The proofs of \ref{prop51} and \ref{val2} are essentially the same as the arguments given in \cite[Section 7]{D}.

\begin{sbprop}\label{prop51} Let $R$ be a commutative ring over $A$, and let $I$ be an ideal of $R$ such that $R \xrightarrow{\sim} \varprojlim_i R/I^i$.  Let $\phi$ be a generalized Drinfeld module over $R$ with trivial line bundle, and suppose it reduces to a Drinfeld module  over $R/I$ of rank $r$. Then there is a unique pair $(\psi, e)$, where $\psi$ is a Drinfeld module over $R$ of rank $r$ with trivial line bundle
and $\psi \equiv \phi \bmod I$, and where $e\in R\ps{z}$ is such that $e \equiv z \bmod I$, the $z^n$-coefficient of $e$ is zero unless $n$ is a power of $q$, the reduction modulo $I^n$ of $e$ is a polynomial for every $n$, and 
$e\circ \psi(a)= \phi(a) \circ e$ for all $a\in A$.
\end{sbprop}

\begin{proof}
As in the first five lines of part (2) of the proof of \cite[Prop.~7.2]{D},
the result follows from \cite[Prop.~5.2]{D} (taking $B = R/I^n$ therein).  More details are given in the argument of \cite[Prop.~3.4]{P2}.  Very roughly, \cite[Prop.~5.2]{D} is used to find the existence of a unique $e$ and $\psi(b)$ such that $e \circ \psi(b) = \phi(b) \circ e$, given a nonconstant $b \in A$.  By the commutativity of $e\phi(a)e^{-1}$ and $e\phi(b)e^{-1}$ for every $a \in A$, we can use \cite[Prop.~5.1]{D} to see that the same $e$ 
works for $\phi(a)$.  
\end{proof}

\begin{sbrem}\label{prop51re} There is a natural generalization of \ref{prop51} to the situation where we do not assume that the line bundle $\cL$ of $\phi$ is trivial. We have still a correspondence  $\phi\mapsto (\psi, e)$ where the line bundle of $\psi$ is $\cL$ and $e= \sum_{i=0}^{\infty} c_i \tau_i$ where $c_i \in R$ and $\tau_i$ is a semi-linear map $\cL\to \cL$ with respect to the ring homomorphism  $R\to R$ given by $x \mapsto x^{q^i}$.

\end{sbrem} 

 In the rest of this subsection, Drinfeld (resp.,  generalized Drinfeld) modules are assumed to have trivialized line bundles. 
\begin{sbpara}\label{val1} 
 
Let $\cV$ be a valuation ring over $A$ with maximal ideal $\fM$, and let $K$ be the field of fractions of $\cV$.  Fix an algebraic closure $\bar K$ of $K$, let $K^{\sep}\subset \bar K$ be the separable closure of $K$, and set $G_K = \Gal(K^{\sep}/K)$. Let $v_K$ be an additive valuation of $K$ associated to $\cV$,
and let $v_{\bar K}$ denote its unique extension to an additive valuation on $\bar K$. 
We suppose that
\begin{itemize}
	\item $\cV$ is of height one, i.e., the value group of $v_K$ is isomorphic as an ordered 
	group to a subgroup of $\R$, and
	\item $\cV$ is complete, i.e., for every nonzero $a\in \fM$, the canonical map $\cV \to \varprojlim_i \cV/a^i\cV$
	is an isomorphism.
\end{itemize}

For a Drinfeld module $\psi$ over $K$, by a $\psi(A)$-lattice in $K^{\sep}$, we mean a finitely generated, projective, $G_K$-stable $\psi(A)$-submodule $\Lambda$ of $K^{\sep}$ such that $\{\la\in \La \mid v_{\bar K}(\la)\geq c\}$ is finite for all $c \in \R$.

We say that a Drinfeld module $\psi$ over $K$ has potentially good reduction if it comes from a Drinfeld module over the integral closure of $\cV$ in some finite extension of $K$.
\end{sbpara}

The following lemma will be useful later.

\begin{sblem}\label{pre0}  
Let the notation be as in \ref{val1}.
Let $\psi$ be a Drinfeld module over $\cV$ of rank $r$, and let $\Lambda$ be a $\psi(A)$-lattice in $K^{\sep}$.
  Then  $v_{\bar K}(\la) <0$ for all $\la \in \La \setminus \{0\}$. 

\end{sblem}

\begin{proof} If $\la\in \La$ is nonzero and $v_{\bar K}(\lam)\geq 0$, then  $v_{\bar K}(\psi(a)\la) \geq 0$ for all $a\in A$, and this contradicts the finiteness of 
$\{\lam\in \La \mid v_{\bar K}(\la)\geq 0\}$.  
\end{proof}

 \begin{sbprop}\label{val2}  Let the notation be as in \ref{val1}.   
  
\begin{enumerate}
	\item[(1)] Generalized Drinfeld modules $\phi$ of generic rank $d$ over $\cV$ are in one-to-one correspondence with  pairs $(\psi, \La)$ consisting of a Drinfeld module $\psi$ over $\cV$ of rank $r\leq d$ and a $\psi(A)$-lattice $\La$ in $K^{\sep}$ of rank $d-r$.
	\item[(2)] Drinfeld modules of rank $d$ over $K$ are in one-to-one correspondence with pairs $(\psi, \La)$ consisting
	of a Drinfeld module $\psi$ over $K$ of rank $r\leq d$ with potentially good reduction and a $\psi(A)$-lattice 
	$\La$  in $K^{\sep}$  of rank $d-r$.
\end{enumerate} 

 \end{sbprop}
 \begin{proof}
	The proof is similar to that of \cite[Prop.~7.2]{D}, so we omit some details. 
	By the semistable reduction theorem \ref{2.2.2}, part (2) is reduced to part (1). We prove part (1). 
	
	The direction $\phi\mapsto (\psi, \La)$ is as follows. 
	 Given a generalized Drinfeld module
	$\phi$ over $\cV$ of generic rank $d$, its reduction modulo $\fM$ is a Drinfeld module of some rank $r \le d$. Hence for some nonzero
	element $t$ in $\fM$, the reduction of $\phi$ modulo $t$ is also a Drinfeld module of rank $r$.  Moreover, as $\cV$ 
	is complete for the ideal generated by $t$, Proposition \ref{prop51} and Remark \ref{prop51re} provide a unique pair 
	$(\psi, e)$ with $\psi$ a
	Drinfeld module over $\cV$ of rank $r$ and
	with $e \in \cV\ps{z}$ such that 
	$e \equiv z \bmod t$ and
	$e \circ \psi(a) = \phi(a) \circ e$ for all $a \in A$, and the coefficient of $z^n$ in $e$ is $0$ unless $n$ is a power of $q$.  
	By the method of \cite[Section 3]{D}, we see that the Taylor series $e$ converges everywhere. 
		
	Since $e \in \cV\ps{z}$ and $e\equiv z \bmod \frak M$, its kernel $\La$ in $K^{\sep}$ is $G_K$-stable with nonzero elements having 
	negative valuation.  Moreover, for $N \in A$ with nonzero image in $\cV$, the map $e$ carries $\psi^{-1}(N)\La/\La$ 
	(i.e., the kernel of $\psi(N)$ on $K^{\sep}/\La$)
	isomorphically onto the finite $A$-module $\ker \phi(N)\cong (A/NA)^d$, and $\La$ contains no nonzero roots of 
	$\psi(N) \in \cV[z]$ because the roots are integral over $\cV$.  We then have a commutative diagram with exact rows
	and columns
	$$
		\SelectTips{cm}{} \xymatrix@R=20pt{
			&& 0 \ar[d] & 0 \ar[d] \\
			& 0 \ar[d] & (A/NA)^r \ar[d] & (A/NA)^d \ar[d] \\
			0 \ar[r] & \Lambda \ar[d]^{\psi(N)} \ar[r] & K^{\sep} \ar[d]^{\psi(N)} \ar[r]^e & K^{\sep} \ar[d]^{\phi(N)} \ar[r] & 0 \\
			0 \ar[r] & \Lambda \ar[r] \ar[d] & K^{\sep} \ar[d] \ar[r]^e & K^{\sep} \ar[r] \ar[d] & 0 \\
			& \Lambda/\psi(N)\Lambda \ar[d] & 0 & 0 \\
			& 0. \\
		}
	$$
	The snake lemma then tells us that $\La/\psi(N)\La\cong (A/NA)^{d-r}$. 
	
	By the general theory of zeros of nonarchimedean functions proven 
	in \cite[Prop.~4]{ML} (see also \cite[Prop.~2.1]{Go}), the set $\{\la \in \La \mid v_{\bar K}(\la)\geq c\}$ is finite for every $c\in \R$. 
	The finiteness of $\La/\psi(N)\La$ and 
	$\{\la \in \La \mid v_{\bar K}(\la)\geq c\}$ imply that $\La$ is a finitely generated $\psi(A)$-module. (This can be seen by imitating the 
	classical proof of the fact that if $L$ is a $\Z$-module with a norm $|\cdot|$ such that  $L/aL$ is finite for some integer $a\geq 2$ and such that $\{x\in L \mid |x|\leq c\}$ is finite for every $c$, then $L$ is a finitely generated $\Z$-module, an important fact used in 
	the proof Mordell-Weil theorem on abelian varieties.)  Hence $\La$ is a lattice of rank $d-r$.
	
	The inverse map 
	$(\psi,\La)\mapsto \phi$ is given as follows. Let $e(z)= z\prod_{\la\in \La\setminus \{0\}} (1-\la^{-1}z)$. By the method of \cite[Section 3]{D}, we see that for each $a\in A$, there is a polynomial $\phi(a)$ over $K$ of degree $|a|^d$ such that $e\circ \psi(a)=\phi(a) \circ e$. By \ref{pre0},  we have that $e\in \cV\ps{z}$ and $e\equiv z\bmod \frak M$. From this, it follows that $\phi(a)$ has coefficients in $\cV[z]$. By the method of \cite[Section 3]{D}, we may conclude that $\phi$ is a generalized Drinfeld module over $\cV$ with generic rank $d$.
\end{proof}

\subsection{Local monodromy over valuation rings} \label{monodromy}

\begin{sbpara}
Let $\cV$ be a strictly Henselian valuation ring over $A$ with maximal ideal $\fM$ and field of fractions  $K$. 
Let $\phi$ be a generalized Drinfeld module over $\cV$ with generic rank $d$.

For a place $v\neq \infty$ of $F$ 
that does not correspond to the kernel of $A\to \cV/\fM$, let $F_v$ denote the completion of $F$ at $v$ and $O_v$ its valuation ring.  
Consider the $v$-adic Tate module 
$$
	T_v(\phi) =\text{Hom}_A(F_v/O_v, \phi\{v\})
$$ 
of $\phi$, where $\phi\{v\}$ is the $v$-primary torsion of $\phi$ in a separable closure $K^{\sep}$. Then $T_v(\phi)$ is a free $O_v$-module of rank $d$ with a continuous  action of $G_K$. Let $V_v(\phi)= F_v\otimes_{O_v} T_v(\phi)$. 
    
For each prime ideal $\frak p$ of $\cV$, let $r(\frak p)$ denote the rank of the Drinfeld module over the residue field $\cV_{\frak p}/{\frak p}\cV_{\frak p}$ of $\frak p$ obtained from $\phi$. Let 
$$
	0 = d(-1) < d(0) <\dots <d(m)=d
$$ 
be such that 
$$
	\{ d(i) \mid 0 \le i \le m\} = \{ r(\frak p) \mid \frak p \in \Spec(\cV) \}.
$$
  
\end{sbpara}

\begin{sbthm}\label{thmlm}   
The action of $G_K$ on $V_v(\phi)$ is quasi-unipotent. More precisely, the following statements hold.
\begin{enumerate}
\item[(1)] 
For a sufficiently small open subgroup $H$ of $G_K$, there are $G_K$-stable $F_v$-subspaces $V_i$ of $V_v(\phi)$ of dimension $d(i)$ 
for $-1 \le i \le m$ such that 
$$
	0=V_{-1}\subset V_0\subset \dots \subset V_m=V_v(\phi)
$$
and $H$ acts trivially on the graded subquotients $V_i/V_{i-1}$ with $0\leq i\leq m$.
\item[(2)] Assume furthermore that the value group $\Gamma$ of $\cV$ has the property that $\Gamma \otimes_{\Z} \Z_{(p)}$ for $p=\mathrm{char}(\F_q)$ is  a finitely generated $\Z_{(p)}$-module.  Then $V_{i-1} = I_H V_i$ for all $0 \le i \le m$, where $I_H$ denotes the augmentation ideal in the group ring $F_v[H]$.
\end{enumerate}
  \end{sbthm}

\begin{proof} We fix a trivialization of the line bundle of $\phi$.  For each positive integer $i \le m$,
let $\frak P_i$ be the union of all prime ideals $\frak p$ of $\cV$ such that $r(\frak p)=d(i)$, and let $\frak Q_i$ be the intersection of all prime ideals $\frak q$ of $\cV$ such that $r(\frak q)=d(i-1)$. Since the set of ideals of a valuation ring is totally ordered, $\frak P_i$ and $\frak Q_i$ are prime ideals of $\cV$ with $r(\frak P_i)=d(i)$, $r(\frak Q_i)= d(i-1)$, and $\frak P_i\subsetneq \frak Q_i$. 
The image of the local ring $\cV_{\frak Q_i}$ in the residue field $\kappa_i$ of $\frak P_i$ is a valuation ring of height one as it has a unique nonzero
prime ideal. Let $\cV_i$ be the completion of this image, and let $\hat \kappa_i$ be its field of fractions.  Let $\phi_i$ be the generalized Drinfeld module over $\cV_i$ induced by $\phi$. Let $(\psi_i, \La_i)$ be the corresponding pair as in \ref{val2}(1).  We also let $\cV_0 = \cV$, $\phi_0 = \phi$, and $\frak P_0 = \fM$.  

For each nonnegative integer $i \le m$, let $\phi'_i$ be the Drinfeld module of rank $d(i)$ over $\kappa_i$ given by $\phi_i$.
We have a $G_K$-stable $F_v$-subspace $V_i$  of $V_v(\phi)$ as follows. 
Let $\phi\{v\}_i$ be the subgroup of $\phi\{v\}$ consisting of all elements which belong to the integral closure of the local ring $\cV_{{\frak P}_i}$ in $K^{\sep}$. 
 Let 
 $$
 	T_i=\Hom_A(F_v/O_v, \phi\{v\}_i)
$$ 
and $V_i=F_v\otimes_{O_v} T_i$. The map $\phi\{v\}_i \to \phi'_i\{v\}$ is an isomorphism by assumption on $v$, and hence we have isomorphisms $T_i \xrightarrow{\sim} T_v(\phi'_i)$ and  $V_i \xrightarrow{\sim} V_v(\phi'_i)$. It follows that $\dim_{F_v}(V_i)=d(i)$.  
 
The action of $G_K$ on $T_i$ factors through the absolute Galois group $G_{\kappa_i}$ of $\kappa_i$. 
For each positive integer $i  \le m$, we have an exact sequence 
$$
	0\to T_{i-1}\to T_i\to O_v \otimes_A \La_i\to 0
$$
of $G_{\hat \kappa_i}$-modules. The action of $G_{\hat \kappa_i}$ on $T_i/T_{i-1}$ factors through a finite quotient group of $G_{\hat \kappa_i}$ because $\La_i$ is finitely generated over $A$ and  each element of $\La_i$ is fixed by some open subgroup of $G_{\hat \kappa_i}$.
Since $\cV$ is Henselian, the local ring $\cV_{\frak Q_i}$  and its image in $\kappa_i$ are also Henselian. The map $G_{\hat \kappa_i} \to G_{\kappa_i}$ is then an isomorphism (this follows from \cite[Prop.~9.1.16]{GaRo}, and is also a special case of \cite[Thm.~5.1.3]{Fu}).
  This shows that the action of $G_{\kappa_i}$ on $V_i/V_{i-1}$ factors through a finite quotient. 
Hence an open subgroup of $G_K$ acts trivially on $V_i/V_{i-1}$. This proves (1) and therefore the quasi-unipotence of the $G_K$-action.

    We prove (2). 
   By replacing $\phi$ by $\phi_i$ ($1\leq i\leq m$) and replacing $\cV$ by $\cV_i$, we may assume that $\cV$ is of height one and complete, $\frak P_i$ is the zero ideal, $\frak Q_i$ is the maximal ideal of $\cV$, and $i=1$. By replacing $\cV$ by its integral closure in some finite separable extension of $K$, we may assume that $G_K$ acts on $\La_1$ trivially.  Write $\La_1$ as $\La$ and $\psi_1$ as $\psi$.

  Take an element $f$ of $A$ such that $(f)= v^a$ for some $a\geq 1$. We identify $T_v(\phi)$ with $\varprojlim_n \phi[f^n]$, where $\phi[f^{n+1}]\to \phi[f^n]$ is  $\phi(f)$. 
Let $\Omega$ be the set of all families $(x_j)_{j\geq 0}$ such that $x_j\in K^{\sep}$, $x_0\in \La$, and $\psi(f)x_{j+1}=x_j$ for all $j\geq 0$. Then we have a commutative diagram of exact sequences
$$
	\SelectTips{cm}{} \xymatrix{  0 \ar[r] & T_0 \ar[r] \ar@{=}[d] & \Omega \ar[r] \ar[d]^g & \La \ar[r] \ar[d] & 0\\
	0 \ar[r] & T_0 \ar[r] & T_v(\phi) \ar[r]   & O_v\otimes_A \La \ar[r] & 0, }
$$
where $g((x_j)_j) = (e(x_j))_j$ for the exponential map $e \in \cV\ps{z}$ which satisfies $e \circ \psi(a) = \phi(a) \circ e$ for all $a \in A$.

Take $\lam \in \La\setminus \{0\}$ and take $\tilde \lam=(\lam_j)_{j\geq 0} \in \Omega$ such that $\lam_0=\lam$.  Regard $\tilde \lam$ as an element of $T_v(\phi)$ by the above commutative diagram.
Let $c$ be the order of the torsion part of $(\Z_{(p)}\otimes_{\Z} \Gamma)/ \Z_{(p)}v_K(\lambda)$. 
As $v_K(\lambda) < 0$ by Lemma \ref{pre0}, we have
$$
	\lam= \psi(f^j)\lam_j= u \lam_j^{|f|^{d(0)j}},
$$	 
where $u$ 
is a unit of $\cV$.  The ramification index of $K(\lambda_j)/K$ is then $\geq |f|^{d(0)j}c^{-1}$, so 
$$
	[K(\lambda_j):K]\geq |f|^{d(0)j}c^{-1}.
$$ 
Since the degree of the latter extension is not bigger than the number of distinct conjugates of $\lambda_j$, the image of the function $G_K\to T_0/\phi(f^j)T_0$ given by $g\mapsto (g-1)\lambda_j$ 
has order at least $|f|^{d(0)j}c^{-1}$. Since $T_0/\phi(f^j)T_0$ has order $|f|^{d(0)j}$, 
this image has index $\leq c$. Since $j$ is arbitrary, the image of the map $G_K\to T_0$ given by $g\mapsto (g-1)\tilde \lambda$ generates an $O_v$-submodule of $T_0$ also of index $\leq c$.  In particular, the image of the latter map generates $V_0$ over $F_v$, proving (2).
\end{proof}

We consider a result for abelian varieties corresponding to Theorem \ref{thmlm}. It should be very well-known, but we provide a proof. 

\begin{sbprop}\label{2.5.2}  Let $\cV$ be a strictly Henselian valuation ring with maximal ideal $\mathfrak{M}$ and field of fractions $K$. Let $B$ be an abelian variety over $K$, and let $\ell$ be a prime number  which does not coincide with the characteristic of $\cV/\mathfrak{M}$. Then 
there is an open subgroup $H$ of $G_K$  such that $(\sig_1-1) (\sig_2-1)=0$ on $T_{\ell}(B)$ for all $\sig_1, \sig_2\in H$. 

\end{sbprop}

\begin{proof} If we replace $K$ by a finite separable extension to give a suitable level structure on $B$, 
we have a morphism from $\Spec(K)$ to the moduli space of polarized abelian varieties associated to $B$. By the valuative criterion, this morphism extends to a morphism from 
$\Spec(\cV)$ to a toroidal compactification $\cT$ of the moduli space, as constructed in \cite[Ch.~IV]{FC}. Let $x\in \cT$ be the image of the closed point of $\Spec(\cV)$, and let $Q$ denote the field of fractions of the completion $\hat \cO_{\cT, x}$ of the local ring $\cO_{\cT,x}$.

The abelian variety $X$ over $Q$ arising from the universal abelian variety over the moduli space with the relevant level structure is obtained by the Mumford construction of \cite[Ch.~III]{FC}, as described in \cite[Ch.~VI, Section 1]{FC}.
That is to say, $X$ is analytically understood as the quotient $G/\La$ of a semi-abelian scheme $G$ over $R$ by a finite rank $\Z$-lattice $\La$ in $G(Q)$. We then have an exact sequence 
$$
	0\to T_{\ell}G \to T_{\ell}X \to \Z_{\ell} \otimes_{\Z} \La \to 0
$$ 
of $\Gal(Q^{\sep}/Q)$-modules. The absolute Galois group of the field of fractions of the strict henselization of $\hat \cO_{\cT,x}$ acts trivially on both $T_{\ell}G$ and $\Z_{\ell}\otimes_{\Z} \La$, and hence we have $(\sig_1-1)(\sig_2-1)=0$  on $T_{\ell}X$ for every $\sig_1, \sig_2$ in this Galois group. The same is then true for action of the absolute Galois group of the field of the fractions of the strict henselization of $\cO_{\cT,x}$. Consequently, the same is true for $G_K$ acting on $T_{\ell}B$. 
\end{proof}

\begin{sbpara}\label{valex} From \ref{thmlm}(2), we see that the analogue of \ref{2.5.2} for Drinfeld modules is not true. For example, 
take $A=\F_q[T]$, let $k$ be a separably closed field over $A$, and let 
$$
	\cV\subset k\ls{t_1}\ps{t_2}
$$ 
be the set of all elements whose images in $k\ls{t_1}$ belong to $k\ps{t_1}$. Then $\cV$ is a strictly Henselian valuation ring and the value group of $\cV$ is isomorphic to $\Z^2$ with the lexicographic order. Consider the generalized Drinfeld module over $\cV$ defined by 
$$
	\phi(T)(z)= Tz + z^q + t_1z^{q^2}+ t_2z^{q^3}.
$$
Let $\frak p_1=(t_1)\supsetneq \frak p_2=(t_2) \supsetneq \frak p_3=0$ be the prime ideals of $\cV$.
Then the Drinfeld module over the residue field of $\frak p_i$ induced by $\phi$ has rank $i$ 
for $i=1,2,3$. Hence \ref{thmlm}(2) shows that if $v$ is a place of $A$ such that $v\neq \infty$ and such that $A\to k$ does not factor through the residue field of $v$, there is no open subgroup $H$ of $G_K$ such that  $(\sig_1-1)(\sig_2-1)=0$ on $V_v(\phi)$ for all $\sig_1, \sig_2\in H$. 
\end{sbpara}

\subsection{Log geometry and toric geometry}\label{tolo}

\begin{sbpara}\label{sat} 
Let $\cS$ be a site with a sheaf of commutative rings $\cO$. A \emph{log structure} $M$ on $\cS$ is a sheaf of commutative monoids 
 endowed with 
a homomorphism $\alpha\colon M \to \cO$  of sheaves of monoids for the multiplicative structure of $\cO$ such that the map $\alpha^{-1}(\cO^\times)\overset{\alpha}\to \cO^\times$   is an isomorphism. Hence, $\cO^\times$ is regarded as a subgroup sheaf of $M$ via $\alpha$. We say the log structure is trivial if $M=\cO^\times$. The semigroup law of the log structure $M$ is written multiplicatively. 

By a \emph{saturated monoid}, we mean a commutative monoid $P$ which satisfies the following two conditions.
\begin{enumerate}
	\item[(i)]  In $P$, if $ab=ac$, then $b=c$. That is, for the group completion  $P^{\gp}= \{ab^{-1}\mid a, b\in P\}$ of $P$, the canonical homomorphism  $P\to P^{\gp}$ is injective.
	\item[(ii)] If $a\in P^{\gp}$ and if $a^n\in P\subset P^{\gp}$ for some $n\geq 1$, then $a\in P$.
\end{enumerate}
A finitely generated saturated monoid is called an \emph{fs monoid}. Recall that a monoid is called \emph{sharp} if $1$ is its only invertible element.

We say that a log structure $M$ is \emph{saturated} if it is a sheaf of saturated monoids.
We call $M$ an \emph{fs log structure} if locally on $\cS$, there is an fs monoid  $P$ and a homomorphism $\alpha \colon P\to \cO$ (called a chart of $M$) such that $M$ is isomorphic to the pushout of $P \leftarrow \alpha^{-1}(\cO^\times)\to \cO^\times$ in the category of sheaves of commutative monoids on $\cS$, which is  endowed with the natural homomorphism to $\cO$. An fs log structure is saturated.

Let $M$ be a saturated log structure on $\cS$, and let $\cL$ be an $\cO$-module on $\cS$ which is locally free of rank $1$.  Then we define a sheaf $\overline{\cL}$ on $\cS$ by
$$
	\overline{\cL} = \cL\cup_{\cL^\times} (M^{-1}\times^{\cO^\times} \cL^\times),
$$  
where $\cup_{\cL^{\times}}$ denotes the quotient of the disjoint union obtained by identifying $\cL^{\times}$ on both sides, and $\times^{\cO^{\times}}$ denotes the maximal quotient of the direct product on which the actions of $\cO^{\times}$ given by multiplying the first and second coordinates agree.
It is regarded as the twist $\cL^\times \times^{\cO^\times} (\cO\cup_{\cO^\times} M^{-1})$ of $\cO\cup_{\cO^\times} M^{-1}$ by the $\cO^\times$-torsor $\cL^\times$.

In this paper, we consider the following types of log structures:
\begin{enumerate}
	\item[(a)] log structures on the \'etale site $\cS$ of a scheme $S$ with $\cO=\cO_S$,
	\item[(b)] log structures on the \'etale site $\cS$  of a locally Noetherian formal scheme $S$ with $\cO=\cO_S$,
	\item[(c)] log structures on the site $\cS$ of open sets of an adic space $S$ with $\cO=\cO^+_S$,
	\item[(d)] log structures on the \'etale site $\cS$ of an adic space $S$ with $\cO=\cO_S$.
\end{enumerate}

In (c) and (d), recall that an adic space has two  sheaves of commutative rings $\cO_S$ and $\cO_S^+\subset \cO_S$ (cf. \cite{Hu}). 
For the \'etale site of a formal scheme, see \cite{GM}. For the \'etale site of an adic space, see \cite{Hu2}. 
In (a) (resp.,  (b)),   $S$ endowed with a log structure is called a \emph{log scheme} (resp., a \emph{log locally Noetherian formal scheme}). A scheme endowed with an fs  log structure is called an \emph{fs log scheme}. The category of fs log schemes has fiber products (see \cite[Ch.~III, 2.1.6]{O}).

In this paper, we will use saturated log structures. In applications later in this paper, we will use log structures on adic spaces in (c) and (d) only in the following restricted ways. When we use (c), 
 we will consider  only the special log structure $M= \cO_S^+\cap \cO_S^\times \subset \cO_S$ endowed with the inclusion map $\alpha \colon M\to \cO=\cO_S^+$. When we use  (d), we will consider  only  whether  the log structure  is trivial or not.  
\end{sbpara}

\begin{sbpara} 
For a line bundle $\cL$ on an fs log scheme $S$, 
we can understand $\overline{\cL}$ as the sheaf $\cF_{\cL}$ of $S$-morphisms to the projective bundle $\mathbb{P}_S(\cL\oplus \cO_S)\supset \cL$ that is endowed with the log structure of the fiber product   $S\times_{S^{\circ}} P$ where $S^{\circ}$ is the scheme $S$ with the trivial log structure and $P = \mathbb{P}_S(\cL \oplus \cO_S)$ is endowed with the log structure associated to the Cartier divisor $P\setminus \cL$.

This can be understood as follows. From \ref{sat}, we have that
$\overline{\cL}= \cL^\times \times^{\cO_S^{\times}} (\cO_S \cup_{\cO_S^{\times}} M_S^{-1})$, 
and it follows that
$$
	\cF_{\cL} =\cL^\times \times^{\cO_S^\times} \cF_{\cO_S}.
$$ 
Hence, we are reduced to the case that $\cL=\cO_S$. In this case, the above projective bundle with log structure is $S \times_{\Spec(\Z)} \mathbb{P}^1_{\Z}$, where $\mathbb{P}^1_{\Z}$ is endowed with the log structure given by the Cartier divisor at infinity, and
$\overline{\cO_S}= \cO_S \cup M_S^{-1}$. An element $f\in \cO_S$ corresponds to the morphism 
$S\to \Spec(\Z[T]) \subset \mathbb{P}^1_{\Z}$ sending $T$ to $f$, while $f\in M_S^{-1}$ corresponds to the morphism 
$S\to \Spec(\Z[T^{-1}]) \subset \mathbb{P}^1_{\Z}$ sending $T^{-1}$ to $f^{-1}\in M_S$. 

This generalizes to a vector bundle $V$ on $S$ as follows. Let $\overline{V}$ be the sheaf  of morphisms to the projective bundle $P=\mathbb{P}_S(V \oplus \cO_S)$ endowed with the log structure defined similarly to the above using the Cartier divisor $P\setminus V$. We have $\overline{V}= 
V\cup_{V^\times} M_S^{-1}V^\times$, where $V^\times$ denotes those elements of $V$ that lie in a basis.
 The authors wonder whether $\overline{V}$ can be used to formulate the notion of a log shtuka, the shtuka version of a log Drinfeld module, for a shtuka is a vector bundle.   
\end{sbpara}

We review basic facts about toric geometry (the geometric meaning of cone decompositions in toric geometry, etc.) formulated in terms of log structures.

\begin{sbpara}\label{MN1} In the rest of this subsection, 
let $L_\Z$ be a finitely generated free $\Z$-module, and let
$$
	L^*_\Z =\Hom_\Z(L_\Z, \Z).
$$ 
Set $L_\R=\R\otimes_\Z L_\Z$ and $L^*_\R=\R\otimes_\Z L^*_\Z$, and let 
$$
	(\,\;,\;) \colon L_\R\times L^*_\R\to \R
$$ 
be the canonical pairing. 

For a finitely generated cone $\sig$ in $L_\R$, let 
$$
	\sig^*= \{x\in L^*_\R \mid  (y,x)\geq 0 \textsp{for all} y\in \sig\}.
$$ 
Then $\sig^*$ is a finitely generated cone in $L^*_\R$, and we have 
$$
	\sig= \{y\in L_\R\mid  (y,x)\geq 0 \textsp{for all} x\in \sig^*\}.
$$ 
Then $\sig$ is rational if and only if $\sig^*$ is rational. 

If $\sig$ is rational, then let $$\sig^{\vee} = \sig^*\cap L_\Z.$$
Then $\sig^{\vee}$ is an fs monoid. 

\end{sbpara}

\begin{sbpara}\label{fan}  Let $L_\Z$ and $L^*_\Z$ be as in \ref{MN1}.

By a \emph{finite rational fan} in $L_\R$, we mean
 a finite set $\Sigma$ of finitely generated rational sharp (i.e., strongly convex) cones in $L_\R$ such that
\begin{enumerate}
	\item[(i)] if $\sig\in \Sig$, then all faces of $\sig$ belong to $\Sig$, 
	\item[(ii)] if $\sig,\tau\in \Sig$, then $\sig\cap \tau$ is a face of $\sig$.
\end{enumerate}

For example, if $\sig$ is a finitely generated rational sharp cone in $L_\R$, then the set $\faces(\sig)$ of all faces of $\sig$ is a finite rational fan in $L_\R$. 

By a \emph{finite rational subdivision} of a finite rational fan $\Sig$ in $L_{\R}$, we mean a finite rational fan $\Sig'$ in $L_\R$ such that
\begin{enumerate}
\item[(iii)] for each $\tau\in \Sig'$, there is $\sig\in \Sig$ such that $\tau\subset \sig$,
\item[(iv)] $\bigcup_{\sig\in \Sig}\sig= \bigcup_{\tau \in \Sig'} \tau$. 
\end{enumerate}

In the case $\Sig=\faces(\sig)$, a finite rational subdivision of $\Sig$ is called a \emph{finite rational cone decomposition} of $\sig$ and is called also a finite rational subdivision of $\sig$. 
\end{sbpara}

\begin{sbpara}\label{toric}
Let $\Sig$ be a finite rational fan in $L_\R$. 
Let $\cS$ be a site with a sheaf of commutative rings $\cO$ and with a saturated log structure $M$. We define the sheaves $\toric(\Sig)$ and $[\Sig]$ on $\cS$ by 
$$
    \toric(\Sig) =\bigcup_{\sig\in \Sig}\toric(\sig)  \subset L_\Z\otimes_{\Z} M^{\gp}, \quad \text{where  } \toric(\sig) := \cH om(\sig^{\vee}, M), 
$$ 
$$
	[\Sig]= \bigcup_{\sig\in \Sig}  [\sig] \subset L_\Z\otimes_{\Z} (M^{\gp}/\cO^\times), \quad \text{where  } [\sig] :=  \cH om(\sig^{\vee}, M/\cO^\times).
$$
Here, $\bigcup$ indicates a union as sheaves, $\cH om$ denotes the sheaf of homomorphisms, and $\sigma^{\vee}$ is viewed as a constant sheaf. In the case $\Sig=\faces(\sig)$, we have $\toric(\Sig)=\toric(\sig)$ and $[\Sig]=[\sig]$.

\end{sbpara}

\begin{sbpara}\label{toric2}

Let $\Sig$ be as in \ref{toric}. Let $\toric_{\Z}(\Sig)$ be the toric variety over $\Z$ associated to $\Sig$ in the classical toric geometry regarded as a log smooth fs log scheme over $\Z$. That is, $\toric_{\Z}(\Sig)$ is the union over $\sig \in \Sig$ of its open sets $\toric_{\Z}(\sig)=\Spec(\Z[\sig^{\vee}])$, each endowed with the log structure associated to $\sig^{\vee}\to \Z[\sig^{\vee}]$.

The toric variety $\toric_{\Z}(\Sig)$ represents the functor
$S\mapsto \Gamma(S, \toric(\Sig))$
from the category of locally ringed spaces with saturated log structures to the category of sets, where $\toric(\Sig)$ is the sheaf on $S$ defined in \ref{toric}. In this story,  the log structure on $S$ is given on the site of open sets of $S$, and we regard the log structure on $\toric_{\Z}(\Sig)$ as a sheaf on the Zariski site. 

In the cases (a), (b), (d) in \ref{sat}, the log structure is given on the \'etale site. In these cases, $\toric(\Sig)(S) := \Gamma(S, \toric(\Sig))$ is identified with the set of morphisms $S\to \toric_{\Z}(\Sig)$ of locally ringed spaces with log structures on the \'etale sites. 

These follow from the fact that for each $\sig\in\Sig$, a morphism $S\to \toric_{\Z}(\sig)$ induces and uniquely comes from a homomorhism $\sig^{\vee}\to M$ of sheaf of monoids.
 
\end{sbpara}

\begin{sbpara}\label{weaker}

 In the case (a) (resp., (b), resp., (d)), by the weaker topology on $\cS$, we mean that we regard the category $\cS$ as another site in which we take open coverings of schemes (resp., formal schemes, resp., adic spaces) as coverings. In particular, the weaker topology on $\cS$ in the case (a) means the Zariski topology.

By the fact that $\toric_{\Z}(\Sig)= \bigcup_{\sig \in \Sig} \toric_{\Z}(\sig)$ is an open covering, we have the following. 
\end{sbpara}

 \begin{sblem}\label{toric3} In the case (a) (resp., (b), resp., (d)), the sheaf $\toric(\Sig)$ on the \'etale site $\cS$ of $S$, which is the union of the subsheaves $\toric(\sig)$ ($\sig \in \Sig$)  for the \'etale topology, is actually the 
 union of the $\toric(\sig)$ as a sheaf for the weaker topology on $\cS$. 
 
\end{sblem}

\begin{sbpara}\label{toric4}  Let $\Sig$ be as above, and let $S$, $\cS$, $\cO$  and $M$ be as in one of the cases (a)--(d) in \ref{sat}.

We have that  
$$ L_{\Z}\otimes_{\Z} \cO^\times \subset  \toric(\Sig)\subset L_\Z\otimes_{\Z} M^{\gp},$$
$L_{\Z} \otimes_{\Z} \cO^\times$ acts on $\toric(\Sig)$, and 
$$\toric(\Sig)/(L_{\Z}\otimes_{\Z} \cO^\times)\cong [\Sig]$$ as  sheaves on $\cS$.

\end{sbpara}

\begin{sblem}\label{toric5}
In the case (a) (resp., (b), resp., (d)), $$\toric(\Sig)/(L_{\Z}\otimes_{\Z} \cO^\times)\cong [\Sig]$$ also as sheaves for the weaker topology on $\cS$. 
\end{sblem}

\begin{proof} The stalks $\cO_{S,s}$ for $s\in S$ are local rings. Hence 
the exact sequence $1 \to \cO^\times\to M^{\gp}\to M^{\gp}/\cO^\times \to 1$ of sheaves on $\cS$ is  exact also for  the weaker topology because for each object $S'$ of $\cS$, each element of $H^1(S', \cO^\times)$ vanishes locally for the weaker topology by Hilbert's theorem 90. This shows that the map $\toric(\Sig)\to [\Sig]$ is surjective as a map of sheaves for the weaker topology.
\end{proof}

\begin{sblem}\label{toric6}
In case (a) (resp., (b), resp., (d)), the sheaf $[\Sig]$ on the \'etale site of $S$, which is the union of the subsheaves $[\sig]$ ($\sig \in \Sig$) for the \'etale topology, is actually the union of $[\sig]$ as the sheaf for the weaker topology.

\end{sblem}

\begin{proof} This follows from \ref{toric3} and \ref{toric5}. 
\end{proof}

\begin{sblem}\label{toric7} Assume we are in one of the cases (a), (b), and (d) in \ref{sat}, let $M$ be a saturated log structure on $\cS$, and let $f,g\in \Gamma(\cS, M^{\gp}/\cO^\times)$.  Then the following (i) and (ii) are equivalent. 
\begin{enumerate}
\item[(i)] \'Etale locally on $S$, we have either $fg^{-1}\in M/\cO^\times$ or $f^{-1}g\in M/\cO^\times$ in $M^{\gp}/\cO^\times$. 
\item[(ii)] Locally on $S$ for the weaker topology, we have either $fg^{-1}\in M/\cO^\times$ or $f^{-1}g\in M/\cO^\times$ in $M^{\gp}/\cO^\times$.
\end{enumerate}
\end{sblem} 

\begin{proof} Let $L_{\Z}=\Z^2$. Set
\begin{eqnarray*}
	\sig_1= \{(x,y)\in \R^2 \mid  x\geq y\geq 0\} & and & \sig_2= \{(x,y)\in \R^2\mid y\geq x\geq 0\}.
\end{eqnarray*} 
Let $\Sig_i=\faces(\sig_i)$ for $i \in \{1,2\}$, and let $\Sig=\Sig_1\cup \Sig_2$. If the condition (i) is satisfied, then the pair 
$$
	(f,g) \in \Gamma(\cS,M^{\gp}/\cO^{\times})^2 = \Gamma(\cS,L_{\Z} \otimes_{\Z} M^{\gp}/\cO^{\times})
$$ 
belongs to $([\sig_1]\cup [\sig_2])(S)$, where $\cup$ is the union of sheaves for the \'etale topology. By \ref{toric6}.  this union $[\Sig]$ is in fact the union of $[\sig_1]$ and $[\sig_2]$ as a sheaf for the weaker topology. Hence the condition (ii) is satisfied. 
\end{proof}

\begin{sbpara}\label{toSig0} Assume we are one of the cases (a)--(d).
Let $\Sig$ and $\Sig'$ be finite rational fans in $L_\R$, and assume that for each $\sig'\in \Sig'$, there exists $\sig\in \Sig$ such that $\sig'\subset \sig$. Then we have injective maps $\toric(\Sig')(S) \to \toric(\Sig)(S)$ and $[\Sig'](S) \to [\Sig](S)$. The morphism $\toric_{\Z}(\Sig') \to \toric_{\Z}(\Sig)$ of fs log schemes is log \'etale. We have $\toric(\Sig) \times_{[\Sig]}
[\Sig'] = \toric(\Sig')$ as subfunctors of $\toric(\Sig)$ by \ref{toric4}.

On the category of fs log schemes, the morphism $[\Sig']\to [\Sig]$ is represented by log \'etale morphisms. In fact, Zariski locally on $S$, 
 any element $a\in [\Sig](S)$ comes from an element of $\toric(\Sig)(S)$ by \ref{toric5}, that is, from a morphism  $S\to \toric_{\Z}(\Sig)$ of fs log schemes. Thus, we have
$$
	S\times_{[\Sig]} [\Sig'] = S \times_{\toric_{\Z}(\Sig)} \toric_{\Z}(\Sig) \times_{[\Sig]} [\Sig'] = S \times_{\toric_{\Z}(\Sig)}  \toric_{\Z}(\Sig').
$$

If $\Sig'$ is a finite rational subdivision of $\Sig$, then the morphism 
 $\toric_{\Z}(\Sig') \to \toric_{\Z}(\Sig)$ is  the well-known proper birational morphism of classical toric geometry. In this case, on the category of fs log schemes, $[\Sig']\to [\Sig]$ is represented by proper morphisms. 
 
\end{sbpara}

\begin{sbpara}\label{toSig1} Assume we are one of the cases (a)--(d). 
Let $\Sig$ be a finite rational fan in $L_\R$. 
Define the topology of the set $\Sig$ by taking the subsets $\faces(\sig)$ for $\sig\in \Sig$ as a base of open sets.

An element $a$ of $[\Sig](S)$ induces a continuous map $\varphi \colon S\to \Sig$ which sends $s\in S$ to the smallest cone $\sig$ of $\Sig$ such that the map $L^*_\Z\to M^{\gp}/\cO^\times$ induced by $a$ sends $\sig^{\vee}\subset L^*_\Z$ to  $M/\cO^\times\subset M^{\gp}/\cO^\times$ at $s$. This map is understood as follows. Locally on $S$, the element $a$ lifts to an element $\tilde a$ of $\toric(\Sig)(S)$. This $\tilde a$ gives a continuous map $S\to \toric_{\Z}(\Sig)$. On the other hand, we have a continuous map $\toric_{\Z}(\Sig) \to \Sig$ which sends $x\in \toric_{\Z}(\Sig)$ to the smallest element $\sig$ of $\Sig$ such that $x$ belongs to $\toric_{\Z}(\sig)$. The map $\varphi$ is locally the composition $S\to \toric_{\Z}(\Sig)\to \Sig$ of continuous maps.
\end{sbpara}

\begin{sbpara}\label{tl2}  

 Consider the case (c) of \ref{sat}. 
We call the log structure  $M=\cO_S^+\cap \cO_S^\times$ with the inclusion map $M\to \cO^+_S$ on the site of open sets of $S$,  the \emph{canonical log structure} on $\cO_S^+$. This is a saturated log structure.

 Note that for $s\in S$, the stalk $\cO_{S,s}$ is a local ring, and there is a valuation ring $\cV$ in the residue field $\kappa(s)$ of $\cO_{S,s}$ such that the field of fractions of $\cV$ coincides with $\kappa(s)$ and  such that the stalk $\cO^+_{S, s}$ coincides with  the inverse image of $\cV$ under $\cO_{S,s}\to \kappa(s)$. Hence the stalk  $M_s$ is the inverse image of $\cV\setminus \{0\}$ in $\cO_{S,s}$.  Hence 
  $M^{\gp}= \cO_S^\times$. We have $M_s^{\gp}=M_s\cup M_s^{-1}$.

\end{sbpara}

\begin{sbprop}\label{tl3}  For an adic space $S$ endowed with the canonical log structure \ref{tl2} on $\cO_S^+$ and for a finite rational subdivision $\Sig'$ of $\Sig$, the canonical map
$[\Sig'](S) \to [\Sig](S)$ is bijective. 
\end{sbprop}

\begin{sbpara} We give a preparation for the proof of \ref{tl3}.

Let $\Sig_{\val}=\varprojlim \Sig'$, where $\Sig'$ ranges over all finite rational subdivisions of $\Sig$. Then $\Sig_{\val}$ is identified with the set of all submonoids $V$ of $L^*_\Z$ satisfying the following conditions:
\begin{enumerate}
\item[(i)] $V\cup (-V)= L^*_\Z$,
\item[(ii)] $\sig^{\vee} \subset V$ for some $\sig\in \Sig$.
\end{enumerate}

In fact, for $x=(x_{\Sig'})_{\Sig'}\in \Sig_{\val}$, the corresponding $V$ is given by $V= \bigcup_{\Sig'}  (x_{\Sig'})^{\vee}\subset L^*_\Z$. 
Conversely, from $V$, we have the corresponding $x=(x_{\Sig'})_{\Sig'}$ where  $x_{\Sig'}$ is the smallest element  $\tau$ of $\Sig'$ such that $\tau^{\vee}\subset V$. 

 With this identification, the inverse limit topology on $\Sig_{\val}$ is described as follows.
The sets $\{V\in \Sig_{\val}\mid I \subset V\}$, where $I$ ranges over all finite subsets of $L^*_{\Z}$,  form a base of open sets.   
\end{sbpara}

\begin{sbpara} We prove \ref{tl3}.

An element $a$ of $[\Sig](S)$ induces a continuous map $\varphi \colon S\to \Sig$ (\ref{toSig1}). 
Furthermore,  $a$ induces a continuous map $\varphi_{\val} \colon S\to \Sig_{\val}$ which sends $s\in S$ to the inverse 
image $V$  of $(M/\cO^\times)_s$ under $L^*_\Z\to (M^{\gp}/\cO^\times)_s$. (Note that since $M_s^{\gp}=M_s\cup M_s^{-1}$, we have $L_{\Z}^*= V \cup(-V)$.)
The composition $S\overset{\varphi_{\val}}\to \Sig_{\val}\to \Sig$ coincides with $\varphi$.

Let $\varphi'$ be the composition 
$S\overset{\varphi_{\val}}\to \Sig_{\val}\to \Sig'$. For each $s\in S$, the map $a \colon L^*_\Z\to (M^{\gp}/\cO^\times)_s$ induces $\varphi'(s)^{\vee}\to (M/\cO^\times)_s$. 
By the continuity of $\varphi'$, there is an open neighborhood $U$ of $s$ in $S$ such that $a$ induces $\varphi'(U)^{\vee}\to (M/\cO^\times)_U$. Hence $a$ comes from $[\Sig'](S)$. 

\end{sbpara}

\subsection{Level structures}\label{ss:level}

We fix an element $N$ of $A$ which does not belong to the total constant field of $F$. 
Fix an integer $d\geq 1$.

  \begin{sbpara}\label{Dlevel}  
    
Recall that for a scheme $S$ over $A$ and a Drinfeld module $(\cL, \phi)$ over $S$ of rank $d$, a \emph{Drinfeld level $N$ structure} on $(\cL, \phi)$  is an $A$-equivariant homomorphism $\iota \colon (\frac{1}{N}A/A)^d \to \cL$, where $A$ acts on $\cL$ via $\phi$, such that 
$$
	\ker(\phi(N))=\sum_{a\in (\frac{1}{N}A/A)^d} [\iota(a)]
$$ 
as Cartier divisors on $\cL$. Here $[\iota(a)]$ is the image of the section $\iota(a) \colon S\to \cL$ regarded as a Cartier divisor on $\cL$. We will call a Drinfeld level $N$ structure simply a \emph{level $N$ structure}. 

Throughout this subsection, we will let $(S,U)$ denote a pair as in \ref{log4}, which is to say a normal scheme $S$ over $A$ and a dense open subset $U$. Recall from \ref{log4} that a generalized Drinfeld module $((\cL,\phi),\iota)$ of rank $d$ over $(S,U)$ with level $N$ structure is a pair consisting of generalized Drinfeld module $(\cL, \phi)$ over $S$ with restriction $(\cL,\phi)|_U$ a Drinfeld module over $U$ of rank $d$ and a level $N$ structure $\iota$ on $(\cL, \phi)|_U$.
  \end{sbpara}

\begin{sblem}\label{levelex} Let $S$ be a normal scheme over $A$, and let  $(\cL, \phi)$ be a Drinfeld module over $S$. Let $U$ be a dense open subset of $S$, and let $\iota$ be a level $N$ structure on $(\cL, \phi)|_U$. Then $\iota$ extends uniquely to a level $N$ structure of $(\cL, \phi)$.

\end{sblem}

\begin{proof} We may assume that $S=\Spec(R)$ for a normal integral domain $R$ with field of fractions $K$ and that the line bundle $\cL$ is trivialized. Then the coefficient of the  polynomial $\phi(N)(z)$  in the highest degree $n=|N|^d$, where $d$ is the rank of $(\cL, \phi)$, is a unit $c$ of $R$. Since this polynomial is a product of polynomials over $K$ 
of degree $1$ by assumption, and since $R$ is normal, $\phi(N)(z)= c\prod_{i=1}^n (z-\alpha_i)$ for some $\alpha_i\in R$. Hence the $A/NA$-homomorphism $\iota \colon (\frac{1}{N}A/A)^d\to \cL|_U$ is extended to $\iota \colon (\frac{1}{N}A/A)^d\to \cL$, 
and we have 
$$
	\ker(\phi(N))=\text{div}(\phi(N)(z))= \sum_i \text{div}(z-\alpha_i)= \sum_{a\in (\frac{1}{N}A/A)^d} [\iota(a)].
$$ 
\end{proof}

\begin{sblem}\label{logD100} Let $R$ be a normal local integral domain with maximal ideal $m_R$, residue field $k$, and field of fractions $K$. Consider a polynomial $f\in R[z]$ in one variable. Assume that the image of $f$ in $k[z]$ is nonzero, and assume that $f$ is a product of polynomials of degree $1$
in $K[z]$. Then 
$$
	f= c \prod_{i=1}^m (z-a_i)  \cdot \prod_{j=1}^n (1-b_jz)
$$
for some $m,n \ge 0$, $a_i \in R$ ($1 \le i \le m$), $b_j \in m_R$ ($1 \le j \le n$), and $c \in R^{\times}$.

\end{sblem}

\begin{proof} By a simple limit argument, we are reduced to the case that $R$ is a local ring of a finitely generated $\Z$-algebra. Assume we are in this case. We can write $f= g \prod_{i=1}^m  (z-a_i)$ with $a_i\in R$ and with $g\in R[z]$ which has no root in $R$.  The completion $\hat R$ of $R$ is also a normal integral domain as $R$ is excellent. By the Weierstrass preparation theorem, $g= hu$ where $h$ is a monic polynomial over $\hat R$ and $u\in \hat R\ps{z}^\times$. If $h$ is of degree $\geq 1$, take a  root $\alpha$ of $h$, which is integral over $\hat R$. Then $\alpha$ is a root of $g$ and hence belongs to $K$. By the normality of $\hat R$, it then belongs to $\hat R\cap K=R$. This is a contradiction. Hence $h=1$ and  $g= u$. Thus the constant term $c$ of $g$ is a unit of $R$. Write $g=c\cdot (1+ \sum_{i=1}^n c_i z^i)$ with $c_i\in R$. Then the inverse of any root $\beta$ of $g$ is a root of the monic polynomial $z^n + \sum_{i=1}^n c_i z^{n-i}$, and hence $\beta^{-1}\in R$. Since $\beta \notin R$, we have $\beta^{-1}\in m_R$. 
\end{proof}

\begin{sbprop}\label{logD53} Let $((\cL, \phi), \iota)$ be a generalized Drinfeld module over $(S,U)$ of rank $d$ with level $N$ structure. Then $\iota \colon (\frac{1}{N}A/A)^d \to \cL|_U$ extends uniquely to a map  $(\frac{1}{N}A/A)^d \to \overline{\cL}$, where $\overline{\cL}$ is defined by the associated log structure of $S$ as in \ref{tL}.

\end{sbprop}

\begin{proof} Working locally on $S$, we
 may assume $\cL=\cO_S$. Let $s\in S$. Note that the polynomial $\phi(N)$ modulo the maximal ideal of $\cO_{S,s}$ is not zero.  We apply Lemma \ref{logD100} to  $R=\cO_{S,s}$ and $f=\phi(N)$. Then we see that for $a\in (\frac{1}{N}A/A)^d$, we have either $\iota(a)\in \cO_{S,s}$ or $\iota(a)^{-1}\in \cO_{S,s}$. We also have $\iota(a)\in j_*(\cO_U)$. 
 Since $\overline{\cL}$ is the union of its subsheaves $\cO_S \subset j_*(\cO_U)$ and $\cO_S^{-1} \cap j_*(\cO_U^{\times}) = \cO_S^{-1} \cap j_*(\cO_U)$, we have that $\iota(a)$ belongs to $\overline{\cL}$ at $s$.
\end{proof}

To reduce various problems to the case of complete discrete valuation rings, the following lemma is useful.

\begin{sblem}\label{todvr} Let $R$ be an excellent integral domain with field of fractions $K$, and let $I$ be an ideal of $R$. Let $\cE$ be the set of all discrete valuation rings $\cV \subset K$ containing $R$ and with maximal ideal $m_{\cV}$ containing $I$.

\begin{enumerate}
\item[(1)] Assume that $R$ is normal and $I$ is contained in the Jacobson radical of $R$. (Note that this condition of $I$ is satisfied  if either $R$ is a local ring with maximal ideal $I$ or $R$ is $I$-adically complete.) Then  $R=\bigcap_{\cV\in \cE} \cV$. 
\item[(2)] Assume that $R/I$ is nonzero and reduced. Then $I=\bigcap_{\cV\in \cE}  m_{\cV}$. 
\end{enumerate}

\end{sblem}

\begin{proof} We first prove the following claim. 
\medskip

{\bf Claim.} If $R$ is local with maximal ideal $m$, then there is a discrete valuation ring $\cV$ such that $R\subset \cV\subset K$ 
	and $R\cap m_{\cV}=m$. 

\begin{proof}[Proof of Claim]
In fact, let $X$ be the normalization of the blow-up of $\Spec(R)$ along $m$, and take a point $x$ of $X$ lying over $m$. Then $m$ generates a principal ideal in $\cO_{X,x}$ generated by an element $t$  which is not a unit. Since $\cO_{X,x}$ is normal and Noetherian by the assumption $R$ is excellent, there is a prime ideal $\frak p$ of height one of $\cO_{X,x}$ which contains $t$. Then the local ring $\cV$ of $\cO_{X,x}$ at $\frak  p$  has the desired property. \phantom\qedhere
\end{proof} 

We prove part (1). Let $f\in K \setminus R$, and let $g \colon X \to \Spec(R)$ be the blow-up along the 
 fractional ideal $R+Rf$. For each point $x$ of $X$, we have either $f\in \cO_{X, x}$ or $f^{-1}\in \cO_{X,x}$. Since $\Gamma(X, \cO_X)=R$,  there exists $x\in X$ such that $f\notin \cO_{X,x}$. Since $g$ is proper, the image in $\Spec(R)$ of the closure of $x$ in $X$ is closed and hence contains a maximal ideal $m$ of $R$. By assumption on $I$, we have $I\subset m$. Let $y\in X$ be the element of the closure of $x$ in $X$ whose image in $\Spec(R)$ is $m$. Then $f\notin \cO_{X, y}$.  Hence $f^{-1}$ belongs to the maximal ideal $m_y$ of $\cO_{X,y}$. 
By the claim applied to $\cO_{X,y}$, there is  a discrete valuation ring $\cV$ such that $\cO_{X,y}\subset \cV\subset K$ and $\cO_{X,y}\cap m_{\cV}=m_y$. We have $I\subset m\subset m_{\cV}$, so $\cV \in \cE$. On the other hand, $f^{-1}\in m_{\cV}$, and hence $f\notin \cV$. 

As for (2), by assumption, $I$ is the intersection of all prime ideals $\frak p$ of $R$ such that $I\subset \frak p$.  By the claim applied to the local ring of $R$ at $\frak p$, there is a discrete valuation ring $\cV$ such that $R\subset \cV \subset K$ and $\frak p=R\cap m_{\cV}$.
\end{proof} 

\begin{sbpara}\label{cpf2} 

As in \ref{val1}, let $\cV$ be a complete valuation ring over $A$ of height $1$, let $K$ be the fraction field of $\cV$, 
let $\bar K$ be an algebraic closure of $K$, and let $K^{\sep}$ be the separable closure of $K$ in $\bar K$.
Let $\phi$ be a generalized Drinfeld module with trivial line bundle over $\cV$, 
and let $(\psi,\La)$ be the associated pair of a Drinfeld module
and a $\psi(A)$-lattice in $K^{\sep}$ from \ref{val2}. Let $e_{\La} \in \cV\ps{Z}$ denote the exponential map of $\La$ (see \ref{prop51}).
Let 
$$\psi(N)^{-1}\La=\{x\in \bar K\mid \psi(N)x\in \La\},$$
\begin{eqnarray*}
	\phi[N]= \{x\in \bar K\mid \phi(N)x=0\} &\text{and}& \psi[N]= \{x\in \bar K\mid \psi(N)x=0\}.
\end{eqnarray*}
Then we have 
$$e_{\La} \colon \psi(N)^{-1}\La/\La \xrightarrow{\sim} \phi[N].$$
We also have an exact sequence 
$$0\to \psi[N] \xrightarrow{e_{\La}} \phi[N]\to \La/\psi(N)\La \to 0.$$
The map $\phi[N]\to \La/\psi(N)\La$ sends $\beta\in \phi[N]$ to the class of any $\lam\in \La$ such that there is a $\lam'\in \bar K$ satisfying $\beta=e_{\La}(\lam')$ and $\psi(N)\lam'= \lam$. 

\end{sbpara}

In this paper, for a sheaf $\cF$ on the \'etale site of a scheme, $\cF_{\bar s}$ denotes the stalk of $\cF$ lying over a point $s$ in the sense of the \'etale topology. 

\begin{sblem}\label{nlevelN} Let $((\cL, \phi),\iota)$ be a generalized Drinfeld module over $(S,U)$ of rank $d$ with level $N$ structure. 
Define the subsheaf $E$ of $(\frac{1}{N}A/A)^d$ on the \'etale site of $S$ to be  the inverse image of $\cL$ under the map $\iota \colon (\frac{1}{N}A/A)^d\to \overline{\cL}$ given by \ref{logD53}. 

\begin{enumerate}
\item[(1)] Let $s\in S$.   We have $E_s := \Gamma(\Spec(\cO_{S,s}), E)\xrightarrow{\sim} E_{\bar s}$. Let  $r$ be the rank of the fiber of $(\cL, \phi)$ at $s$. Then the $A/NA$-module $E_s$ is  free of rank $r$.  Let $E'_s= (\frac{1}{N}A/A)^d\setminus E_s$. If the line bundle $\cL$ is trivialized at $s$, the polynomial $\phi(N)(z)$ is expressed over $\cO_{S,s}$ as 
$$\phi(N)(z)= c \prod_{a\in E_s}  (z-\iota(a)) \cdot \prod_{a\in E'_s} (1- \iota(a)^{-1}z)$$
for some unit $c$ of $\cO_{S,s}$, and $\iota(a)^{-1}$ belongs to the maximal ideal $m_s $ of $\cO_{S,s}$ for all $a\in E_s'$. 
\item[(2)] Let $I$ be a quasi-coherent ideal of $\cO_S$ such that $\cO_S/I$ is a sheaf of reduced rings. Assume that $(\cL, \phi)\bmod I$ is a  Drinfeld module of rank $r$ over $\cO_S/I$ and that the line bundle $\cL$ is trivialized. Then the image of the composite map
$$(\tfrac{1}{N}A/A)^d\setminus E\overset{\iota}\to M_S^{-1}\cL^\times =M_S^{-1}\overset{b}\to M_S \to \cO_S,$$
where $b(f)=f^{-1}$, is contained in $I$. 
\end{enumerate}
\end{sblem}

\begin{proof} (1)  The bijectivity $E_s \xrightarrow{\sim} E_{\bar s}$ is clear. Trivialize the line bundle $\cL$ at $s$. By \ref{logD100}, we obtain the product formula for $\phi(N)(z)$ stated above and that $E_s$ has order $|N|^r$. To prove $E_s\cong (A/NA)^r$, we may assume that $S$ is of finite type over $A$ and hence is excellent. By \ref{todvr}(1), we may take a discrete valuation ring $\cV'$ in the field of fractions of $\cO_{S,s}$ such that $\cO_{S,s}\subset \cV'$ and $m_s\subset m_{\cV'}$. Let $\cV$ be the completion of $\cV'$, and let $K$ be the field of fractions of $\cV$. By replacing $\cO_{S,s}$ by $\cV$, we may assume that $\cO_{S,s}=\cV$. 

As in as in \ref{cpf2}, we then have a pair $(\psi,\La)$ and an exact sequence 
of finite flat group schemes over $K$ in which the cokernel of the exponential map $e_{\La} \colon \psi[N] \to \phi[N]$ is $\La/\psi(N)\La$.  Concerning the level $N$ structure $\iota$ over $K$, for $a\in (\frac{1}{N}A/A)^d$ with $\iota(a)\in e(\psi[N])$, we have $a\in E_s$. Since the finite flat group scheme $\psi[N]$ has rank $|N|^r$, and this is equal to the order of $E_s$, we have that $\iota(a)\in e_{\La}(\psi[N])$ if and only if $a\in E_s$. Hence the above exact sequence shows that $(\frac{1}{N}A/A)^{d-r}\cong \La/N\La$ as an $A/NA$-module. Hence $E_s\cong (A/NA)^r$ as an $A/NA$-module. 

(2) We may assume that $S$ is of finite type over $A$ and hence is excellent. By \ref{todvr}(2), we may assume $S=\Spec(\cV)$ for a complete discrete valuation ring $\cV$. Then this (2) follows from the formula for $\phi(N)(z)$ in (1) at the closed point $s$ of $\Spec(\cV)$, 
\end{proof}

In the remainder of this subsection, we provide conditions for the automorphism group of a generalized Drinfeld module $((\cL, \phi), \iota)$ with level $N$ structure to be trivial: see \ref{auto} below.

\begin{sblem}\label{MNx} Let $S$ be a normal integral scheme over $A$ with function field $K$, and let $\phi$ be a generalized Drinfeld module over $S$ with line bundle $\cO_S$. Let $N_1, N_2 \in A$ be such that $N_1$ is invertible on $S$. 
Suppose that $x\in K$ satisfies $\phi(N_1N_2)x=0$ and $\phi(N_2)x\neq 0$. Then $x^{-1}\in \Gamma(S, \cO_S)$.

\end{sblem}

\begin{proof} We write $\cO_S$ for $\Gamma(S,\cO_S)$ for short.
First assume $N_2=1$. Since $N_1$ is invertible on $S$ and 
$$
	0=\phi(N_1)(x)= N_1x +\sum_{i=2}^n a_ix^i
$$ 
for some $n\geq 2$ and $a_i\in \cO_S$ for $2 \le i \le n$, we have that $x^{-1}$ is integral over $\cO_S$. Since $S$ is normal,  $x^{-1}$ belongs to $\cO_S$. 

In general, assume that $x^{-1}$ does not belong to $\cO_S$. Since $S$ is normal, there is a valuation ring $\cV\subset K$ which dominates a point of $S$ such that $x^{-1}$ does not belong to $\cV$. Then  $x$ belongs to the maximal ideal of $\cV$. Hence  $\phi(N_2)x$ belongs to the maximal ideal of $\cV$. Hence $(\phi(N_2)x)^{-1}$ does not belong to $\cV$. But $\phi(N_1)\phi(N_2)x=0$ and we have seen $(\phi(N_2)x)^{-1} \in \cO_S$, which is a contradiction. \end{proof}

\begin{sbpara}\label{*S} In the remainder of the section, let  $((\cL, \phi), \iota)$ be
 a generalized Drinfeld module over $(S, U)$  of rank $d$ with level $N$ structure. Assume we are in one of the following two situations:
\begin{enumerate}
\item[(i)] $N$ has at least two prime divisors, or
\item[(ii)] $N$ is invertible on $S$. 
\end{enumerate}

\end{sbpara}

\begin{sbprop}\label{*Simage} Assume we are in situation (i) or (ii) of \ref{*S}. We let $(\frac{1}{N}A/A)^d_{*}$ be
the nonempty subset of $(\frac{1}{N}A/A)^d$ consisting of
\begin{itemize}
\item in situation (i), all elements $a$ such that  the ideal $\{b\in A \mid ba=0\}$ of $A$ has at least two prime divisors, and
\item in situation (ii), all nonzero elements.
\end{itemize}
Then we have 
$$\iota((\tfrac{1}{N}A/A)^d_{*,S})\subset M_S^{-1}\cL^\times \subset \overline{\cL}.$$
\end{sbprop}

\begin{proof} 
We may assume that the line bundle $\cL$ is trivial.  Let $a \in (\frac{1}{N}A/A)^d_*$. Working locally on $S$, let $v$ be a maximal ideal of $A$ which contains the ideal $\{b\in A\mid ba=0\}$ and which is not in the image of $S\to \Spec(A)$. 
Let $N_1 \in A$ generate the highest power of $v$ dividing $N$, and let $N_2 = N/N_1$.  Then the conditions of Lemma \ref{MNx} are satisfied for 
$x = \iota(a)$,
so $\iota(a)^{-1} \in \Gamma(S,\cO_S)$, which implies the result.
\end{proof}

\begin{sbprop}\label{unit} Let the situation be as in \ref{*S}. 
\begin{enumerate}
\item[(1)] Let $a\in (\frac{1}{N}A/A)^d$, and assume that $a$ is a part of a basis of the free  $A/NA$-module $(\frac{1}{N}A/A)^d$. Then $\iota(a)\in M_S^{-1}\cL^\times$.
\item[(2)] 
Locally on $S$, there exists  $a\in (\frac{1}{N}A/A)^d$ such that $a$ is a part of a basis of the free
$A/NA$-module $(\frac{1}{N}A/A)^d$ and 
$\iota(a) \in \cL^{\times}$.
\end{enumerate}
\end{sbprop}

\begin{proof} Part (1) follows from \ref{*Simage}. We prove (2).   Let $s\in S$,  trivialize $\cL$ at $s$, and let $E_s$ be as in \ref{nlevelN}. 
Since $E_s\cong (A/NA)^r$ as an $A/NA$-module with $r\geq 1$ by \ref{nlevelN}, there exists $a\in E_s$ which is a part of a basis of the free $A/NA$-module $(\frac{1}{N}A/A)^d$.  Since $a\in E_s$, the element
$\iota(a)$ belongs to $\cO_{S,s}$. On the other hand, since $a$ belongs to $(\frac{1}{N}A/A)^d_{*}$, we have that $\iota(a)^{-1}$ belongs to $\cO_{S,s}$ by  \ref{*Simage}. Hence $\iota(a)\in \cO_{S,s}^\times$. 
 \end{proof}

\begin{sbprop}\label{auto} Under the assumption in \ref{*S},  the automorphism group of $((\cL, \phi), \iota)$ is trivial.
\end{sbprop}

\begin{proof} Any automorphism $f \colon \cL \to \cL$ of $(\cL,\phi)$ with $f \circ \iota = \iota$ must fix a local basis of $\cL$ by \ref{unit}(2), so it must be trivial. \end{proof}

\subsection{Log Drinfeld modules}\label{Nlev}

We fix an element $N$ of $A$ which does not belong to the total constant field of $F$. 
Fix an integer $d\geq 1$.  We prove basic results on log Drinfeld modules of rank $d$ with level $N$ structure.

 \begin{sbpara}\label{logreg}  Let $S$ be a scheme with a log structure $M$.  
 We say  $S$ is \emph{log regular} if $S$ is locally Noetherian, the log structure is fs,  and the following two conditions (i) and (ii) 
 are satisfied for every $s\in S$. Let $I$ be the ideal of $\cO_{S,\bar s}$ generated by the image of $M_{S,\bar s}\setminus \cO_{S,\bar s}^\times \to \cO_{S,\bar s}$.  
\begin{enumerate}
	\item[(i)]  The local ring $\cO_{S,\bar s}/I$ is regular, and
	\item[(ii)] $\dim(\cO_{S,\bar s})= \dim(\cO_{S,\bar s}/I) + \text{rank}_{\Z}((M_S^{\gp}/\cO_S^\times)_{\bar s})$. 
\end{enumerate}

Regarding log regularity, see \cite{KK} for log structures on the Zariski site, and see for instance \cite[Expos\'e VI]{ILO} and \cite[Section 12.5]{GaRo}  for log structures on the \'etale site (which we use  in this paper).  We call a log regular log scheme a \emph{log regular scheme} for brevity.

 \end{sbpara}
 
 \begin{sbpara}\label{logD20} For a log regular scheme $S$, the following (1)--(4) hold by 4.1, 11.6, 8.2, and 3.1 of \cite{KK},
 in that order.
  \begin{enumerate}
	\item[(1)] The underlying scheme of $S$ is  normal. 
 	\item[(2)] We have $M_S=\cO_S\cap j_*(\cO^\times_U)\subset j_*(\cO_U)$. Here $U$ is the open set of $S$ consisting all points at which the log structure is trivial, and $j \colon U\to S$ is the inclusion map. 
	\item[(3)] Any log smooth scheme over $S$ is log regular.
	\item[(4)] Let $k$ be a perfect field, and suppose that $S$ is a scheme  over $k$ of finite type. 
	Then $S$ is log smooth over $k$. 
\end{enumerate}

  \end{sbpara}

\begin{sbpara}\label{Deflog1} 
For a scheme $S$ over $A$ with a saturated log structure, we recall from \ref{logD6} that a log Drinfeld module $((\cL, \phi), \iota)$ over $S$ of rank $d$ with level $N$ structure is a pair consisting of a generalized Drinfeld module $(\cL, \phi)$ over $S$ and a map 
$\iota \colon (\tfrac{1}{N}A/A)^d \to \overline{\cL}$ 
such that \'etale locally on $S$, there are a morphism $f \colon S \to S'$ of log schemes over $A$, where $S'$ is log regular and locally of finite type over $A$,
and a generalized Drinfeld module $((\cL', \phi'), \iota')$ over $(S',U')$ of rank $d$ with level $N$ structure, 
where $U'$ denotes the dense open subset of $S'$ at which the log structure of $S'$ is trivial, such that $((\cL, \phi), \iota)$ and the pullback of $((\cL', \phi'), \iota')$ by $f$ are isomorphic, where $\iota'$ is regarded as a map $(\frac{1}{N}A/A)^d\to \overline{\cL'}$ as in \ref{logD53}.

Let us prove (1) of \ref{logD8}. That is, 
if the log structure of $S$ is trivial, then the notion of a log Drinfeld module over $S$ of rank $d$ with level $N$ structure is equivalent to the notion of a Drinfeld module over $S$ of rank $d$ with level $N$ structure.

Assume we are given a log Drinfeld module over $S$ of rank $d$ with level $N$ structure. By definition, \'etale locally on $S$,  it comes from a log regular  $S'$. The morphism $S\to S'$ factors as $S\to U\to S'$, where $U$ is the part of $S'$ with trivial log structure, which is open in $S'$. On $U$, we have a Drinfeld module of rank $d$ with level $N$ structure in the usual sense. Since we pull it back, the original object is a Drinfeld module of rank $d$ with level $N$ structure. Conversely, assume that we have a Drinfeld module over $S$ of rank $d$ with level $N$ structure. It is obtained via a morphism $S\to \cM^d_N$ to the moduli space of Drinfeld modules of rank $d$ and level $N$ as the pullback of the universal object over the regular scheme $\cM^d_N$, yielding a log Drinfeld module with level $N$ structure.

\end{sbpara}

\begin{sbpara}\label{strong} 

Let $(S, U)$ be as in \ref{log4}, and let $((\cL, \phi), \iota)$ be a  generalized Drinfeld module  over $(S, U)$ of rank $d$ with level $N$ structure. We consider the following condition on divisibility.

\begin{enumerate}
\item[(div)] For every $a,b\in (\frac{1}{N}A/A)^d$, locally on $S$ we have either 
\begin{eqnarray*}
	\pole(\iota(a))\pole(\iota(b))^{-1}\in M_S/\cO_S^\times &\text{or}& \pole(\iota(b))\pole(\iota(a))^{-1}\in M_S/\cO_S^\times
\end{eqnarray*} 
in $M_S^{\gp}/\cO_S^\times$. 
\end{enumerate}

Note that by \ref{toric7}, we have the same condition (div) if we replace ``locally'' (this means Zariski locally) by ``\'etale locally''.

\end{sbpara}

\begin{sbthm}\label{regdiv}
Let $S$ be a log regular scheme over $A$, let $U$ be the dense open set of $S$ consisting of all points at which the log structure is trivial, and let $((\cL, \phi), \iota)$ be a generalized Drinfeld module over $(S,U)$ of rank $d$ with level $N$ structure. Then the condition (div) is satisfied.

\end{sbthm}

Note that Theorem \ref{regdiv} implies Proposition \ref{logD7}, which asserts that a log Drinfeld module with level $N$ structure $\iota$ over an $A$-scheme $S$ with saturated log structure $M_S$ satisfies the condition (div) of \ref{strong}. That is, it is
sufficient to work with a pair $(S',U')$ as in $(S,U)$ of \ref{regdiv} by the definition of log Drinfeld module in \ref{logD6}.

We prove Theorem \ref{regdiv}. Since we can work \'etale locally on $S$ by the remark at the end of \ref{strong}, this theorem follows from the following lemma.

\begin{sblem}\label{1logD10} Let $S$ be a log regular scheme over a field, and assume it is strictly Henselian. Let $s$ be the closed point of $S$.  Let $x, y\in M_{S,s}$, and assume that $x^{-1}+y^{-1}\in \cO_{S,s}\cup M_{S,s}^{-1}$. Then either $xy^{-1}\in M_{S,s}$ or $yx^{-1}\in M_{S,s}$. 
\end{sblem}

\begin{proof} Let $m_s$ be the maximal ideal of $\cO_{S,s}$.  We may assume that $x,y\in m_s$.  If 
$$
	x^{-1}+y^{-1}=c\in \cO_{S,s},
$$ 
then $xy^{-1}=cx-1\in \cO_{S,s}^\times$.  Hence we may assume that $x^{-1}+y^{-1}=z^{-1}$ for some $z\in M_{S,s}$.

The completion $R$ of $\cO_{S,s}$ is isomorphic to $k\ps{P}\ps{T_1,\dots, T_n}$ for some field $k$ and sharp fs monoid $P$ such that the log structure of $\Spec(R)$ is given by $P\to R$ \cite[3.1]{KK}. It is sufficient to prove that if $\gamma_1,\gamma_2,\gamma_3\in P$ and $u_1,u_2,u_3\in R^\times$ are such that 
$$
	u_1\gamma_1^{-1}+u_2\gamma_2^{-1}= u_3\gamma_3^{-1},
$$ 
then either $\gamma_1\gamma_2^{-1}\in P$ or $\gamma_2\gamma_1^{-1}\in P$. Assume this is not satisfied. 

As is well known, if $\gamma\in P^{\gp}$ is such that $f(\gamma)\geq 0$ for the canonical extension of every homomorphism $f \colon P\to \N$, then $\gamma\in P$. Taking $\gamma=\gamma_1\gamma_2^{-1}$ and also $\gamma= \gamma_2\gamma_1^{-1}$, we see that there exist homomorphisms $g,h \colon P\to \N$ 
such that $g(\gamma_1)>g(\gamma_2)$ and $h(\gamma_1)<h(\gamma_2)$. The pair $(g,h)$ induces $P\to \N\times \N$ and hence induces 
a ring homomorphism 
$$
	k\ps{P}\ps{T_1,\dots, T_n}\to k\ps{\N\times \N}=k\ps{t_1,t_2},
$$ 
where each $T_i$ is sent to $0$. Let $a=g(\gamma_1)$, $b=g(\gamma_2)$, $c=h(\gamma_1)$, 
$d=h(\gamma_2)$ so that $a,b,c,d\in \N$ with $a>b$ and $c<d$. Write the image of $\gamma_3$ in $k\ps{t_1,t_2}$  as $t_1^et_2^f$ with $e,f\in \N$. Then there are $v_1,v_2,v_3\in k\ps{t_1,t_2}^\times$ 
such that 
$$
	v_1t_1^{-a}t_2^{-c}+v_2t_1^{-b}t_2^{-d}= v_3t_1^{-e}t_2^{-f}.
$$ 
Thus we have 
$$
	v_1t_2^{d-c}+v_2t_1^{a-b}= v_3t_1^{a-e}t_2^{d-f}
$$ 
with $a-b>0$ and $d-c>0$, which is impossible. 
\end{proof}

\begin{sbrem}\label{notlog}   Without log regularity, the condition (div) in \ref{strong} need not be satisfied.  See \ref{Sdiv}.

\end{sbrem}

\begin{sbprop}\label{unit2} Let   $((\cL, \phi), \iota)$ be 
 a log Drinfeld module of rank $d$ with level $N$ structure over a scheme $S$ over $A$ with a saturated log structure. 
 Suppose  that either $N$ has at least two prime divisors or $N$ is invertible on $S$. 
\begin{enumerate}
\item[(1)] Let $a\in (\frac{1}{N}A/A)^d$ and assume that $a$ is a part of a basis of the free $A/NA$-module $(\frac{1}NA/A)^d$. Then $\iota(a)\in M_S^{-1}\cL^\times$. 
\item[(2)] Locally on $S$, there exists an $a\in (\frac{1}{N}A/A)^d$ which is a part of a basis
of the free $A/NA$-module $(\frac{1}NA/A)^d$ and
such that $\iota(a) \in \cL^\times$.
\end{enumerate}
\end{sbprop}

\begin{proof} Using the definition \ref{logD6} of a log Drinfeld module, this reduces to the case of a generalized Drinfeld module with level $N$ structure over a log regular scheme, and hence to \ref{unit}, as log regular schemes are normal.
\end{proof}

\begin{sbprop}\label{auto2} Under the assumptions of Proposition \ref{unit2}, the automorphism group of $((\cL, \phi), \iota)$ is trivial. 
\end{sbprop}

\begin{proof} This follows from \ref{unit2}(2). \end{proof}

\subsection{Bruhat-Tits buildings}\label{BT}
 
 Like the toroidal compactifications considered in \cite{AMRT} and \cite{FC},  our toroidal compactifications are associated to cone decompositions. Our cone decompositions 
are based on the decomposition of the Bruhat-Tits building of $\PGL_n(F_{\infty})$ into  simplices.

We give an almost self-contained review of Bruhat-Tits buildings for $\PGL_n$ in a style which matches this paper, using the conditions $C(S)$ and $C(S,\mu)$ introduced in \ref{C(S)}.  These conditions are about inequalities of norms, and such inequalities give cones and cone decompositions used in this paper. At the end of this subsection, we prove Proposition \ref{normcv}, via which Bruhat-Tits buildings are related to Drinfeld modules.

Let $E$ be a complete discrete valuation field with finite residue field $\F_q$. 
Let $O_E$ be the valuation ring of $E$, and let $m_E$ be the maximal ideal of $O_E$. Let $|\;\,| \colon E\to \R_{\geq 0}$ be the standard absolute value, sending a uniformizer to $q^{-1}$.
Let $n\geq 1$.

\begin{sbpara} \label{norm} 
Let $V$ be an $E$-vector space of finite dimension $n$. 
By a \emph{norm} on $V$, we mean a map $\mu \colon V\to \R_{\geq 0}$ which satisfies the following two equivalent conditions.
\begin{enumerate}
	\item[(i)] There exists a basis $(e_i)_{1\leq i\leq n}$ of $V$ such that there are $(s_i)_{1 \le i \le n} \in \R_{>0}^d$ 
	for which
	$$\mu\left(\sum_{i=1}^n x_ie_i\right) = \max\{s_i|x_i| \mid 1\leq i\leq n\}$$
for all $x_i\in E$. 
	\item[(ii)] We have
	\begin{enumerate}
		\item[(a)] $\mu(ax)= |a|\mu(x)$ for all $a\in E$ and $x\in V$, 
		\item[(b)] $\mu(x+y)\leq \max(\mu(x), \mu(y))$ for all $x,y\in V$, and
		\item[(c)] $\mu(x)>0$ for all $x\in V \setminus \{0\}$. 
	\end{enumerate}
\end{enumerate}
See \cite{GI} for the equivalence.  Note that in (b), we have $\mu(x+y) = \max(\mu(x),\mu(y))$ if $\mu(x) \neq \mu(y)$.
We will call a basis of $V$ satisfying (i) an {\it orthogonal basis} for $\mu$.
\end{sbpara}

\begin{sbpara} The rank $n$ \emph{Bruhat-Tits building} $|\BT_n|$ is a topological space with underlying set consisting of all homothety classes of norms (\ref{sat}) on the $E$-vector space $E^n$.  The topology on $|\BT_n|$ is defined by the embedding 
$$|\BT_n| \hookrightarrow \prod_{(x,y) \in (E^n-\{0\})^2} \R_{>0}, \quad \mu\mapsto  \mu(x)\mu(y)^{-1},$$ where for $\mu \in |\BT_n|$, the expression $\mu(x)\mu(y)^{-1}$ means $\tilde \mu(x)\tilde \mu(y)^{-1}$ for a norm $\tilde \mu$ with class $\mu$. 

The group $\PGL_n(E)$ acts continuously on $|\BT_n|$. That is, the class $g\in \PGL_n(E)$ of  $\tilde g\in \GL_n(E)$ sends the class $\mu\in |\BT_n|$  of a norm $\tilde \mu$ to the class $g\mu\in |\BT_n|$ of the norm $x\mapsto \tilde \mu(\tilde g^{-1}(x))$. 

\end{sbpara}

\begin{sbpara}\label{BTn}

The Bruhat-Tits building $|\BT_n|$ is identified with the geometric realization of a simplicial complex $\BT_n$ which is also called the Bruhat-Tits building.

The set of $0$-simplices of $\BT_n$ will be introduced in \ref{0sim}. In \ref{sim}, we will
 introduce general simplices of $\BT_n$ by using a condition $C(S)$ in \ref{C(S)}, and for each simplex $S$ of $\BT_n$, we will introduce its geometric realization $|\BT_n|(S)\subset |\BT_n|$ by using   a condition $C(S, \mu)$ in \ref{C(S)}. $|\BT_n|$ is the union of these $|\BS_n|(S)$

\end{sbpara}

\begin{sbpara}\label{0sim}

A \emph{$0$-simplex} of $\BT_n$ is the class of an $O_E$-lattice in $E^n$. Here, by an \emph{$O_E$-lattice in $E^n$}, we mean a finitely generated $O_E$-submodule of $E^n$ which generates $E^n$ over $E$. Two $O_E$-lattices $L$ and $L'$ are said to be equivalent if $L'=aL$ for some $a\in E^\times$. 

For an $O_E$-lattice $L$ in $E^n$, we have the associated norm $\mu_L$ on $E^n$ defined as
$$\mu_L(x) = \min\{|a|\mid a\in E, x\in aL\}.$$ 
Via the map $\class(L)\mapsto \class(\mu_L)$, we regard a $0$-simplex of $\BT_n$ as an element of $|\BT_n|$.

\end{sbpara}

\begin{sbpara}\label{C(S)}  We introduce conditions $C(S)$ and $C(S,\mu)$.
For $\mu \in |\BT_n|$ and $x, y \in E^n$, we write $\mu(x)\geq \mu(y)$ to indicate that $\tilde \mu(x)\geq \tilde \mu(y)$ for a norm $\tilde \mu$ with class $\mu$. 
 
 For a subset $S$ of $|\BT_n|$, the condition $C(S)$ is as follows.

\begin{description}
\item[$C(S)$] 
For each pair $(x,y)$ of nonzero elements of $E^n$, we have either $\mu(x)\geq \mu(y)$ for all $\mu\in S$ or $\mu(x)\leq \mu(y)$ for all $\mu\in S$. 
\end{description}

For a subset $S$ of $|\BT_n|$ and $\mu\in |\BT_n|$, the condition $C(S,\mu)$ is as follows. 

\begin{description}
\item[$C(S, \mu)$] For any pair $(x,y)$ of nonzero elements of $E^n$ such that $\mu'(x)\geq \mu'(y)$ for all $\mu'\in S$, we have $\mu(x)\geq \mu(y)$. 
\end{description}
\end{sbpara}

\begin{sbpara}\label{sim} A \emph{simplex} of $\BT_n$ is a nonempty set $S$ of $0$-simplices of $\BT_n$ satisfying the condition $C(S)$. 
It follows from \ref{CS1}(ii) below that a simplex of $\BT_n$ is a finite set. A simplex is called  an \emph{$r$-simplex} if it has cardinality $r+1$. 

The \emph{geometric realization} of a simplex $S$ of $\BT_n$ is 
$$
	|\BT_n|(S) = \{ \mu\in |\BT_n| \mid C(S,\mu) \text{ is satisfied} \}.
$$ 

\end{sbpara}
The traditional definition of a simplex of $\BT_n$ is given by the following \ref{CS1}(ii).

\begin{sbprop}\label{CS1}
For any nonempty set $S$ of $0$-simplices of $\BT_n$, the condition $C(S)$ is equivalent to each of the following two conditions:
\begin{enumerate}
\item[(i)] For any two $O_E$-lattices $L$ and $L'$ in $E^n$ with classes in $S$, we have either $L\subset L'$ or $L'\subset L$. 
\item[(ii)] There are $O_E$-lattices $L^0, \dots, L^r$ in $E^n$ such that $S$ is the set of equivalence classes of $L^0, \dots, L^r$  and $L^0\supsetneq L^1\supsetneq \dots \supsetneq L^r \supsetneq m_EL^0$. 
\end{enumerate}
\end{sbprop}

\begin{proof}
$C(S) \Rightarrow$ (i): Assume that there are $O_E$-lattices $L, L'$  in $E^n$ with classes in $S$ such that $L\not\subset L'$ and $L'\not\subset L$. Take $x, y\in E^n$ such that $x\in L$, $x\notin L'$, $y\in L'$, $y\notin L$. Then $\mu_L(x)<\mu_L(y)$ and $\mu_{L'}(x)>\mu_{L'}(y)$, so $C(S)$ fails.

(i) $\Rightarrow $ (ii): Fix an $O_E$-lattice $L^0$ in $E^n$ with class in $S$. For each element $\mu$ of $S$, take the $O_E$-lattice $L(\mu)$ in $E^n$ with class $\mu$ such that  
$L(\mu)\subset L^0$ and $L(\mu) \not\subset m_EL^0$, By assumption, the set $\{L(\mu)\mid \mu \in S\}$ is totally ordered by inclusion, and $L(\mu) \supsetneq m_EL^0$ for all $\mu \in S$ by definition. 

(ii) $\Rightarrow C(S)$: Take $O_E$-lattices $L^0, \dots, L^r$ satisfying (ii). Let $L^{i+(r+1)t}= m_E^tL^i$ for $0\leq i\leq r$ and $t\in \Z$. Then the $L^i$ for $i\in \Z$ form a decreasing filtration on $E^n$, and $\bigcup_{i\in \Z} L^i=E^n$. Let $x,y \in E^n \setminus \{0\}$. Let $i,j\in \Z$ be such that $x\in L^i$, $x\notin L^{i+1}$, $y\in L^j$, $y\notin L^{j+1}$. If $i\leq j$ (resp., $i\geq j$),  we have $\mu_{L^m}(x)\geq \mu_{L^m}(y)$ (resp., $\mu_{L^m}(x) \leq \mu_{L^m}(y)$) for all $m\in \Z$. 
\end{proof}

\begin{sbpara}  The group $\PGL_n(E)$ acts on $\BT_n$ in the natural way, and this action is compatible with the action of $\PGL_n(E)$ on $|\BT_n|$. For each $r$, the  action of $\PGL_n(E)$ on the set of all $r$-simplices in $\BT_n$ is transitive.

\end{sbpara}

\begin{sblem}\label{linnor} Let $(\mu_i)_{i\in I}$ be a nonempty finite family of norms on $E^n$ such that the set $S$ of the classes of $\mu_i$ satisfies $C(S)$. Let $a_i\in \R_{>0}$ for each $i\in I$. Then $\sum_{i \in I} a_i\mu_i$ is a norm on $E^n$. 
\end{sblem}

\begin{proof}
Set $\mu = \sum_{i \in I} a_i\mu_i$.
Let $x$ and $y$ be nonzero elements of $E^n$. We prove that $\mu(x+y)\leq \max(\mu(x), \mu(y))$. We may assume that $\mu_i(x)\geq \mu_i(y)$ for all $i$. We then have 
$$
	\mu(x+y)=\sum_{i\in I} a_i\mu_i(x+y) \leq \sum_{i \in I} a_i \mu_i(x)= \mu(x).
$$ 
\end{proof}

\begin{sbprop}\label{CS2}  Let $S$ be a simplex of $\BT_n$, and let $\mu\in |\BT_n|$. Then $\mu \in |BT_n|(S)$ if and only if 
for norms $\tilde \mu$ with class $\mu$ and $\tilde \nu$ with class $\nu$ for each $\nu \in S$, there exist $a_\nu \in \R_{\geq 0}$ for $\nu \in S$ such that
$\tilde \mu= \sum_{\nu \in S}  a_\nu \tilde \nu$.
\end{sbprop}

\begin{proof} 
Let us suppose $\mu\in |\BT_n|(S)$, the opposite direction being clear.
Let $L^i$ ($i\in \Z$) be as in  (ii) in \ref{CS1}. For $m\in \Z$, let $Y_m = L^m \setminus L^{m+1}$.
 Let $P$ be the set of all maps $E^n \to \R$ such that $f(0)=0$, $f(x)=f(y)$ if $x, y\in Y_m$ for some $m \in \Z$, and $f(ax)= |a|f(x)$ if $a\in E$ and $x\in E^n$. For $m\in \Z$, take an element $x_m$ of $Y_m$. 
 \medskip
 
 {\bf Claim 1.} We have $\tilde \mu\in P$ and $\tilde \mu(x_m)\geq \tilde \mu(x_{m+1})$ for all $m \in \Z$. 

\begin{proof}[Proof of Claim 1.] If $x,y\in Y_m$, then $\mu_{L^i}(x)=\mu_{L^i}(y)$ for all $i$, and hence the condition $C(S, \mu)$ shows that $\tilde \mu(x)= \tilde \mu(y)$. Furthermore, since $\mu_{L^i}(x_m) \geq \mu_{L^i}(x_{m+1})$ for all $i$, the condition $C(S, \mu)$ shows that  $\tilde\mu(x_m)\geq \tilde \mu(x_{m+1})$. \phantom\qedhere
\end{proof}

 {\bf Claim 2.} If $f\in P$, then $f= \sum_{i=0}^r a_i\mu_{L^i}$ with $a_i= (f(x_{i-1})-f(x_i))(q-1)^{-1}$.
  
\begin{proof}[Proof of Claim 2.] We have an isomorphism   
$$
	P\xrightarrow{\sim} \R^{r+1}, \quad f \mapsto (f(x_m))_{0\leq m\leq r}.
$$ 
On the other hand,   $\mu_{L^i}(x_m)$ is $q^c$ where $c$ is the smallest integer such that $c\geq (i-m)(r+1)^{-1}$. Hence for $0\leq i\leq r$ and $0\leq m\leq r$, the difference $\mu_{L^i}(x_{m-1})- \mu_{L^i}(x_m)$ is $q-1$ if $i=m$ and $0$ if $i\neq m$.
Hence we see that $f(x_{m-1}) - f(x_m) = \sum_{i=0}^r a_i (\mu_{L^i}(x_{m-1})- \mu_{L^i}(x_m))$, and the expression for $f$ then follows from the fact that $f(x_m) = q^jf(x_{m+j(r+1)})$
for all $j$.
 \phantom\qedhere
\end{proof}

Since $\tilde{\mu} \in P$ by Claim 1, the result follows from Claim 2.
\end{proof}

 \begin{sblem}\label{Ss}  Let $\nu$ be a $0$-simplex of $\BT_n$.
 \begin{enumerate}
 \item[(1)] We have $|\BT_n|(\{\nu\})= \{\nu\}$. 
 \item[(2)] For a simplex $S$ of $\BT_n$, we have $\nu\in |\BT_n|(S)$ if and only if $\nu\in S$. 
 \end{enumerate}
 \end{sblem}
 
\begin{pf}
Part (1) follows from \ref{CS2} applied to $\{\nu\}$, and the ``if'' statement in (2) is clear. Now assume $\nu \in |\BT_n|(S)$. Then $C(S \cup \{\nu\})$ is satisfied and hence $S\cup \{\nu\}$ is a simplex. Hence it is sufficient to prove that if $S' = S \cup \{\nu\}$ is a simplex and $\nu\notin S$, then $C(S'\setminus \{\nu\}, \{\nu\})$ is not satisfied. This is seen from the presentation of $S'$ as in \ref{CS1}(ii) (in which $S$ is replaced by $S'$). 
\end{pf}

\begin{sbpara}\label{AP1} Let $|\AP_n|\subset |\BT_n|$  be the set of all classes of norms $\mu$ on $E^n$ such that the standard base $(e_i)_{1\leq i\leq n}$ of $E^n$ is an orthogonal base for $\mu$. We have a canonical bijection $$ |\AP_n|\cong \R^n_{>0}/\R_{>0},$$
where the multiplicative group $\R_{>0}$ acts on $\R^n_{>0}$ diagonally,
which sends the class of a norm $\mu$ which belongs to $|\AP_n|$  to the class of $(\mu(e_i))_i$. The inverse map sends 
the class of $s\in \R^n_{>0}$ to the class of the norm $\mu_s$ defined by 
$$\mu_s(x)= \max\{ s_i|x_i| \mid 1 \le i \le n \} \;\textsp{for} x\in E^n.$$
This 
$|\AP_n|\subset |\BT_n|$ is called  the \emph{standard apartment} in $|\BT_n|$.

The map 
$$
	\PGL_n(E) \times \R^n_{>0}\to |\BT_n|, \quad (g, s) \mapsto g\class(\mu_s)
$$ 
is surjective, as is seen from the condition (i) on a norm in \ref{norm}.  
The topology of $|\BT_n|$ is also described as the  quotient topology of the topology of 
$\PGL_n(E)\times \R_{>0}^n$
under this surjection.
\end{sbpara}

\begin{sbpara}\label{apcone1}

 For an $O_E$-lattice $L$ in $E^n$, the class of $\mu_L$ in $|\BT_n|$ belongs to $|\AP_n|$ if and only if $L=\bigoplus_{i=1}^n m_E^{a_i} e_i$ for some $a_i \in \Z$ where $(e_i)_i$ denotes the standard basis of $E^n$. For this $L$, we have $\mu_L=\mu_s$ where $s=s(L) :=  (q^{a_1}, \dots, q^{a_n})\in \R^n_{>0}$. Let $\AP_n$ be the simplicial subcomplex of $\BT_n$ with simplicies those $S$ such that every element of $S$ is the class of such an $L$. 

For a simplex $S$ of $\AP_n$,  let $$\sig(S) = \{s\in \R^n_{>0} \mid \text{class}(\mu_s) \in |\BT_n|(S)\}\cup \{0\}.$$ 
This has the following two properties:
\begin{enumerate}
\item[(1)]  For any diagonal matrix $h$ in $\GL_n(E)$ with $i$th entries $a_i\in E^\times$, we have $g\sig(S)=\sig(hS)$, where
$g \in \GL_ n(\R)$ is the diagonal matrix $g$ with $i$th entry $|a_i|$.
\item[(2)] For any permutation matrix $g$, we have $g\sig(S)=\sig(gS)$.
\end{enumerate}
\end{sbpara}

\begin{example}\label{ExC(S)} By the \emph{standard $(n-1)$-simplex} of $\BT_n$, we mean the $(n-1)$-simplex $S$ consisting of the classes of the $O_E$-lattices $L^i$ ($0\leq i\leq n-1$) defined by
$$L^i = \sum_{j=1}^{n-i} O_E e_j + \sum_{j=n-i+1}^n m_E e_j.$$
Note that
$$\sum_{i=1}^n O_E e_i = L^0 \supset L^1\supset \dots \supset L^{n-1}\supset m_EL^0.$$
We make the following remarks.
\begin{enumerate}
\item[(1)] We have $\sig(S)= \{s\in \R^n_{\geq 0}\mid s_1\leq s_2\leq \dots \leq s_n \leq qs_1\}.$ 
\item[(2)] Let $J$ be a nonempty subset of $\{1,\dots, n\}$, and let $S'\subset S$ be the simplex consisting of all classes of $L^i$ such that $i\in J$. Then $\sig(S')$ consists of all $s\in \sig(S)$ satisfying the following two conditions:
\begin{enumerate}
\item[(i)]  if $2 \leq i\leq n$ and $i\notin J$, then $s_{i-1}=s_i$, and
\item[(ii)] if $1\notin J$, then $s_n=qs_1$. 
\end{enumerate}
\item[(3)] Every simplex in $\BT_n$ is written as $gS'$ for some $g\in \PGL_n(E)$ and $S'$ as above. 

\item[(4)] Let $T$ be the set of diagonal matrices in $\PGL_n(E)$, and let $N_T$ be the normalizer of $T$ in $\PGL_n(E)$. That is, $N_T$ is the subgroup of $\PGL_n(E)$ generated by $T$ and the permutation matrices. Then every simplex in $\AP_n$ is written as $gS'$ for some $g\in N_T$  and $S'$ as above. 
\end{enumerate}
Statements (1) and (2) follow from \ref{CS2}, whereas (3) and (4) are clear by the understanding of a simplex as in \ref{CS1}(ii). 
\end{example}

\begin{sbprop}\label{apcone0} Let $S$ be a simplex in $AP_n$.
\begin{enumerate}
\item[(1)] The set $\sig(S)$
is the subcone of $\R^n_{>0}\cup\{0\}$ generated by $s(L)\in \R^n_{>0}$  for the $O_E$-lattices $L$ in $E^n$ whose classes belong to $S$. 
\item[(2)] The map taking $s \in \R^n_{> 0}$ to $\class(\mu_s)$ restricts to a bijection $\sig(S)\setminus \{0\}\to 
 |\BT_n|(S)$.
\end{enumerate}
\end{sbprop}

\begin{pf} Part (1) is reduced to \ref{ExC(S)}(2) by \ref{ExC(S)}(4). 
Part (2) follows from \ref{CS2} and \ref{ExC(S)}(2).
\end{pf}

\begin{sbprop}\label{apcone2}  Let $I=\{(h,i,j)\in \Z^3\mid 1\leq j \le i \leq n\}$. For each map $\alpha \colon I \to \{\R_{\leq 0}, \{0\}, \R_{\geq 0}\}$ satisfying the condition that
\begin{itemize}
\item for every $(i,j) \in \Z^2$ with $1\leq j \le i\leq n$, there exists $h \in \Z$ such that $\alpha(h,i,j) \neq \R_{\leq 0}$ and $\alpha(h-1, i, j) \neq \R_{\geq 0}$,
\end{itemize}
define 
$$c(\alpha) = \{s\in \R_{\geq 0}^n\mid q^hs_j -s_i\in \alpha(h,i,j)\textsp{for all} (h,i,j)\in I\}.$$
Then the set $\Sig$ of cones $c(\alpha)$ for such maps $\alpha$ is a cone decomposition of $\R_{> 0}^n\cup\{0\}$, and 
the map $S\mapsto \sig(S)$ is a bijection from $\AP_n$ to $\Sig\setminus \{0\}$. 
\end{sbprop} 

This cone decomposition of  $\R^n_{> 0}\cup\{0\}$
will be the base of the cone decompositions  for our toroidal compactifications. 

\begin{pf}
That $\sig(S) \in \Sig\setminus \{0\}$ for $s \in \AP_n$ follows from (2) and (4) of \ref{ExC(S)}. It is clear that the map $S\mapsto \sig(S)$ is injective. We prove that it is surjective. 

By (1) (resp., (2)) in \ref{apcone1}, if $\tau\in \Sig\setminus \{0\}$ and $g \in \GL_n(E)$ is a diagonal matrix with entries in $q^{\Z}$ (resp., a permutation matrix), then $\tau$ is in the image of $S\mapsto \sig(S)$ if and only if $g\tau$ is.  
Let $\tau\in \Sig\setminus \{0\}$. For each $2\leq i\leq n$, let $h(i)$ denote the smallest integer such that $s_1\leq q^{h(i)}s_i$ for all $s\in \tau$. Since $\tau \in \Sig$, we then have $s_1 \geq q^{h(i)-1}s_i$ for all $s\in \tau$. Replacing $\tau$ by $g\tau$ for a diagonal matrix $g$ with entries in $q^{\Z}$,  we may assume that $h(i)=0$ for all $2\leq i\leq n$. Then for $i,j\in \{2,\dots,n\}$,  $s_i \leq qs_1\leq qs_j$. Hence we have either $s_i \geq s_j$ for all $s\in \tau$ or $s_i\leq s_j$ for all $s\in \tau$. By replacing $\tau$ with $g\tau$ for a permutation  $g$ on $\{2,\dots,n\}$,  we may assume that $s_1\leq s_2\leq \dots \leq s_n\leq qs_1$ for all $s\in \tau$. Then $\tau=\sig(S')$ for some $S'$ in \ref{ExC(S)}. 

 We claim that $\Sig$ is a cone decomposition of $\R^n_{>0}\cup \{0\}$. It is clear that $\R_{>0}^n \cup \{0\}$ is the union of the
 cones in $\Sig$, and any two cones in $\Sig$ intersect in a face of each cone. It remains only to show that each face of a given
 $\tau \in \Sig$ is also in $\Sig$. We have $c(\alpha) = \{0\}$ for many choices of $\alpha$, e.g., if $\alpha(1,i,i) = \R_{\leq 0}$ for all $i$.
 Suppose that $\tau \neq \{0\}$. 
 Note $\tau=\sig(S)$ for some $S\in \AP_n$ as we have seen above. By \ref{ExC(S)}(4) and (1) and (2) of \ref{apcone1},
 we may  assume that $\tau=\sig(S')$, where $S'$ is as in \ref{ExC(S)}(2), and for such a $\tau$, 
 the desired property holds.
\end{pf}

\begin{example}\label{APn=2}
In the case $n=2$, the cone decomposition of $\R^2_{> 0} \cup \{0\}$ given by $\sig(S)$ for simplices $S$ of $\AP_2$ consists of the cones 
$$
 	\sig(S_h)= \{s\in \R^2_{\geq 0}\mid q^{h-1}s_1\leq s_2\leq q^hs_1\},
$$
for $h \in \Z$, where 
$$
	S_h =\{\class(O_E e_1+m_E^{h-1} e_2), \class(O_E e_1+ m_E^he_2)\},
$$ 
and the cones 
$$
	\sig(\{\class(L_h)\})=\{s\in \R^2_{\geq 0}\mid s_2=q^hs_1\}
$$
for $h \in \Z$, where $L_h= O_E e_1+ m_E^hO_E e_2$, as well as the zero cone.

\end{example}

\begin{sblem}\label{1CS}  Let $\mu\in |\BT_n|$ and let $S_{\mu}$ be the set of all $0$-simplices $\nu$ of $\BT_n$ such that $C(\{\mu\}, \nu)$ is satisfied.
Then
\begin{enumerate}
\item[(1)] $S_{\mu}$ is the unique simplex $S$ such that any norm $\tilde \mu$ with class $\mu$ may be written as a sum $\sum_{\nu \in S} a_{\nu} \tilde{\nu}$ with $a_{\nu} \in \R_{> 0}$,
where $\tilde{\nu}$ is a norm with class $\nu$, and
\item[(2)] $S_{\mu}$ is the smallest simplex $S$ such that $\mu\in |\BT_n|(S)$. 
\end{enumerate}
\end{sblem}

\begin{pf} We prove first that
$\mu\in |\BT_n|(S)$ for some simplex $S$, which by \ref{CS2} then contains a simplex satisfying the condition on $S$ in (1).
Since the map
$$
 	\PGL_n(E) \times |\AP_n|\to |\BT_n|, \quad (g, \mu)\mapsto g\mu
$$ 
is surjective, we are reduced to the case  $\mu\in |\AP_n|$ and hence to \ref{apcone0}(2) and \ref{apcone2}. 

It remains to prove that if $S$ is a simplex satisfying the condition in part (1), then $S=S_{\mu}$. 
Let $L^i$ for $i \in \Z$ be as in the proof of \ref{CS1}, and set $Y_i = L^i \setminus L^{i+1}$ as in the proof of \ref{CS2}. If $x\in Y_i$ and $y\in Y_j$, then  $\mu(x)\geq \mu(y)$ if and only if $i\leq j$. For $\mu'\in |\BT_n|$, the condition $\mu' \in |\BT_n|(S)$ (which is to say that $C(S,\mu')$ is satisfied) is equivalent to the following condition:
\begin{itemize}
	\item If $x\in Y_i$ and $y\in Y_j$ and $i\leq j$, then $\mu'(x) \geq \mu'(y)$.
\end{itemize}

On the other hand, $C(\{\mu\},\mu')$ is also equivalent to this condition as a consequence of \ref{CS2}. That is, $C(\{\mu\}, \mu')$ and $C(S, \mu')$ are equivalent. Consequently, all $\nu \in S_{\mu}$ satisfy $C(S, \nu)$ and hence lie in $S$ by \ref{Ss}(2). Conversely, all $\nu \in S$ satisfy $C(S, \nu)$ trivially and therefore also $C(\{\mu\}, \nu)$, so are contained in $S_{\mu}$ by its definition. Therefore $S = S_{\mu}$, as claimed.
\end{pf}

\begin{sbprop}\label{2CS} Let $S$ be a subset of $|\BT_n|$ and let $T= \bigcup_{\mu\in S} S_{\mu}$, where $S_{\mu}$ is as in \ref{1CS}.
\begin{enumerate}
\item[(1)] The condition $C(S)$ is satisfied if and only if $T$ is a simplex.
\item[(2)] If the condition $C(S)$ is satisfied, then $S\subset |\BT_n|(T)$. 
\item[(3)] The condition $C(S)$ is satisfied if and only if $S \subset |\BT_n|(S')$ for some simplex $S'$. 
\end{enumerate}
\end{sbprop}

\begin{proof}
If $C(S)$ is satisfied, then  $C(S\cup T)$ is satisfied, and hence $C(T)$ is satisfied. Conversely, assume $T$ is a simplex. Then for each $\mu \in S$, since $\mu \in |\BT_n|(S_{\mu})$ by \ref{1CS}(2), we have  $\mu\in |\BT_n|(T)$. Hence $S\subset |\BT_n|(T)$ and hence $C(S)$ is satisfied. 
Thus (1) and (2) hold, and (3) follows from (1) and (2). 
\end{proof}

\begin{sbprop} Let $S$ be a subset of $|\BT_n|$ satisfying $C(S)$ and let $\mu\in |\BT_n|$. Let $T$ be the simplex defined in \ref{2CS}. Then  
$C(S, \mu)$ is satisfied  if and only if $C(T,\mu)$ is satisfied.  

\end{sbprop}

\begin{proof}
 If $C(S, \mu)$ is satisfied, then $C(T,\mu)$ is satisfied since $S\subset |\BT_n|(T)$ by \ref{2CS}(2). 
Conversely, if $C(T,\mu)$ is satisfied, then $C(S, \mu)$ is satisfied since all elements $\nu$ of $T$ satisfy $C(S, \nu)$, 
\end{proof}

\begin{sbprop}\label{normcv}  Let $\cV$ be a complete valuation ring of height one over $A$ as in \ref{val1}. 
Let $\psi$ be a Drinfeld module over $\cV$ of rank $r$ with trivial line bundle, and let $\La$ be a $\psi(A)$-lattice in $K^{\sep}$.
Then there exists a unique norm $\mu$ on the $F_{\infty}$-vector space $F_{\infty}\otimes_A \La$, where $A$ acts on $\La$ via $\psi$, having the property that if $z \in \bar K$ and $a\in A\setminus \{0\}$ are such that $\psi(a)z$ is a nonzero element of  $\La$, then
$$
	-v_{\bar K}(z)=|a|^{-r}\mu(\psi(a)z)^r.
$$ 
In particular, we have $v_{\bar K}(\psi(a)z) = |a|^r v_{\bar K}(z)$ whenever $\psi(a)z \in \Lambda \setminus \{0\}$.
\end{sbprop}

\begin{proof}
Let $a\in A$ and $z\in {\bar K}$, and write 
$$
	\psi(a)(z)=\sum_{i=1}^m a_iz^i \quad \text{with}\;\; a_1=a, \; m= |a|^r, \; a_m \in \cV^\times.
$$ 
If $v_{\bar K}(z)<0$, then since $a_i\in \cV$ for each $i$ and $a_m \in \cV^\times$, we have $\psi(a)(z)=u z^{|a|^r}$ for some $u\in \cV^\times$, so 
$$
	v_{\bar K}(\psi(a)z)= |a|^rv_{\bar K}(z).
$$   

Since $v_{\bar K}$ takes negative values on nonzero elements of $\Lambda$ by Lemma \ref{pre0}, there is then a well-defined map $\mu \colon \La\to \R_{\ge 0}$ given by $\mu(\lam) = (-v_{\bar K}(\la))^{1/r}$ for $\lam\neq 0$ and by $\mu(0)=0$ and satisfying $\mu(\psi(a)\la)= |a|\mu(\la)$ for all $a \in A$ and $\la\in \La$.  This map $\mu$ uniquely extends to a map  $\mu \colon F\otimes_A \La \to \R_{\geq 0}$ such that $\mu(a\otimes \la)= |a|\mu(\la)$ for all $a\in F$ and $\la\in \La$, and in turn, it extends uniquely to a continuous map  $\mu \colon F_{\infty}\otimes_A \La\to \R_{\geq 0}$. This $\mu$ satisfies the condition (ii) of a norm in \ref{norm}. 

If $z\in \bar K$ and $a\in A$ are such that $\psi(a)z\in \La\setminus \{0\}$, then $v_{\bar{K}}(\psi(a)z) < 0$ by \ref{pre0}, which forces $v_{\bar K}(z) < 0$.  We then have 
$$
	|a|^rv_{\bar K}(z) = v_{\bar K}(\psi(a)z)= -\mu(\psi(a)z)^r.
$$ 
The uniqueness of the norm $\mu$ is evident. 
\end{proof}

\subsection{Norm-ordered bases} \label{normord}

In this subsection, we suppose that $A = \F_q[T]$. We first study of orthogonal bases of free $A$-modules such that each element has minimal norm among elements not in the $A$-span of the prior elements. We use this to characterize the orbits of the Bruhat-Tits building $|\BT_n|$ under $\PGL_n(A)$. 

A large part of the following proposition is contained in Lemma 4.2 of \cite{taguchi}.
It will be applied to the study of generalized Drinfeld modules over $A$, through \ref{val2} and \ref{normcv}.  

\begin{sbprop}\label{diagbase} 

Let $\La$ be a free $A$-module of finite rank $n$, and let $\mu$ be a norm on $F_{\infty}\otimes_A \La$. 
\begin{enumerate}
	\item[(1)] For a family $(\lam_i)_{1\leq i\leq n}$ of elements of $\La$, the following three properties are equivalent.
	\begin{enumerate}
	         \item[(i)] The collection $(\lam_i)_i$ is an $A$-basis of $\La$ 
	         such that $\mu(\la_i)\leq \mu(\la_{i+1})$ for $1\leq i\leq n-1$
	         and is orthogonal as an $F_{\infty}$-basis of 
	         $F_{\infty}\otimes_A \La$ in the sense of \eqref{norm}.   
		\item[(ii)] For each $i$ with $1\leq i\leq n$, we have $\lam_i\notin \La_{i-1} :=  \sum_{j=1}^{i-1} A\lam_j$, and $\lam_i$ has the 
		smallest norm under $\mu$ among elements of $\La \setminus \La_{i-1}$. 
		 \item[(iii)] The collection $(\lam_i)_i$ is an $A$-basis of $\La$ 
	         such that $\mu(\la_i)\leq \mu(\la_{i+1})$ for $1\leq i\leq n-1$ and such that for every $1\leq i\leq n$ and every $a_j\in A$ with 
	         $1\leq j\leq i-1$, we have
	         $\mu(\la_i) \leq \mu(\la_i+ \sum_{j=1}^{i-1} a_j\la_j)$. 
 	\end{enumerate}
 	\item[(2)] There is a family $(\lam_i)_{1\leq i\leq n}$ satisfying the equivalent conditions in (1).
	\item[(3)] Let $(\la_i)_{1\leq i\leq n}$ be a family of elements of $\La$ satisfying the equivalent 
	conditions in (1). Let $\la_i'= \sum_{j=1}^n  a_{ij}\la_j$ for $1\leq i \leq  n$, where $a_{ij}\in A$. 
	Then $(\la_i')_{1\leq i\leq n}$   satisfies  the equivalent conditions in (1) if and only if the following two conditions are satisfied:
	\begin{enumerate}
	\item[(i)]  $\mu(a_{ij}\la_j)\leq \mu(\la_i)$ for all $i,j$. In particular, $a_{ij}=0$ if $\mu(\la_j)>\mu(\la_i)$. 
        \item[(ii)] Let $0\leq k< k+m\leq n$ and assume that $\mu(\la_i)=\mu(\la_{k+1})$ if and only if $ k+1\leq i\leq k+m$. 
	Then the matrix 
	$(a_{k+i, k+j})_{1\leq i,j\leq   m}$, which belongs to $M_m(\F_q)$ by (i), belongs to $\GL_m(\F_q)$.
        \end{enumerate}
        \item[(4)]  If $(\la_i)_{1\leq i\leq n}$ and $(\la'_i)_{1\leq i\leq n}$ are families of elements of $\La$ satisfying the equivalent 
        conditions in (1), we have $\mu(\la_i)=\mu(\la'_i)$ for $1\leq i\leq n$. 
 \end{enumerate}
 \end{sbprop}
 
\begin{proof} 
	It is clear that for part (2) we can find $(\lambda_i)_i$ satisfying the condition (ii) of part (1) recursively.  
	The implications (i) $\Rightarrow$ (ii) and (i) $\Rightarrow$ (iii) of part (1) are easily seen.  
	We prove the implications (ii) $\Rightarrow$ (i) and (iii) $\Rightarrow$ (i). Let $ (\lam_i)_i$ be as in either
	 (ii)  or (iii).  Let 
	$$
		\La'= \left(\sum_{i=1}^{n-1} F\lam_i\right) \cap \La.
	$$ 
	By induction on $n$, we may suppose that the tuple $(\lam_i)_{1\leq i\leq n-1}$ is an $A$-basis of $\La'$ and is an orthogonal basis 
	for the restriction of $\mu$ to $F_{\infty} \otimes_A \La'$. Furthermore, we have
        $\mu(\lambda_{i-1}) \le \mu(\lambda_i)$ for $2\leq i \le n$.

 \medskip
 
 {\bf Claim 1.} $(\lam_i)_{1\leq i\leq n}$ is an orthogonal basis for $\mu$. 
 
 \medskip
 
We prove Claim 1.  
Given $(x_i)_{1\le i \le n} \in F_{\infty}^n$, we must show that
$$
	\mu\left(\sum_{i=1}^n  x_i\lam_i\right)=\max\{|x_i|\mu(\lam_i)\mid  1\leq i\leq n \}.
$$
By induction, we have this if $x_n = 0$, so we may assume $x_n \neq 0$.  By replacing $x_i$ with $x_n^{-1}x_i$ for $i \le n-1$, we may further assume that $x_n = 1$.  By induction, we are quickly reduced to the case that
$$
	\mu(\lam_n) =  \max\{|x_i|\mu(\lam_i)\mid 1\leq i\leq n-1 \},
$$
since the result otherwise follows quickly from the triangle inequality for the norm.  Let us assume this equality holds.

Since
$F_{\infty}= A + m_{\infty}$, where $m_{\infty}$ is the maximal ideal of the valuation ring of $F_{\infty}$, 
we can write $x_i= a_i+r_i$ for some $a_i \in A$ and $r_i\in m_{\infty}$ for all $1 \le i \le n-1$.
We then have 
$$
	\sum_{i =1}^n x_i\lam_i= \left(\sum_{i=1}^{n-1} r_i\lam_i\right) + \left(\lam_n+\sum_{i=1}^{n-1} a_i\lam_i\right).
$$  
Note that
$$
	\mu\left(\sum_{i=1}^{n-1} r_i\lam_i\right) = \max\{ |r_i| \mu(\lam_i) \mid 1 \le i \le n-1\} <  \mu(\lam_n).
$$
By the condition (ii) or the condition (iii), we also have 
$$
	\mu\left(\lam_n +\sum_{i=1}^{n-1} a_i\lam_i\right) \geq \mu(\lam_n),
$$ 
and it follows from these two inequalities that
$$
	\mu\left(\sum_{i=1}^n x_i \lambda_i\right) = \mu\left(\lam_n +\sum_{i=1}^{n-1} a_i\lam_i\right).
$$
On other hand, for $1\leq i\leq n-1$ such that $a_i\neq 0$, we have
$\mu(a_i\lam_i)= \mu(x_i\lam_i)$. Hence 
$$
	\mu\left(\lam_n +\sum_{i=1}^{n-1} a_i\lam_i\right)\leq 
	\max(\mu(\lam_n), \max\{|x_i|\mu(\lam_i)\mid 1\leq i\leq n-1\}) = \mu(\lam_n).
$$ 
Therefore, we have 
$\mu(\sum_{i=1}^n x_i \lambda_i) =  \mu(\lam_n)$,	
as desired. This proves (iii) $\Rightarrow$ (i). For the proof of (ii) $\Rightarrow$ (i), it remains to prove
 
\medskip

{\bf Claim 2.} $(\lam_i)_{1\leq i\leq n}$ is an $A$-basis of $\La$. 

\medskip

We prove Claim 2.  We can write any nonzero $\lam \in \Lam$ as 
$$
	\lam= \left(\sum_{i=1}^{n-1} a_i\lam_i\right)+b\lam_n
$$ 
with $a_i\in A$ and $b\in F$. Suppose $b \notin A$, and write $b=a_n+r$ with $a_n\in A$ and $r \in F$ such that $|r|<1$. Replacing $\lam$ by $\lam-\sum_{i=1}^n a_i\lam_i$, we may assume that $\lam=r\lam_n$. Then $\mu(\lam)<\mu(\lam_n)$. This contradicts condition (ii).

The proof of (3) is straightforward, and (4) follows from (3). 
 \end{proof}

We shall refer to a basis of $\La$ as in \ref{diagbase} satisfying the equivalent properties of \ref{diagbase}(1) 
with respect to a given norm $\mu$ as \emph{norm-ordered}.

\begin{sbprop}\label{diagbase2} The map $\R^n_{>0}\to |\AP_n|$ given by $s\mapsto \class(\mu_s)$, where $\mu_s$ is as in \ref{AP1}, induces a bijection
$$\{s\in \R^n_{>0}\mid s_1\leq \dots \leq s_n\}/\R_{>0} \xrightarrow{\sim} \PGL_n(A)\bs |\BT_n|.$$

  \end{sbprop}

\begin{proof} Let $P= \{s\in \R^n_{>0}\mid s_1\leq \dots \leq s_n\}/\R_{>0}$ and $Q= \PGL_n(A)\bs |\BT_n|$, and let $R$ be the set of all isomorphism classes of pairs $(\La, \mu)$, where  $\La$ is a free $A$-module of rank $n$ and $\mu$ is a homothety class  of a norm on the $F_{\infty}$-vector space $F_{\infty}\otimes_A \La$. We have a canonical map 
$Q\to R$ sending $\mu$ to $(A^n, \mu)$.  We have the map $R \to P$ which sends the class of $(\La, \mu)$ to the class of $(\tilde \mu(\la_1), \dots, \tilde \mu(\la_n))$, where $(\la_i)_{1\leq i\leq n}$ is norm-ordered $A$-basis of $\La$, for $\tilde \mu$ a norm with class $\mu$. This map $R\to P$ is well defined by \ref{diagbase}(4). Then as is easily seen, the compositions $P\to Q \to R \to P$ and $Q\to R \to P \to Q$ are the identity maps. 
\end{proof}

\section{Cone decompositions}\label{cone_decomp}

In this section, we consider cone decompositions. We suppose throughout that $A = \F_q[T]$ so that $\infty$ is the standard infinite place
of $F = \F_q(T)$. The main idea is as follows.

As in the introduction, we use cone decompositions to define toroidal compactifications of moduli spaces. As in \ref{Cd}, the cone gives a condition that the poles of torsion points of a log Drinfeld module satisfy a certain divisibility. 

On the other hand, as we saw in Section \ref{BT}, the decomposition of the Bruhat-Tits building into simplices gives a cone decomposition. In \ref{normcv}, the Bruhat-Tits building, which is a space of norms, is related to Drinfeld modules because in the degeneration over a complete valuation ring of height one, we have a lattice $\La$ which has a norm defined by the poles of elements of $\La$. 

The main result of this section is Theorem \ref{Sigk3}, which connects these two kinds of cone decompositions: one is related to poles of torsion points, and the other is related to the poles of elements in the lattice $\La$ and to the Bruhat-Tits building. The former and the latter poles are related by the exponential map of the Drinfeld module, which is studied in Sections \ref{expon} and \ref{expon2}. Theorem \ref{Sigk3} is proved in Section \ref{ss:cone2}. In Section \ref{mofu}, we define moduli functors for toroidal compactifications using cone decompositions and poles of torsion points.  

We shall employ the following notation throughout this section. Let $\cV$ denote a complete valuation ring over $A$ of height $1$, and as in \ref{val1}, let $K$
denote its field of fractions, $\bar{K}$ an algebraic closure of $K$, and $K^{\sep}$ the separable closure of $K$ in $\bar{K}$. 
Let $v_{\bar{K}}$ denote the additive valuation on $\bar{K}$ extending the valuation $v_K$ of $K$.

We use $\phi$ to denote a generalized Drinfeld module over $\cV$ of generic rank $d$ and with trivial line bundle. 
Let $(\psi, \La)$ be a pair of a Drinfeld module of rank $r \ge 1$ and an $\psi(A)$-lattice $\La$ of rank $n = d-r$ in $K^{\sep}$ associated to
$\phi$ as in \ref{val2}(1). Let $\mu$ be the norm on $F_{\infty} \otimes_A \La$ attached to $\psi$ by \ref{normcv}.
Let $e_{\La}$ denote the exponential map of the lattice $\La$ (see \ref{prop51} and \ref{cpf2}). 

Finally, we let $N$ denote an element of $A$ not contained in $\F_q$.

\subsection{Cone decompositions and torsion points 1}\label{ss:cone1}
 
We state the four results \ref{Ntor}, \ref{Sigk3}, \ref{Sig1}, \ref{Sigk5} concerning torsion points of generalized Drinfeld modules over complete valuation rings of height one. The proofs of these results  are given in Section \ref{ss:cone2} below.  

\begin{sbpara} \label{cphi} Let  
$$
	C_d=\{(s_1,\dots, s_{d-1})\in \R^{d-1} \mid 0\leq s_1\leq \dots \leq s_{d-1}\}.
$$  

We attach an invariant
$$
	c(\phi) \in C_d/\R_{>0}
$$ 
to the generalized Drinfeld module $\phi$ over $\cV$ taken at the start of Section \ref{cone_decomp}. We define $c(\phi)$ as the class of  
$$
	(0^{r-1},(\mu(\la_i))_{1 \le i \le n})\in C_d,
$$ 
where $(\la_i)_{1 \le i \le n}$ is a norm-ordered $A$-basis of $\La$, as in \ref{diagbase}(1). This $c(\phi)$ is well defined by \ref{diagbase2}.

\end{sbpara}
\begin{sbpara}\label{Sigk1} Let $k\geq 1$ be an integer. Let 
${}_d\Sig^{(k)}$, or more simply $\Sig^{(k)}$, be the cone decomposition of $C_d$ defined as follows. Let 
$$
	{}_d I^{(k)} = I^{(k)} =\{(h, i, j)\in \Z^3\mid 1\leq j \le i\leq d-1 \textsp{and} 0\leq h\leq k-1\}.
$$
For each map $\alpha \colon I^{(k)} \to \{\R_{\leq 0}, \{0\}, \R_{\geq 0}\}$ with $\alpha(0,i,j) \neq \R_{\geq 0}$ for all $1 \le j \le i \le d-1$, let 
$$c(\alpha)= \{s\in C_d \mid q^hs_j-s_i\in \alpha(h,i,j)\textsp{for all} (h,i,j)\in I^{(k)} \}.$$
Then $\Sig^{(k)}$ is the set of the cones $c(\alpha)$ for these maps $\alpha$. 

This is a coarser version of the restriction of the cone decomposition $\{\sigma(S) \mid S\in \AP_{d-1}\}$ (along with $\{0\}$) of $\R_{>0}^{d-1} \cup \{0\}$ in \ref{apcone2} to the intersection of the latter set with $C_d$. For $\alpha$  as above, if $\alpha(k-1, i, j) \neq \R_{\leq 0}$ for all $1\leq j\leq i \leq d-1$, then we have either $c(\alpha)
=\sigma(S)$ for some $S \in \AP_{d-1}$ or $c(\alpha)=\{0\}$. However, in general, it can happen that 
$c(\alpha)$ contains infinitely many $\sigma(S)$ with $S \in \AP_{d-1}$ as in the case of $\sig^{k-1}\in \Sig^{(k)}$ in \ref{Sigd=3} below.

 \end{sbpara}
 
 \begin{example}\label{Sigd=3} Assume $d=3$.

 For integers $h\geq 1$, let 
$$
	\sig_h = \{(s_1,s_2)\in \R^2_{\geq 0}\mid  q^{h-1}s_1\leq s_2\leq q^hs_1\}.
$$ 
This is the cone $\sig(S_h)$ in \ref{APn=2} associated to the simplex $S_h$ of $\AP_2$. For $h\geq 0$, let
$$
	\sig^h = \{(s_1,s_2)\in \R^2_{\geq 0}\mid  q^h s_1\leq s_2\}= \left(\bigcup_{h'>h}\sig_{h'}\right) \cup \{(0,s)\mid  s\geq 0\}.
$$
Then for $k\geq 1$, the set $\Sig^{(k)}$ consists of $\sig_h$ for $1\leq h\leq k-1$ and $\sig^{k-1}$, and their faces. 

 \end{example}
 
We state our results on torsion points and cone decompositions: see Section \ref{ss:cone2} for the proofs.

\begin{sbprop}\label{Ntor} \ 
\begin{enumerate}
\item[(1)] 
Let $(\beta_i)_ {1\leq i\leq n}$ be a family of $N$-torsion points of $\phi$ in $\bar K$.
Then the following two conditions on $(\beta_i)_{1\le i \le n}$ are equivalent. 
\begin{enumerate}
	\item[(i)] We have $\beta_i= e_{\La}(\tilde \la_i)$ for $1\leq i\leq n$, where the $\tilde \la_i$ are elements of $\bar K$ such that  $(\psi(N)\tilde \lam_i)_{1\leq i\leq n}$ is a norm-ordered $A$-basis of $\La$. 
	\item[(ii)] For  $1\leq i\leq n$, we have
$$
	-v_{\bar K}(\beta_i) = \min\left\{-v_{\bar K}(\beta)\mid \beta \in \phi[N] \setminus \left(e_{\La}(\psi[N])+ \sum_{j=1}^{i-1}\phi(A)\beta_j\right)\right\}.
$$ 
\end{enumerate}
\item[(2)] For given $a_i\in A$ with $1\leq i\leq n$, the valuation $v_{\bar K}(\sum_{i=1}^n \phi(a_i)\beta_i)$ is independent of the choice of $(\beta_i)_i$ satisfying the conditions in (1). In particular, $v_{\bar K}(\beta_i)$ is independent of the choice of such $(\beta_i)_i$. 
\end{enumerate}
\end{sbprop}

\begin{sbpara}\label{Ntor(2)} 
We define 
$$c(\phi, N)\in C_d/\R_{>0}$$ as the class of  $(0^{r-1}, (-v_{\bar K}(\beta_i))_{1\leq i\leq n})$, where $r$ and $(\beta_i)_{1\leq i\leq n}$ are as in \ref{Ntor}. 

\end{sbpara}

\begin{sbthm}\label{Sigk3} Let $k$ and $k'$  positive integers with $k\leq k'$. 
\begin{enumerate}
\item[(1)] For each $\sig \in \Sig^{(k')}$, there is a unique finitely generated rational subcone $\tau$ of $C_d$ such that we have the equivalence 
$$c(\phi)\in \sig/\R_{>0}\quad  \Leftrightarrow \quad c(\phi, N)\in \tau/\R_{>0}$$
for every 
\begin{itemize}
	\item complete valuation ring $\cV$ of height one over $A=\F_q[T]$,
	\item generalized Drinfeld module $\phi$ over $\cV$ of generic rank $d$, and 
	\item $N\in A$ of degree $k$,
\end{itemize}	
where $c(\phi)$ is as in \ref{cphi}, and $c(\phi,N)$ is as in \ref{Ntor(2)}.
\item[(2)] The subcone $\tau$ in (1) remains unique if we instead restrict the set of $\cV$ in (i) to complete discrete valuation rings over $A$.
\item[(3)] The set of all $\tau$ as in (1) for $\sig \in \Sig^{(k')}$ forms a finite rational subdivision $\Sig_{k,k'}={}_d\Sig_{k,k'}$ of $C_d$. 
\end{enumerate}
\end{sbthm}

We will denote the fan $\Sig_{k,k}$ by $\Sig_k$. (See \ref{fan} for the definition of a finite rational fan.)

\begin{sbprop}\label{Sig1}
The fan $\Sig_1$ coincides with the fan of all faces of $C_d$.
\end{sbprop}

\begin{sbprop}\label{Sigk5} 
Let $k$ and $k'$ be positive integers with $k\leq k'$.  
\begin{enumerate}
	\item[(1)] There is a one-to-one correspondence between the set of all finitely generated rational subcones $\sig$ of $C_d$ such that $\sig$ is a subcone of some $\sig'\in \Sig_{k'}$ and the set of all finitely generated rational subcones $\tau$ of $C_d$ such that $\tau$ is a subcone of some $\tau'\in \Sig_{k,k'}$ characterized by the following property. If $\sig\leftrightarrow \tau$, we have the equivalence 
$$c(\phi, N') \in \sig/\R_{>0}\quad   \Leftrightarrow\quad  c(\phi, N)\in \tau/\R_{>0}$$
 for every 
\begin{itemize}
	\item complete  valuation ring $\cV$ of height one over $A=\F_q[T]$,
	\item generalized Drinfeld module $\phi$ over $\cV$ of generic rank $d$, and 
	\item $N \mid N'$ in $A$ with $\deg N = k$ and $\deg N' = k'$.
\end{itemize}
	\item[(2)] This correspondence induces a one-to-one correspondence between the set of all finite rational subdivisions of $\Sig_{k'}$ and the set of all finite rational subdivisions of $\Sig_{k,k'}$.
	\item[(3)] The characterization of the correspondence in (1) remains true if we instead restrict to complete discrete valuation rings $\cV$ 
	over $A$. 
	\end{enumerate}
\end{sbprop}

\subsection{Study of Drinfeld exponential maps 1} \label{expon}

In this subsection, we introduce piecewise-linear, increasing functions $\epsilon_s^{r,n}$ on $\R_{>0}$ that give the comparison \ref{eep} between negatives of valuations of elements $x$ of $\bar K$ (more precisely, the maximum of such valuations among elements of $x+\Lambda$) and their values $e_{\La}(x)$ under exponential maps of the $\psi(A)$-lattice $\La$. Here, $s$ is given by the negatives of the valuations of a norm-ordered basis. This is applied
to determine the poles of $N$-torsion points of $\phi$ in Section \ref{expon2}.

\begin{sbpara} \label{epsilon} 
Fix integers $r \ge 1$ and $n \ge 0$. Let $s=(s_i)_{1\leq i\leq n}\in \R_{>0}^n$. We define a map
$$\epsilon^{r,n}_s \colon \R_{>0} \to \R_{>0}$$
on $x \in \R_{>0}$ by
$$
	\epsilon^{r,n}_s(x) = \sum_{y\in A^n}  \max(x- \max\{|y_i|^r s_i\mid 1\leq i\leq n\}, 0).
$$
In the case $n = 0$, we have $\epsilon_s^{r,n}(x) = x$, for $s$ the empty tuple.

We relate this map $\epsilon^{r,n}_s$ to the Drinfeld exponential map in \ref{eep2} and \ref{eep}.

\end{sbpara}

\begin{example}\label{513} 
Consider the case $n=1$.  Let $s=s_1\in \R_{>0}$,  $x\in \R_{\ge 0}$,   $h \in \Z_{\ge 1}$, and assume $q^{(h-1)r}s\leq x \leq q^{hr}s$.
Then
$$
	\epsilon^{r,1}_s(x)=q^h x - (q^{h(r+1)}-1)\frac{q-1}{q^{r+1}-1}s.
$$
If $x < s$, then $\epsilon^{r,1}_s(x) = x$.
\end{example}

\begin{sbpara}\label{sigr} This is a preparation for the following \ref{explin}. 

For $n\geq 1$ and an simplex $S$ of $\AP_n$, we defined a finitely generated rational cone $\sig(S)$ in $\R^n_{\geq 0}$ in \ref{apcone1}. We define here a modified version $\sig_r(S)$ of $\sig(S)$ for each integer $r\geq 1$ by
$$
	\sig_r(S) = \{(y_1^r,\dots, y_n^r)\mid y \in \sig(S)\}.
$$
This $\sig_r(S)$ is also a finitely generated rational cone in $\R_{\ge 0}^n$. In fact, if $\alpha$ is a map 
$$
	I=\{(h, i,j)\in \Z^3\mid 1\leq j \le i \leq n\}\to \{\R_{\leq 0}, \{0\}, \R_{\geq 0}\}
$$  
as in \ref{apcone2}
such that $\sig(S)$ coincides with the set of all $y \in \R^n_{\geq 0}$ satisfying $q^hy_j-y_i\in \alpha(h,i,j)$ for all $(h,i,j)\in I$,
then $\sig_r(S)$ coincides with the set of 
$y \in \R^n_{\geq 0}$ satisfying $q^{hr}y_j-y_i\in \alpha(h,i,j)$ for all $(h,i,j)\in I$.

\end{sbpara}

\begin{sblem}\label{explin} Let $S$ be a simplex of $ \AP_{n+1}$. Then there is a linear map $l \colon \R^{n+1}\to \R$  such that $\epsilon^{r,n}_s(x)= l(s_1, \dots, s_n, x)$ for all $s\in\R^n_{>0}$ and $x\in \R_{>0}$ satisfying $(s_1, \dots, s_n, x)\in \sig_r(S)$. 
\end{sblem}

This follows from the definition of $\epsilon^{r,n}_s$.

\begin{sbcor}
The map $\epsilon^{r,n}_s \colon \R_{\ge 0} \to \R_{\ge 0}$ is a piecewise linear, increasing homeomorphism.
\end{sbcor}

\begin{sbpara}\label{pre1} 
For any subgroup $\Omega$ of the additive group $\bar K$ such that the set $\{\lambda \in \Omega \mid  v_{\bar K}(\lambda) \geq c\}$ is finite for all $c$,
the associated Drinfeld exponential map $e_{\Omega} \colon \bar K\to \bar K$ is given by
$$e_{\Omega}(z)= z \prod_{\la\in \Omega\setminus \{0\}}  (1-\lam^{-1}z).$$

\end{sbpara}

The following  \ref{eep2}  and \ref{eep} relate the function $\epsilon^{r,n}_s$ and the Drinfeld exponential map $e_{\La}$ of
the $\psi(A)$-lattice $\Lambda$ in the pair $(\psi,\La)$ attached to our generalized Drinfeld module $\phi$.

\begin{sbprop}\label{eep2} 
Let $(\lam_i)_{1 \le i \le n}$ be a norm-ordered basis of the lattice $\Lambda$, as in \ref{diagbase}(1).
Let $s_i=-v_{\bar K}(\lambda_i)$
for $1 \le i \le n$. 
Let $\bar \cV$ be the integral closure of $\cV$ in $\bar K$.
 Let $\alpha\in \bar K^\times$, let $E_1$ be the finite set $\{\la\in \La \setminus \{0\} \mid  \la\in \alpha\bar \cV\}$ (see \ref{pre0}), and let $\beta= \alpha \prod_{\la\in E_1} (\alpha \la^{-1})$.
\begin{enumerate}
	\item[(1)] We have $-v_{\bar K}(\beta)= \epsilon^{r,n}_s(-v_{\bar K}(\alpha))$. 
	\item[(2)] We have $\beta^{-1}e_{\La}(\alpha z)\in \bar \cV\ps{z}$ with coefficients converging to $0$ and at least 
	one coefficient a unit. 
\end{enumerate}
\end{sbprop}

\begin{proof} Part (1) follows from the definition of $\epsilon^{r,n}_s$ in \ref{epsilon}. That is, 
$$
	-v_{\bar K}(\beta) = \sum_{\lambda \in \La} \max(-v_{\bar K}(\alpha\lambda^{-1}),0)
	= \sum_{y \in A^n} \max\left(-v_{\bar K}(\alpha)+ v_{\bar K}\left(\sum_{i=1}^n\psi(y_i)\lambda_i\right),0\right),
$$
and
$v_{\bar K}(\sum_{i=1}^n\psi(y_i)\lambda_i) = -\max\{|y_i|^r s_i \mid 1 \le i \le n\}$ by \ref{normcv}.

Setting $E_2=\La\setminus (E_1 \cup\{0\})$, 
we have $$e_{\La}(\alpha z)= \alpha z\prod_{\la\in E_1}  (1-\la^{-1}\alpha z) \prod_{\la \in E_2} (1-\la^{-1}\alpha z)=  \beta z \prod_{\la \in E_1} (\alpha^{-1} \la -z)  \prod_{\la \in E_2}  (1-\la^{-1}\alpha z).$$
Thus, we see that $\beta^{-1}e_{\La}(\alpha z)$ reduces modulo the maximal ideal of $\bar \cV$ to a polynomial in $z$ with leading coefficient
$\pm 1$, which gives the first part of (2). The coefficients of $\beta^{-1}e_{\La}(\alpha z)$ converge to $0$ since $E_1$ is finite.
\end{proof}

\begin{sbpara}

In the notation of \ref{pre1}, for $x\in \bar K$, let 
$$
	v_{\bar K, \Omega}(x) = \max\{v_{\bar K}(x-\la)\mid  \la\in \Omega\}.
$$ 
Note that our assumption on $\Omega$ implies that 
$\{\la\in \La \mid v_{\bar K}(x-\la)\geq c\}$ is finite for all $c$ as well, so this maximum exists.
The map $v_{\bar K, \Omega}$ factors through the projection $\bar K\to \bar K/\Omega$, and we use the same symbol to denote the resulting map. 

\end{sbpara}
\begin{sbprop}\label{eep} Let the notation be as in \ref{eep2}. For $x\in \bar K$, we have 
$$-v_{\bar K}(e_{\La}(x))= \epsilon^{r,n}_s(-v_{\bar K, \La}(x)).$$
\end{sbprop}

\begin{proof} Let $\alpha$ be an element of $x+\La$ such that $v_{\bar K}(\alpha)= v_{\bar K, \La}(x)$. In the proof of \ref{eep2}, let $z=1$. For $\la \in E_1$, we have that $\alpha^{-1}\la-1$ lies in $\bar \cV$ by definition of $E_1$ and is a unit therein by choice of $\alpha$. For $\la \in E_2$, we have that $1-\la^{-1}\alpha \in \bar \cV^{\times}$ by definition of $E_2$. By the proof of \ref{eep2}, we then see that $\beta^{-1}e_{\La}(\alpha) \in \bar \cV^{\times}$. The result then follows by \ref{eep2}, as $v_{\bar K}(\beta)
= v_{\bar K}(e_{\La}(\alpha)) = v_{\bar K}(e_{\La}(x))$.
\end{proof}

\begin{sbrem} Our theory of cone decomposition for toroidal embeddings is based on simplices of the Bruhat-Tits building. The key point is the following. By \ref{explin} and \ref{eep2}, \ref{eep}, the poles of the Drinfeld exponential map have a linear property on simplices, and hence poles of torsion points of generalized Drinfeld modules have linear properties if we introduce cone decompositions related to simplices.

\end{sbrem}

\begin{sbpara}\label{hat(e)}

For integers $r$ and $n$ such that $r\geq 1$, $n\geq 0$, and for  $s\in \R^n_{>0}$, define
$$\delta^{r,n}(s_1,\dots,s_n) =\frac{q-1}{q^{r+n}-1} \sum_{i=1}^n q^{n-i}\epsilon^{r,i-1}_{s_1, \dots, s_{i-1}}(s_i).$$ 
If $n = 0$, this is zero. We describe a relationship between $\delta^{r,n}$ and generalized Drinfeld modules over $\cV$ in \ref{T4.9}.

We define a modified version ${\hat \epsilon}^{r,n}_s$ of $\epsilon^{r,n}_s$ by 
$${\hat \epsilon}^{r,n}_s(x) = \epsilon^{r,n}_s(x)- \delta^{r,n}(s_1,\dots, s_n).$$

\end{sbpara}

\begin{example}\label{514}  
Let $n=1$, $s=s_1\in \R_{>0}$. Then
$$
	\delta^{r,1}(s)= \frac{q-1}{q^{r+1}-1}s.
$$
As in \ref{513}, we also let $x \in \R_{> 0}$ and $h \in \Z_{\ge 0}$ and assume that $q^{(h-1)r}s\leq x \leq q^{hr}s$ if $h \ge 1$ and $x \le s$ if $h = 0$. Then
$$
	\hat \epsilon^{r,1}_s(x)=\epsilon^{r,1}(x)-\delta^{r,1}(s)= q^h x - q^{h(r+1)}\frac{q-1}{q^{r+1}-1}s.
$$

\end{example}

Formulas concerning $\epsilon^{r,n}_s$ have simpler presentations when we use ${\hat \epsilon}^{r,n}_s$ in place of $\epsilon^{r,n}_s$: see for example \ref{prope2} and \ref{prope3} below.

The following is the $\hat \epsilon$-version of \ref{explin}. 
\begin{sblem}\label{explin2} 

 Let $S$ be a simplex of $\AP_{n+1}$. Then there is a linear map $l \colon \R^{n+1}\to \R$  such that $\hat \epsilon^{r,n}_s(x)= l(s_1, \dots, s_n, x)$ for all $s\in\R^n_{>0}$ and $x\in \R_{>0}$ satisfying $(s_1, \dots, s_n, x)\in \sig_r(S)$. 

\end{sblem}

\begin{proof}
This follows from \ref{explin}. 
\end{proof}

\begin{sbpara}\label{usehat}

 It follows from the definitions that
$$
	\delta^{r,n+1}(s_1, \dots, s_{n+1}) = \delta^{r,n}(s_1,\ldots, s_n)+
	\frac{q-1}{q^{r+n+1}-1}{\hat \epsilon}^{r,n}_{s_1,\ldots,s_n}(s_{n+1}).
$$
The following results \ref{deltasum}--\ref{prope1} are proven in \ref{31pf1}--\ref{31pf3} below.
\end{sbpara}

\begin{sbprop} \label{deltasum}
Assume that $s_i\leq s_{i+1}$ for $1\leq i\leq n$.   We have
$$
	\delta^{r,n+1}(s_1,\dots, s_{n+1})= \delta^{r,1}(s_1)+\delta^{r+1, n}(s'),
$$ 
where $s'=({\hat \epsilon}^{r,1}_{s_1}(s_{i+1}))_{1\leq i\leq n}$. 
\end{sbprop}

\begin{sbprop}\label{prope2} 
Assume that $s_i\leq s_{i+1}$ for $1\leq i\leq n+m-1$.  Then 
we have
$$
	{\hat \epsilon}^{r,n+m}_s = {\hat \epsilon}^{r+m, n}_{s'} \circ {\hat \epsilon}^{r,m}_t,
$$
where $t =(s_1,\dots, s_m)\in \R_{>0}^m$ and $s' \in \R_{>0}^n$ with
$s'_i= {\hat \epsilon}^{r,m}_t(s_{m+i})$. 

\end{sbprop}

We note that Proposition \ref{prope2} has the following corollary.

\begin{sbcor} \label{prope3} Assume that $s_i\leq s_{i+1}$ for $1\leq i\leq n-1$. Then we have
$${\hat \epsilon}^{r,n}_s = {\hat \epsilon}^{r+n-1,1}_{ s'_n}\circ \dots \circ {\hat \epsilon}^{r+1,1}_{ s'_2}\circ  {\hat \epsilon}^{r,1}_{s'_1}$$
where $s'_i={\hat \epsilon}^{r,i-1}_{s_1,\dots, s_{i-1}}(s_i)$ for $1 \le i \le n$.
\end{sbcor}

\begin{sbprop}\label{prope1} Assume that $s_i\leq s_{i+1}$ for $1\leq i \leq n-1$ and $s_n \le x$. Then we have
$${\hat \epsilon}^{r,n}_s(x)= q^{r+n}{\hat \epsilon}^{r,n}_s(q^{-r}x).$$

\end{sbprop}

We prove these propositions. 

\begin{sbpara}\label{31pf1} 
First, the case $n=1$ of \ref{prope1} follows from \ref{514}.  

\end{sbpara}

\begin{sbpara} We will treat \ref{deltasum} and the case $m=1$ of \ref{prope2} simultaneously by induction on $n$. 
Both are easy to see for $n = 0$. 
Proposition \ref{prope2} for arbitrary $m$ follows from the case $m=1$ by induction, as follows.  Assuming the
result for $m$ and $s_i \le s_{i+1}$ for all $1 \le i \le n+m$, we have
$$
	\hat{\epsilon}_s^{r,n+m+1} = \hat{\epsilon}_{s''}^{r+m,n+1} \circ \hat{\epsilon}_{s_1,\ldots,s_m}^{r,m} 
	= \hat{\epsilon}_{s'}^{r+m+1,n} \circ \hat{\epsilon}_{s''_1}^{r+m,1} \circ \hat{\epsilon}_{s_1,\ldots,s_m}^{r,m}
	= \hat{\epsilon}_{s'}^{r+m+1,n} \circ \hat{\epsilon}_{s_1,\ldots,s_{m+1}}^{r,m+1},
$$
where $s'' \in \R_{>0}^{n+1}$ with $s''_i = {\hat \epsilon}^{r,m}_{s_1, \ldots, s_m}(s_{m+i})$, and where $s' \in \R_{>0}^n$ with $s'_i  
= {\hat \epsilon}^{r,m+1}_{s_1, \ldots, s_{m+1}}(s_{m+i+1})$.
\end{sbpara}

\begin{sbpara}\label{31pf2} 
We show that \ref{prope2} and \ref{deltasum} for $n-1$ imply \ref{deltasum} for $n$.
To this end, for $s = (s_i)_{i=1}^{n+1}$, $t = (s_i)_{i=1}^n$, $s' = (s'_i)_{i=1}^n$, 
and $t' = (s'_i)_{i=1}^{n-1}$ with $s'_i = \hat{\epsilon}_{s_1}^{r,1}(s_{i+1})$ for
$1 \le i \le n$, we compute
\begin{align*}
	\delta^{r,n+1}(s) &= \delta^{r,n}(t)+ \frac{q-1}{q^{r+n+1}-1}{\hat \epsilon}^{r,n}_{t}(s_{n+1})\\
	&= \delta^{r,1}(s_1)+\delta^{r+1,n-1}(t')+ \frac{q-1}{q^{r+n+1}-1}{\hat \epsilon}^{r,n}_{t}(s_{n+1})\\
	&= \delta^{r,1}(s_1)+\delta^{r+1,n}(s')+\frac{q-1}{q^{r+n+1}-1}({\hat\epsilon}_{t}^{r,n}(s_{n+1})
	-{\hat \epsilon}_{t'}^{r+1,n-1}({\hat\epsilon}_{s_1}^{r,1}(s_{n+1})))\\
	&= \delta^{r,1}(s_1)+\delta^{r+1,n}(s').
\end{align*}
Here, the first step is \ref{usehat}, the second step is \ref{deltasum} for $n-1$ and, the third step is \ref{usehat} again, and 
the fourth step is \ref{prope2} for $n-1$.
\end{sbpara}

\begin{sbpara}
We complete the proof of \ref{deltasum} and $m = 1$ of \ref{prope2} by showing that \ref{deltasum} for $n$ 
and the proven case of \ref{prope1} imply \ref{prope2} for $m = 1$ and $n$.

Let $s\in \R_{>0}^{n+1}$. Let $y=(y_i)_{1\leq i\leq n+1}\in A^{n+1}$, and consider 
$M_y = \max(P_y-Q_y, 0)$, where 
$$P_y=\max(x-|y_1|^rs_1, 0), \quad Q_y=\max\left(\max_{2\leq i\leq n+1}\{|y_i|^rs_i\}- |y_1|^r s_1,0\right).$$
If $x\leq \max_{2\leq i\leq n+1}\{|y_i|^rs_i\}$, then $M_y = 0$.
If $x\geq \max_{2\leq i\leq n+1}\{|y_i|^rs_i\}$, then
$$
	M_y = \begin{cases}  \max(x-|y_1|^rs_1, 0) & \text{if } \max_{2\leq i\leq n+1}\{|y_i|^rs_i\}\leq |y_1|^rs_1, \\
	\max(x- \max_{2\leq i\leq n+1}\{|y_i|^rs_i\}, 0) & \text{otherwise}.
	\end{cases}
$$
Hence $M_y = \max(x-\max_{1 \le i \le n+1}\{|y_i|^rs_i\},0)$ for all $y\in A^{n+1}$, and therefore
$$
	\sum_{y \in A^{n+1}} M_y = \epsilon^{r,n+1}_s(x).
$$

On the other hand, 
if we fix $y_2, \dots, y_{n+1}$, 
then the sum of $M_y$ for $y$ with these fixed 
$y_2, \dots, y_{n+1}$ is equal by definition to 
\begin{align*}
	\sum_{y_1 \in A} M_y &= \max\left(\epsilon^{r, 1}_{s_1}(x)-\epsilon^{r,1}_{s_1}\left(\max_{2\leq i\leq n+1}\{|y_i|^rs_i\}\right), 0\right)\\
	&=\max\left(\epsilon^{r, 1}_{s_1}(x)- \max_{2\leq i\leq n+1}\{\epsilon^{r,1}_{s_1}(|y_i|^rs_i)\}, 0\right)\\
	&=\max\left({\hat \epsilon}^{r, 1}_{s_1}(x)- \max_{2\leq i\leq n+1}\{{\hat \epsilon}^{r,1}_{ s_1}(|y_i|^rs_i)\}, 0\right).
\end{align*}
By the case $n=1$ of \ref{prope1} (see \ref{31pf1}), we have 
$$ \max_{2\leq i\leq n+1}\{{\hat \epsilon}^{r,1}_{s_1}(|y_i|^rs_i)\}=\max_{2\leq i\leq n+1}\{|y_i|^{r+1}{\hat \epsilon}^{r,1}_{s_1}(s_i)\}$$
unless 
$y_2 = \cdots = y_{n+1} = 0$.
It follows that the sum of all $M_y$ is also given by
$$
	\sum_{y \in A^{n+1}} M_y = (\epsilon^{r+1, n}_{s'}\circ {\hat \epsilon}^{r,1}_{s_1})(x) + \delta^{r,1}(s_1),
$$ 
where $s'= ({\hat \epsilon}^{r,1}_{s_1}(s_{i+1}))_{1\leq i\leq n}$. 
By comparing with our first calculation of the sum and applying \ref{deltasum} for $n$, we obtain 
$$
	\epsilon^{r,n+1}_s(x)= (\epsilon^{r+1,n}_{s'}\circ {\hat \epsilon}^{r,1}_{s_1})(x) + \delta^{r,n+1}(s)
	-\delta^{r+1,n}(s'),
$$ 
which implies \ref{prope2} for $m = 1$ and $n$.
\end{sbpara}

\begin{sbpara}\label{31pf3} 
	We prove \ref{prope1} using \ref{prope2} (for $m = 1$) by induction on $n$. 
	Suppose it holds for $n$: we prove it for $n+1$.  
	For $s = (s_i)_{i=1}^{n+1}$, $t=(s_i)_{i=1}^n$, 
	and $s' = \epsilon_t^{r,n}(s_{n+1})$, we have
	$$
		\hat{\epsilon}_s^{r,n+1}(x) = \hat{\epsilon}_{s'}^{r+n,1} \circ \hat{\epsilon}_t^{r,n}(x)
		= \hat{\epsilon}_{s'}^{r+n,1}(q^{r+n}\hat{\epsilon}_t^{r,n}(q^{-r}x)),
	$$
	the first equality by \ref{prope2} and our assumption on the $s_i$, and the second equality by induction 
	and our assumption on the $s_i$ and $x$.
	By induction and the fact that $s_n \le s_{n+1}$, we have
	$$
		q^{-r-n}\hat{\epsilon}_t^{r,n}(s_{n+1}) = \hat{\epsilon}_t^{r,n}(q^{-r}s_{n+1}) \le  \hat{\epsilon}_t^{r,n}(q^{-r}x),
	$$
	the inequality following from the fact that $s_{n+1} \le x$, since $\hat{\epsilon}_t^{r,n}$ is increasing.
	By the equality in the case $n=1$ proven in \ref{31pf1}, as well as \ref{prope2} again, we then have
	$$
		\hat{\epsilon}_{s'}^{r+n,1}(q^{r+n}\hat{\epsilon}_t^{r,n}(q^{-r}x)) = q^{r+n+1}\hat{\epsilon}_{s'}^{r+n,1}(\hat{\epsilon}_t^{r,n}(q^{-r}x))
		= q^{r+n+1}\hat{\epsilon}_s^{r,n+1}(q^{-r}x).
	$$
\end{sbpara}

\begin{sbprop}\label{T4.9} Let the notation be as in \ref{eep2}. Let $f=c(T, q^d)$ be the coefficient of $\phi(T)$ of the highest
 degree. Then $$v_K(f)= (q^d-1)\delta^{r,n}(s_1, \dots, s_n).$$
\end{sbprop}

\begin{proof} 
Take $z\in \bar K$ such that
$v_{\bar K}(z)\ll 0$, $v_{\bar K}(z)=v_{\bar K, \La}(z)$, and $v_{\bar K}(\psi(T)z)= v_{\bar K, \La}(\psi(T)z)$.
Since $v_{\bar K}(\psi(T)z)= q^r v_{\bar K}(z)$ by \ref{normcv}, Propositions \ref{eep} and \ref{prope1} imply that
\begin{align*}
	v_{\bar K}(e_{\La}(\psi(T)z)) 
	&= -\epsilon_s^{r,n}(-q^rv_{\bar K}(z))\\
	&= -q^d\hat\epsilon_s^{r,n}(-v_{\bar K}(z)) - \delta^{r,n}(s_1,\ldots,s_n)\\
	&= q^dv_{\bar K}(e_{\La}(z)) + (q^d-1)\delta^{r,n}(s_1,\dots, s_n).
\end{align*}
On the other hand, we have
$$                    
	v_{\bar K}(e_{\La}(\psi(T)z))= v_{\bar K}(\phi(T)(e_{\La}(z))) = q^d v_{\bar K}(e_{\La}(z)) + v_{\bar K}(f),
$$ 
the latter statement in that $v_{\bar K}(e_{\La}(z))$ can be made arbitrarily negative by making $v_{\bar K}(z)$ so.
\end{proof}

\subsection{Study of Drinfeld exponential maps 2} \label{expon2}

The main result of this subsection is Proposition \ref{xicdv1}, which describes the poles of $N$-torsion points by the poles of a norm-ordered basis  of $\La$. For this comparison, we introduce a bijection $\xi_k^d$ from $C_d$ to itself, where $k = \deg N$, with coordinates described by the maps $\epsilon_s^{r,d-r}$ of Section \ref{expon}, where $s$ is given the negatives of the valuations of a basis
of $\La$ as in \ref{diagbase}.

We maintain the notation introduced at the beginning of this section, and we let $e_{\La}$ denote the exponential
map of $\La$ and
let $\mu$ be the norm on $F_{\infty} \otimes_A \La$ attached to $(\psi,\La)$ by Proposition \ref{normcv}: it satisfies 
$\mu^r = -v_{\bar{K}}$ on $\La\setminus \{0\}$.

\begin{sbprop}\label{Ntor1} Let $(\lam_i)_{1\leq i\leq n} \in \La^n$ satisfy the equivalent conditions in \ref{diagbase}(1).  Let  $\la \in \bar K$ be such that $\psi(N)\la\in \La$, and for $1 \le i \le n$, let $a_i\in A$ with $|a_i|<|N|$ be such that $\psi(N)\la\equiv \sum_{i=1}^n \psi(a_i)\lam_i\bmod \psi(N)\La$.   Then
$$-v_{\bar K}(e_{\La}(\la))= \epsilon^{r,n}_s(|N|^{-r}\max\{|a_i|^rs_i \mid 1\leq i\leq n\})$$
where $s_i=-v_{\bar K}(\lam_i)$. 

\end{sbprop}

\begin{proof} 
	By assumption on the $a_i$, the element  $\sum_{i=1}^n\psi(a_i)\la_i$ of $\La$ 
	has minimal norm under $\mu$ in its congruence class modulo $\psi(N)\La$. 
	By Proposition \ref{normcv}, we then have that 
	\[
		|N|^r v_{\bar{K},\La}(\la) = -\mu\left(\sum_{i=1}^n\psi(a_i)\la_i\right)^r
		= -\max\{|a_i|^rs_i \mid 1\leq i\leq n\}.
	\]
	By this, the result is reduced to a direct consequence
	of Proposition \ref{eep}.
\end{proof}

\begin{sbpara}\label{defcone2} 

For $k\geq 0$, we define a map 
$$
	\xi^d_k \colon C_d\to C_d,
$$ 
where $C_d$ is as in \ref{Cd}.
Let 
$s = (s_i)_{i=1}^{d-1} \in C_d$ and set
$$
	r = r(t) = \min\{i \mid s_i\neq 0\}
$$ 
if it exists and $r = d$ otherwise. Define $\xi_k^d(s)$ by
$$
	\xi_k^d(s)_i =\begin{cases} 0 & \text{if} \; 1 \leq i<r\\  
	\epsilon_{s_r^r, \dots, s_{d-1}^r}^{r,d-r}(q^{-kr}s_i^r) & \text{if}\;  r \leq i <d.
	\end{cases}
$$

Since $\epsilon_{s_r^r, \dots, s_{d-1}^r}^{r,d-r}$ is an increasing bijection $\R_{\ge 0} \to \R_{\ge 0}$, the map $\xi_k^d$ is easily seen to provide a bijection $C_d \to C_d$.

\end{sbpara}

\begin{sbprop}\label{xicdv1} 
Let $(\lam_i)_{1\leq i\leq n}$ be a norm-ordered basis of $\La$ associated to $\mu$ as in \ref{diagbase}.
For each $1 \le i \le N$, let $\tilde \la_i$ be an element of $\bar K$ such that $\psi(N)\tilde \la_i= \la_i$, and let $\beta_i =e_{\La}(\tilde \la_i)$. Let $a$ be a nonzero element of $A$ such that $|a|<|N|$. 
Then
 we have
$$\xi^d_{k-k'}(0^{r-1},\mu(\lam_1), \dots, \mu(\lam_n))_{i+r-1}=  - v_{\bar K}(\phi(a)\beta_i),$$
where $k$ is the degree of $N$, and $k'$ is the degree of $a$, and $0^{r-1}=(0, \dots, 0) \in \R^{r-1}$.

In particular, for $a=1$ we have
$$
	\xi^d_k(0^{r-1},\mu(\lam_1), \dots, \mu(\lam_n))= (0^{r-1},  - v_{\bar K}(\beta_1),\dots, -v_{\bar K}(\beta_n)).
$$
\end{sbprop}

\begin{proof} 
Set $s_i = -v_{\bar{K}}(\lambda_i)$ for $1 \le i \le n$ and $s = (s_i)_i$.  By definition of $\mu$, we have $s_i = \mu(\lambda_i)^r$.   Thus
$$
	\xi^d_{k-k'}(0^{r-1},\mu(\lam_1), \dots, \mu(\lam_n))_{i+r-1} = \epsilon^{r,n}_s(q^{-(k-k')r}s_i).
$$
By Proposition \ref{Ntor1}, we have $\epsilon^{r,n}_s(q^{-(k-k')r}s_i) = -v_{\bar K}(e_{\La}(\psi(a)\tilde{\la}_i)) = 
-v_{\bar K}(\phi(a)\beta_i)$.  
\end{proof}

\subsection{Cone decompositions and torsion points 2}\label{ss:cone2}

We prove the results \ref{Ntor}, \ref{Sigk3}, \ref{Sig1}, and \ref{Sigk5} stated in Section \ref{ss:cone1}.

The idea of the proof of the main theorem \ref{Sigk3} is as follows. We prove it by showing that the map $\xi_k^d \colon C_d \to C_d$ defined in \ref{defcone2} 
induces a bijection $\Sig^{(k')} \to \Sig_{k,k'}$ and sends each invariant $c(\phi)$ to $c(\phi,N)$. A difficulty is that, as in Example \ref{Ex8} below, the map $\xi^d_k$ is not continuous for $d \ge 3$. Furthermore, this map $\xi^d_k$ involves  the operator $x\mapsto x^r$ for $r\geq 2$, so it is far from linear. Hence, the fact that the images of the cones $\sig\in \Sig^{(k')}$ under $\xi^d_k$ are still  cones is highly nontrivial. We will show that for each cone $\sig$ of $\Sig^{(k')}$, the restriction of $\xi_k^d$ to $\sig$ is a composition of a map $\pi_{\sig}^{(k')}$ followed by a linear map, and show that this map $\pi_{\sig}^{(k')}$ has the property that its image is a set defined by explicit linear inequalities and hence is a cone. 

\begin{sbpara}\label{cpf3} We prove \ref{Ntor}.  

(1) The implication (i) $\Rightarrow$ (ii) follows from \ref{Ntor1}, recalling the exact sequence of \ref{cpf2}. As for (ii) $\Rightarrow$ (i), let $(\la_i)_i$ be a basis of $\La$ satisfying
\ref{diagbase}(1), and let $(\beta_i)_i$ be a family of $N$-torsion points of $\phi$ in $\bar K$ satisfying (ii).  We will exhibit a basis
$(\la_i')_i$ of $\La$ satisfying \ref{diagbase}(1) such that for $\tilde{\la}'_i \in \bar K$ with $\psi(N)\tilde{\la}'_i = \la'_i$ we have
$\beta_i = e_{\La}(\tilde{\la}'_i)$. By induction on $j$ with $0 \le j \le n$, we construct a basis $(\la^{(j)}_i)_i$ of $\La$ satisfying \ref{diagbase}(1) 
such that $(\la^{(n)}_i)_i$ provides the desired basis $(\la_i')_i$. 

First, let $\la^{(0)}_i= \la_i$. Assume $1\leq j\leq n$. Suppose by induction that we have a basis  $(\la^{(j-1)}_i)_i$ of $\La$ satisfying \ref{diagbase}(1) such that if we define $\beta^{(j-1)}_i = e_{\La}(\tilde \la^{(j-1)}_i)$ 
for some $\tilde \la^{(j-1)}_i \in \bar K$ satisfying $\psi(N)\tilde \la^{(j-1)}_i=\la^{(j-1)}_i$, then $\beta_k =\beta^{(j-1)}_k$ for $1\leq k<j$. 
Noting \ref{cpf2}, we write 
$$
	\beta_j= \xi+ \sum_{k=1}^n \phi(a_k)\beta^{(j-1)}_k,
$$ 
where $\xi\in e_{\La}(\psi[N])$ and $a_k \in A$. It follows from \ref{Ntor1} and \ref{diagbase}(3) that if $k\geq j$ and $a_k \notin NA$, 
then $a_k \in \F_q + NA$ and $-v_{\bar K}(\beta^{(j-1)}_k)=-v_{\bar K}(\beta_j)$. Let $s$  be the largest integer $k$ such that $k\geq j$ and $a_k \notin NA$. For $k \ge 1$, define 
$$
	\lambda_k^{(j)} = \begin{cases} \sum_{h=1}^s a_h \lambda_h^{(j-1)} & \text{if } k = j \\
	\lambda_k^{(j-1)} & \text{if } k < j \textsp{or} k > s \\
	\lambda_{k-1}^{(j-1)} & \text{if } j < k \le s. \end{cases}
$$
 Then $(\la^{(k)}_i)_i$ is a basis of $\La$ satisfying \ref{diagbase}(1) and $\beta_k = e_{\La}(\tilde \la^{(j)}_k)$ for $1\leq k\leq j$ for some $\tilde \la^{(j)}_k\in \bar K$ such that $\psi(N)\tilde \la^{(j)}_k=\la^{(j)}_k$.

(2) This follows by \ref{Ntor1} and \ref{diagbase}(4), given the relation between the family $(\beta_i)_i$ and a basis $(\la_i)_i$ of $\La$
satisfying \ref{diagbase}(1) in condition (i) of part (1).

\end{sbpara}

We give preparations for the proofs of  \ref{Sigk3}, \ref{Sig1}, and \ref{Sigk5}.

\begin{sbpara}\label{sigc}

Let $k\geq 1$, and let $\sig\in \Sig^{(k)}$, where $\Sig^{(k)}$ is as in \ref{Sigk1}. We define a map $$\pi_{\sig}^{(k)} \colon \sig \to \R_{\geq 0}^{d-1}$$ as follows.
Let $\theta^{(k)}$, or more briefly $\theta$, be the function on $i \in \Z$ with $1 \le i \le d-1$ defined by
$$
	\theta^{(k)}(i) = \min\{j\in \Z \cap [1,i] \mid q^{k-1}s_j\geq s_i\textsp{for all}s\in \sig\}.
$$ 
(Note that this set of $j$ is not empty because it contains $i$.) For $s = (s_i)_{i=1}^{d-1} \in C_d$, let $r=r(s)$ be the minimal
$i$ with $s_i \neq 0$ as in \ref{defcone2}, and set 
 $$
	\pi_{\sig}^{(k)}(s)_i = \begin{cases} 0 & \text{if } 1 \leq i < r \\
	\hat \epsilon_{s_r^r, \dots, s_{\theta(i)-1}^r}^{r,\theta(i)-r}(s_i^r) & \text{if } r \leq i < d.
	\end{cases}
$$

\end{sbpara}

\begin{example}\label{1sig}  Let $d=4$ and $k=2$, and let  
$$\sig=\{(s_1, s_2, s_3)\in \R_{\geq 0}^3\mid qs_1\leq s_2 \leq s_3, qs_2\geq s_3\}\in \Sig^{(2)}.$$ We have  $\theta(1)=1$, $\theta(2)=\theta(3)=2$. Let $s\in \sig$. Then $\pi_{\sig}^{(2)}(s)$ is as follows. 

If $s_1=0$, then $\pi_{\sig}^{(2)}(s)= (0, s_2^2, s_3^2)$. 
The case that $s_1>0$ is described as follows. Take $h,h'\in \Z$ such that 
$q^{h-1}s_1\leq s_2 \leq q^hs_1$ and $q^{h'-1}s_1\leq s_3\leq q^{h'}s_1$ (in which case $h\geq 2$ and $h' \in \{h,h+1\}$). By \ref{514}, we have 
$$\pi_{\sig}^{(2)}(s)= (s_1, q^hs_2-q^{2h}(q+1)^{-1}s_1, q^{h'}s_3-q^{2h'}(q+1)^{-1}s_1).$$  
\end{example}

\begin{sbprop}\label{lin1}  Let $k$ and $\sig$ be as in \ref{sigc}, and let $0\leq k'\leq k$. Let $i, i' \in \Z$ with 
$\theta(i)-1\leq i' \le i \leq d-1$. 
Let $\pi_{\sig, k', i,i'} \colon \sig\to \R_{\geq 0}$ be the map defined on $s\in\sig$ by
$$\pi_{\sig, k', i, i'}(s) =  \begin{cases} 0 & \text{if } 1 \leq i < r \\
	\hat \epsilon_{s_r^r, \dots, s_{i'}^r}^{r,i'+1-r}(q^{-k'r}s_i^r) & \text{if } r \leq i < d,
	\end{cases}$$
where $r=r(s)$ is as in \ref{sigc}.
Then there is a linear map $l_{k',i,i'} \colon \R^{d-1}\to \R$ such that
$$\pi_{\sig, k', i,i'}(s) = l_{k', i,i'} (\pi_{\sig}^{(k)}(s))$$
for every $s\in \sig$.
\end{sbprop}

\begin{proof}  
Note that $\pi_{\sig}^{(k)}(s)_i= \pi_{\sig, 0, i, \theta(i)-1}(s)$, where $\theta = \theta^{(k)}$.  Our assumption on $i$ implies that $i \ge r$.
For $j \in \Z$ such that $\theta(i) \leq j\leq i$, let 
$$
	u_j= \hat \epsilon_{s_r^r, \dots, s_{\theta(i)-1}^r}^{r,\theta(i)-r}(s_j^r).
$$ 
Then $u_i=\pi_{\sig}^{(k)}(s)_i$ by definition. 
By induction on $i \ge r$, we may suppose that $u_j$ is a linear function of $\pi_{\sig}^{(k)}(s)_{j'}$ for $j' \le j$ for each $j \in \Z$ with $\theta(i)\leq j<i$. By \ref{prope1} and \ref{prope2}, we have
$$
\pi_{\sig, k',i,i'}(s)=\hat \epsilon^{\theta(i), i'+1-\theta(i)}_{u_{\theta(i)}, \dots, u_{i'}}(q^{-k'\theta(i)}u_i).$$
For a simplex $S$ of $\AP_{i'-\theta(i)+2}$ (see \ref{sigr}) such that $(s'_{\theta(i)}, \ldots, s'_{i'}, q^{-k'}s'_i) \in \sigma(S)$ for all $s' \in \sig$, we see from \ref{prope1} that $(u_{\theta(i)}, \ldots, u_{i'}, q^{-k'\theta(i)}u_i)$ similarly ranges in $\sig_{\theta(i)}(S)$. Hence by \ref{explin2}, the right hand side is a linear function of $u_{\theta(i)}, \dots, u_{i'},$ and $u_i$.  The result then follows by
the inductive hypothesis.
\end{proof}

\begin{sbprop}\label{piineq}  
Let $\sig = \sig(\alpha)$ be defined by $\alpha \colon I^{(k)} \to \{\R_{\leq 0}, \{0\}, \R_{\geq 0}\}$ as in \ref{Sigk1}.
The map $\pi_{\sig}^{(k)} \colon \sig\to \R_{\geq 0}^{d-1}$ is injective, and its image coincides with the cone consisting of all $x\in C_d$ such that
$$q^{h(c(i,j)+1)} l_{0,j,c(i,j)}(x) - l_{0, i,c(i,j)}(x) \in \alpha(h,i,j)$$ 
for all $(h,i,j) \in I^{(k)}$ with $\theta(i)-1 \le j$, where each $c(i,j)$ is a chosen integer such that  $\theta(i)-1\leq c(i,j)\leq j$, and
$l_{0,j,c(i,j)}$ and $l_{0,i,c(i,j)}$ are any choices of linear functions as in \ref{lin1}.
 \end{sbprop}
 
 \begin{proof} Let $d\geq 2$, and suppose by induction that the result holds for $d-1$.  
 Let $\sig'$ be the subcone of $C_{d-1}$ consisting of all elements $a$ such that
 $q^h a_j -a_i \in \alpha(h,i, j)$ for all $0 \le h \le k-1$ and $1 \le j<i \le d-2$. 
Let 
$$
	J^{(k)} = \{ (h,j) \in \Z^2 \mid  0 \le h \le k-1, 1 \le j \le d-2 \}.
$$ 
For $(h,j) \in J^{(k)}$ with $j \le \theta(i)-1$, we have $\alpha(h,d-1,j) = \R_{\le 0}$ by definition of $\theta(i)$. Since $q^h s_{\theta(i)-1} \le s_{d-1}$ implies that $q^h s_j \le s_{d-1}$ for $s \in C_d$ and $j \le \theta(i)-1$, we see that
$$
	\sig = \{s\in \sig' \times \R_{\ge 0} \mid q^hs_j-s_{d-1} \in \alpha(h,d-1,j) \textsp{for all} (h,j) \in J^{(k)} \textsp{with} j \ge \theta(i)-1 \}.
$$

For $(h,j) \in J^{(k)}$ with $j \ge \theta(i)-1$, we have $q^h s_j - s_{d-1} \in \alpha(h,d-1,j)$ if and only if 
$$
	q^{h(c(d-1,j)+1)}\pi_{\sig, 0, j, c(d-1,j)}(s) - \pi_{\sig, 0,d-1, c(d-1,j)}(s) \in \alpha(h,d-1,j)
$$
by \ref{prope1}. 
By \ref{lin1}, we then have that $\pi_{\sig}^{(k)}(\sig)$ consists of all $a \in \pi_{\sig'}^{(k)}(\sig') \times \R_{\ge 0}$ such that
$$
	q^{h(c(d-1,j)+1)}l_{0,j,c(d-1,j)}(a) - l_{0,d-1,c(d-1,j)}(a)  \in \alpha(h,d-1,j)
$$
for all $(h,j) \in J^{(k)}$ with $j \ge \theta(i)-1$.
(In particular, $a_j \le a_{d-1}$ for all $j$ for all such $a$.)
 The result follows from this by the inductive hypothesis.
 \end{proof}

\begin{example}\label{5sig}
Let $\sig$ be as in \ref{1sig}.
Take $c(2,1)=c(3,1)=c(3,2)=1$. Since 
$\pi_{\sig, 0, i,1}(s)= \pi_{\sig}^{(2)}(s)_i$ for $i=2,3$ and  
$$
	\pi_{\sig,0, 1,1}(s)= q(q+1)^{-1}s_1=q(q+1)^{-1}\pi_{\sig}^{(2)}(s)_1,
$$  
it follows from the description in \ref{piineq} that 
$$\pi_{\sig}^{(2)}(\sig)=  \{(a_1, a_2, a_3)\in \R^3_{\geq 0}\mid q^3(q+1)^{-1}a_1\leq a_2\leq a_3, q^2a_2\geq a_3\}.$$

\end{example}

\begin{sblem}\label{lin2}  For $1\leq i\leq d-1$, let $\tilde \delta_i \colon C_d\to \R$ be the map defined on $s\in C_d$ by
$$\tilde \delta_i(s)=  \begin{cases} 0 & \text{if } 1 \leq i < r \\
	\delta^{r,i+1-r}(s_r^r, \dots, s_i^r) & \text{if } r \leq i < d.
	\end{cases}
$$
Let $k\geq 0$ and $\sig\in \Sig^{(k)}$. Then there is a linear map $l \colon \R^{d-1}\to \R$ such that
$$\tilde \delta_i(s)= l(\pi_{\sig}^{(k)}(s))$$
for every $s\in \sig$.

\end{sblem}

\begin{proof} For $i\geq 1$, we have 
$$\tilde \delta_i(s)= \tilde \delta_{i-1}(s)+ \frac{q-1}{q^{i+1}-1}\pi_{\sig,0, i, i-1}(s)$$
by \ref{usehat} and the definition of $\pi_{\sig,0,i,i-1}$ in \ref{lin1}. The lemma then follows from \ref{lin1} by induction on $i$.
\end{proof}

\begin{sbprop}\label{lin3} Let $k \ge 1$, let $\sig \in \Sig^{(k)}$, and let $0 \le k' \le k$. Then 
there is a linear map $l_{k'} \colon \R^{d-1}\to \R^{d-1}$ such that
$$\xi^d_{k'}(s)= l_{k'}(\pi_{\sig}^{(k)}(s))$$
for every $s\in \sig$.

\end{sbprop}

\begin{proof}
The first $r-1$ coordinates of $\xi_{k'}^d(s)$ are zero. Since $s \in C_d$, the definition of the maps in \ref{epsilon} implies that we have equalities
$$
	\xi^d_{k'}(s)_i= \epsilon_{s_r^r, \dots, s_i^r}^{r,i+1-r}(q^{-k'r}s_i^r)
$$
for all $i$  with $r \le i \le d-1$. Hence 
$$
	\xi^d_{k'}(s)_i= \pi_{\sig,k', i, i}(s) + \tilde \delta_i(s)
$$
for all $i$. Hence \ref{lin3} follows from \ref{lin1} and \ref{lin2}. 
\end{proof}

\begin{example}\label{4sig} Let $\sig \in \Sig^{(2)}$ be as in \ref{1sig}.
We compute the maps $\xi^4_1$ and $\xi^4_2$ on $s \in \sig$.  

First assume $s_1>0$ and take $h,h'\in \Z$ as in \ref{1sig}. Then by \ref{514}, for $k \in \{1,2\}$, we have
$$\xi^4_k(s_1, s_2, s_3)=(q^{-k}s_1, q^{h-2k}s_2-(q^{2h-2k}-1)(q+1)^{-1}s_1, q^{h'-2k}s_3-(q^{2h'-2k}-1)(q+1)^{-1}s_1).$$

 Next assume $s_1=0$. Then  
$$ \xi^4_1(s)=  (0, q^{-2}s_2^2, q^{-2}s_3^2), \quad \xi^4_2(s)=  (0, q^{-4}s_2^2, q^{-4}s_3^2).$$  

Compare this with the computation of $\pi_{\sig}^{(2)}$ in  \ref{1sig}. 
We have linear maps $l_k$ for $k \in \{1,2\}$ such that $\xi^4_k=  l_k \circ \pi_{\sig}^{(2)}$ on $\sig$ given  by 
$$l_k(a_1, a_2, a_3)=(q^{-k}a_1, q^{-2k}a_2+(q+1)^{-1}a_1, q^{-2k}a_3+(q+1)^{-1}a_1).$$

\end{example}

\begin{sblem}\label{sandV} Let $s\in C_d\cap \Q^{d-1}$, and let $r = r(s)$ be as in \ref{sigc}. Then there are a complete discrete valuation ring $\cV$ and a generalized  Drinfeld module $\phi$ over $\cV$ of generic rank $d$ such that the rank of $\phi\bmod m_{\cV}$ is $r$ and such that the class $c(\phi)$ in $C_d/\R_{>0}$ coincides with the class of $(s_i^{1/r})_{i=1}^{d-1}$, where $c(\phi)$ is as in \ref{cphi}.

\end{sblem}

\begin{proof} Take a discrete valuation ring $\cV'$ over $A$ and a Drinfeld module $\psi$ over $\cV'$ of rank $r$. Let $n=d-r$, and let $\cV$ be the completion of the local ring of the polynomial ring $\cV'[u_1, \dots, u_n]$ in $n$ variables at the prime ideal generated by $m_{\cV'}$. Let $K$ be the field of fractions of $\cV$. Take an integer $c\geq 1$ such that $cs_i\in \Z$ for all $1\leq i\leq d-1$, take elements $\la'_i$ ($1\leq i\leq n$) of $\cV'$ such that $\text{ord}_K(\la_i')= -cs_{i+r-1}$,  let $\la_i=\la_i'u_i$, and let $\La= \sum_{i=1}^n \psi(A)\la_i$. Then $\La$ is a $\psi(A)$-lattice in $K^{\sep}$ of rank $n$ in the sense of \ref{val1}. Let $\phi$ be the generalized Drinfeld module over $\cV$ of generic rank $d$ corresponding to $(\psi, \La)$ in the correspondence \ref{val2}. Then $c(\phi)$ is the class of $(s_i^{1/r})_{i=1}^{d-1}$ by its definition and \ref{normcv}. 
\end{proof}

\begin{sbpara} \label{Sigk3pf}

We prove Theorem \ref{Sigk3}. Let $\sig \in \Sig^{(k')}$. It follows from \ref{piineq} that $\pi_{\sig}^{(k')}(\sig)$ is a cone in $C_d$. In turn, it follows from \ref{lin3} (applied by exchanging $k$ and $k'$) that $\tau = \xi_k^d(\sig)$ is a cone in $C_d$.
Given $\cV$, $\phi$, and $N$ of degree $k$ as in \ref{Sigk3}(1), we see from \ref{xicdv1} that the map $\xi_k^d$ carries a representative of $c(\phi)$ to $c(\phi,N)$ for any $N \in A$ of degree $k$.  It follows in particular that $\tau$ corresponds to $\sig$ in the sense of
\ref{Sigk3}(1). In fact, by \ref{sandV}, the  correspondence $\sig \mapsto \tau$ is characterized by the condition in \ref{Sigk3}(2) because the map $\xi_k^d$ is continuous on the subset $\{s\in C_d\mid r(s) = r\}$ of $C_d$ for each $r$. Part (3) defining $\Sig_{k,k'}$ follows from the fact that $\xi_k^d \colon C_d \to C_d$ is a bijection.

\end{sbpara}

\begin{sbpara} 
We prove Proposition \ref{Sig1}. By \ref{Sigk3}, the bijection $\xi^d_1 \colon C_d\to C_d$ induces a bijection $\Sig^{(1)}\to \Sig_1$ given by $\sig\mapsto  \xi^d_1(\sig)$. 
The fan $\Sig^{(1)}$ is the fan of all faces of $C_d$. A face of $C_d$ has the form $\{s\in C_d \mid  s_i=s_j \textsp{if} (i,j)\in J \}$ for some subset $J$ of $\{1,\dots, d-1\}^2$. For a face $\sig$ of $C_d$, since $\xi^d_1(C_d)=C_d\in\Sig_1$, the cone $\xi^d_1(\sig)\in \Sig_1$ must be a face of $C_d$. We have that $\xi^d_1(\sig) \subset \sig$ since $\epsilon_{s_r^r, \ldots, s_{d-1}^r}^{r,d-r}$ is strictly increasing. By the same reasoning, the cone $\xi^d_1(\sig)$ is strictly larger than $\tau$ for every face $\tau$ of $C_d$ such that $\tau \subsetneq \sig$. Hence  
$\xi^d_1(\sig)=\sig$. 
\end{sbpara}

\begin{example}\label{6sig}
For $\sig$ as in \ref{1sig}, one sees from the descriptions of $\pi_{\sig}^{(k)}$  in \ref{5sig} and $l_k$ in \ref{4sig} that the cones $\xi^4_k(\sig)\in \Sig_k$ for $k \in \{1,2\}$ are given by 
 $$\xi^4_k(\sig)=\{(b_1, b_2, b_3)\in \R^3_{\geq 0}\mid  qb_1\leq b_2\leq b_3, q^2b_2 \ge b_3+q^k(q-1)b_1 \}.$$ 
\end{example}

\begin{sbprop}\label{thmcone}  Let $k$ and $k'$ be nonnegative integers. Set
$$
	\xi_{k,k'}^d = \xi_{k'}^d \circ (\xi_k^d)^{-1} \colon C_d \to C_d.
$$
\begin{enumerate}
	\item[(1)] The map $\xi^d_{k,k'} \colon C_d\to C_d$ is a homeomorphism. 
	\item[(2)] If $1\leq k \leq k'$, the map $\xi^d_{k,k'}$ induces a bijection $\Sig_{k,k'} \to \Sig_{k'}\;;\; \tau\mapsto \xi^d_{k,k'}(\tau)$.
	\item[(3)] Under the assumption of (2), for each $\tau\in \Sig_{k,k'}$, 
	there exists a bijective linear map $l \colon \R^{d-1} \to \R^{d-1}$ such that the restrictions of 
	$l$ and $\xi^d_{k,k'}$ to $\tau$ coincide. 
	 \end{enumerate}
\end{sbprop}

 \begin{proof} Part (2) follows from \ref{Sigk3}(3). Part (3) follows from  \ref{lin3}. Part (1) follows from (2) and (3). 
 	 \end{proof}

\begin{sbpara} \label{Sigk5pf} We prove  Proposition \ref{Sigk5}. By \ref{thmcone}, for a subcone $\tau$ of some element of $\Sig_{k,k'}$
	($1\leq k\leq k'$), 
	the image $\tau' = \xi_{k,k'}(\tau)$ is a subcone of some element of $\Sig_{k'}$. The map $\tau' \mapsto \tau$ induced by $\xi_{k,k'}^{-1}$ provides the correspondence in \ref{Sigk5}, as follows by the 
	analogous arguments to those of \ref{Sigk3pf}.
\end{sbpara}

\begin{example}\label{Ex8}  Assume $d=3$.
\begin{enumerate}
	\item[(1)]  Assume $s_1>0$.
	Let $h$ be the integer such that   $q^{h-1}s_1\leq s_2\leq q^hs_1$. Then 
$$
	\xi^3_k(s_1, s_2) = \begin{cases}
	(q^{-k}s_1, q^{-k}s_2) & \text{if } h\leq k, \\
	(q^{-k}s_1, q^{h-2k}s_2-(q^{2(h-k)}-1)(q+1)^{-1}s_1) & \text{if } h\geq k.\\
	\end{cases}
$$

	On the other hand, 
	we have $\xi^3_k(0, s)= (0, q^{-2k}s_2)$. 
	
	In particular, $\xi^3_k$ is not continuous. In fact, for $h \ge k$, we have 
	$$ 
		\xi^3_k(q^{-h}, 1)= (q^{-k-h}, (q^{h+1-2k}+q^{-h})(q+1)^{-1}),
	$$ 
	and as $h$ tends to $\infty$, these values tend to $(0, \infty)$ so do not converge to $\xi_k^3(0,1)=(0, q^{-2k})$. 
       	\item[(2)] Let $1\leq k\leq k'$.  Let $\sig_h$ for $1 \le h \le k'-1$ and $\sig^{k'-1}$ be as in \ref{Sigd=3}.  Then $\Sig_{k,k'}$ 
	 consists of all faces of
	 $$
	 	\xi^3_k(\sig_h)= \begin{cases} \{(s_1, s_2) \in \R_{\geq 0}^2\mid q^{h-1}s_1\leq s_2\leq q^hs_1\} & \text{if } h \le k, \\
		\{(s_1, s_2) \in \R_{\geq 0}^2 \mid \frac{q^{2h-1-k}+q^k}{q+1} s_1\leq s_2\leq \frac{q^{2h+1-k}+q^k}{q+1}s_1\} & \text{if } h \ge k.
		\end{cases}
	$$ 
	 and 
	 $$
	 	\xi^3_k(\sig^{k'-1})= \{(s_1,s_2)\in \R_{\geq 0}^2\mid (q^{2k'-1-k}+q^k)s_1\leq (q+1)s_2\}.
	$$                               
        \item[(3)] If $1\leq k\leq k'$, then $\xi_{k,k'}$ restricts to the linear maps
       	\begin{align*}
		&(x,y) \mapsto (q^{k-k'}x, q^{k-k'}y) &&\text{on } \xi_k^3(\sig_h) \textsp{for} 1 \le h \le k, \\
		&(x,y) \mapsto (q^{k-k'}x, q^{2k-k'-h}y+q^{3k-k'-h}\tfrac{q^{2(h-k)}-1}{q+1}x) 
        		&&\text{on } \xi_k^3(\sig_h) \textsp{for} k \le h < k', \\
        		&(x,y)\mapsto (q^{k-k'}x, q^{2(k-k')}y -q^k\tfrac{q^{2(k-k')}-1}{q+1}x)
		&&\text{on } \xi_k^3(\sig^{k'-1}).
	\end{align*}
 \end{enumerate}

\end{example}

\subsection{Moduli functors and cone decompositions}\label{mofu}

In this subsection, we prove Proposition \ref{logD9} of the introduction.  Recall that it states that a log Drinfeld module with level $N$ structure $\iota$ over an $A = \F_q[T]$-scheme $S$ with saturated log structure $M_S$ satisfies the following two conditions:
\begin{enumerate}
	\item[(1)] Locally on $S$, there exists an $A/NA$-basis $(e_i)_{0\leq i\leq d-1}$ of $(\frac{1}{N}A/A)^d$ such that 
	$$
		\frac{\pole(\iota(a))}{\pole(\iota(e_i))} \in M_S/\cO^\times_S \subset M_S^{\gp}/\cO^\times_S
	$$
	for all $a\in (\frac{1}{N}A/A)^d - \sum_{j=0}^{i-1} (A/NA)e_j$, for each $0 \le i \le d-1$.
	\item[(2)] The values $\pole(\iota(e_i))$ are independent of the choice of $(e_i)_i$ in (1), 
	and $\pole(\iota(e_0))=1$.  
\end{enumerate}
We then consider the moduli functors which appeared in Section \ref{main_result} of the introduction, as well as related moduli functors.

\begin{sbprop}\label{logD90}  
Let $S$ be a normal $A$-scheme, let $U$ be a dense open subset of $S$, let $((\cL, \phi), \iota)$ be a generalized Drinfeld module over $(S, U)$ of rank $d$ with level $N$ structure, and let $$\iota \colon (\tfrac{1}{N}A/A)^d \to \overline{\cL} $$ be the unique extension of $\iota \colon (\frac{1}{N}A/A)^d\to \cL|_U$ on $U$, as in \ref{logD53}. Assume that the condition (div) in \ref{strong} concerning $\pole(\iota(a))$ for $a \in  (\tfrac{1}{N}A/A)^d$ is satisfied.  Then the statements (1) and (2) of \ref{logD9} hold in this situation. 
\end{sbprop}

\begin{proof} We may assume that $S$ is of finite type over $A$ and hence is excellent. 

Let $s\in S$, let $r$ be the rank of the fiber of the generalized Drinfeld module $(\cL, \phi)$ at $s$, and let $n=d-r$. Let $E$ be the
inverse image of $\cL$ under $\iota \colon (\tfrac{1}{N}A/A)^d \to \overline{\cL}$, as in 
\ref{nlevelN}.  By the condition (div), the poles of elements in the image of $\iota$ are totally ordered, so 
there are elements $b_i$ ($1 \le i \le N$) of the finite set $(\tfrac{1}{N}A/A)^d$ such that
$$
	\pole(\iota(b))\pole(\iota(b_i))^{-1}\in M_{S,s}/\cO_{S,s}^\times
$$ 
 for all elements $b$ of $(\frac{1}{N}A/A)^d$ which do not belong to $E_s+ \sum_{j=1}^{i-1} (A/NA)b_j$. 

We define elements $e_i$ of $(\frac{1}{N}A/A)^d$ for $0 \le i \le d-1$ as follows.
Let $(e_i)_{0\leq i\leq r-1}$ be a basis of the free $A/NA$-module $E_s$ of rank $r$. Let $e_i= b_{i-r+1}$ for $r\leq i\leq d-1$. 
Note that $\pole(e_i) = 1$ for all $i \le r-1$, so to show condition (1) of \ref{logD9} for $(e_i)_{0 \le i \le d-1}$, it suffices to 
show that $(e_i)_{0\leq i\leq d-1}$ is an $A/NA$-basis of $(\frac{1}{N}A/A)^d$ at $s$.  
By \ref{todvr}(1), we can show this over discrete valuation rings containing $\cO_{S,s}$ as in (0) of the proof of \ref{todvr}. 
Let $\cV$ be the completion of such a discrete valuation ring, and view $\phi$ as a generalized Drinfeld module over $\cV$. Setting $\beta_i = \iota(b_i)$ 
for $1 \le i \le n$, in the notation of \ref{cpf2}, we are left to show that the images $\lambda_i$ of the $\beta_i$ in $\Lambda/\psi(N)\Lambda$
form a basis. This will follow if we can verify condition 
(ii) of \ref{Ntor}(1).  The condition on $\pole(\iota(b_i))$ above tells us exactly this.

That the statement (2) of \ref{logD9} holds reduces to the statement that the values $\pole(\iota(b_i)) \in M_{S,s}/\cO_{S,s}^{\times}$ 
are independent of the choice of the $b_i$. Again by \ref{todvr}(1), it suffices to see this over each $\cV$ as above, where it 
follows from the statement of \ref{Ntor}(2) that the valuations of the $\beta_i$ are independent of the choice of $\beta_i$
(in that the valuation determines $\beta_i$ up to unit).
\end{proof}

\begin{sbpara}\label{3logD9} 
We prove \ref{logD9}. This is reduced to the case $S$ is log regular and hence to \ref{logD90} by \ref{regdiv}. 
\end{sbpara}

\begin{sbpara} \label{categories}

Recall that $\cC_{\log}$ denotes the category of schemes $S$ with saturated log structures, over $A$ if $N$ has at least two prime divisors and over $A[\frac{1}{N}]$ otherwise.

We define $\cC_{\nl}$ to be the category of pairs $(S,U)$ of a normal scheme $S$, over $A$ if $N$ has at least two prime divisors and over $A[\frac{1}{N}]$ otherwise, and a dense open subset $U$ of $S$.
\end{sbpara}

\begin{sbpara}\label{moduli}

We define moduli functors  $$\begin{matrix} \overline{\fM}^d_N,\   \overline{\fM}^d_{N, \Sig},\ \overline{\fM}^d_{N, +, \sig}  \colon \cC_{\log}\to \text{(Sets)} && (\overline{\fM}^d_N \supset \overline{\fM}^d_{N, \Sig},\; \overline{\fM}^d_N \supset \overline{\fM}^d_{N, +, \sig})\\ 
\overline{\fM}^d_{N, \text{Sa}},\ \overline{\fM}^d_N,\ \overline{\fM}^d_{N, \Sig},\ \overline{\fM}^d_{N, +, \sig} \colon  \cC_{\nl}\to \text{(Sets)}&& (\overline{\fM}^d_{N, \text{Sa}}\supset \overline{\fM}^d_N\supset \overline{\fM}^d_{N, \Sig},\; \overline{\fM}^d_N \supset \overline{\fM}^d_{N, +, \sig}).\end{matrix}$$
 Here, $\Sig$ is a finite rational cone decomposition of $C_d$, and $\sig$ is a finitely generated rational subcone of $C_d$. 
 In particular, we review the definitions of $\overline{\fM}^d_N$ and $\overline{\fM}^d_{N,\Sig}$ on $\cC_{\log}$ from \ref{toroidal_functor} and
 \ref{Cd}, respectively.

\begin{itemize}
	\item For an object
	$S$ of $\cC_{\log}$, let $\overline{\fM}^d_N(S)$ be the set of all log Drinfeld modules
	over $S$ of rank $d$ with level $N$ structure.
	\item For an object $(S, U)$ of $\cC_{\nl}$, let 
$\overline{\fM}^d_{N,\Sa}(S, U)$  be the set of all isomorphism classes of  generalized Drinfeld modules over $(S, U)$ of  rank $d$ with  level $N$ structure (\ref{log4}). Let  $\overline{\fM}^d_N(S, U)$ be the subset of $\overline{\fM}^d_{N,\Sa}(S, U)$ consisting of  those which satisfy (div) of \ref{strong}. 
	\item On both $\cC_{\log}$ and $\cC_{\nl}$, the functor $\overline{\fM}^d_{N, +, \sig}$ (resp., $\overline{\fM}^d_{N,  \Sig}$) is the part of $\overline{\fM}^d_N$ classifying objects that,  
 for the standard basis (resp., locally for some $\sig\in \Sig$ and for some  basis) $(e_i)_{0\leq i\leq d-1}$ of the free $A/NA$-module $(\frac{1}{N}A/A)^d$, satisfy the following two properties:
\begin{enumerate}
\item[(i)] For every $1\leq i\leq d-1$ and for every family $(a_j)_{0\leq j\leq d-1}$ of elements of $A/NA$ such that $a_j \neq 0$ for some $j \geq i$, we have 
$$
	\pole\left(\iota\left(\sum_{j=0}^{d-1} a_je_j\right)\right) \cdot \pole(\iota(e_i))^{-1}\in M_S/\cO_S^\times
$$ 
in $M_S^{\gp}/\cO_S^\times$.
\item[(ii)] We have 
$$
	(\pole(\iota(e_i)))_{1\leq i\leq d-1}\in [\sig](S)
$$ in the sense of \ref{toric} and \ref{toric2}. That is, if $(b_i)_{1 \le i \le d-1} \in \Z^{d-1}$ 
is such that $\sum_{i=1}^{d-1}b_i s_i\geq 0$ for all $s\in \sig$, then 
	$\prod_{i=1}^{d-1} \pole(\iota(e_i))^{b_i}\in M_S/\cO^\times_S$ in $M_S^{\gp}/\cO_S^\times$. 
\end{enumerate}
\end{itemize}
\end{sbpara}

\begin{sbpara}\label{log22} We have a functor $\cC_{\nl} \to \cC_{\log}$ which sends $(S, U)$ to $S$ endowed with the log structure \ref{log2}. This functor is not fully faithful. 

In fact, for pairs $(S, U)$ and $(T, V)$ of normal schemes and dense open subsets, we have the natural map $\pi \colon P \to L$ from the set $P$ of morphisms $(S,U) \to (T, V)$ of pairs to the set $L$ of morphisms $S \to T$ of log schemes, which has the following properties:
\begin{enumerate}
\item[(1)] The map $\pi \colon P\to L$ is injective. 
\item[(2)] The map $\pi$ is not necessarily surjective. For example, if $S=T=U=\Spec(k[T_1, T_2])$ for a field $k$,  and if $V$ is the complement of the origin $(0,0)$ in $T$, the associated log structures of $S$ and $T$ are trivial and hence we have the identity morphism $S\to T$ of log schemes, but this does not come from a morphism of pairs $(S, U) \to (T, V)$. 
\item[(3)] If $V$ coincides with the set of points of $T$ at which the associated log structure is trivial, then $\pi$ is bijective.
\end{enumerate}
\medskip

We used the same notation ($\overline{\fM}^d_N$,  $\overline{\fM}^d_{N, \Sig}$ and $\overline{\fM}^d_{N, +,\sig}$) for functors on different categories $\cC_{\log}$ and $\cC_{\nl}$. We will see in Proposition \ref{M=M} that in the cases that we use the same notation for a functor $F_{\log}$ on $\cC_{\log}$ and a functor $F_{\nl}$ on $\cC_{\nl}$, the functor $F_{\nl}$ coincides with the composition of $F_{\log}$ and $\cC_{\nl}\to \cC_{\log}$. We do not use this coincidence until we  prove \ref{M=M}. 

\end{sbpara}

\begin{sbpara}\label{modulirem} We have a morphism of functors $\overline{\fM}^d_{N, \Sig}\to [\Sig]$ on $\cC_{\log}$ taking an object
with level $N$ structure $\iota$ to 
$\text{pole}(\iota(e_i)))_{1\leq i\leq d-1}$,  where $(e_i)_{1\leq i\leq d-1}$ is as in the definition of $\overline{\fM}^d_{N, \Sig}$ in \ref{moduli}.
This is well-defined by \ref{logD9}.

\end{sbpara}

\begin{sblem} \label{moduliremarks} Let $\Sig$ be a finite rational cone decomposition of $C_d$.
\begin{enumerate} 
\item[(1)] On both $\cC_{\log}$ and $\cC_{\nl}$, 
as a sheaf functor for Zariski topology, we have
$$\overline{\fM}^d_{N, \Sig}= \bigcup_{\sig \in \Sig} \bigcup_{g \in \GL_d(A/NA)} g(\overline{\fM}^d_{N, +, \sig}).$$
\item[(2)] For $\sig\in \Sig$, the functor
$\overline{\fM}^d_{N, +, \sig}$  on $\cC_{\log}$ is an open subfunctor of $\overline{\fM}^d_{N,\Sig}$. That is, for an object $S$ of $\cC_{\log}$ and an element of $\overline{\fM}^d_{N,\Sig}(S)$, the fiber product 
of the functors $S\to \overline{\fM}^d_{N, \Sig}\leftarrow \overline{\fM}^d_{N, +, \sig}$ is represented by an open subset of $S$. 
\item[(3)] For an object $(S, U)$ of $\cC_{\nl}$ and an element  of $\overline{\fM}^d_{N,\Sig}(S,U)$, the fiber product 
of the functors $(S,U)\to \overline{\fM}^d_{N, \Sig}\leftarrow \overline{\fM}^d_{N, +, \sig}$ is represented by $(S',U)$ for an open subset $S'$ of $S$ containing $U$.  
\end{enumerate}
\end{sblem}

\begin{pf} In part (1), the element $g$ changes the basis of $(\frac{1}{N}A/A)^d$, from which the statement is clear, in that it holds locally.

We prove part (2).
Let $\overline{\fM}^d_{N,+,\Sig}$ be the subfunctor of $\overline{\fM}^d_{N, \Sig}$ defined by the condition that the standard basis $(e_i)_{0 \le i \le d-1}$ of $(\frac{1}{N}A/A)^d$ can be used as the basis in the definition of $\overline{\fM}^d_{N,\Sig}$. It is an open subfunctor of $\overline{\fM}^d_{N,\Sig}$ because the condition which defines this subfunctor is that, for $1 \le i \le d-1$, the element 
 $\pole(\iota(e_i))\pole(\iota(a))^{-1}$ is invertible for all $a\in (\frac{1}{N}A/A)^d$ such that $a \notin \sum_{j=0}^{i-1} (A/NA)e_j$ and $\pole(\iota(a)) \mid \pole(\iota(e_i))$ (note that the condition (div) is satisfied).
  We have the cartesian diagram
 $$
 \SelectTips{cm}{} \xymatrix{
  \overline{\fM}^d_{N,+,\sig} \ar[r] \ar[d] & [\sig] \ar[d]  \\
 \overline{\fM}^d_{N,+, \Sig} \ar[r] & [\Sig]},
$$ 
where the lower horizontal arrow is the canonical map factoring through $\overline{\fM}^d_{N, \Sig}$.
The right vertical arrow is also an open immersion. Part (2) follows from this.

The proof of (3) is similar to that of (2).
\end{pf}

In the case of the category $\cC_{\log}$, the following result is Theorem \ref{shape}(1) of the introduction.

\begin{sbprop}\label{Sigk8} Let $k$ be the degree of $N$. 
We have $\overline{\fM}^d_N= \overline{\fM}^d_{N, \Sig_k}$ on both $\cC_{\log}$ and $\cC_{\nl}$.
\end{sbprop}

\begin{proof} Let $S$ (resp., $(S, U)$) be an object of $\cC_{\log}$ (resp., $\cC_{\nl}$), and let  $((\cL, \phi), \iota)$ be an element of $\overline{\fM}^d_N(S)$ (resp., $\overline{\fM}^d_N(S,U)$).  Let $s\in S$.  Working locally at $s$, by \ref{logD9} (resp., \ref{logD90}), we have a basis $(e_i)_{0\leq i\leq d-1}$ of $(\frac{1}{N}A/A)^d$ satisfying the condition (1.1) in \ref{logD9}.  Let $I^{(k)} \subset \Z^3$ be as in \ref{Sigk1}.
We define a map $\alpha \colon I^{(k)} \to \{\R_{\leq 0}, \{0\}, \R_{\geq 0}\}$. Let $(h,i,j)\in I^{(k)}$. Take an element $a$ of $A$ such that $|a|=q^h$. We have either $\pole(\iota(e_i))\pole(\iota(ae_j))^{-1}\in M_S/\cO_S^\times$ or $\pole(\iota(ae_j))\pole(\iota(e_i))^{-1}\in M_S/\cO_S^\times$ at $s$. In the former case, let $\alpha(h,i,j)=\R_{\leq 0}$. Otherwise, let $\alpha(h,i,j)=\R_{\geq 0}$. Define $\tau\in \Sig^{(k)}$ by 
$$
	\tau= \{(s_i)_{1\leq i\leq d-1}\in C_d\mid q^hs_j - s_i\in \alpha(h,i,j)\textsp{for all} (h,i,j)\in I^{(k)}\}
$$ 
(see \ref{Sigk1}). Let $\sig$ be the cone of $\Sig_k$ corresponding to $\tau$ (\ref{Sigk3}). Then by \ref{todvr}, we have $(\pole(\iota(e_i)))_{1 \leq i\leq d-1}\in [\sig](S')$ for some open neighborhood $S'$ of $s$ in $S$.  
\end{proof}

\begin{sbpara}\label{plan45} We describe the plan of Sections \ref{s:Tate} and \ref{toroidal} for our toroidal compactifications. Let $k$ be the degree of the polynomial $N$. 

In Section \ref{s:Tate}, we consider the formal scheme version of the moduli functor $\overline{\fM}^d_{N,+, \sig}$. Assuming that $\sig$ is contained in some cone in $\Sig_k$, we will prove there that it is represented by a formal scheme which is strongly related to  the toric variety $\toric_{\F_p}(\sig) =\toric_\Z(\sig)\otimes_{\Z} \F_p$ defined in \ref{toric2}. 

In Section \ref{toroidal}, we prove that the moduli functor $\overline{\fM}^d_{N, \Sig}$ on $\cC_{\log}$ is  represented by an fs log scheme  $\overline{\cM}^d_{N,\Sig}$ over $A$ which has the properties in Theorem \ref{main} and Theorem \ref{shape}. The proofs use the relation with formal toric varieties proved in Section \ref{s:Tate}.

The category $\cC_{\nl}$ is important because in Section \ref{toroidal}, the theory of toroidal compactification is based on the theory of Satake compactification in Section \ref{Satake} and the latter compactification is constructed  as an object of $\cC_{\nl}$. 
\end{sbpara}

\section{Adic spaces, formal schemes, and formal moduli}\label{s:Tate} 

In this section, we study the following two subjects (1) and (2).

\begin{enumerate}
	\item[(1)] Degeneration of Drinfeld modules on adic spaces, Tate uniformizations: see Section \ref{ss:Tate}.
	\item[(2)] Degeneration of Drinfeld modules on formal schemes, 
 iterated Tate uniformizations, formal moduli spaces: see Sections \ref{iterated}--\ref{itTa}. 
\end{enumerate}

In the subject (1), we will generalize the correspondence 
$$\phi\;  \leftrightarrow \; (\psi, \La)$$
of Drinfeld for complete discrete valuation fields in  \ref{val2} to normal adic spaces, regarding \ref{val2} as a theory on the adic space $\Spa(K, \cV)$. 
This is the theory of Tate uniformizations for Drinfeld modules on adic spaces. Here $\phi$ is obtained as the quotient of $\psi$ by  $\La$. The main result is Theorem \ref{thmTate}. However, as is described in Section \ref{how_to}, unlike the corresponding theory for abelian varieties, this is not so useful for toroidal compactifications of the moduli spaces of Drinfeld modules. This is because the action of local monodromy on $\La$ is too big.

In the subject (2), in the case $A = \F_q[T]$, 
 we give the moduli theory of log Drinfeld modules on formal schemes using iterated Tate uniformizations, which is useful as the local theory for our toroidal compactifications. 
  There we obtain $\phi$ from $\psi$ as 
$$\psi=\phi_0, \phi_1, \phi_2, \cdots, \phi_n=\phi,$$
where $\phi_{i+1}$ is obtained as a quotient of $\phi_i$ by an $A$-lattice of rank $1$ on which the action of the local monodromy is trivial. This is the theory of iterated Tate uniformizations. The main result of the formal moduli  is Theorem \ref{fthm}. It says that the formal moduli space is an open set of a formal toric variety and hence is (formally) log regular.

\subsection{Construction of the quotient} \label{ss:constrquot}

This subsection explains a simple but central construction underlying the Tate and iterated Tate uniformizations that we treat later on in this section. We we work over a complete excellent normal domain $R$. As laid out more carefully in \ref{Fuji2} below, given a pair $(\psi,\Lambda)$ of a generalized Drinfeld module $\psi$ of generic rank $r$ and a projective Galois-stable $\psi(A)$-lattice $\Lambda$ of rank $n$ in the separable closure of the fraction field of $R$, we construct a generalized Drinfeld module $\phi$ of generic rank $d = r+n$ over $R$ as the quotient of $\psi$ by $\Lambda$ via the exponential map of $\Lambda$. This $\phi$ agrees with $\psi$ modulo the ideal of definition of $R$. The reader should note that, unlike in the related construction \ref{val2} for complete valuation rings, $\psi$ is not assumed to be a Drinfeld module.

\begin{sbpara}\label{Fuji2}  As in the introduction, we let $F$ be a function field in one variable over a finite field, and let $A$ be its subring of elements integral outside of a fixed place $\infty$. 

Let $R$ be an excellent normal integral domain over $A$, and let $Q$ be the field of fractions of $R$.  Let $\bar Q$ be an algebraic closure of $Q$, let $Q^{\sep}\subset \bar Q$ be the separable closure of $Q$, and 
let $\bar R$ be the integral closure of $R$ in $\bar Q$. 
Let $I$ be an ideal of $R$ such that $R/I$ is reduced, and assume that $R$ is $I$-adically complete. Let $\bar I$ be the radical of the ideal 
$I\bar{R}$ of $\bar{R}$.

Let $\psi$ be a generalized Drinfeld module over $R$ of generic rank $r$ with trivial line bundle. 
Let $\La$ be an $A$-submodule of $Q^{\sep}$ for the action of $A$ via $\psi$. 
We suppose that the following conditions are satisfied:
\begin{enumerate}
\item[(i)]  $\La$ is stable under the action of $\Gal(Q^{\sep}/Q)$,
\item[(ii)] $\La$ is a projective of rank $n$ as an $A$-module,
\item[(iii)] $\lam^{-1}\in \bar I$ for every nonzero element $\lam$ of $\La$, and
\item[(iv)] for each $k \geq 1$, we have $\lam^{-1} \in I^k\bar R$ for almost all nonzero elements $\lam$ of $\La$. 
\end{enumerate}

We will construct a generalized Drinfeld module $\phi$ over $R$ of generic rank $d=r+n$ such that $\phi \equiv \psi \bmod I$ in \ref{Fuji3}--\ref{Fuji9}.  It can be regarded as the quotient of $\psi$ by $\La$. 

\end{sbpara}

\begin{sbpara}  \label{cases} We will use the following cases of this construction:

\begin{enumerate}
	\item[(1)] $\psi$ is a Drinfeld module of rank $r$. 
	
	 This case is useful in Section \ref{ss:Tate} in the direction $(\psi, \La) \mapsto \phi$ of the correspondence in the Tate uniformization.   
	 It was considered by K. Fujiwara and R. Pink \cite[page 181]{P1}. (The formulation there is slightly different from what we discuss here.) 
 	\item[(2)] $A=\F_q[T]$ and $n=1$.  
	
	This case is useful for iterated Tate uniformizations in Sections \ref{itTa0} and \ref{itTa}.
\end{enumerate}
 \end{sbpara}

\begin{sbpara}\label{Fuji3} Define a formal power series $e(z)$ over $R$ in one variable $z$ by
$$e(z)= z \prod_{\lam} \left(1- \frac{z}{\lam}\right)\in R\ps{z},$$
where $\lam$ ranges over all nonzero elements of $\La$. 
We have  $$e(z)\equiv z \bmod I.$$
Hence $e(z)$ has an inverse function $e^{-1}(z)\in R\ps{z}$ satisfying $e^{-1}(z)\equiv z \bmod I$.
\end{sbpara}

\begin{sbpara}\label{Fuji4}

For $a\in A$, let 
$$\phi(a)(z) = (e\circ \psi(a)\circ e^{-1})(z)\in R\ps{z}.$$ 
We have $\phi(a)(z) \equiv \psi(a)(z) \bmod I$. 
\end{sbpara}

\begin{sbprop}\label{Fuji5} The power series $\phi(a)(z)$ is a polynomial of degree $|a|^d$. 
\end{sbprop}

\begin{proof} 
By \ref{todvr}, we are reduced to the case that $R$ is a complete discrete valuation ring. We can then imitate the proof of Drinfeld in \cite{D} in the case that $\psi$ is a Drinfeld module over $R$.  Let $\psi(a)^{-1}\La= \{x\in \bar Q\mid \psi(a)(x)\in \La\}$. By comparing zeros using nonarchimedean analysis over $Q$, we have 
$$
	e(\psi(a)z)= c \prod_{\beta \in \psi(a)^{-1}\La/\La} (e(z)-e(\beta))^{m(\beta)}
$$ 
for some constant $c\in Q^\times$, with $m(\beta)$ defined as follows. Take a lift $\tilde \beta$ of $\beta$ in $\psi(a)^{-1}\La$, and let $\gamma = \psi(a)(\tilde \beta)$. Then $m(\beta)$ denotes the multiplicity of the root $\tilde \beta$ of the polynomial $\psi(a)(x)- \gamma$, which is independent of the choice of $\tilde \beta$. Hence $e(\psi(a)(z))=\phi(a)e(z)$ is a polynomial in $e(z)$ of degree $|a|^d$.
\end{proof}

\begin{sbpara}\label{Fuji9} By \ref{Fuji5}, we have constructed a generalized Drinfeld module $\phi$ over $R$ of generic rank $d$. 
such that $\phi\equiv \psi \bmod I$. 

\end{sbpara}

\subsection{Tate uniformizations in adic geometry}\label{ss:Tate}

This subsection is devoted to the proof of the following theorem, Theorem \ref{thmTate}, which provides a generalization of the correspondence of \ref{val2} to normal adic spaces. 

\begin{sbthm}\label{thmTate} 
Let $S$ be an adic space over $A$ which has an open covering, each member of which is isomorphic to an open subspace of the adic space $\Spa(R, R)$ associated to the formal scheme $\mr{Spf}(R)$ for some excellent ring $R$ complete with respect to an ideal $I$ of definition. 
We suppose that $S$ is normal (that is, all local rings of $S$ are normal). 

Fix integers $d, r\geq 1$ such that $r\leq d$. Let $n=d-r$.  Then the two categories (a) and (b) which follow are equivalent:

\begin{enumerate}
	\item[(a)] the category of rank $d$ Drinfeld modules $\phi_S$ over $\cO_S$ on $S$ which come from a generalized Drinfeld module 
	$\phi$ over $\cO^+_S$ such that $\phi \bmod I$ is a Drinfeld module of rank $r$ over $\cO^+_S/I$ (here,
	$I$ is the ideal of definition of $\cO^+_S$ such that $\cO^+_S/I$ is reduced),
	\item[(b)] the category of pairs $(\psi, \La)$ with $\psi$ a Drinfeld module of rank $r$ over $\cO^+_S$  and $\La$ 
	an $A$-submodule of $\cO_S\otimes_{\cO^+_S} \cL$ on the \'etale site $S_{\et}$ of $S$, where $\cL$ is the line bundle of 
	$\psi$ and the action of $A$ on $\cO_S\otimes_{\cO^+_S} \cL$ is via $\psi$ (not the usual action),  satisfying the following conditions:
\begin{enumerate}
	\item[(i)] Locally on $S_{\et}$, the $A$-module $\La$ is isomorphic to the constant sheaf associated to a projective $A$-module of rank $n$. 
	\item[(ii)] Every nonzero local section $\lambda$ of $\La$ is a local basis of the line bundle $\cO_S\otimes_{\cO^+_S} \cL$
	and $\lambda^{-1}$ is topologically nilpotent.
	\item[(iii)] Locally on $S_{\et}$, for every ideal $I$ of definition of $\cO^+_S$, we have 
	$\lambda^{-1}\in I\cL$ for almost all nonzero local sections $\lambda$ of $\Lambda$. 
\end{enumerate}
\end{enumerate}
\end{sbthm}

\begin{sbpara} \label{functors} An outline of the proof of Theorem \ref{thmTate} is as follows.

\begin{itemize}
	\item The functor from (b) to (a) is given by the construction of Section \ref{ss:constrquot} by the fact that $S$ locally has the form $\Spa(B,B^+)$, where $B^+$ is $R$ of Section  \ref{ss:constrquot}.
	\item The functor from (a) to (b) is $\phi \mapsto (\psi, \ker(e))$, 
where the pair $(\psi, e)$ is as in Proposition \ref{prop51}. (Note that Lemma \ref{stlem} tells us that $\phi$ in (a) is uniquely determined.) 
The crucial problem is to  show that $\ker(e)$ is big enough, and this is discussed  in \ref{bigLa}--\ref{big2}. 
\end{itemize}

In the proof of Theorem \ref{thmTate}, the theory of simplices of the Bruhat-Tits building and the fan $\Sig_{1,k}$ of \ref{Sigk3} play key roles: see \ref{Ta1k} and \ref{bigLa2}. 

\end{sbpara}

\begin{sbpara}  If $r=d$ (so $n=0$), then the categories (a) and (b) coincide, and hence the theorem is evident. We therefore assume for the rest of this subsection that $r<d$ (i.e., that $n>0$).

\end{sbpara}

\begin{sblem} If the category (a) or the category (b) is nonempty, then the adic space $S$ is  Tate (that is, a topologically nilpotent unit exists locally).

\end{sblem}

\begin{proof}

Assume that the category (a) is nonempty. Locally on $S$, trivialize the invertible $\cO_S^+$-module $\cL_{\phi}$ of $\phi$. For any element $a$ of $A$ which is not in the total constant finite field of $F$, writing $\phi(a)(z)=\sum_{i=1}^m c_iz^i$ with $m=|a|_{\infty}^d$ and $c_i\in \cO_S^+$, we have that $c_m$ is a topologically nilpotent unit of $\cO_S$.

Assume that the category (b) is nonempty. Locally on $S$, trivialize the invertible $\cO_S^+$-module $\cL_{\psi}$ of $\psi$. Then  
the inverse $\lambda^{-1}$ of any nonzero local section  $\la$ of $\La$ is a topologically nilpotent unit of $\cO_S$. 
\end{proof}

For the rest of this subsection, we assume that $S$ is Tate. 
\begin{sbpara} We show that we may assume that $A=\F_q[T]$ in the proof of \ref{thmTate}.

There is a finite flat ring homomorphism  $f \colon \F_q[T]\to A$, and via this homomorphism, a Drinfeld $A$-module of rank $d$ is regarded as a Drinfeld $\F_q[T]$-module of rank $dd'$ where $d'$ is the degree of $f$. A generalized Drinfeld $A$-module is similarly regarded as a generalized Drinfeld $\F_q[T]$-module. The category (a) (resp., (b)) for $(A, d)$ is equivalent to the category of objects $E$ of the category  (a) (resp., (b)) for $(\F_q[T], dd')$ endowed with a ring homomorphism $h \colon A \to \text{End}(E)$ over $\F_q[T]$ such that $\Lie(h(a))=a$ for all $a\in A$. 

\end{sbpara}

In the rest of this subsection, we can suppose that $A=\F_q[T]$.
We prove the following proposition in \ref{big1}--\ref{big2}.

\begin{sbprop}\label{bigLa} Let $\phi$ be as in (a) of \ref{thmTate}, and let $(\psi, e)$ be the associated pair in \ref{prop51}. 
On the \'etale site of $S$, consider the sheaf  $\La :=  \ker(e)$.
\begin{enumerate}
\item[(1)] The sheaf $\La$ satisfies the conditions (i)--(iii) on $\La$ in (b) in \ref{thmTate}. 
\item[(2)] For every complete discrete valuation field $K$  with valuation ring $\cV$ with a morphism $\Spa(K, \cV)\to S$, 
the pullback of $\La$ to $\cV$ coincides with the kernel in the separable closure of $K$ of the map $e$ as in \ref{prop51}
associated to the pullback of $\phi$ to $\cV$. 
\end{enumerate}
\end{sbprop}

\begin{sbpara}\label{big1}  We begin the proof of \ref{bigLa}. We may assume that $T$ is invertible in $\cO_S^+$. In fact, $S$ is covered by the open subspaces $S'$ and $S''$ such that $T$ is invertible on $\cO^+_{S'}$ and $T-1$ is invertible on $\cO^+_{S''}$. On $S''$, we can argue replacing $T-1$ by $T$. 

So, we assume that $T$ is invertible in $\cO_S^+$. 
Working \'etale locally on $S$, we may place a level $T$ structure on our Drinfeld module $\phi_S$ over $\cO_S$ on $S$.
With this assumption, we trivialize the invertible $\cO^+_S$-module $\cL_{\psi}=\cL_{\phi}$ of $\psi$ and $\phi$ by using a nonzero $T$-torsion point of $\psi$, which is a basis of $\cL_{\psi}$ by \ref{unit}. 

\end{sbpara}

\begin{sblem}\label{bb'} Let $\Phi = \phi_S[T]$ be the group of all $T$-torsion points of $\phi_S$. Let $\Psi = e(\psi[T]) \subset \Phi$.
\begin{enumerate}
\item[(1)] The nonzero elements of $\Psi$ are units in $\cO_S^+$. 
\item[(2)] If $\beta\in \Phi\setminus \Psi$, then $\beta^{-1}\in \cO_S^+$ is a topologically nilpotent unit in $\cO_S$. 
\end{enumerate}
\end{sblem}

\begin{proof} Part (1) follows from \ref{*Simage}, as the subset $\Psi$ of $\Phi$ is contained in
$\cO_S^+$. For part (2), we note that \ref{nlevelN} tells us that $\beta^{-1}$ lies in every ideal of definition of $\cO_S^+$, whereas
\ref{*Simage} again tells us that $\beta \in \cO_S^{\times}$.
\end{proof}

\begin{sbpara}\label{TaSig}

We work locally on $S$. Given sections $f$ and $g$ of $\cO_S^+\cap \cO_S^\times$, at least one of $f$ and $g$ divides the other in $\cO^+_S$
by \ref{tl2}.  Hence by \ref{bb'}, there exists a basis $(\gamma_i)_{0\leq i\leq d-1}$ of  the $\F_p$-vector space $\Phi$ such that $\Psi=\sum_{i=0}^{r-1} \F_p \gamma_i$ and such that if we put  $\beta_i=\gamma_{i+r-1}$ for $1\leq i\leq n$, then $\beta_i^{-1} \mid \beta^{-1}$ in $\cO^+_S$ for all $i$ for every $\beta\in \Phi\setminus (\Psi+\sum_{j=1}^{i-1} \F_p\beta_j)$. 
 By \ref{bb'}(2), we have that 
 $\beta_n^{-1} \mid \beta_1^{-c}$ in $\cO^+_S$ for some $c\geq 1$.
 
 Let $\Sig$ be the fan of all faces of the cone
$$\{(a_i)_i \in C_d \mid  a_i=0 \textsp{for} 1\leq i\leq r-1\textsp{and} q^ca_r\geq a_{d-1}\}.$$
Via the map $\pole$ of \ref{tL}, the tuple $(\gamma_i^{-1})_i$ determines an element of $[\Sig](S)$, as defined in \ref{toric}, for the canonical log structure $\cO^+_S\cap \cO_S^\times$ on the sheaf $\cO_S^+$ of \ref{tl2}.

For a positive integer $k$, let $\Sig_{1,k}*\Sig$ be the join of the fans $\Sig_{1,k}$ (\ref{Sigk3}) and $\Sig$. 
By \ref{tl3}, we have that $[\Sig_{1,k}* \Sig](S)=[\Sig](S)$.  
Hence,
$(\gamma_i^{-1})_i$ determines an element of $[\tau_1](S)$ for some  cone  $\tau_1$ in $\Sig_{1,k}* \Sig$. 
Take $\tau\in \Sig_{1,k}$ such that $\tau_1\subset \tau$. We have $\tau=\xi^d_1(\sig)$ for some  $\sig\in \Sig^{(k)}$: that is, recall that 
$\xi_1^d = \xi_{k,1}^d \circ \xi_k^d$ by \ref{thmcone}, that $\xi_k^d$ gives the correspondence $\Sig^{(k)} \to \Sig_k$ of Theorem \ref{Sigk3} by
\ref{Sigk3pf}, and that $\xi_{k,1}^d = (\xi_{1,k}^d)^{-1}$ gives the correspondence $\Sig_k \to \Sig_{1,k}$ of Theorem \ref{Sigk5} by \ref{Sigk5pf}.

\end{sbpara}

\begin{sblem}\label{Ta1k} Suppose that $q^{k-1}>c^{1/r}$. For the face $\sig'$ of $\sig$ given by
$$
	\sig' = \{s\in \sig\mid s_i=0 \textsp{for} 1\leq i\leq r-1\},
$$ 
the set 
$$
	\sig_n =\{(s_{i+r-1})_{1\leq i\leq n}\mid s\in \sig'\} \subset \R^n
$$ 
is the cone associated to some 
 simplex of $\AP_n$ as in \ref{apcone1}. 
\end{sblem}

\begin{proof} We may suppose that $\cO_S^+$ is an excellent ring $R$ complete with respect to $I$, since
$S$ is as in \ref{thmTate}.
Since $S$ is Tate, there is a prime ideal $\frak p$ of $R$ of height one such that $I\subset \frak p$. Let $\cV$ be the completion of the local ring $R_{\frak p}$, let $K$ be the field of fractions of $\cV$, and consider the pullback under $\Spa(K, \cV) \to S$.

Let $s$ be the element of $\sig_n \cap \R_{>0}^n$ such that 
$$
	\xi^d_1(0^{r-1}, (s_i)_{1\leq i\leq n})= (0^{r-1}, (-v_K(\beta_i))_{1\leq i\leq n}).
$$ 
Let $\epsilon=\epsilon_{s_1^r, \dots, s_n^r}^{r,n}$. Since $\beta_n^{-1} \mid \beta_1^{-c}$ in $R$, we have
$-v_K(\beta_n) \le -cv_K(\beta_1)$, and the definition of $\xi^d_1$ in \ref{defcone2} then tells us that
$\epsilon(q^{-r}s_n^r) \leq c\epsilon(q^{-r}s_1^r)$. Hence we have 
 $$q^{-r}s_1^r= \epsilon(q^{-r}s_1^r)\geq c^{-1}\epsilon(q^{-r}s_n^r) \geq c^{-1}q^{-r}s_n^r,$$
and therefore $c^{1/r}s_1 \geq s_n$. Suppose that $\sig_n$ is not associated to a simplex of $\AP_n$, and note
that $\sig_n \in {}_{n+1} \Sig^{(k)}$ as $\sigma \in {}_d \Sig^{(k)}$. As noted in the definition \ref{Sigk1} of $\Sig^{(k)}$, we 
must then have $s_n\geq q^{k-1}s_1$. But then $c^{1/r}s_1\geq q^{k-1}s_1$, and this is a contradiction as 
$s_1>0$ and $q^{k-1}>c^{1/r}$.   
\end{proof}

Proposition \ref{bigLa} is reduced to the following.

\begin{sbprop}\label{bigLa2} Assume $S=\Spa(B, B^+)$ with $R := B^+$.  
Assume the $T$-level structure of $\phi_S$ is such that $(\gamma_i^{-1})_{0\leq i\leq d-1}$  gives an element of $[\tau](S)$ for some $\tau\in \Sig_{1,k}$ with $k\geq 1$ and that the cone $\sig\in \Sig^{(k)}$ such that   $\tau=\xi^d_1(\sig)$  satisfies the condition that for 
$$
	\sig' := \{s\in \sig\mid s_i=0 \textsp{for} 1\leq i\leq r-1\},
$$ 
the cone
$$
	\sig_n := \{(s_{i+r-1})_{1\leq i\leq n}\mid s\in \sig'\}\subset \R^n,
$$ 
corresponds to  a simplex  of $\AP_n$.

Then the following holds locally on $\Spf(R)$: There is a finite separable extension $Q'$ of the field of fractions $Q$ of $R$ such that the integral closure $B'$ of $B$ in $Q'$ is \'etale over $B$ and $\La := \{x\in Q'\mid e(x)=0\}$ has the following two properties.
\begin{enumerate}
\item[(1)] $\La$ satisfies the conditions (i)--(iii) on $\La$ in \ref{thmTate}(b). That is,
\begin{enumerate}
	\item[(i)]  $\La$ is a projective $A$-module of rank $n$,
	\item[(ii)] every nonzero element $\lambda$ of $\La$ is a unit of $B'$ such that $\lambda^{-1}$ is topologically nilpotent,
	\item[(iii)] for each $n\geq 1$ and ideal of definition $I$ of $R$, 
	almost all nonzero elements $\la \in \La$ satisfy $\la^{-1}\in (IR')^n$, where $R'$ denotes
	the integral closure of $R$ in $Q'$.
\end{enumerate}
\item[(2)] For every complete discrete valuation field $K$ and morphism $\Spa(K, \cV)\to S$ with $\cV$ the valuation ring of $K$, 
the kernel of the map $e$ associated to the pullback of $\phi$ to $\cV$ in the separable closure $K^{\sep}$ of $K$ coincides with the pullback of $\La$ to $K^{\sep}$. 
\end{enumerate}
\end{sbprop}

The proof of \ref{bigLa2} is given in \ref{alpha'}--\ref{big2}. 

\begin{sblem}\label{Qdiv} Let $p$ be a prime number, let $R_1$ be a normal integral domain over $\F_p$ with field of fractions $K$, and let $R_2$ be a normal subring of $K$ such that $R_1\subset R_2$. Let $u_i$ ($1\leq i\leq t$) be   elements of $R_1$ which are invertible in $R_2$, and let $m\geq 1$ be an integer. Then there are a finite separable extension $K'$ of $K$ and elements $v_i$ ($1\leq i \leq t$) of $K'$ such that if we denote by $R'_j$ ($j=1,2$) the integral closure of $R_j$ in $K'$, then $R'_2$ is \'etale over $R_2$, $v_i  \in (R_2')^\times$, and $R_1'v_i^m=R_1'u_i$ ($1\leq i\leq t$).

\end{sblem}

\begin{proof} Let $p$ be the characteristic of $\F_q$. Write $m=ap^b$ with integers  $a\geq 1$ and $b \geq 0$ where $a$ is not divisible by $p$. Take $w_i \in K^{\sep}$ ($1\leq i\leq t$) such that $w_i^a=  u_i$ 
and take  $v_i \in (K^{\sep})^\times$ ($1\leq i\leq t$) such that $(v_i^{-1})^{p^b}-v_i^{-1}= w_i^{-1}$. Let $K'= K(v_1, \dots, v_t)$. 
Then $R'_2=R_2[v_1, \dots, v_t]$ and $R'_2$ is \'etale over $R_2$ since the $u_i$ are units in $R_2$. 

By \ref{todvr}(1), it remains to prove the claim that for every additive valuation $\nu \colon K' \to \Gamma\cup\{\infty\}$ (with $\Gamma$ a totally ordered abelian group) for which $\nu(x)\geq 0$ for all $x\in R_1$, we have   $m\nu(v_i)= \nu(u_i)$. 
For such a $\nu$, we have $a\nu(w_i)= \nu(u_i)$ evidently. That $p^b\nu(v_i)=\nu(w_i)$ is easily seen in the case $\nu(v_i) > 0$. The case that $\nu(v_i)<0$ cannot occur, as it would force $\nu(w_i)<0$, but $\nu(w_i) \ge 0$ since $u_i \in R_1$. If $\nu(v_i) = 0$, then $\nu(w_i) \le 0$, so $\nu(w_i) = 0$. 
Together, these imply that $p^b\nu(v_i)=\nu(w_i)$, so we have the claim.
\end{proof}

\begin{sbpara}\label{alpha'} Let $\tau$, $\sig$, $\sig'$, $\sig_n$   be as \ref{bigLa2}, and let 
$$
	\tau'=\{a\in \tau\mid a_i=0 \textsp{for} 1\leq i\leq r-1\}
$$ 
and
$$
	\tau_n= \{(a_{i+r-1})_{1\leq i\leq n}\mid a\in \tau'\}\subset \R^n.
$$ 
Then by \ref{explin}, there is a  $\Q$-linear map $l \colon \Q^n\to \Q^n$ such that the bijection $\xi^d_1 \colon \sig'\to \tau'$ sends $(0^{r-1}, x)\in \sig'$ for $x\in \sig_n$ to $(0^{r-1}, l(x))$.  
Let  $l^{-1} \colon \Q^n\to \Q^n$ be a $\Q$-linear map such that the inverse bijection $(\xi^d_1)^{-1} \colon \tau'\to \sig'$ sends $(0^{r-1}, x) \in \tau'$ for $x\in \tau_n$ to $(0^{r-1}, l^{-1}(x))$. 
Write $l^{-1}(e_i)= \sum_{j=1}^n a_{ij}e_j$ with $a_{ij}\in \Q$, where $(e_i)_{1\leq i\leq n}$ denotes the standard basis of $\Q^n$. Take an integer $m\geq 1$ such that $ma_{ij}q^{-r}\in \Z$ for all $i, j$. 
By \ref{Qdiv} applied to the case $R_1=R$, $R_2=B$, $K=Q$, $p$ the characteristic of $\F_q$, $t=n$ and $u_i= \beta_i^{-1}$ ($1\leq i\leq n$), there are a finite separable extension $Q'$ of $Q$ and $\alpha'_i \in Q'$ ($1\leq i\leq n$) such that 
the integral closure $B'$ of $B$ in $Q'$ is \'etale over $B$, $\alpha'_i\in (B')^\times$, and the integral closure $R'$ of $R$ in $Q'$ satisfies 
$$
	R'(\alpha_i')^m = R'\prod_{j=1}^n \beta_j^{ma_{ij}q^{-r}}
$$ 
for $1\leq j\leq n$.

Replacing $\Spa(B, R)$ by $\Spa(B', R')$, we assume that $\alpha'_i\in B$.

\end{sbpara}

The elements $\alpha'_i$ have the following property. 

\begin{sblem}\label{sbeta} 

Let $K$ be a complete discrete valuation field with valuation ring $\cV$ and with a morphism
$\Spa(K, \cV)\to S$.
Define $s\in \R^n_{\geq 0}$ by 
$$s_i ^r= -q^rv_{\bar K}(\alpha'_i)=-v_{\bar K}(\psi(T)\alpha'_i)\quad (1\leq i\leq n).$$ 
Then 
$$
	\xi_1^d(0^{r-1}, s)=  (0^{r-1}, (-v_{\bar K}(\beta_i))_{1\leq i\leq n}).
$$ 
 
\end{sblem}

\begin{proof} This follows from the definitions of the elements $\alpha_i'$ and $\beta_i$ in \ref{alpha'} using the linear function $\ell$, noting \ref{normcv} for the second equality in the
definition of $s$.
\end{proof}

\begin{sblem}\label{bea} Let $1\leq i\leq n$, and let  $f(z)=\beta_i^{-1}e(\alpha'_iz)=\sum_{j=1}^{\infty} c_jz^j$. Then we have
\begin{enumerate}
\item[(1)] $c_j\in R$ for all $j\geq 1$,
\item[(2)] $c_j\to 0$ as $j\to \infty$, and
\item[(3)] locally on $\Spf(R)$, there exists $m \ge 1$ such that $c_m\in R^\times$.
 \end{enumerate}
 \end{sblem}

\begin{proof} 
    By \ref{todvr}, we may assume that $S=\Spa(K, \cV)$, where $\cV$ is a complete discrete valuation ring over $A$ and $K$ is the field 
    of fractions of $\cV$. 
    Consider the finite set  
    $$
    	E_1= \{\la\in \La\setminus \{0\} \mid \la \in O_{\bar K}\alpha'_i\},
    $$
    and let $E_2=\La\setminus (E_1\cup \{0\})$. 
    By \ref{sbeta}, \ref{eep2}, and the definition of $\xi_1^d$ in \ref{defcone2}, we have
    $\beta_i\sim \alpha_i'\prod_{\la\in E_1} \alpha'_i\la^{-1}$, where $\sim$ means that the ratio is a unit of $O_{\bar K}$. Thus 
    $$
    	\beta_i^{-1}e(\alpha'_iz) \sim z \prod_{\la \in E_1} ((\alpha'_i)^{-1}\la-z) \prod_{\la \in E_2} (1-\alpha'_i \la^{-1}z),
    $$ 
    as in the proof of \ref{eep2}.
    This modulo $m_{\bar K}$ is $z\prod_{\la \in E_1}((\alpha_i')^{-1}\la-z)$. 
\end{proof}

\begin{sblem}\label{lemp} Let $p$ be a prime number, let $R$ be a normal integral domain over $\F_p$, and let $f(z)=\sum_{i=0}^{\infty} c_iz^{p^i} \in R\ps{z}$ such that $c_0\neq 0$ and some $c_i$ is a unit. Suppose that there is  a multiple $\pi$ of $c_0$ in $R$ such that $R$ is $\pi$-adically complete and $c_i \to 0$
in the $\pi$-adic topology.

Let $a\in R^\times$. Then the following holds locally on $\Spf(R)$ (for the $\pi$-adic topology): there is a finite separable extension $Q'$ of the field of fractions $Q$ of $R$ having the property that  there is unit $u$ in the integral closure $R'$ of $R$ in $Q'$  such that $f(u)=a$ and $R'[c_0^{-1}]$ is \'etale over $R[c_0^{-1}]$.

\end{sblem}

\begin{proof} Take $m\geq 0$ such that $c_i\in \pi^3R$ for all $i>m$. Consider the polynomial 
$$
	h(z) := -a + \sum_{i=0}^m c_iz^{p^i}
$$
By \ref{logD100} (applied to the completion of a localization of the integral closure of $R$ in the splitting field of $h$) and the fact that some $c_i \in R^{\times}$, there is an irreducible monic polynomial $P(z)$ over $R$ such that $P(z) \mid h(z)$. 
Let $v$ be a root of $P(z)$, and let $Q'=Q(v)$. 
  Since $v$ divides $a\in R^\times$ in $R'$, we have $v\in (R')^\times$. For $i>m$, write $c_i=\pi^3b_i$ for some $b_i\in R$, and note that these $b_i$ converge to $0$ in the $\pi$-adic topology.

As $R$ is $\pi$-adically complete, there exists $w\in R'$ such that 
$$
	w+ \sum_{i=1}^\infty c_i(\pi c_0^{-1})\pi^{2p^i-3}w^{p^i}= (\pi c_0^{-1})\sum_{i=m+1}^{\infty} b_i v^{p^i}.
$$ 
(Here, note that $2p^i-3>0$ for all $i\geq 1$.) We have 
$$
	f(\pi^2w)= \sum_{i=0}^{\infty} c_i\pi^{2p^i}w^{p^i} = c_0\pi^2\left(w+\sum_{i=1}^{\infty} c_i(\pi c_0^{-1})\pi^{2p^i-3}w^{p^i}\right)= \pi^3\sum_{i=m+1}^{\infty} b_iv^{p^i}= \sum_{i=m+1}^{\infty} c_iv^{p^i}.
$$ 
Let $u = v-\pi^2w\in (R')^\times$. Since $R'$ is an $\F_p$-algebra, we have 
$$
	f(u)= f(v)- f(\pi^2w)=\sum_{i=0}^m c_iv^{p^i}=a.
$$

We prove that $R'[c_0^{-1}]$ is \'etale over $R[c_0^{-1}]$. By replacing $R$ by the $m_R$-adic completion of $R$ for a maximal ideal $m_R$ of $R$, we may assume that $R$ is a local ring and complete for the $m_R$-adic topology. Let $R\langle z\rangle $ be the $m_R$-adic completion of the polynomial ring $R[z]$. Let $F=-a+f$, and let $n$ be the degree of the polynomial $F\bmod m_R\in (R/m_R)[z]$. Then $R\langle z\rangle/(F)$ is generated by $z^i$ ($0\leq i\leq n-1$) as an $R$-module. Hence we have 
$$
	z^n\equiv \sum_{i=0}^{n-1} r_iz^i\bmod F
$$ 
for some $r_i\in R$. Let $P_1$ be the polynomial $z^n-\sum_{i=0}^{n-1} r_iz^i$, and write $P_1= gF$ with $g \in R\langle z\rangle$. Then $g \bmod m_R\in (R/m_R)[z]$ is an element of $(R/m_R)^\times$, and hence $g \in R\langle z\rangle^\times$. Hence $(F)=(P_1)$ as ideals of $R\langle z\rangle$. Let $P_2\in R[z]$ be the monic irreducible polynomial of $u$ over $Q$. Since $F(u) = 0$, we have $P_2 \mid P_1$. Since $\frac{dF}{dz} =c_0$, the product rule tells us that $\frac{dP_2}{dz}$ is invertible in $R[c_0^{-1}]\langle z\rangle/(P_2)$. This proves that $R[c_0^{-1}][z]/(P_2)$ is \'etale over $R[c_0^{-1}]$ and $R'[c_0^{-1}]= R[c_0^{-1}][z]/(P_2)$. 
\end{proof}

\begin{sbpara}\label{big2} We complete the proof of \ref{bigLa2} (and hence of \ref{bigLa}).

By \ref{bea} and \ref{lemp}, there exists (locally on $\Spf(R)$) a finite separable extension $Q'$ of $Q$ such that $B'$ is \'etale over $B$ and for which there is an $x\in (R')^\times$ such that $\beta_i^{-1}e(\alpha'_i x)=1$.  
Let $\alpha_i = \alpha_i'x$ so that $e(\alpha_i)=\beta_i$.

Let $$\La=\sum_{i=1}^n  \psi(AT)\alpha_i.$$ Since $e(\psi(AT)\alpha_i)= \phi(AT)e(\alpha_i)= \phi(AT)\beta_i=0$, we have $\La\subset \ker(e)$. By \ref{sbeta},  $$\xi^d_1(0^{r-1}, s)= (0^{r-1}, (-v_{\bar K}(\beta_i))_{1\leq i\leq n})$$ for the $s\in \R^n_{\geq 0}$ such that $$s_i^r= -v_{\bar K}(\psi(T)\alpha_i)\quad (1\leq i\leq n).$$
By \ref{xicdv1} and \ref{normcv}, we then see that for every complete discrete valuation field $K$ a with valuation ring $\cV$ and with a morphism $\Spa(K, \cV)\to S$, the pullback of $\La$ in $K^{\sep}$ coincides with the kernel of $e$ for the pullback of $\phi$ to $\cV$. Hence we have $\La= \ker(e)$.
 By reduction to the case of $\Spa(K, \cV)$ by pullbacks (using \ref{todvr}), we see from \ref{val2} that $\La$ satisfies the conditions (i)--(iii).

\end{sbpara}

\begin{sbpara}\label{Fuji1} As in \ref{functors}, we now have the functor from category (a) to category (b) of Theorem \ref{thmTate}. The quasi-inverse functor from category (b) to category (a) is obtained by Section \ref{ss:constrquot}.  

The fact they are the quasi-inverse to each other  is reduced to the case of complete discrete valuation rings by  \ref{todvr}. 

\end{sbpara}

\begin{sbpara}\label{4.2Tv}

In the notation of Theorem \ref{thmTate},
let $v$ be a maximal ideal of $A$, and assume that the image of $S\to \Spec(A)$ induced by $A\to \Gamma(S, \cO_S^+)$ does not contain $v$. Let $A_v$ be the $v$-adic completion of $A$. Then we have an exact sequence 
$$
	0\to T_v\psi\overset{e}\to T_v\phi_S \to A_v\otimes_A \La\to 0
$$
of $v$-adic Tate modules on $S_{\et}$. 

\end{sbpara}

\begin{sbpara}\label{4.2Tv2} We compare abelian varieties and Drinfeld modules. 

Let $\cT$ be a toroidal compactification of the moduli space of polarized abelian varieties, let $R$ be the completion of the local ring $\cO_{\cT,x}$ at some $x\in \cT$ (see \cite[Ch.~IV]{FC}), and let $Q$ be the field of fractions of $R$. Let $X$ be the abelian variety over $Q$ obtained from the universal abelian variety over the moduli space. 
As noted in the proof of \ref{2.5.2}, the abelian variety $X$ is analytically the quotient 
 $G/\La$ of a semi-abelian scheme $G$ over $R$ by a finite rank $\Z$-lattice $\La$ of $G(Q)$. This $G$ is similar to the Drinfeld module $\psi$ in this subsection, and the above $\La$ is similar to the $\La$ in this subsection.  
For any prime number which is invertible in $R$, we have an exact sequence of Tate modules 
$$
	0\to T_{\ell}G\to T_{\ell}X\to \Z_{\ell}\otimes_{\Z} \La \to 0
$$ 
of $\Gal(Q^{\sep}/Q)$-modules that is similar to the exact sequence in \ref{4.2Tv}. 

However a big difference is that in the theory of abelian varieties, 
the lattice $\La$ appears in $G(Q)$, but in the theory of Drinfeld modules in this subsection, the lattice $\La$  appears on $S_{\et}$ but cannot appear  in $\cL_{\psi}(Q)$, or even in $\cL_{\psi}(Q^{\sep}) $, for $Q$ the field of fractions of the $R$ studied in this subsection, and $\cL_{\psi}$ the line bundle of $\psi$. If $v$ is not contained in the image of $\Spec(R)\to \Spec(A)$, then
we have a $\Gal(Q^{\sep}/Q)$-module $T_v\phi$, but as the examples  \ref{Gal1} and \ref{Gal2} we provide below show, the image of $\Gal(Q^{\sep}/Q)$ in $\text{Aut}(T_v\phi/T_v\psi)$ 
can be very big so that $T_v\phi/T_v\psi$ cannot have the form $A_v\otimes_A \La$, where the 
action of $\Gal(Q^{\sep}/Q)$ on $\La$ factors through a finite quotient of $\Gal(Q^{\sep}/Q)$. 
Taking a complete discrete valuation ring $\cV\supset R$ which dominates $R$, and letting $K$ denote the field of fractions of $\cV$, we can identify $T_v\phi/T_v\psi$ with $A_v\otimes_A \La$, where $(\psi, \La) $ corresponds to $\phi$ over $K$ and $\La\subset \cL_{\psi}(K^{\sep})$, but this $\La$ does not in general appear in  $\cL_{\psi}(Q^{\sep})\subset \cL_{\psi}(K^{\sep})$. 
\end{sbpara}

\subsection{Log Drinfeld modules in formal geometry} \label{logD_formal}

In this Section \ref{logD_formal}, we extend the theories of generalized Drinfeld modules and log Drinfeld modules for schemes in Section \ref{Drin_mod} to formal schemes. 

\begin{sbpara}\label{fvsadic}  

If $S$ is a locally Noetherian formal scheme and $S_{\adic}$ is the associated adic space, we have a morphism of locally ringed spaces 
$$
	(S_{\adic}, \cO_{S_{\adic}}^+) \to (S, \cO_S).
$$

In the case that $S=\Spf(R)$ and $S_{\adic}=\Spa(R, R)$, the map $S_{\adic}\to S$ sends $x\in S_{\adic}$ to the prime ideal $\{f\in R\mid |f(x)|<1\}$ of $R$. The homomorphism of sheaves of rings from 
the inverse image of 
$\cO_S$ on $S_{\adic}$ to $\cO_{S_{\adic}}^+$ induces the identity map $R=\Gamma(S, \cO_S) \to R=\Gamma(S_{\adic}, \cO_{S_{\adic}}^+)$. 
 
 Hence if $S$ is over $A$, a generalized Drinfeld module over $S$ defines a generalized Drinfeld module over $\cO_{S_{\adic}}^+$ on $S_{\adic}$ by pullback and hence a generalized Drinfeld module over $\cO_{S_{\adic}}$ on $S_{\adic}$.

\end{sbpara}

\begin{sbpara}\label{fpair}
We consider pairs $(S, U)$, where 
$S$ is a locally Noetherian formal  scheme such that $\cO_{S.s}$ is a normal integral domain for every $s\in S$ and $U$ is a dense open subset of $S_{\adic}$.

\end{sbpara}

\begin{sbpara}\label{fpair2}   Let $(S, U)$ be as in \ref{fpair}, and assume that $S$ is over $A$. 

By a \emph{generalized Drinfeld module over $(S, U)$ of rank $d$ with level $N$ structure}, we mean a pair $((\cL, \phi), \iota)$ consisting of a generalized Drinfeld module over $S$ with pullback to $U$ via \ref{fvsadic}(1) a Drinfeld module over $\cO_U$ of rank $d$ and a level $N$ structure $\iota$ on this Drinfeld module over $\cO_U$.

\end{sbpara}

\begin{sbrem}\label{remU} 
\begin{enumerate}
	\item[(1)] In \ref{fpair2},  we consider a dense open subset $U$ of $S_{\adic}$ rather than a dense open subset of $S$. The reason is as follows. In the case $A=\F_q[T]$, a standard situation we consider is that of a generalized Drinfeld module $\phi$ over the $T$-adic completion $\F_q\ps{T}$ of $A$ which is not a Drinfeld module but which induces a Drinfeld module over $\F_q\ls{T}$. Then $\phi$ induces a generalized Drinfeld module over the formal scheme $S=\Spf(\F_q\ps{T})$. But $S$  is just a one point set, so we cannot have a nonempty open set of $S$
on which the pullback of $\phi$ is a Drinfeld module. On the other hand, the pullback of this $\phi$ on $S$ to $S_{\adic}=\Spa(\F_q\ps{T}, \F_q\ps{T})$ induces a Drinfeld module over $\cO_U$ on the dense open set $U=\Spa(\F_q\ls{T}, \F_q\ps{T})$ of $S_{\adic}$. 
	\item[(2)] Let $(S, U)$ and $((\cL, \phi), \iota)$ be as in \ref{fpair2}. Then via the morphism $(U, \cO_U^+)\to (S, \cO_S)$ induced by (1) in \ref{fvsadic}, we obtain on $U$ a generalized Drinfeld module over $\cO_U^+$ which becomes a Drinfeld module over $\cO_U$. This situation is considered in  Section \ref{ss:Tate}, the present $U$ being the $S$ in Section \ref{ss:Tate}.  
\end{enumerate}
\end{sbrem}

\begin{sbpara}\label{fpairlog} For $(S, U)$ as in \ref{fpair}, define the saturated log structure $M$ on $S$ associated to $U$ as  $M= \cO_S\cap j_*(\cO_U^\times)$ in $j_*(\cO_U)$, where $j$ is the composition $U\to S_{\adic}\to S$ and $j_*$ is for the \'etale topologies of $S$ and $U$.
For an object $S'$ of the \'etale site of $S$, the set of $S'$-sections $M(S')$ may be described as follows. We have the induced morphism $S'_{\adic}\to S_{\adic}$ of adic spaces. Let $U'\subset S'_{\adic}$ be the inverse image of $U$. Then $M(S')$ is identified with the set of all elements $f\in \cO(S')$ whose pullbacks to $\cO_{S'_{\adic}}(U')$ are invertible. 

For a line bundle $\cL$ on $S$, this log structure $M$ defines 
$\overline{\cL}= \cL\cup_{\cL^\times} M^{-1}\cL^\times\subset j_*(\cL|_U)$
where $\cL|_U$ means the invertible $\cO_U$-module on $U$ induced by $\cL$. 

\end{sbpara}

\begin{sbprop}\label{fextend}

Let $(S,U)$ be as in \ref{fpair2}, and let $((\cL, \phi), \iota)$ be a generalized Drinfeld module over $(S, U)$ of rank $d$ with level $N$ structure. Then the map $\iota \colon (\frac{1}{N}A/A)^d \to \cL|_U$ on $U$ extends uniquely to a map $\iota \colon (\frac{1}{N}A/A)^d \to \overline{\cL}$ on $S$. 
\end{sbprop}

\begin{proof} This is reduced to \ref{logD53} applied to the normal schemes $\Spec(\cO_{S,s})$ for $s\in S$. 
\end{proof}

\begin{sbpara}\label{flogD}  Let $S$ be a locally Noetherian formal scheme over $A$ with a saturated log structure. By a \emph{log Drinfeld module over $S$ of rank $d$ with level $N$ structure}, we mean a pair $((\cL, \phi), \iota)$ of a generalized Drinfeld module $(\cL, \phi)$  over $S$ and a map $\iota \colon (\tfrac{1}{N}A/A)^d\to \overline{\cL}$
such that \'etale locally on $S$, there are 
\begin{itemize}
	\item a log  scheme $S'$  over $A$ satisfying the equivalent conditions (i)--(iv) in \ref{log3}, 
	\item a closed subscheme $Y$ of $S'$, 
	\item a morphism $f \colon S \to \hat S'$ of log formal schemes over $A$, where $\hat S'$ denotes the formal completion of $S'$ along $Y$, 
	\item a generalized Drinfeld module $((\cL',\phi'), \iota')$ over $(\hat S', U)$ of rank $d$ with level $N$ structure in the sense of \ref{fpair2},
	where $U$ denotes the dense open set of $\hat S'_{\adic}$ consisting of all points at which the log structure on $\cO_{\hat S'_{\adic}}$ 
	is trivial, and 
	\item an isomorphism between $((\cL, \phi), \iota)$ and the pullback of $((\cL', \phi'), \iota')$ under $f$, where $\iota'$ is regarded as a map 
 	$(\frac{1}{N}A/A)^d\to \overline{\cL'}$ as in \ref{fextend}.
\end{itemize}

This is a formal scheme version of \ref{logD6}. 
\end{sbpara}

\begin{sbpara}\label{fstrong} Let $(S, U)$ be as in \ref{fpair2}, and let $((\cL, \phi), \iota)$ be a generalized Drinfeld module over $(S, U)$ of rank $d$ with level $N$ structure. Let $\iota \colon (\frac{1}{N}A/A)^d\to \overline{\cL}$ be the induced map. We consider the following formal scheme version of the condition (div) in \ref{strong}. 

\begin{enumerate}
\item[(div)] For every $a,b\in (\frac{1}{N}A/A)^d$, locally on $S$ we have either 
\begin{eqnarray*}
	\pole(a)\pole(b)^{-1}\in M_S/\cO_S^\times &\text{or}& \pole(b)\pole(a)^{-1}\in M_S/\cO_S^\times
\end{eqnarray*}
in $M_S^{\gp}/\cO_S^\times$. 
\end{enumerate}

We note that, by \ref{toric7}, we have the same condition (div) if we replace ``locally''  by ``\'etale locally''. 

By \ref{regdiv},  if $\Spec(\cO_{S,s})$ with the log structure induced by $M_S$ is log regular for every $s\in S$  (for example, if $S$ is the formal completion of a log regular scheme of finite type over $A$ along a closed subscheme), and if $U$ is the dense open subset of
$S_{\adic}$ consisting of all points at which the log structure on $\cO_{S_{\adic}}$ is trivial, then the condition (div) satisfied.

We therefore have the following proposition.

\end{sbpara}

 \begin{sbprop} Let $S$ be a locally Noetherian formal scheme over $A$ with a saturated log structure, and let $((\cL, \phi), \iota)$ be a log Drinfeld module over $S$ of rank $d$ with level $N$ structure. Then for every $a, b\in (\frac{1}{N}A/A)^d$, locally on $S$ we have either $\pole(\iota(a))\pole(\iota(b))^{-1}\in M_S/\cO_S^\times$ or $\pole(\iota(b))\pole(\iota(a))^{-1}\in M_S/\cO^\times_S$ in $M^{\gp}_S/\cO_S^\times$.

 \end{sbprop}

\begin{sbprop}
Suppose that we are in one of the following two situations:
\begin{enumerate}
\item[(i)] $S$ is a locally Noetherian formal scheme over $A$ with a saturated log structure, and $((\cL, \phi), \iota)$ is a log Drinfeld module over $S$ of rank $d$ with level $N$ structure, or
\item[(ii)] $(S, U)$ is as in \ref{fpair2}, and $((\cL, \phi), \iota)$ is a generalized Drinfeld module over $(S,U)$ of rank $d$ with level $N$ structure. 
\end{enumerate}
If, furthermore, either $N$ has at least two prime divisors or $N$ is invertible on $S$, 
then the automorphism group of $((\cL, \phi), \iota)$ is trivial. 
\end{sbprop}

\begin{proof} 
This is reduced to \ref{unit} and \ref{unit2}. 
\end{proof}

\subsection{Formal moduli}\label{iterated}

In this subsection, we take $A = \F_q[T]$ and consider the formal scheme versions of the moduli functors for schemes defined in Section \ref{mofu}. We state our main theorem \ref{fthm} in the setting of formal moduli. It will be proved in Section \ref{itTa}.

\begin{sbprop}\label{flogD90}  Suppose that we are in one of the following two situations:
\begin{enumerate}
\item[(i)] $S$ is a locally Noetherian formal scheme over $A$ with a saturated log structure, and $((\cL, \phi), \iota)$ is a log Drinfeld module over $S$ of rank $d$ with level $N$ structure, or
\item[(ii)] $S$ is a formal scheme over $A$  which has an open covering by  $\Spf(R)$ with $R$ an excellent normal domain, $U$ is a dense open subset of the adic space $S_{\adic}$, and $((\cL, \phi), \iota)$ is a generalized Drinfeld module over $(S,U)$ of rank $d$ with level $N$ structure satisfying the divisibility condition (div) in \ref{fstrong}. 
\end{enumerate}
Then we have the same statements as (1) and (2)  in \ref{logD9} concerning $\pole(\iota(a))\in M_S/\cO_S^\times$ for $a\in (\frac{1}{N}A/A)^d$ in this formal situation.

\end{sbprop}

 \begin{proof} This can be proved in the same way as \ref{logD9} (in \ref{3logD9}) for (i) and \ref{logD90} for (ii), 
 using the reduction \ref{todvr} to the case of complete discrete valuation rings described in \ref{logD90}. 
 \end{proof}

\begin{sbpara}\label{formal} 

In the case  $N$ has at least two prime divisors (resp., $N$ has only one prime divisor),  we define the categories $\hat \cC_{\log}$ and $\hat \cC_{\nl}$ as follows. 
\begin{itemize}
	\item Let $\hat \cC_{\log}$ be the category of locally Noetherian formal schemes $S$  over $A$  (resp., $A[\frac{1}{N}]$) endowed with a saturated log structure $M_S$ on the \'etale site $S_{\et}$. 
	\item Let $\hat \cC_{\nl}$ be the category of pairs $(S, U)$, where $S$ is a formal scheme over $A$ (resp., $A[\frac{1}{N}]$) which has an open covering by  $\Spf(R)$ with $R$ an excellent normal domain and $U$ is a dense open subset of the adic space $S_{\adic}$. 
\end{itemize}

As excellent rings are noetherian, for any object $(S,U)$ of $\hat \cC_{\nl}$, the formal scheme $S$ lies in $\hat \cC_{\log}$, taking the saturated log structure associated to $U$, as in \ref{fpairlog}.
 
 \end{sbpara}

\begin{sbpara}\label{sigrnd}

Let $r\geq 1$, $n\geq 0$, and $d=r+n$.

Let $\sig$ be a finitely generated rational subcone of $C_d$.
We define functors 
$$\overline{\fM}^{r,n}_N, \, \overline{\fM}^{r,n}_{N, +, \sig} \colon \hat \cC_{\log}\to  \text{(Sets)}\qquad (\overline{\fM}^{r,n}_N \supset \overline{\fM}^{r,n}_{N, +, \sig})$$ 
$$\overline{\fM}^{r,n}_N, \, \overline{\fM}^{r,n}_{N, +, \sig} \colon \hat \cC_{\nl}\to  \text{(Sets)}\qquad (\overline{\fM}^{r,n}_N \supset \overline{\fM}^{r,n}_{N, +, \sig})$$  as follows.

For an object $S$ of $\hat \cC_{\log}$, let $\overline{\fM}^{r,n}_N(S)$ be the set of all isomorphism classes of log Drinfeld modules $((\cL, \phi), \iota)$ over $S$ of rank $d$ with level $N$ structure satisfying the condition (*) below.

For $(S,U)\in \hat \cC_{\nl}$, let $\overline{\fM}^{r,n}_N(S,U)$ be the
set of all isomorphism classes of  
generalized Drinfeld modules $((\cL, \phi), \iota)$  over $(S,U)$ of rank $d$ with  level  $N$ structure 
satisfying the divisibility condition (div) of \ref{fstrong} and the following condition  (*).

\begin{enumerate}
\item[(*)] Let $I\subset \cO_S$ be the ideal of definition such that $\cO_S/I$ is reduced. Then $(\cL,\phi) \bmod I$ is a Drinfeld module over $\cO_S/I$ of rank $r$. 
\end{enumerate}

On both $\hat \cC_{\log}$ and $\hat \cC_{\nl}$,  the subfunctor $\overline{\fM}^{r,n}_{N, +, \sig}$   of $\overline{\fM}^{r,n}_N$  is defined 
 by putting conditions on  the divisibility of $\pole(\iota(a))$ for $a\in (\frac{1}{N}A/A)^d$ in the same way as the cases of the categories $\cC_{\log}$ and $\cC_{\nl}$ in \ref{moduli}, respectively.

We use the same notation for functors on different categories $\hat \cC_{\log}$ and $\hat \cC_{\nl}$, but this is fine, as is seen in \ref{fM=M} below.

\end{sbpara}

\begin{sbpara} When we consider  $((\cL,\phi), \iota)\in \overline{\fM}^{r,n}_{N, +, \sig}$, we trivialize the line bundle $\cL$ by $\iota(e_0)$ which is a basis of $\cL$, where $(e_i)_{0 \le i \le d-1}$ is the standard basis of $(\frac{1}{N}A/A)^d$.

\end{sbpara}

\begin{sbpara}\label{toMrn} On the category  $\hat \cC_{\log}$ (resp.,  $\hat \cC_{\nl}$), we have a canonical morphism of functors $((\cL, \phi), \iota)\mapsto ((\cL, \psi), \iota')$  from 
$\overline{\fM}^{r,n}_{N,+,\sig} $ to the functor $S\mapsto \text{Mor}(S, \cM^r_N)$ (resp., $(S,U)\mapsto  \text{Mor}(S, \cM^r_N)$), where $\text{Mor}$ is the set of morphisms of locally ringed spaces over $A$. Here, $\psi$ is that of $(\psi, e)$  in \ref{prop51}, and $\iota'$ is  defined by $e(\iota'(a_0, \dots, a_{r-1}))= \iota(a_0, \dots,a_{r-1}, 0, \dots, 0)$.

\end{sbpara}

\begin{sbpara}\label{sig'n} Let $k$ be the degree of the polynomial $N$. Let $r\geq 1$ and $n\geq 0$ with $r+n=d$. Let $\sig$ be a finitely generated rational subcone of $C_d$ such that $\sig\subset \tau$ for some $\tau\in \Sig_k={}_d\Sig_k$.  Let $\sig'$ be the face $\{s\in \sig\mid s_i=0\textsp{for} 1\leq i\leq r-1\}$ of $\sig$, and let $\sig_n= \{(s_{i+r-1})_{1\leq i\leq n}\mid s\in \sig'\}$ so that $\sig'=\{0\}^{r-1}\times \sig_n$.

Let $p$ be the characteristic of $F$, and let $\toric_{\F_p}(\sig_n) $ be the toric variety $\toric_{\Z}(\sig_n) \otimes_{\Z} \F_p$ over $\F_p$ associated to $\sig_n$ (see \ref{toric2}). 
 That is,  $\toric_{\F_p}(\sig_n)= \Spec(\F_p[\sig_n^{\vee}])$, where 
 $$
 	\sig_n^{\vee}=\left\{(b_i)_{1 \le i \le n} \in \Z^n\mid\sum_{i=1}^n b_is_i\geq 0\textsp{for all}s\in \sig_n\right\}.
$$

\end{sbpara}

\begin{sbpara}\label{cM(sig)} Let the notation be as in \ref{sig'n}.

Let $\overline{\cM}^{r,n}_{N,(+,\sig)}$ be the formal completion of $\cM_N^r \times_{\F_p} \toric_{\F_p}(\sig_n)$ along the closed subscheme that is the inverse image of $(0, \dots, 0) \in \mathbb{A}_{\F_p}^n$ under 
$$
	\cM^r_N \times_{\F_p}\toric_{\F_p}(\sig_n)\to \toric_{\F_p}(\sig_n) \to \mathbb{A}^n_{\F_p}.
$$ 
Here the latter morphism $\toric_{\F_p}(\sig_n) \to  \mathbb{A}^n_{\F_p}$ of toric varieties is induced by
the embedding $\sig_n\subset \R^n_{\geq 0}$ of cones. Endow $\overline{\cM}^{r,n}_{N, (+,\sig)}$ with the inverse image of the standard log structure of $\toric_{\F_p}(\sig_n)$.

The object $\overline{\cM}^{r,n}_{N,(+,\sig)}$ represents the functor 
$$\hat \cC_{\log} \to (\text{Sets}), \quad S\mapsto \{((\psi, \iota), t_1,\dots, t_n)\},$$
where $\psi$ is a Drinfeld module over $S$ with level $N$ structure $\iota$, the bundle of which is trivialized by $\iota(e_0)$, and where the $t_i \in \Gamma(S, M_S)$ satisfy the following conditions.
\begin{enumerate}
\item[(i)] The image of $t_i$ in $\cO_S$  is topologically nilpotent. 
\item[(ii)]  The element  $(t_i\bmod \cO^\times_S)_{1\leq i\leq n}$ of $\Gamma(S, (M^{\gp}_S/\cO_S^\times)^n) $ belongs to $[\sig_n](S)$. That is, 
 $\prod_{i=1}^n t_i^{b_i} \in M_S$ for all $(b_i)_{1 \le i \le n} \in \sigma_n^{\vee}$.
 \end{enumerate}
 
 In fact, first, $\cM^r_N\times \toric_{\F_p}(\sig_n)$ represents the functor $S\mapsto \{((\psi, \iota), h)\}$ where $(\psi, \iota)$ is as above and $h$ is a homomorphism $\sig_n^{\vee} \to M_S$. Next, via the homomorphism $\bN^n \to \sig_n^{\vee}$ corresponding to the embedding $\sig_n \subset \R^n_{\geq 0}$, this functor may be rewritten as  $S\mapsto \{((\psi, \iota), t_1, \dots, t_n)\}$, where $(\psi, \iota)$ is as above and $t_1,\dots, t_n\in \Gamma(S,M_S)$ satisfy the above condition (ii). This follows from \ref{toSig0} applied to the case that $\Sig$ is the fan of all faces of $\R^n_{\geq 0}$ and $\Sig'$ is the fan of all faces of $\sig_n$. Hence the formal completion represents the above functor.

 Let $\cM^{r,n}_{N,(+,\sig)}$ be the dense open set of the adic space 
$(\overline{\cM}^{r,n}_{N, (+,\sig)})_{\adic}$ associated to $\overline{\cM}^{r,n}_{N, (+,\sig)}$ consisting of all points at which the log structure is trivial. 
Then  $\cM^{r,n}_{N,(+,\sig)}$  is the  open set of $(\overline{\cM}^{r,n}_{N, (+, \sig)})_{\adic}$ consisting of all points at which $t_1\cdots t_n$ is invertible (i.e., they are invertible in $\cO_W$ for $W = \cM^{r,n}_{N,(+,\sig)}$, not necessarily in $\cO_W^+$). 

Since $\cM^r_N$ is a regular scheme (see \ref{002}) and $\toric_{\F_p}(\sig_n)$ is log smooth over $\F_p$, the log scheme $\overline{\cM}^{r,n}_{N,(+,\sig)}$ is log regular as the completion of a log regular scheme (see \ref{logD20}(3)) along a closed subscheme. 
 
  \end{sbpara}

 \begin{sbpara} \label{theta}
 On the category $\hat \cC_{\log}$ (resp., $\hat \cC_{\nl}$), we have a canonical morphism 
 \begin{eqnarray*}&\theta \colon \overline{\fM}^{r,n}_{N, +, \sig}\to \overline{\cM}^{r,n}_{N, (+, \sig)}\quad (\text{resp., } \overline{\fM}^{r,n}_{N, +, \sig}\to (\overline{\cM}^{r,n}_{N, (+, \sig)},\cM^{r,n}_{N, (+, \sig)}))\\
&(\phi, \iota)\mapsto  ((\psi, \iota'),  t_1, \dots, t_n),
\end{eqnarray*}
 where $(\psi, \iota')\in \cM^r_N$ is as in \ref{toMrn} and 
 $$t_i := \frac{\iota(e_0)}{\iota(e_{i+r-1})} \textsp{for} 1\leq i\leq n.$$

  \end{sbpara}
  
  The following is the main theorem of the formal moduli. 

\begin{sbthm}\label{fthm} Let $\sig$ be a finitely generated rational subcone of an element of $\Sig_k$. 
There is a unique open subset $\overline{\cM}^{r,n}_{N, +, \sig}$ of the formal scheme $\overline{\cM}^{r,n}_{N, (+, \sig)}$
of \ref{cM(sig)} such that the morphism $\theta$ 
of \ref{theta} induces an isomorphism
$$\overline{\fM}^{r,n}_{N, +, \sig}\xrightarrow{\sim} \overline{\cM}^{r,n}_{N, +, \sig}$$ 
of functors on $\hat \cC_{\log}$ and such that if $\cM^{r,n}_{N,+,\sig}$ denotes the open set of $(\overline{\cM}^{r,n}_{N, +, \sig})_{\adic}$ consisting of all points at which the log structure is trivial, then $\theta$ induces an isomorphism
$$\overline{\fM}^{r,n}_{N, +, \sig}\xrightarrow{\sim} (\overline{\cM}^{r,n}_{N, +, \sig},\cM^{r,n}_{N, +, \sig})$$ 
of functors on $\hat \cC_{\nl}$. 

\end{sbthm}

\begin{sbpara} \label{fthmn=0}
	If $n = 0$, then the functor $\overline{\fM}_{N,+,\sig}^{r,0} = \overline{\fM}_N^{r,0}$ applied to $S \in \hat \cC_{\log}$ or $(S,U) \in \hat \cC_{\nl}$ is the set of isomorphism classes of Drinfeld modules over $S$ of rank $r$ with level $N$ structure, and $\overline{\cM}_{N,(+,\sig)}^{r,0} = \cM_N^r$ is the corresponding moduli space. Thus, Theorem \ref{fthm} holds in this case.
\end{sbpara}

The following proposition follows from Theorem \ref{fthm}. 

\begin{sbprop}\label{fM=M} Let $F \in \{\overline{\fM}^{r,n}_N, \overline{\fM}^{r,n}_{N, +, \sig}\}$. Then the functor $F$ on the category $\hat \cC_{\nl}$ coincides with the functor sending $(S,U)$ to $F(S)$ for the functor $F$ on $\hat \cC_{\log}$ applied to $S$ endowed with the log structure associated to $U$ (see \ref{log2}) .
\end{sbprop}

\subsection{Iterated Tate uniformizations 1}\label{itTa0}

We prove Theorem \ref{fthm} in Section \ref{itTa} using the concept of iterated Tate uniformization. In preparation, we consider here the case in which the basis is  a complete discrete valuation ring. 

Let $\cV$ be a complete discrete valuation ring over $A = \F_q[T]$. Let $K$ be the field of fractions of $\cV$, let $\bar K$ be an algebraic closure of $K$, and let $\bar \cV$ be the integral closure of $\cV$ in $\bar K$. Let $m_{\cV}$ and $m_{\bar \cV}$ be the maximal ideals of $\cV$ and of $\bar \cV$, respectively. Let $v_{\bar K} \colon \bar K \to \Q$ be an additive  valuation induced by an additive valuation $v_K$ of $\cV$. 

Let $N$ be an element of $A$ of positive degree. If $N$ has only one prime divisor, we assume $N$ is invertible in $\cV$.
Let $r\geq 1$ and $n\geq 0$, and set $d=r+n$.

\begin{sbpara}\label{iTatedvr}  For a finitely generated subcone $\sig$ of $C_d$, we will write $\overline{\fM}^{r,n}_{N, +, \sig}(S, U)$ with $S=\Spf(\cV)$ and $U=\Spa(K, \cV)$ simply by  $\overline{\fM}^{r,n}_{N,+,\sig}(\cV)$.
In \ref{iTatedvr}--\ref{iotai}, we take 
$(\phi, \iota)\in \overline{\fM}^{r,n}_{N, +, C_d}(\cV)$.

Let $(\psi, \iota')$ be the image of this element in $\cM^r_N(\cV)$ (see \ref{toMrn}). We discuss the iterated Tate uniformization of $(\phi, \iota)$. That is, we explain how $(\phi, \iota)$ is obtained from $(\psi, \iota')$ by constructing $(\phi_i, \iota_i)\in \overline{\fM}^{r,i}_{N, +, C_{r+i}}(\cV)$ with $0\leq i\leq n$ such that  $(\phi_0,\iota_0)=(\psi, \iota')$ and $(\phi_n,\iota_n)=(\phi, \iota)$ and such that $(\phi_i, \iota_i)$ for $1\leq i\leq n$ is the quotient of $(\phi_{i-1}, \iota_{i-1})$ by an $A$-lattice of rank $1$. 
 
Let $\La$ be the $\psi(A)$-lattice in $K^{\sep}$ corresponding to $\phi$ as in \ref{val2}. By \ref{Ntor}(1), there is a $\psi(A)$-basis $(\la_i)_i$ of $\La$ satisfying the equivalent conditions (i)--(iii) in \ref{diagbase}(1) for which there are $\la'_i\in K^{\sep}$ with $1\leq i\leq n$ such that $\psi(N)(\la'_i)= \la_i$ and $e_{\La}(\la_i')= \iota(e_{i+r-1})$. Here $(e_i)_{0\leq i\leq d-1}$ is a standard $A/NA$-basis of $(\frac{1}{N}A/A)^d$. 
We have 
$$0=\La_0 \subsetneq \La_1\subsetneq \dots \subsetneq \La_n=\La,$$
and the $\La_i/\La_{i-1}$ for $1\leq i\leq n$ are of $A$-rank one.

For $0\leq i\leq n$, let $\phi_i$ be the generalized Drinfeld module over $\cV$ of generic rank $r+i$ corresponding to the pair $(\psi, \La_i)$ as in \ref{val2}. In particular, we have $\phi_0=\psi$ and $\phi_n=\phi$. Let $e_{i,0} \colon \bar K \to \bar K$ be the exponential map of the $\psi(A)$-lattice $\La_i$. 

For $1\leq i\leq n$, let $\gamma_i=e_{i-1,0}(\la_i)$, which is an element of $K$ since $\gamma_i = e_{i-1,0}(\psi(N)\lambda'_i) = \phi_{i-1}(N)\iota(e_{i+r-1})$. Then $\phi_i$  for $1\leq i\leq n$ is obtained from the pair $(\phi_{i-1}, \phi_{i-1}(A)\gamma_i)$ by the method of Section \ref{ss:constrquot}.   That is, $\phi_i$ is obtained from $\phi_{i-1}$ as its quotient by the $A$-lattice $\phi_{i-1}(A)\gamma_i$ of rank one, placing us in case (2) of \ref{cases}.

For $1\leq i\leq n$, let  $e_{i,i-1}$ be the exponential map associated to the $\phi_{i-1}(A)$-lattice $\phi_{i-1}(A)\gamma_i$ so that $e_{i,0}= e_{i,i-1}\circ e_{i-1,0}$. We may then define $e_{i,j}$ for $0\leq j\leq i\leq n$ as the composition $e_{i,i-1}\circ e_{i-1, i-2}\circ \dots \circ e_{j+1,j}$. 

\end{sbpara}

\begin{sblem}\label{xi0}  For $1\leq i\leq n$, let $s_i = (-v_{\bar K}(\la_i))^{1/r}$ be the norm of $\la_i$ (see \ref{normcv}). Let $k$ be the degree of $N$.
Then we have
$$ (0^{r-1}, (-v_K(\gamma_i))_{1\leq i\leq n})= \xi^d_0(0^{r-1}, (s_i)_{1\leq i\leq n})=\xi^d_{k,0}((-v_K(\iota(e_i)))_{0\leq i\leq d-1}).$$

\end{sblem}

\begin{proof}
	As the $\lambda_i$ for $1 \le i \le n$ satisfy the conditions of \ref{diagbase}(1), we have $v_{\bar K, \Lambda_{i-1}}(\lambda_i) = v_{\bar K}(\lambda_i)$. From \ref{eep}, 
	it follows that $-v_K(\gamma_i) = \epsilon_{s_1^r,\ldots, s_{i-1}^r}^{r,i-1}(s_i^r) = \epsilon_{s_1^r, \ldots, s_n^r}^{r,n}(s_i^r)$. By definition of $\xi_0^d$ from \ref{defcone2},
	we then have the first equality. 
	
	Since $\iota(e_{i+r-1}) = e_{\La}(\lambda_i')$ and $\psi(N)(\lambda'_i) = \lambda_i$
	for $1 \le i \le N$, Proposition \ref{xicdv1} implies that $\xi_k^d(0^{r-1},(s_i)_{1 \le i \le n}) = (0^{r-1},-v_K(\iota(e_{i+r-1}))_{1 \le i \le n})$. 
	Moreover, since $\iota(e_i) \in \ker \psi(N)$ for $0 \le i \le r-1$, we have $-v_K(\iota(e_i)) = 0$ for such $i$ by \ref{normcv}. 
	As $\xi^d_{k,0}
	= \xi_0^d \circ (\xi_k^d)^{-1}$ by definition (see \ref{thmcone}), the second equality holds.
\end{proof}

\begin{sblem}\label{econg}  Let $1\leq i\leq n$. 
\begin{enumerate}
\item[(1)] The map  $e_{i,i-1} \colon \bar K \to \bar K$ induces an automorphism of $\gamma_i m_{\bar \cV}$. 
\item[(2)] For $b\in m_{\bar \cV}$, the map $e_{i,i-1}$ induces an automorphism  of $\gamma_i b \bar \cV$ that reduces to the identity map on $\gamma_i b \bar \cV/\gamma_i b^2 \bar \cV$. 
\item[(3)] We have $\iota(e_j) \in \gamma_i m_{\cV}$ for $0\leq j\leq i+r-1$. 
\end{enumerate}
\end{sblem}

\begin{proof} Let $b$ be a nonzero element of $m_{\bar \cV}$. We prove that 
$$
	e_{i,i-1}(\gamma_i b z)\equiv \gamma_i  b z \bmod \gamma_i b^2\bar \cV\ps{z}.
$$ 
In fact, 
$$
	(\gamma_i b)^{-1}e_{i,i-1}(\gamma_i b z)= z \prod_{a \in A\setminus \{0\}} (1- (\phi_{i-1}(a)\gamma_i)^{-1}\gamma_i bz).
$$ 
It is sufficient to show that $\gamma_i \in (\phi_{i-1}(a)\gamma_i)\cV$ for all $a\in A\setminus \{0\}$. 
Employing \ref{eep} (noting that $v_{\bar K,\Lambda_{i-1}}(\psi(a)\lambda_i) = v_{\bar K}(\psi(a)\lambda_i)$) and \ref{normcv}, we have 
$$-v_K(\phi_{i-1}(a)\gamma_i)= \epsilon^{r, i-1}_{s_1^r, \dots, s_{i-1}^r}(|a|^rs_i^r)\geq  \epsilon^{r, i-1}_{s_1^r, \dots, s_{i-1}^r}(s_i^r)=-v_K(\gamma_i).$$

Hence $e_{i,i-1}$ induces the identity map of $\gamma_i b^h \bar\cV/\gamma_i b^{h+1}\bar \cV$ for every $h\geq 1$ and hence induces an isomorphism $\gamma_i b \bar \cV \xrightarrow{\sim} \gamma_i b \bar \cV$. Hence it induces an isomorphism $\gamma_i m_{\bar \cV}\xrightarrow{\sim} \gamma_i m_{\bar \cV}$.  That is, we have proven (1) and (2).

We prove (3). Let $k$ be the degree of $N$. For $1 \leq j\leq i$, we have
$$- v_K(\gamma_i) = \epsilon^{r,i-1}_{s_1^r, \dots, s_{i-1}^r}(s_i^r)=\epsilon^{r,n}_{s_1^r, \dots, s_n^r}(s_i^r) \geq \epsilon^{r,n}_{s_1^r, \dots, s_n^r}(q^{-kr}s_j^r)= -v_K(\iota(e_{r+j-1})).$$
\end{proof}

\begin{sbpara}\label{iotai}
For $0\leq i\leq n$, we define a level $N$ structure $\iota_i$ for the pullback of $\phi_i$ to $K$ as follows. For the standard basis $(e_j)_{0\leq j\leq i+r-1}$ of $(\frac{1}{N}A/A)^{r+i}$, let $\iota_i(e_j)$ (where $0\leq j\leq i+r-1$) be the unique element of $K$ satisfying $v_K(\iota_i(e_j))= v_K(\iota(e_j))$ and
 $e_{n,i}(\iota_i(e_j))= \iota(e_j)$. (The existence and the uniqueness of such an element follow from 
\ref{econg}.)

By construction, we have $(\phi_i, \iota_i)\in  \overline{\fM}^{r,i}_{N, +, C_{r+i}}(\cV)$.

\end{sbpara}

\begin{sbpara}\label{(sharp)} 
Suppose now that $n \ge 1$.
The goal of the rest of this section is to prove Proposition \ref{4.5prop} which gives a bijection of the form 
$$
	(\phi, \iota) \mapsto  ((\phi_{n-1}, \iota_{n-1}), t)
$$ 
for $t$ in a certain subset of $m_{\cV} \setminus \{0\}$,
as a refinement of the above map $(\phi, \iota)\mapsto (\phi_{n-1}, \iota_{n-1})$.

We start with a larger set of $t$ than will occur in this bijection. That is, let 
$$
	\overline{\fM}^{r,n-1}_{N,+, C_{d-1}}(\cV)^{(\sharp)} = \{ ((\phi', \iota'), t) \in {\fM}^{r,n-1}_{N, +, C_{d-1}}(\cV) \times (m_\cV \setminus \{0\}) \mid v_K(t) \geq -v_K(\iota(e_{d-2})) \}.
$$ 
Note that the condition on $t$ in this definition is exactly that $((-v_K(\iota(e_i)))_{0\leq i\leq d-2}, v_K(t)) \in C_d$.

\end{sbpara}

\begin{sbpara}\label{n-1objects}

Let $((\phi', \iota'), t)\in \overline{\fM}^{r,n-1}_{N, +, C_{d-1}}(\cV)^{(\sharp)}$. 
Let $(\psi, \La')$ be the pair in \ref{val1} corresponding to $\phi'$, and let $e_{\La'}$ be the exponential map associated to $\La'$. 

Replacing $(\phi, \iota)$ in \ref{iTatedvr} and \ref{xi0} by $(\phi', \iota')$,
let $(\la_i)_{1\leq i\leq n-1}$ with $\la_i\in \La'$ and $(\gamma_i)_{1\leq i\leq n-1}$ be chosen and defined as before, and set
$s_i= (-v_{\bar K}(\la_i))^{1/r}$.

Let $\alpha$ be  an element of $\bar K$ such that $e_{\La'}(\alpha)= t^{-1}$ and such that $v_{\bar K,\La'}(\alpha) = v_{\bar K}(\alpha)$, which is to say $v_{\bar K}(\alpha) \geq v_{\bar K}(\alpha')$ for all $\alpha'\in \bar K$ such that $e_{\La'}(\alpha')= t^{-1}$. Then $v_K(t)= \epsilon^{r,n-1}_{s_1^r, \dots, s_{n-1}^r}(-v_{\bar K}(\alpha))$ by \ref{eep} (and, in particular, $v_{\bar K}(\alpha) < 0$). Let $\la= \psi(N)\alpha$. Since $\psi(N)(z)$ is a monic polynomial over $\cV$ and $v_{\bar K}(\alpha)<0$, we have $v_{\bar K}(\la)= |N|^r v_{\bar K}(\alpha)$. It then follows from the definition of $\xi^d_k$ in \ref{defcone2} that 
$$
	\xi^d_k(0^{r-1}, s_1, \dots, s_{n-1}, -v_{\bar K}(\la))= ((- v_K(\iota'(e_i)))_{0\leq i\leq d-2}, v_K(t)).
$$

\end{sbpara}

\begin{sblem}\label{pxixi}
Let $((\phi', \iota'), t)\in \overline{\fM}^{r,n-1}_{N, +, C_{d-1}}(\cV)^{(\sharp)}$. 
\begin{enumerate}
\item[(1)] There exists $w\in m_{\cV}\setminus \{0\}$ such that the $(d-1)$th (i.e., last) coordinate of the image of $((-v_K(\iota(e_i)))_{0\leq i\leq d-2}, v_K(t)) \in C_d$ under the map 
$\xi_{k,0}^d \colon C_d\to C_d$ is 
$v_K(w)$.
\item[(2)]  For $w$ as in (1), we have $\phi'(N)(t^{-1}z)= w^{-1}h(z)$, where $h(z)\in \cV[z]$ is primitive (i.e., has a coefficient in $\cV^{\times}$). In particular, we have $v_K(\phi'(N)(t^{-1})) \geq -v_K(w)$. 
\end{enumerate}
\end{sblem}

\begin{proof} 
Since $\phi'(N)(t^{-1}z)\in K[z]$, there is $w\in K^\times$ such that $\phi'(N)(t^{-1}z)=w^{-1}h(z)$ for some primitive $h(z)\in \cV[z]$. 
We prove that $w\in m_{\cV}$ and that $v_K(w)$ has the property stated in (1).

Let $\alpha \in \bar K$ with $e_{\La'}(\alpha) = t^{-1}$ be as in \ref{n-1objects}, and again let $\la = \psi(N)\alpha$. 
Since $v_K(t)= \epsilon^{r,n-1}_{s_1^r, \dots, s_{n-1}^r}(-v_{\bar K}(\alpha))$, we have by \ref{eep2} that $e_{\La'}(\alpha z)= t^{-1}h_1(z)$ for some $h_1(z)\in \bar \cV\ps{z}$ with a unit coefficient. 
Since $\psi(N)(z) \in \cV[z]$ is monic of degree $|N|^r$ and $v_{\bar K}(\lambda) = |N|^r v_{\bar K}(\alpha) < 0$, we have $\psi(N)(\alpha z)= \la h_2(z)$ for a primitive polynomial $h_2(z)\in\bar \cV[z]$. Let $w'$ be an element of $\bar K$ such that $v_{\bar K}(w')$ is the $(d-1)$th coordinate of $\xi_{k,0}^d((-v_K(\iota(e_i)))_{0\leq i\leq d-2}, v_K(t))$. As this $(d-1)$th coordinate
is $\epsilon_{s_1^r, \ldots, s_{n-1}^r}^{r,n-1}(-|N|^rv_{\bar K}(\alpha))$ by definition of $\xi_k^d$, again by \ref{eep2} we have 
$e_{\La'}(\la z)=  (w')^{-1}h_3(z)$ for some $h_3(z)\in \bar \cV\ps{z}$ having a unit coefficient.

Putting this all together, we have
$$e_{\La'}(\psi(N)(\alpha z))= \phi'(N)e_{\La'}(\alpha z)=\phi'(N)(t^{-1}h_1(z))= w^{-1}h(h_1(z)),$$ 
as well as
 $$e_{\La'}(\psi(N)(\alpha z))= e_{\La'}(\la h_2(z))= (w')^{-1}h_3(h_2(z)).$$
As $w^{-1}h(h_1(z))=(w')^{-1}h_3(h_2(z))$, the primitivity of the power series $h \circ h_1$ and $h_3 \circ h_2$ forces $w'w^{-1}\in \bar \cV^\times$, so $v_K(w) = v_{\bar K}(w')$.
 \end{proof}

\begin{sbpara}\label{sharp}

Let $\overline{\fM}^{r,n-1}_{N, +, C_{d-1}}(\cV)^{\sharp}$ be the subset of  $\overline{\fM}^{r,n-1}_{N, +, C_{d-1}}(\cV)^{(\sharp)}$ consisting of all elements 
 $((\phi', \iota'), t)$ such that 
$$
 	v_K(\phi'(N)(t^{-1}))= -v_K(w),
$$ 
where $w$ is as in \ref{pxixi}(1).   Recall that
we have $v_K(\phi'(N)(t^{-1})) \geq -v_K(w)$ for any $((\phi', \iota'), t) \in \overline{\fM}^{r,n-1}_{N, +, C_{d-1}}(\cV)^{(\sharp)}$ from \ref{pxixi}(2).

Put differently, we require that
$-v_K(\phi'(N)(t^{-1}))$ coincides with the $(d-1)$th coordinate of $\xi_{k,0}^d((-v_K(\iota(e_i))_{0\leq i\leq d-2}, v_K(t))$. 
That is, an element $((\phi',\iota'),t)$ of $\overline{\fM}^{r,n-1}_{N, +, C_{d-1}}(\cV)^{(\sharp)}$ lies in $\overline{\fM}^{r,n-1}_{N, +, C_{d-1}}(\cV)^{\sharp}$ if and only if
$$
	-v_K(\phi'(N)(t^{-1})) = \epsilon_{s_1^r, \ldots, s_{n-1}^r}^{r,n-1}(-v_{\bar K}(\la)),
$$
where $\la = \psi(N)(\alpha)$ with $\alpha$ of maximal valuation such that $e_{\La'}(\alpha)=t^{-1}$, as in \ref{n-1objects}.

In \ref{4.5prop} below, we will obtain a bijection $\nu \colon \overline{\fM}^{r,n-1}_{N, +, C_{d-1}}(\cV)^{\sharp}\xrightarrow{\sim} \overline{\fM}^{r,n}_{N, +, C_d}(\cV)$. 

\end{sbpara}

\begin{sblem}\label{1step}  

Let $((\phi', \iota'), t)\in \overline{\fM}^{r,n-1}_{N, +, C_{d-1}}(\cV)^{\sharp}$, 
and let 
$\gamma=\phi'(N)(t^{-1})$.
\begin{enumerate}
	\item[(1)] The map 
	$$
		A\to \phi'(A)\gamma,\quad a \mapsto \phi'(a)\gamma
	$$ 
	is injective, and  $\phi'(A)\gamma$ satisfies the conditions on $\La$ in \ref{Fuji2} with $n=1$ (taking $R$ to be $\cV$, 
	$I$ to be $m_{\cV}$, and $\psi$ to be $\phi'$ therein). 
	\item[(2)] Let $a\in A\setminus \{0\}$. Then 
	$$v_K(\phi'(a)\gamma) \leq q^{(d-1)(\deg a-1)}v_K(\gamma).$$ 
\end{enumerate}

\end{sblem}

\begin{proof} 
It suffices to show the claimed inequality, as it yields (iii) and (iv) of \ref{Fuji2}, and in particular the injectivity of $a \mapsto \phi'(a)\gamma$
(which gives (ii), and (i) holds as $\gamma \in K$).
Fix $a$, and set $m = \deg a$.

Let  $x=-v_{\bar K}(\la)$, where $\la$ is as in \ref{pxixi}. Set $g=(s_1^r, \dots, s_{n-1}^r)$. By \ref{Ntor1} and \ref{prope1}, we have    
\begin{multline*} 
	-v_K(\phi'(a)\gamma)=\epsilon^{r,n-1}_g(q^{rm}x)= q^{(d-1)m}\epsilon_g^{r,n-1}(x)-(q^{(d-1)m}-1)\delta^{r,n-1}(g)\\
  	\geq q^{(d-1)m}(\epsilon_g^{r,n-1}(x)- \delta^{r,n-1}(g)).
\end{multline*}
Let $R=-v_K(\gamma)= \epsilon_g^{r,n-1}(x)$, the latter equality by the condition in \ref{sharp}.
  
Note that $x\geq -v_{\bar K}(\la_{n-1})= s_{n-1}^r$.  Thus, in the definition \ref{hat(e)} of $\delta^{r,n-1}$ as 
$$
	\delta^{r,n-1}(g) = \frac{q-1}{q^{d-1}-1}\sum_{i=1}^{n-1}q^{n-1-i} \epsilon^{r,i-1}_{s_1^r, \dots, s_{i-1}^r}(s_i^r),
$$ 
we have $\epsilon^{r,i-1}_{s_1^r, \dots, s_{i-1}^r}(s_i^r) = \epsilon^{r,n-1}_g(s_i^r) \leq R$ for $1\leq i\leq n-1$. From this, we have 
$$
	\delta^{r,n-1}(g)\leq \frac{q^{n-1}-1}{q^{d-1}-1}R.
$$

Combining the above, we obtain
$$
  	-v_K(\phi'(a)\gamma)\geq  q^{(d-1)m}R\left( 1- \frac{q^{n-1}-1}{q^{d-1}-1}\right).
$$
Since
$$
	1- \frac{q^{n-1}-1}{q^{d-1}-1} = \frac{q^{d-1}-q^{n-1}}{q^{d-1}-1}
  	\geq \frac{1}{q^{d-1}-1} \geq q^{-(d-1)},
$$
the inequality in the result follows.
\end{proof}
  
 \begin{sbpara} As in Lemma \ref{1step}, let $((\phi', \iota'), t)\in \overline{\fM}^{r,n-1}_{N, +, C_{d-1}}(\cV)^{\sharp}$ and $\gamma=\phi'(N)(t^{-1})$. 
 
 Let $\phi$ be the generalized Drinfeld module over $\cV$ obtained from the pair $(\phi', \phi'(A)\gamma)$ by the method of Section \ref{ss:constrquot}, which we can apply by Lemma \ref{1step}. Let $e_{n,n-1}$ be the exponential map of the lattice $\phi'(A)\gamma$. 
 Define a level $N$ structure $\iota$ on the pullback of $\phi$ to $K$ by
 $\iota(e_i)= e_{n,n-1}(\iota'(e_i))$ for $0\leq i\leq d-2$ and by $\iota(e_{d-1})= e_{n,n-1}(t^{-1})$.  The following lemma provides the
 existence of the map claimed in \ref{sharp}.
 
 \end{sbpara} 
 
 \begin{sblem}\label{4.5lem}
 	We have 
	$(\phi, \iota)\in \overline{\fM}^{r,n}_{N, +, C_d}(\cV)$.
  \end{sblem}
  
  \begin{proof} Let $\la_n= \lam = \psi(N)\alpha$ for $\alpha \in \bar K$ such that $v_{\bar K,\La'}(\alpha) = v_{\bar K}(\alpha)$ with $e_{\La'}(\alpha) = t^{-1}$. Set $\La=\sum_{i=1}^n A\la_i$. Then $\lambda_n$ is such that $(\la_i)_{1\leq i\leq n}$ satisfies the condition 
   \ref{diagbase}(1)(iii). Hence the divisibility conditions as in \ref{categories} on the $\iota(e_i)$ defining $\overline{\fM}^{r,n}_{N,+, C_d}$
   in \ref{moduli} 
 are satisfied. 
   \end{proof}
   
\begin{sbprop}\label{4.5prop} The map $$\rho \colon \overline{\fM}^{r,n-1}_{N,+, C_{d-1}}(\cV)^{\sharp}\to \overline{\fM}^{r,n}_{N,+,C_d}(\cV), \quad ((\phi', \iota'), t) \mapsto (\phi, \iota)$$
  is a bijection.

\end{sbprop}

\begin{proof} We have the inverse map  
$$\rho^{-1} \colon \overline{\fM}^{r,n}_{N,+,C_d}(\cV) \to \overline{\fM}^{r,n-1}_{N,+, C_{d-1}}(\cV)^{\sharp}, \quad (\phi, \iota)\mapsto ((\phi', \iota'), t),$$
where $(\phi',\iota')$ is the $(\phi_{n-1}, \iota_{n-1})$ in \ref{iotai} and $t$ is the unique element of $m_{\cV}\setminus \{0\}$ such that $-v_K(t)=v_K(\iota(e_{d-1}))$ and  
$e_{n,n-1}(t^{-1})= \iota(e_{d-1})$, which we have by \ref{econg}.  
\end{proof}

\subsection{Iterated Tate uniformizations 2}\label{itTa}

In \ref{4.6.1}--\ref{generalpf}, 
we prove Theorem \ref{fthm} for $\hat \cC_{\nl}$ by induction on $n$ (fixing $r$) using the method of iterated Tate uniformizations. 
In \ref{hatlogpf}, we prove Theorem \ref{fthm} for $\hat \cC_{\log}$.  Recall that the former (resp., latter) theorem concerns the injectivity and image of the morphism
$\theta \colon \overline{\fM}^{r,n}_{N, +, \sig} \to (\overline{\cM}^{r,n}_{N,(+,\sig)},\cM^{r,n}_{N,(+,\sig)})$ (resp.,  $\theta \colon \overline{\fM}^{r,n}_{N, +, \sig} \to \overline{\cM}^{r,n}_{N, (+, \sig)}$) of \ref{theta} given by
$$
	\theta(\phi,\iota) = ((\psi,\iota'),(\iota(e_{i+r-1})^{-1})_{1 \le i \le n}),
$$
where $(\psi,\iota') \in \cM^r_N$ is as in \ref{toMrn}. (Here, and in this subsection, we often omit $S \in \hat{\cC}_{\log}$ and $(S,U) \in \hat{\cC}_{\nl}$ in the notation.)

\begin{sbpara}\label{4.6.1}

For the proof of \ref{fthm} for $\hat \cC_{\nl}$, we may assume $\sig\in \Sig_k$. In fact, by assumption in \ref{sig'n}, there exists $\tau \in \Sig_k$ such that $\sig\subset \tau$. If an open set $\overline{\cM}^{r,n}_{N, +, \tau}$ of $\overline{\cM}^{r,n}_{N, (+, \tau)}$ satisfies 
$\theta \colon \overline{\fM}^{r,n}_{N,+,\tau} \xrightarrow{\sim} (\overline{\cM}^{r,n}_{N, +, \tau},\cM^{r,n}_{N,+,\tau})$,
then the inverse image $\overline{\cM}^{r,n}_{N, +, \sig} $ of $\overline{\cM}^{r,n}_{N, +, \tau}$ under  $\overline{\cM}^{r,n}_{N, (+, \sig)} \to \overline{\cM}^{r,n}_{N, (+, \tau)} $ satisfies   $\theta \colon \overline{\fM}^{r,n}_{N, +, \sig} \xrightarrow{\sim} (\overline{\cM}^{r,n}_{N, +, \sig},\cM^{r,n}_{N,+,\sig})$.

\end{sbpara}

\begin{sblem}\label{'sig}

Let  $\sig\in \Sig_k={}_d\Sig_k$. Then the image ${}'\sig$ of $\sig$ under the projection $C_d\to C_{d-1}$ lies in ${}_{d-1}\Sig_k$.
 
\end{sblem}

\begin{proof} Assume $n\geq 1$. Let
 $\tilde \sig\in {}_d\Sig^{(k)}$ be the cone corresponding to $\sig \in \Sig_k$ as in \ref{Sigk3}. By definition of $\Sig^{(k)}$, this $\tilde \sig$ is described as the set of $s\in C_d$ such that for some $\alpha \colon {}_d I^{(k)} \to \{ \R_{\le 0}, \{0\}, \R_{\ge 0}\}$, we have
$q^h s_j -s_i\in \alpha(h,i,j)$ for all $(h,i,j) \in {}_d I^{(k)}$ (see \ref{Sigk1}). 
Let $ {}' {\tilde\sig} \in {}_{d-1}\Sig^{(k)}$ be the cone of all $s\in C_{d-1}$ satisfying the following conditions for the same $\alpha$: $q^h s_j -s_i\in \alpha(h,i,j)$ for all $(h,i,j) \in {}_{d-1}I^{(k)}$ (i.e., for $1\leq j \leq i\leq d-2$ and $0\leq h\leq k-1$). 
 That is, ${}'\tilde \sig$ is the image of $\tilde \sig$ under the projection $C_d\to C_{d-1}$ given by $(s_i)_{1\leq i\leq d-1}\mapsto (s_i)_{1\leq i\leq d-2}$. 
 Define ${}'\sig\in {}_{d-1}\Sig_k$ to be the cone corresponding to ${}'\tilde \sig$. Then ${}'\sig$ is the image of $\sig$ under the projection $C_d\to C_{d-1}$ since, as in the proof \ref{Sigk3pf} of \ref{Sigk3}, we have that ${}'\sig = \xi_k^{d-1}({}'\tilde \sig)$, and the first $d-2$ coordinates of 
 $\xi_k^d((s_i)_{1 \le i \le d-1})$ coincide with $\xi_k^{d-1}((s_i)_{1 \le i \le d-2})$ (see \ref{defcone2}).
 \end{proof}

 \begin{sbpara}  \label{induct}

We have \ref{fthm} for $\hat \cC_{\nl}$ in the case $n=0$ by \ref{fthmn=0}.
By induction on $n$, we assume that for some $n\geq 1$ we have the open subset $\overline{\cM}^{r,n-1}_{N, +, {}'\sig}$ of $\overline{\cM}^{r,n-1}_{N, (+, {}'\sig)}$ 
and an isomorphism  $\theta \colon \overline{\fM}^{r,n-1}_{N,+,  {}'\sig} \xrightarrow{\sim} (\overline{\cM}^{r,n-1}_{N, +, {}'\sig}, \cM^{r,n-1}_{N,+,{}'\sig})$, where $\cM^{r,n-1}_{N,+,{}'\sig}$ denotes the inverse image of $\cM^{r,n-1}_{N,(+,{}'\sig)}$ in $(\overline{\cM}^{r,n-1}_{N, +, {}'\sig})_{\adic}$.

Let $V$ be the open subset of $\overline{\cM}^{r,n}_{N, (+, \sig)}$ consisting of those $((\psi,\iota'), t_1, \dots, t_n)$ such that $((\psi,\iota'), t_1, \dots, t_{n-1})$ belongs to $\overline{\cM}^{r,n-1}_{N, +, {}'\sig}$. 
In \ref{VtoM}, we will define an open subset $\overline{\cM}^{r,n}_{N,+, \sig}$ of $V$. In \ref{fTc5}, we will
define a map $\nu \colon (\overline{\cM}^{r,n}_{N, +, \sig}, \cM^{r,n}_{N, +, \sig})\to \overline{\fM}^{r,n}_{N, +, \sig}$ of functors on $\hat \cC_{\nl}$ by the method of iterated Tate uniformization, which is nearly the inverse of $\theta$.  
\end{sbpara}

We first give a preparatory lemma. Let $W$ be the inverse image of $\cM^{r,n}_{N, (+, \sig)}$ in $V_{\adic}$. 
Let $((\psi, \iota'), t_1, \dots, t_n)\in V$. Then we have $({}'\phi, {}'\iota)\in \overline{\fM}^{r,n-1}_{N, +, {}'\sig}$ corresponding to 
$$
	((\psi, \iota'), t_1, \dots, t_{n-1})\in \overline{\cM}^{r,n-1}_{N, +, {}'\sig}.
$$

\begin{sblem}\label{VtoM0} Locally on $V$, there exists $w\in \cO_V$ which is invertible on $W$ and a primitive polynomial $h \in \cO_V[z]$ such that ${}'\phi(N)(t_n^{-1}z)= w^{-1} h(z)$.
\end{sblem}

\begin{proof}

By \ref{thmcone}(3), there exists a linear map  $l \colon \Q^{d-1}\to \Q^{d-1}$ that coincides with $\xi^d_{k,0}$ on $\sig$. 
Let $l_{d-1} \colon \Q^{d-1}\to \Q$ be the projection of $l$ to the $(d-1)$th coordinate. Let $\text{Div}(V)$ be the group of Cartier divisors on $V$ (which we specify as sections of $\cO_V$), and let $l_{d-1} \colon \text{Div}(V)^{d-1}\to \text{Div}(V)\otimes \Q$ be the homomorphism induced by $l_{d-1}$. 

By \ref{pxixi}(1), upon applying the first part of \ref{todvr}, we have that locally on $V$, there exists $w\in \cO_V$ whose divisor coincides with the image of $(\pole(\iota(e_i))_{1\leq i\leq d-2},  (t_n))\in \text{Div}(V)^{d-1}$ under $l_{d-1}$. (In particular, this image lies in $\text{Div}(V) \subset \text{Div}(V) \otimes \Q$.) 

Similarly, by \ref{pxixi}(2) and \ref{todvr}, the polynomial $h(z) =w\cdot {}'\phi(N)(t_n^{-1}z)$ lies in $\cO_V[z]$. We prove that locally on $S$, some coefficient of $h(z)$ is invertible. Let $x$ be a point of $V$, and let $m$ be the maximal ideal of the local ring $\cO_{V,x}$. Take a discrete valuation ring $\cV$ in the field of fractions of $\cO_{V,x}$ such that $\cO_{V,x}\subset \cV$ and the maximal ideal of $\cO_{V,x}$ is $m_{\cV}\cap \cO_{V,x}$ (as in the proof of \ref{todvr}). By \ref{pxixi}, some coefficient of $h(z)$ is invertible in $\cV$ and therefore in $\cO_{V,x}$ as well.
\end{proof}

\begin{sbpara}\label{VtoM} 

Let $\overline{\cM}^{r,n}_{N, +, \sig}$ be the open set of $V$ consisting of all points at which $w\cdot  {}'\phi(N)(t_n^{-1}) \in \cO_V$, with $w$ as in \ref{VtoM0}, is invertible. 
\end{sbpara}

\begin{sbpara}\label{fTc2} 

Let $X=\overline{\cM}^{r,n}_{N, +, \sig}$.  Let $((\psi,\iota'),t_1,\ldots,t_n) \in X$, and set $\gamma = {}'\phi(N)(t_n^{-1})$. Choose an open cover of $X$ by formal spectra $\Spf(R)$, where the rings $R$ are excellent normal domains. By \ref{1step}, the $A$-module ${}'\phi(A)\gamma= {}'\phi(AN)(t_n^{-1})$ satisfies the conditions on $\La$ in \ref{Fuji2} (with $n = 1$ therein) over complete discrete valuation rings $\cV$ containing such an $R$. 
By \ref{todvr} and by the last inequality in \ref{1step} (which we use to treat the condition (iv) in \ref{Fuji2}), it then satisfies these conditions over $R$, taking $I$ to be the ideal of definition of $R$ such that $R/I$ is reduced. 

As in \ref{Fuji4}, we obtain a generalized Drinfeld module $\phi$ on $X$ by dividing ${}'\phi$ by this lattice ${}'\phi(A)\gamma$
via its exponential map, which we denote $e_{n,n-1}$. 

\end{sbpara}

\begin{sblem}\label{nn-1} Let $I$ be the ideal of definition of $\cO_X$ such that $\cO_X/I$ is reduced. 
Then for $1\leq i\leq n$, we have $t_ie_{n, n-1}(t_i^{-1}z)\in \cO_X\ps{z}$ and $t_i e_{n,n-1}(t_i^{-1}z) \equiv z \bmod I\cO_X\ps{z}$.

\end{sblem}

\begin{proof} By \ref{todvr} and \ref{4.5prop}, this is reduced to \ref{econg}(2) (with $b = t_i^{-1}$). \end{proof}

\begin{sbpara}\label{fTc4}  We define the level $N$ structure  $\iota$ on $\phi$ using the level $N$ structure ${}'\iota$ on ${}'\phi$ as 
$\iota(e_i) = e_{n,n-1}({}'\iota(e_i))$ for $0\leq i\leq d-2$ and $\iota(e_{d-1}) = e_{n,n-1}(t_n^{-1})$.  
By  \ref{nn-1}, we then have 
\begin{eqnarray*}
	\iota(e_{i+r-1})t_i \in \cO_X^\times &\mr{and}& \iota(e_{i+r-1})t_i \equiv 1 \bmod I
\end{eqnarray*}   
for $1\leq i\leq n$.
\end{sbpara}

\begin{sbpara}\label{fTc5} By \ref{todvr} and \ref{4.5lem}, we have $(\phi, \iota)\in \overline{\fM}^{r,n}_{N, +, C_d}$.
Since $((\psi,\iota'),t_1, \ldots, t_n) \in \overline{\cM}^{r,n}_{N,+,\sig}$, by \ref{cM(sig)} and \ref{fTc4}, we have that
$(0^{r-1},\pole(\iota(e_{i+r-1}))_{1 \le i \le n})$ satisfies the condition (ii) of \ref{moduli}, and thus $(\phi,\iota) \in 
\overline{\fM}^{r,n}_{N, +, \sig}$ (see \ref{sigrnd}).
We therefore obtain a map $$\nu \colon (\overline{\cM}^{r,n}_{N,+,\sig}, \cM^{r,n}_{N,+,\sig})\to \overline{\fM}^{r,n}_{N, +, \sig}$$
which  sends $((\psi,\iota'), t_1, \dots, t_n)$ to $(\phi, \iota)$. (Recall that we also have the map $\theta \colon \overline{\fM}^{r,n}_{N, +, \sig}
\to (\overline{\cM}^{r,n}_{N, (+, \sig)}, \cM^{r,n}_{N, (+, \sig)})$
of $\ref{theta}$.)

\end{sbpara}

  \begin{sblem}\label{MMM} Let $X =\overline{\cM}^{r,n}_{N,+,\sig}$ and $U=\cM^{r,n}_{N,+,\sig}$. Then 
  the image of the composition 
  $$
  	\theta\circ \nu \colon (X,U) \to (\overline{\cM}^{r,n}_{N, (+, \sig)}, \cM^{r,n}_{N, (+, \sig)})
$$ 
is $(X,U)$, and $\theta\circ \nu$ induces an 
  automorphism  of $(X, U)$. 
   It  induces the identity morphism of  the log formal scheme $(X,\cO_X/I)$, where $I$ is as in \ref{nn-1}.

\end{sblem} 

\begin{proof} It is sufficient to prove that for $t\in \prod_{i=1}^n t_i^{\Z}$, the pullback $(\theta \circ \nu)^*(t)$ under $\theta\circ \nu \colon X \to \overline{\cM}^{r,n}_{N, (+, \sig)}$ satisfies $(\theta\circ \nu)^*(t)t^{-1}\in \cO_X^\times$ and $(\theta\circ \nu)^*(t)t^{-1}\equiv 1 \bmod I$. It is sufficient to prove this in the case $t=t_i$. We have $\theta^*(t_i)= \iota(e_{i+r-1})^{-1}$, and hence $(\theta \circ \nu)^*(t_i)= (\nu^*\iota(e_{i+r-1}))^{-1}= e_{n,n-1}(t_i^{-1})^{-1}$. Hence we are reduced to \ref{fTc4}. 
\end{proof}

\begin{sbpara}\label{dvrpf}  We prove \ref{fthm} for complete discrete valuation rings by induction on $n$.
 Let $\cV$ be a complete discrete valuation ring over $A$ with field of fractions $K$. In the case $N$ has only one prime divisor, we assume $N$ is invertible in $\cV$. Let $S=\Spf(\cV)$ and $U=\Spa(K, \cV)\subset S_{\adic}$.
 
Let 
$$
 	P_n = \overline{\fM}^{r,n}_{N,+, \sig}(\cV)= \overline{\fM}^{r,n}_{N, +, \sig}(S,U),
$$ 
as in \ref{iTatedvr}.
 Let $R_n$ (resp., $R'_n$) be the set of all morphisms in $\hat \cC_{\nl}$ from $(S,U)$ to $(\overline{\cM}^{r,n}_{N,+,\sig}, \cM^{r,n}_{N, +, \sig})$ (resp., to 
$(\overline{\cM}^{r,n}_{N,(+,\sig)}, \cM^{r,n}_{N, (+, \sig)})$).  

Let $'\sigma \in C_{d-1}$ be as in \ref{'sig}, and define $P_{n-1}$, $R_{n-1}$, and $R'_{n-1}$ analogously to the above, with $'\sig$ in place
of $\sig$. We suppose by induction that the map $\theta \colon P_{n-1} \to R'_{n-1}$ of \ref{theta} is injective with image $R_{n-1}$, and 
we show that the map $\theta \colon P_n \to R'_n$ of \ref{theta} is
injective with image $R_n$.
 
Let $Q_n$ be the subset of $\overline{\fM}^{r,n-1}_{N, +, C_{d-1}}(\cV)^{\sharp}$ (as in \ref{sharp}) consisting of those elements of
the product $\overline{\fM}^{r,n-1}_{N,+,'\sig}(S,U) \times (\cV \setminus \{0\})$ which are carried under
the bijection 
$$\rho \colon \overline{\fM}^{r,n-1}_{N,+,C_{d-1}}(\cV)^{\sharp}\xrightarrow{\sim} \overline{\fM}^{r,n}_{N, +, C_d}(\cV)$$
of \ref{4.5prop} to an element of $P_n$.
By definition, $\rho$ then induces a bijection $\rho_{\sig} \colon Q_n\xrightarrow{\sim} P_n$. 
This sends $((\phi_{n-1},\iota_{n-1}),t)$ to $(\phi,\iota)$, where $\phi$ is the quotient of $\phi_{n-1}$ by the lattice
$\phi_{n-1}(AN)(t^{-1})$ with exponential $e_{n,n-1}$ such that $\iota(e_i) = e_{n,n-1}(\iota_{n-1}(e_i))$ for $0 \le i \le d-2$ and 
$\iota(e_{d-1}) = e_{n,n-1}(t^{-1})$. Set  $t_i = \iota_{n-1}(e_{i+r-1})^{-1}$ for $1 \le i \le n-1$ and $t_n = t$.

The bijection $\theta \colon P_{n-1}\xrightarrow{\sim} R_{n-1}$ also induces a bijection $\theta_{\sig}
\colon Q_n \xrightarrow{\sim} R_n$,
which sends $((\phi_{n-1},\iota_{n-1}),t)$ to $((\psi,\iota'),t_1,\ldots,t_n)$, noting that
$((\psi,\iota'),t_1,\ldots,t_{n-1}) = \theta(\phi_{n-1},\iota_{n-1})$. 
(That $\theta_{\sig}$ takes image in $R_n$ inside $R_n'$ follows from the condition defining $\overline{\fM}^{r,n-1}_{N,+,C_{d-1}}(\cV)^{\sharp}$ inside $\overline{\fM}^{r,n-1}_{N,+,C_{d-1}}(\cV)^{(\sharp)}$.)
Tracing through the definition of the map $\nu \colon R_n \to P_n$ of \ref{fTc5}, one sees that it takes $((\psi,\iota'),t_1,\ldots,t_n)$ to
the same $(\phi,\iota)$ as above: that is, the level structure on the image $\phi$ is connected to $\iota_{n-1}(e_i)$ for $0 \le i \le d-2$
and $t^{-1}$ via  $e_{n,n-1}$, just as before. We conclude that $\nu$ is the composition $\rho_{\sig} \circ \theta_{\sig}^{-1}$, 
and therefore a bijection.

By \ref{MMM}, the image of the composition $R_n \overset{\nu}\to P_n \overset{\theta}\to R'_n$ is contained in $R_n\subset R'_n$, and the induced map $R_n \to R_n$ is a bijection. Since $\nu$ is surjective, the image of $\theta \colon P_n \to R'_n$ is contained in $R_n$. Since the composition $R_n \overset{\nu}\to P_n \overset{\theta}\to R_n$ is bijective, the map $\theta \colon P_n \to R_n'$  is injective
with image $R_n$.
   
   \end{sbpara}
   
   \begin{sbpara}\label{generalpf}
   
   We prove Theorem \ref{fthm} for $\hat \cC_{\nl}$ in general. Let $P= \overline{\fM}^{r,n}_{N, +, \sig}$, let $R=(\overline{\cM}^{r,n}_{N, +, \sig}, \cM^{r,n}_{N, +, \sig})$, and let $R'=(\overline{\cM}^{r,n}_{N,(+, \sig)}, \cM^{r,n}_{N, (+, \sig)})$. The argument amounts to the following three statements.
      
 \begin{enumerate} 
   \item[(1)] The image of $\theta \colon P\to R'$ is contained in $R$.
   \item[(2)] The map $\theta \colon P \to R'$ is injective. 
   \item[(3)] The composition $R \overset{\nu}\to P \overset{\theta}\to R'$ is an isomorphism onto $R$.  
\end{enumerate}

   Statements (1) and (2) are reduced to the case of a complete discrete valuation ring treated in \ref{dvrpf}, while (3) is simply Lemma \ref{MMM}. By (2) and (3), the map $\theta \colon P\to R$ that we have by (1) is an isomorphism.

      \end{sbpara}

\begin{sbpara}\label{itTate} The isomorphism $\nu \colon \overline{\cM}^{r,n}_{N,+,\sig}\xrightarrow{\sim} \overline{\fM}^{r,n}_{N, +, \sig}$ in the above proof of \ref{fthm} for $\hat \cC_{\nl}$ is regarded as one step in the following iterated Tate uniformization.

Let $(S,U)$ be an object of $\hat \cC_{\nl}$, and let $(\phi, \iota)$ be an element of $\overline{\fM}^{r,n}_{N, +,\sig}(S, U)$. Let 
$$((\psi, \iota'), t_1, \dots, t_n)=\nu^{-1}(\phi, \iota), \quad (\phi_{n-1}, \iota_{n-1}) = \theta^{-1}((\psi, \iota'), t_1, \dots, t_{n-1}).$$
Then $\phi$ is obtained from the pair $(\phi_{n-1}, \phi_{n-1}(N)(t_n^{-1}))$ by the method of \ref{1step}, that is, $\phi$ is the quotient of $\phi_{n-1}$ by the lattice $\phi_{n-1}(AN)(t_n^{-1})$.

Define $(\phi_i, \iota_i)$ ($0\leq i\leq n$) in this way starting from $(\phi_n,\iota_n) :=(\phi, \iota)$ by downward induction on $n$. The recursion ends with $(\phi_0, \iota_0)= (\psi, \iota')$. Thus, $\phi=\phi_n$ is obtained from $\psi$ as an ``iterated quotient'' by $A$-lattices of rank $1$. 

For $1\leq i\leq n$, let $e_{i,i-1}$ be the exponential map which connects $\phi_{i-1}$ and $\phi_i$ (i.e., we have $\phi_i(a)\circ e_{i,i-1}=e_{i,i-1}\circ \phi_{i-1}(a)$ for all $a\in A$). 
Then 
$$
	e_{i,0} :=e_{i,i-1}\circ e_{i-1, i-2}\circ \dots \circ e_{1,0}
$$ 
is the exponential map which connects $\psi$ and $\phi_i$.

We have local systems of $A$-modules 
$$0=\La(\phi_0)\subset \La(\phi_1)\subset \dots \subset \La(\phi_{n-1})\subset \La(\phi_n)=\La$$
on the \'etale site of $U$, where $\La$ is associated to $\phi$ by Theorem \ref{thmTate} and $\La(\phi_i)=\ker(e_{i,0})$ is associated to $\phi_i$ by \ref{thmTate}. The quotients $\La(\phi_i)/\La(\phi_{i-1})$ ($1\leq i\leq n$) are constant sheaves and are isomorphic to $A$. 

\end{sbpara}

\begin{sbpara}\label{hatlogpf}  We prove Theorem \ref{fthm} for $\hat \cC_{\log}$.

By construction as an open set in the log regular scheme $\overline{\cM}^{r,n}_{N,(+,\sig)}$, the log scheme 
$\overline{\cM}^{r,n}_{N, +,\sig}$ of \ref{cM(sig)} is log regular. The identity morphism on the object $(\overline{\cM}^{r,n}_{N,+,\sig},\cM^{r,n}_{N,+,\sig})$ of $\hat{\cC}_{\nl}$ maps under $\theta^{-1}$ (which exists by \ref{fthm} for $\hat \cC_{\nl}$) to its universal generalized Drinfeld module with level $N$ structure. This in turn provides a universal log Drinfeld module $((\cL,\phi),\iota)$ with level $N$ structure on $\overline{\cM}^{r,n}_{N, +,\sig}$
in $\hat \cC_{\log}$ by \ref{fextend}.
This universal object determines a morphism $\overline{\cM}^{r,n}_{N,+,\sig}\to \overline{\fM}^{r,n}_{N,+,\sig}$ of functors on $\hat \cC_{\log}$, taking a morphism $S \to \overline{\cM}^{r,n}_{N,+,\sig}$ to the pullback of $((\cL,\phi),\iota)$. We prove that this is an isomorphism. 

The surjectivity of this morphism as a morphism of sheaves for the \'etale topology  is straightforward. In fact, let $S$ be an object of $\hat \cC_{\log}$, and let $((\cL, \phi), \iota)\in \overline{\fM}^{r,n}_{N,+, \sig}(S)$. By the definition \ref{logD6} of a log Drinfeld module, \'etale locally on $S$, we have a morphism $S\to S'$ with $S'$ log regular such that $((\cL, \phi), \iota)$ is the pullback of some
 $((\cL', \phi'), \iota')\in \overline{\fM}^{r,n}_{N,+, \sig}(S')$. Let $U'$ be the open set of $S'_{\adic}$  consisting of all points at which the log structure is trivial. Since $S'$ is normal, $(S', U')\in \hat \cC_{\nl}$ and hence by the part of Theorem \ref{fthm} concerning the category $\hat \cC_{\nl}$ proven in \ref{generalpf}, there is a morphism $(S', U') \to (\overline{\cM}^{r,n}_{N, +, \sig}, \cM^{r,n}_{N,+,\sig})$ such that $((\cL', \phi'), \iota')$ comes from the universal object on $(\overline{\cM}^{r,n}_{N, +, \sig}, \cM^{r,n}_{N,+,\sig})$ by pullback. Hence $((\cL,\phi), \iota)$ comes from a morphism $S\to \overline{\cM}^{r,n}_{N, +,\sig}$. 

We prove the injectivity. 
We have the morphism of functors $\theta \colon \overline{\fM}^{r,n}_{N,+,\sig}\to \overline{\cM}^{r,n}_{N,(+,\sig)}$ on $\hat \cC_{\log}$ in the reverse direction from \ref{theta}. In fact, we have the following.

\medskip

{\bf Claim 1}. The image of the morphism $\theta$ on $\hat \cC_{\log}$ is contained in $\overline{\cM}^{r,n}_{N,+,\sig}$. 

\begin{proof}[Proof of Claim 1.]  For an object $S$ of $\hat \cC_{\log}$, an element of $\overline{\fM}^{r,n}_{N,+,\sig}(S)$ comes \'etale locally on $S$ from a log regular scheme $S'$ by a morphism $S\to S'$. Since $S'$ is normal, it comes from an element of $\overline{\fM}^{r,n}_{N,+,\sig}(S', U')$ for $U'$ as above, and the image of this element in $ \overline{\cM}^{r,n}_{N,(+,\sig)}(S',U')$ belongs to   $\overline{\cM}^{r,n}_{N,+,\sig}(S',U')$ by \ref{generalpf}.
\end{proof}

Now the injectivity in question is a consequence of the following claim.

\medskip

{\bf Claim 2}.  The composition $\overline{\cM}^{r,n}_{N,+,\sig}\to \overline{\fM}^{r,n}_{N,+,\sig}\overset{\theta}\to  \overline{\cM}^{r,n}_{N,+,\sig}$ for $\hat \cC_{\log}$ is the identity map. 

\begin{proof}[Proof of Claim 2.] For an object $S$ of $\hat \cC_{\log}$, an element $\overline{\cM}^{r,n}_{N,+,\sig}(S)$ comes \'etale locally  from a log regular $S'$. By the reduction to the category $\hat \cC_{\nl}$ as in the above proof of Claim 1, we are reduced to the same statement as Claim 2 for $\hat \cC_{\nl}$ which follows from Theorem \ref{fthm} for the category $\hat \cC_{\nl}$. 
\end{proof}
\end{sbpara}

\section{Compactifications} \label{toroidal}

In Section \ref{Satake}, we construct the Satake compactification of $\cM_N^d$ using the method of Pink \cite{P2}, slightly changing the formulation so that it is defined integrally. In Section \ref{ss:tornl}, we construct the toroidal compactifications and prove that they have properties stated in the main theorems in Section \ref{main_result}. In preparation for this, we assume the existence of these toroidal compactifications in Section \ref{ss:AMFM} and consider their relation to the formal moduli spaces of Section \ref{iterated}. The properties of the toroidal compactifications are then proved using the properties of the formal moduli spaces so derived.
In Section \ref{ss:example}, we describe the toroidal compactifications explicitly in a special case, and we describe the local monodromy of the related generalized Drinfeld modules. 

In this section, we again suppose that $A=\F_q[T]$. Let $d\geq 1$, let $N$ be a nonconstant element of $A$, and let $k$ be the degree of the polynomial $N$.

\subsection{Satake compactifications}  \label{Satake}

We construct the Satake compactification of $\cM^d_N$ over $A$ if $N$ has at least two prime divisors and over $A[\tfrac{1}{N}]$ otherwise.
The Satake compactification of $\cM^d_N$ over $F$ was constructed by Kapranov in \cite{K} by a rigid analytic method and by Pink in \cite{P2} by  a purely algebraic method. We follow the construction in \cite{P2}.

There are small differences between \cite{P2} and the construction in this subsection.
As mentioned, in both \cite{P2} and \cite{K}, the Satake compactification is constructed over $F$, not $A$ or $A[\frac{1}{N}]$.
Another difference is that 
we characterize the Satake compactification as an object which represents the moduli functor  $\overline{\fM}^d_{N, \Sa}$ as in  Theorem \ref{Sathm} below, whereas the Satake compactification of \cite{P2} was characterized by a certain delicate condition. Because these differences exist and because Theorem \ref{Sathm} is used for our construction of toroidal compactifications in Section \ref{ss:tornl},
we give a proof of this theorem following the arguments in \cite{P2} but slightly modifying them to provide the stated result.

 \begin{sbpara}

 Recall that in the case $N$  has at least two prime divisors (resp., only one prime divisor), 
 we denote by $\cC_{\nl}$  the category of pairs $(S, U)$ where $S$ is a normal scheme over $A$ (resp., $A[\frac{1}{N}]$) and $U$ is a dense open subset of $S$. Recall also that $$\overline{\fM}^d_{N, \Sa} \colon \cC_{\nl}\to (\text{Sets})$$
 is the functor which sends $(S, U)$ to the set of all isomorphism classes of $((\cL, \phi), \iota)$, where $(\cL, \phi)$ is a generalized Drinfeld module over $S$ with restriction $(\cL,\phi)|_U$ to $U$ a Drinfeld module of rank $d$, and where $\iota$ is a Drinfeld level $N$ structure on $(\cL, \phi)|_U$.
 
 \end{sbpara}

   \begin{sbthm}\label{Sathm} The functor $\overline{\fM}^d_{N, \Sa}$ is represented by a pair $(\overline{\cM}^d_{N, \Sa}, \cM^d_N)$ for some projective normal scheme $\overline{\cM}^d_{N, \Sa}$ over $A$ (resp. $A[\frac{1}{N}]$)  which contains $\cM^d_N$ as a dense open subscheme. 
  
 \end{sbthm}

 \begin{sbpara}\label{Sapf1a} The proof of the case $N=T$ of Theorem \ref{Sathm} is given  in \ref{Sapf1a}--\ref{Sapf1d}, and it is extended to a general $N$ in \ref{Sapf2}--\ref{Sapf4}. 
 
 First, we recall the known explicit construction of $\cM^d_T$. 
 Let $V$ be a $d$-dimensional vector space over $\F_q$ with base $(u_i)_{0\leq i\leq d-1}$.  Let 
 $$
 	\text{Sym}(V)=\F_q[u_0, \dots, u_{d-1}]
$$ 
be the symmetric algebra of $V$ over $\F_q$.
Let $\Omega$ be the open subscheme of $\mathbb{P}^{d-1}_{\F_q}= \text{Proj}(\text{Sym}(V))$ obtained by inverting all elements of $V\setminus \{0\}\subset \text{Sym}(V)$. In $\mathbb{P}^{d-1}_{\F_q}$, this $\Omega$ is the complement 
   of the union of all $\F_q$-rational hyperplanes.
   
   \end{sbpara}
   
 \begin{sblem}\label{Sapf1b} The scheme $A[\frac{1}{T}]\otimes_{\F_q} \Omega$ 
 represents $\fM_T^d$, or in other words, $\cM^d_T = A[\frac{1}{T}]\otimes_{\F_q} \Omega$. 
 \end{sblem}
   
 \begin{pf}  We have the universal Drinfeld module $(\cL, \phi)$ over $A[\frac{1}{T}]\otimes_{\F_q} \Omega$, where $\cL$ is the pullback of $\cO(1)$ of $\mathbb{P}^{d-1}_{\F_q}$ to $\Omega$ and $\phi$ is given by 
   $$
 	\phi(T)(z)= Tz \prod_{v\in V\setminus \{0\}} (1-v^{-1}z).
$$ 
It is endowed with the $T$-level structure 
$$
	(T^{-1}a_i\bmod A)_{0\leq i\leq d-1}\mapsto \sum_{i=0}^{d-1} a_i u_i\in \cL
$$ 
for $(a_i)_{0 \le i \le d-1} \in \F_q^d$.
A Drinfeld module of rank $d$ with $T$-level structure over an $A[\frac{1}{T}]$-scheme $S$  comes uniquely from this universal one by a morphism $S\to A[\frac{1}{T}]\otimes_{\F_q} \Omega$ over $A[\frac{1}{T}]$. \end{pf}

  \begin{sbpara}\label{Sapf1c}
We define a projective variety $Q$ over $\F_q$ which contains $\Omega$ as a dense open subscheme.

 Let $R$ be the $\F_q$-subalgebra of the rational function field $\F_q(u_0, \dots, u_{d-1})$
 generated by $v^{-1}$ for $v\in V\setminus \{0\}$. 
 Regarding $R$ as a graded ring in which each such $v^{-1}$ has degree $1$, let $Q=\text{Proj}(R)$. Then $\Omega$ is identified with the open set of $Q$ obtained by inverting all $v^{-1} \in R$ with $v\in V\setminus \{0\}$.

On $\Omega$, the pullback of $\cO(1)$ from $\mathbb{P}^{d-1}_{\F_q}$ and the pullback of $\cO(-1)$ from $Q$ coincide. The universal Drinfeld module over $A[\frac{1}{T}] \otimes_{\F_q} \Omega$ extends uniquely to a generalized Drinfeld module over $A[\frac{1}{T}]\otimes_{\F_q} Q$ with the line bundle $\cO(-1)$.

By \cite{PS}, the variety $Q$ is normal, and hence  $A[\frac{1}{T}]\otimes_{\F_q} Q$ is normal. As a consequence of the following lemma, the functor $\overline{\fM}^d_{T, \Sa}$ is then represented by $(A[\frac{1}{T}]\otimes_{\F_q} Q, \cM_T^d)$.
\end{sbpara}

  \begin{sblem}\label{Sapf1d} Let $(S, U)$ be an object of $ \cC_{\nl}$. A morphism 
  $U \to  A[\frac{1}{T}]\otimes_{\F_q} \Omega$ of $A[\frac{1}{T}]$-schemes extends to a morphism 
  $S \to  A[\frac{1}{T}]\otimes_{\F_q} Q$ of $A[\frac{1}{T}]$-schemes if and only if the induced Drinfeld module over $U$ extends to a generalized Drinfeld module over $S$.  Moreover, any such extensions of morphisms and Drinfeld modules are unique. 
  \end{sblem}

\begin{pf} Note that any extension of morphisms from $U$ to $S$ as in the statement is 
unique as $Q$ is separated, whereas any extension of Drinfeld modules from $U$ to $S$ is unique by \ref{stlem}.

If we have an extended morphism $S\to  A[\frac{1}{T}]\otimes_{\F_q} Q$, we have the pullback of the generalized Drinfeld module to $S$. Conversely, assume that the Drinfeld module over $U$ extends to $S$ as a generalized Drinfeld module with line bundle  $\cL$. Locally on $S$,  some  $v_0\in V$ is a base of $\cL$ by \ref{unit}(2), and $v_0v^{-1}\in \cO_S$ for all $v\in V\setminus \{0\}$ by \ref{unit}(1). This proves that the morphism on $U$ extends to a morphism $S\to  A[\frac{1}{T}]\otimes_{\F_q} Q$. \end{pf}

To proceed from level $T$ to general level, we use part (1) of the following simple lemma on the category of pairs $(S,U)$, where $S$ is a normal scheme and $U$ is a dense open subset of $S$ (with the evident definition of morphisms). Its part (2) will be used in Section \ref{ss:tornl}.

\begin{sblem}\label{nllem} 
 Let $f \colon (S, U)\to (P, W)$ be a morphism in the category of pairs of a normal scheme and a dense open subset.
\begin{enumerate}
\item[(1)] Let $W'\to W$ be a finite morphism of schemes with $W'$ normal, and let $P'$ be the integral closure of $P$ in $W'$. Then, every morphism $U\to W'$ of schemes over $W$ extends uniquely to a morphism $(S,U) \to (P', W')$ over $(P, W)$.
\item[(2)] Let $I_i$ for $1 \le i \le n$ be invertible ideals of $\cO_P$ such that, for each $i$, the pullback of $I_i$ to $S$ is an invertible ideal $I_{i,S}$, and the restriction of $I_i$ to $W$ coincides with $\cO_W$. Let $P'$ be the normalization of the blow-up of $P$ along the product of the ideals $I_i+I_j$ for all pairs $(i,j)$. Then we have a morphism $f' \colon (S,U) \to (P', W)$ over $(P, W)$ if and only if, for each pair $(i,j)$, either $I_{i,S}\subset I_{j,S}$ or $I_{j,S}\subset I_{i,S}$. Furthermore, $f'$ is unique if it exists. 
\end{enumerate}
\end{sblem}

\begin{pf} Part (1) is clear. We prove part (2). 
 Let $\tilde{P}$ be the blow-up of $P$ along the product of the ideals $I_i+I_j$. (Roughly speaking, $\tilde{P}$ is the minimal space over $P$ on which each ideal  $I_i+I_j$ generates an invertible ideal.) By the definition of the blow-up, for each point $x$ of $\tilde{P}$ with image $y$ in $P$, the ring $\cO_{\tilde{P},x}$ is generated  over the ring $\cO_{P,y}$ by elements $t_{i,j}$ for the pairs $(i,j)$ such that for each $(i,j)$, we have either $I_{j,\tilde{P}}= t_{i,j}I_{i,\tilde{P}}$ or $I_{i,\tilde{P}}=t_{i,j}I_{j,\tilde{P}}$. 
Part (2) follows from this. 
\end{pf}

\begin{sbpara}\label{Sapf2} Assume that $N$ is divisible by $T$. The morphism $A[\frac{1}{T}]\otimes_A \cM_N^d \to \cM_T^d$ is finite. Let $A[\frac{1}{T}]\otimes_A\overline{\cM}^d_{N, \Sa}$ be the integral closure of $\overline{\cM}^d_{T, \Sa}$ in $A[\frac{1}{T}]\otimes_A \cM_N^d$. 

We prove that 
$(A[\frac{1}{T}]\otimes_A \overline{\cM}^d_{N, \Sa}, A[\frac{1}{T}]\otimes_A \cM_N^d)$ represents the restriction of the functor $\overline{\fM}^d_{N,\Sa}$ to the full subcategory of $\cC_{\nl}$ consisting of the objects on which $T$ is invertible.

Let $(S, U)$ be an object of the latter category. 
The set $\overline{\fM}^d_{N,\Sa}(S,U)$ is identified with the set of pairs of an element $f$ of $\overline{\fM}^d_{T,\Sa}(S,U)$ and a morphism $U \to A[\frac{1}{T}]\otimes_A \cM_N^d$ over $\cM^d_T$. We identify $f$ with a morphism $(S,U)\to (\overline{\cM}^d_{T, \Sa},\cM^d_T)$ of $\cC_{\nl}$. By Lemma \ref{nllem}(1) applied to this $f$ and $W' = A[\frac{1}{T}]\otimes_A \cM_N^d$, such a pair corresponds bijectively to a morphism $(S,U) \to (A[\frac{1}{T}]\otimes_A\overline{\cM}^d_{N, \Sa}, A[\frac{1}{T}]\otimes_A \cM_N^d)$ in $\cC_{\nl}$.

\end{sbpara}

\begin{sbpara}\label{Sapf3} Assume that $N$ is not divisible by $T$.  Let $T'=TN$ if $N$ has only one prime divisor, and let $T'=T$ otherwise. Let $\tilde S_0=A[\frac{1}{T'}]\otimes_A \overline{\cM}^d_{TN, \Sa} =A[\frac{1}{T'}]\otimes_{A[\frac{1}{T}]} A[\frac{1}{T}]\otimes_A \overline{\cM}^d_{TN, \Sa}$, let 
$$
	G=\ker(\GL_d(A/TNA) \to \GL_d(A/NA))\cong \GL_d(A/TA),
$$ 
and let $S_0=A[\frac{1}{T'}]\otimes_A \overline{\cM}^d_{N, \Sa}$ be the quotient of 
$\tilde S_0$ by the action of $G$, which we can take because $\tilde S_0$ is projective over $A[\frac{1}{T'}]$. Then the scheme $U_0=A[\frac{1}{T'}]\otimes_A \cM_N^d$, which is the quotient of $\tilde U_0=A[\frac{1}{T'}]\otimes_A \cM^d_{TN}$ by the action of $G$, is regarded as a dense open subscheme of $S_0$.

In the rest of this \ref{Sapf3}, we prove that $(S_0, U_0)$ represents the restriction of the functor $\overline{\fM}^d_{N,\Sa}$ to the full subcategory of $\cC_{\nl}$ consisting of all objects on which $T'$ is invertible. 

We show that the universal Drinfeld module over $U_0$ extends to a generalized Drinfeld module over $S_0$. Let $(\cL,\phi)$ be the 
 universal generalized Drinfeld module over $\tilde S_0$, and let $(\cL', \phi')$ be the universal Drinfeld module over $U_0$. Then $(\cL, \phi)$ has a $TN$-level structure and hence has an $N$-level structure. By \ref{unit}(2), for some $e\in (\frac{1}{N}A/A)^d$, the image of $e$ by this level structure is a base of $\cL$. Let $\pi \colon \tilde S_0\to S_0$ be the canonical projection. Since the pullback of this $N$-level structure of $(\cL, \phi)$ to $\tilde U_0$ comes from the $N$-level structure of $(\cL', \phi')$, this $e$ gives a section of $\pi_*\cL$, which we also denote by $e$. Consider the line bundle $\cO_{S_0}e$ on $S_0$. The pullback of $\cO_{S_0}e$ to $\tilde S_0$ is $\cL$ and that to $U_0$ is $\cL'$. The action of $A$ on $\cL$ by $\phi$ and the action of $A$ on $\cL'$ by $\phi'$ induce the same action of $A$ on $\cO_{S_0}e$, and by this, the line bundle $\cO_{S_0}e$ becomes a generalized Drinfeld module over $S_0$  which induces $(\cL, \phi)$ on $\tilde S_0$ and $(\cL', \phi')$ on $U_0$. 
 
 Hence for an object $(S,U)$ of $\cC_{\nl}$ on which $T'$ is invertible, a morphism 
 $(S, U) \to (S_0, U_0)$ induces an object of $\overline{\fM}^d_{N,\Sa}(S, U)$. 
 Conversely, assume we are given an object of $\overline{\fM}^d_{N,\Sa}(S, U)$. It induces a morphism $U\to U_0$. Let $\tilde U$ be the fiber 
 product of $U\to U_0 \leftarrow \tilde U_0$, and let $\tilde S$ be  the integral closure of $S$ in $\tilde U$. Since $T$ is invertible on $U_0$, the morphism  $\tilde U_0\to U_0$ is a finite \'etale Galois covering with Galois group $G$, and hence $\tilde U$ is normal and hence $\tilde S$ is normal. Thus, we have a morphism $(\tilde S,\tilde U)\to (\tilde S_0,\tilde U_0)$ in $\cC_{\nl}$, which in turn gives a morphism $(\tilde S,\tilde U) \to (S_0,U_0)$. Since the last morphism  is $G$-invariant and $S$ is the quotient of $\tilde S$ by $G$, we have a morphism $(S,U) \to (S_0,U_0)$. 

\end{sbpara}

\begin{sbpara}\label{Sapf4} We complete the proof of \ref{Sathm}. 

First assume $N$ has at least two prime divisors. By \ref{Sapf2} and \ref{Sapf3}, we have $A[\frac{1}{T}]\otimes_A \overline{\cM}^d_{N, \Sa}$ over $A[\frac{1}{T}]$. 
Similarly, we have $A[\frac{1}{T-1}]\otimes_A \overline{\cM}^d_{N,\Sa}$ over $A[\frac{1}{T-1}]$. 
Their pullbacks over $A[\frac{1}{T(T-1)}]$ coincide.
We obtain $\overline{\cM}^d_{N, \Sa}$ as the union of $A[\frac{1}{T}]\otimes_A \overline{\cM}^d_{N, \Sa}$ and $A[\frac{1}{T-1}]\otimes_A \overline{\cM}^d_{N,\Sa}$.

Next assume $N$ has only one prime divisor. If the prime divisor is $(T)$, then $A[\frac{1}{T}]\otimes_A \overline{\cM}^d_{N, \Sa}$ obtained in \ref{Sapf2} is $\overline{\cM}^d_{N, \Sa}$. If the prime divisor is  $(T-1)$, we obtain $\overline{\cM}^d_{N, \Sa}$ similarly. If the prime divisor is neither $(T)$ nor $(T-1)$,  then we have $A[\frac{1}{T}]\otimes_A \overline{\cM}^d_{N, \Sa}$ by \ref{Sapf3} and $A[\frac{1}{T-1}]\otimes_A \overline{\cM}^d_{N, \Sa}$ similarly, and $\overline{\cM}^d_{N, \Sa}$ is obtained as the union of these. 

\end{sbpara}

  \begin{sbpara} By construction, our Satake compactification $\overline{\cM}^d_{N, \Sa}$  of $\cM_N^d$ coincides over $F$ with that of Pink \cite{P2}.

  \end{sbpara}

\subsection{Algebraic moduli and formal moduli} \label{ss:AMFM}

This is a preliminary subsection. In \ref{tildeu} and \ref{sandeta}, we review some relations between schemes, formal schemes, and adic spaces. In the remainder of the subsection, assuming that the moduli functor $\overline{\fM}^d_{N, \Sig}$ on $\cC_{\nl}$ is representable, we deduce many results on the moduli space from the formal moduli theory of Section \ref{s:Tate}. In Section \ref{ss:tornl}, we will prove the representability of this functor. The results of this subsection then imply many properties of the resulting moduli spaces (i.e., the toroidal compactifications), including the main theorems of this paper from Section \ref{main_result}. We show that once the algebraic moduli space exists, then its formal completion is the formal moduli space, which has beautiful properties. The algebraic moduli space then has beautiful properties as well.

\begin{sbpara}\label{tildeu}

We first review basic relations between schemes, formal schemes, and adic spaces. 
 Let $X$ be a locally Noetherian formal scheme, and let $\tilde X=X_{\adic}$ be the adic space associated to $X$.  
We then have two locally ringed spaces $(\tilde X, \cO_{\tilde X}^+)$ and $(\tilde X,\cO_{\tilde X})$. 

Let $S$ be a scheme, and assume that we are given a morphism $u \colon X\to S$ of locally ringed spaces. Then we have two morphisms of locally ringed spaces 
$$\tilde u^+ \colon (\tilde X, \cO_{\tilde X}^+) \to S, \quad \tilde u \colon (\tilde X, \cO_{\tilde X}) \to S$$ 
whose underlying maps $\tilde X\to S$ need not coincide, as follows. The morphism $\tilde u^+$ is the composition  $(\tilde X,  \cO^+_{\tilde X}) \to X\to S$ of the morphism discussed in \ref{fvsadic} with $u$. The morphism $\tilde u$ is described in the following paragraph.

Recall that for a locally ringed space $\cS$ and for a ring $R$, the set of morphisms $\cS\to \Spec(R)$ of locally ringed spaces is in bijection with the set of ring homomorphisms $R\to \Gamma(\cS, \cO_{\cS})$. 
The definition of $\tilde u$ is reduced to the case $S$ is an affine scheme $\Spec(R)$. In this setting, we have ring homomorphisms
$$R = \Gamma(S,\cO_S)\to \Gamma(X, \cO_X) \to \Gamma(\tilde X, \cO_{\tilde X}^+) \to \Gamma(\tilde X, \cO_{\tilde X}).$$
The morphism $\tilde u^+$ of locally ringed spaces already defined corresponds to the composition  $R\to \Gamma(\tilde X, \cO_{\tilde X}^+)$, while the morphism $\tilde u$ is defined to be the one which corresponds to the composition $R\to  \Gamma(\tilde X, \cO_{\tilde X})$.
The map $\tilde{u}^+$ sends $x \in X$ to the prime ideal
$\tilde u^+(x) = \{f\in R\mid |f(x)|<1\}$ of $R$, whereas $\tilde u(x) = \{f\in R\mid |f(x)|=0\} \in \Spec(R)$.

For $x\in \tilde X$, let $y^+ = \tilde u^+(x)$ and $y=\tilde u(x)$, and let $z$ be the image of $x$ under the canonical map $\tilde X\to X$ of \ref{fvsadic}. Then the definitions of $\tilde u^+$ and $\tilde u$ show that $y^+$ belongs to the closure of $y$ in $S$, and we have a commutative diagram
$$
	\SelectTips{cm}{} \xymatrix{  \cO_{S,y^+} \ar[r] \ar[d] & \cO_{X, z} \ar[r] & \cO^+_{\tilde X, x} \ar[d] \\
	\cO_{S, y} \ar[rr] && \cO_{\tilde X, x}, }
$$
in which the composition of the upper row is induced by $\tilde u^+$ and the lower row is induced by $\tilde u$.

\end{sbpara}

\begin{example}\label{sandeta}
Let $E$ be a field, let $S=\Spec(E[t])$ with $t$ an indeterminate, let $X$ be the formal completion $\Spf(E\ps{t})$ of $S$ along the ideal $(t)$, and let $u \colon X\to S$ be the canonical morphism. Then 
$$
	\tilde{X} = X_{\adic}=\Spa(E\ps{t}, E\ps{t})
$$ 
consists of two points $s$ and $\eta$ for which $s$ is in the closure of $\eta$. For the maps $\tilde X \to X \to S$ whose composition is $\tilde u^+$, the first map sends $s$ and $\eta$ to the unique point of $X$, and the second map sends this point to the prime ideal $(t)$ of $E[t]$. The map $\tilde u \colon \tilde X\to S$ sends $s$ and $\eta$ to the prime ideals $(t)$ and $(0)$ of $E[t]$, respectively. 
The commutative diagram in \ref{tildeu} for $x=\eta$ becomes
$$
	\SelectTips{cm}{} \xymatrix{    E[t]_{(t)} \ar[r] \ar[d] & E\ps{t} \ar@{=}[r] & E\ps{t} \ar[d] \\
	E(t) \ar[rr] && E\ls{t}.}
$$

\end{example}

\begin{sbpara}\label{Massump} In the rest of this subsection, we fix a finite rational subdivision $\Sig$ of $\Sig_k$ for $k = \deg N$, and we assume that the following hypothesis is true.

\begin{itemize}
\item There exists a normal scheme $\overline{\cM}^d_{N,\Sig}$ of finite type over $A$ (over $A[\frac{1}{N}]$ if $N$ has only one prime divisor) with an open immersion $\cM^d_N\to \overline{\cM}^d_{N, \Sig}$ over the base scheme $A$ or $A[\frac{1}{N}]$ such that the pair $(\overline{\cM}^d_{N, \Sig}, \cM^d_N)$ represents the functor $\overline{\fM}^d_{N,\Sig}$ on $\cC_{\nl}$. 
\end{itemize}

Hence all the results in the rest of this subsection are hypothetical results and should be regarded as announcements of the results with the plans of the proofs. In Section \ref{ss:tornl}, we will prove that the above hypothesis holds, and the results stated below will become unconditional. 

Let $\sig\in \Sig$. By the above assumption and \ref{moduliremarks}, there is an open subscheme $\overline{\cM}^d_{N,+,\sig}$ of $\overline{\cM}^d_{N, \Sig}$ containing $\cM^d_N$ such that the pair $(\overline{\cM}^d_{N, +,\sig}, \cM^d_N)$ represents the functor $\overline{\fM}^d_{N,+,\sig}$ on $\cC_{\nl}$.
\end{sbpara}

Recall the category $\hat \cC_{\nl}$ of pairs $(S,U)$ of a formal $A$-scheme $S$ (over $A[\frac{1}{N}]$ if $N$ has only one prime divisor) that is locally $\Spf$ of an excellent normal domain and a dense open subset $U$ of its adic space from \ref{formal}.

\begin{sbprop}\label{fMcM}   
Let $r\geq 1$ and $n\geq 0$ be integers such that $d=r+n$. The functor $\overline{\fM}^{r,n}_{N, +, \sig}$ on $\hat \cC_{\nl}$ is isomorphic to the functor that sends an object $(X, V)$ of $\hat \cC_{\nl}$ to the set 
of all morphisms $u \colon X\to \overline{\cM}^d_{N, +, \sig}$ of locally ringed spaces satisfying the following two conditions.
\begin{enumerate}
\item[(i)] For every $x\in X$, the fiber at $u(x)$ of the universal generalized Drinfeld module on $\overline{\cM}^d_{N, +, \sig}$ has rank $r$.
\item[(ii)] The image of $V$ under the map $\tilde u \colon X_{\adic}\to \overline{\cM}^d_{N, +, \sig}$ is contained in $\cM^d_N$.
\end{enumerate}

\end{sbprop}

\begin{proof}
Let $F$ be the functor on $\hat \cC_{\nl}$ defined in the statement of the proposition.
We define a natural transformation $F\to \overline{\fM}^{r,n}_{N, +,\sig}$ as follows. Fix $u\in F(X,V)$. Let $((\cL, \phi), \iota)$ be
the universal generalized Drinfeld module on $(\overline{\cM}^d_{N, +,\sig}, \cM^d_N)$. 
Let $(\cL', \phi')$ be the generalized Drinfeld module on $X$ obtained as the pullback of $(\cL, \phi)$ by $u$. Then every fiber of $(\cL', \phi')$ has rank $r$ by condition (i).  

Let $\tilde X=X_{\adic}$. The pullback of $(\cL', \phi')$ to 
$(\tilde  X, \cO_{\tilde X})$ by the morphism $(\tilde X, \cO_{\tilde X}) \to X$ of \ref{fvsadic} is identified with the pullback of $(\cL,\phi)$ under $\tilde u \colon (\tilde X, \cO_{\tilde X})\to \overline{\cM}^d_{N, +, \sig}$ by
trivializing $\cL$ using the basis $\iota(e_0)$ and tracing the commutative diagram in \ref{tildeu}, 
which shows that the pullback on $\tilde X$ of $\phi(a)$ for $a\in A$ by $\tilde u^+$  coincides with that by $\tilde u$.  
Since $\tilde u$ sends $V$ into $\cM^d_N$ by condition (ii), the pullback of $(\cL', \phi')$ to $(V, \cO_V)$ via $(\tilde X, \cO_{\tilde X})\to X$ is endowed with a canonical level $N$ structure. Thus we obtain an element of $\overline{\fM}^{r,n}_{N, +,\sig}(X,V)$.
 
 The converse morphism  $\overline{\fM}^{r,n}_{N, +,\sig}\to F$ is defined as follows. Let $X=\overline{\cM}^{r,n}_{N,+,\sig}$, let $V=\cM^{r,n}_{N,+,\sig}\subset X_{\adic}$, and let  $((\cL, \phi),\iota)$ be the universal object on $X$, where 
 $\iota \colon (\frac{1}{N}A/A)^d\to \overline{\cL}$. 
We identify $\cL$ with $\cO_X$ via 
the basis $\iota(e_0)$. Let $R$ be the ring $\cO(X)$.
Then $a \mapsto \phi(a)$ for $a\in A$ defines a generalized Drinfeld module over $R$ such that the 
coefficient of $z^{|a|^r}$ in $\phi(a)(z)$ is invertible in $R$ for every nonzero element $a$ of $A$.  Let $c\in R$ be the coefficient of $z^{|N|^d}$ in $\phi(N)(z)$.  Let $R'$ be the localization $R[c^{-1}]$ of $R$. Set $\cS=\Spec(R)$ and $\cS'=\Spec(R')$ so that $(\cS,\cS')$ is an 
object of $\cC_{\nl}$.
Then the restriction of the generalized Drinfeld module over $\cS$ to $\cS'$ is a Drinfeld module of rank $d$. 

On $V$, we have
$$
	\phi(N)(z) = c \prod_{b \in (\frac{1}{N}A/A)^d} (z-\iota(b)).
$$
Since $R'$ is normal, this formula for $\phi(N)$ implies that $\iota(b)$ takes values in $R'$. The map $\iota \colon (\frac{1}{N}A/A)^d\to \cL$ on $\cS'$ is a homomorphism, as is checked on $V$. Hence, the formula also shows that $\iota$ on $\cS'$ is a level $N$ structure of this Drinfeld module. 
Hence $\phi$ and $\iota$ define an element of 
$\overline{\fM}^d_{N,+,\sig}(\cS,\cS')$. Let $f \colon \cS \to \overline{\cM}^d_{N, +, \sig}$ be the corresponding morphism. By definition, it restricts to a morphism $\cS' \to \cM_N^d$.

Let $w \colon X \to \cS$ be the canonical morphism.
Let $u = f\circ w \colon X \to \overline{\cM}^d_{N,+,\sig}$.  Then $\tilde u \colon V \to \overline{\cM}^d_{N,+,\sig}$ coincides with the composition 
$$
	V \to X \xrightarrow{w} \cS \xrightarrow{f} \overline{\cM}^d_{N,+,\sig},
$$ 
which factors as 
$$
	V \xrightarrow{\tilde w}  \cS' \to \cM^d_N \to \overline{\cM}^d_{N,+,\sig}
$$ 
where $\tilde{w}$ is induced by $w$.
We have $u\in F(X,V)$, and this 
defines 
$(X,V) \to F$ and hence $\overline{\fM}^{r,n}_{N, +,\sig} \to F$.
\end{proof}

The following proposition follows from \ref{fMcM}.

\begin{sbprop}\label{MandM} Let $S$ be the open set of $\overline{\cM}^d_{N, +, \sig}$ of points at which the fiber of the universal generalized Drinfeld module $(\cL, \phi)$ has rank $\geq r$. 
Let $\hat S$ be the formal completion of $S$ along its closed subset of points at which the fiber of $(\cL, \phi)$ has rank $r$. Let $U$ be the inverse image of $\cM^d_N\subset S$ in $\hat S_{\adic}$ under the morphism $\tilde u \colon \hat S_{\adic}\to S$ of \ref{tildeu}.
 Then the  object $(\hat S, U)$ of $\hat \cC_{\nl}$
 is isomorphic to the object $(\overline{\cM}^{r,n}_{N, +, \sig}, \cM^{r,n}_{N,+,\sig})$ that represents the functor $\overline{\fM}^{r,n}_{N, +, \sig}$ on $\hat \cC_{\nl}$ in Theorem \ref{fthm}.

\end{sbprop}

\begin{sbpara}\label{shape0}
Let the notation be as in \ref{MandM} for some $\sig \in \Sig$.
As in \ref{sig'n}, we consider the cone 
$$
	\sig_n = \{(s_i)_{1\leq i\leq n} \mid (0^{r-1}, s_1, \dots, s_n) \in \sig\},
$$ 
which is identified with a face of $\sig$. With the identification $\hat S\cong \overline{\cM}^{r,n}_{N,+, \sig}$ given by \ref{MandM}, we have a morphism 
$$
	\hat S \to \toric_{\F_p}(\sig_n), \quad 
	((\cL, \phi), \iota)\mapsto \left(\frac{\iota(e_0)}{\iota(e_{i+r-1})}\right)_{1\leq i\leq n}.
$$
Using the above identification and the definition \ref{cM(sig)} of $\overline{\cM}^{r,n}_{N,(+,\sig)}$, there is also a canonical morphism $\hat S \to \cM^r_N$. Moreover, $\hat S$ is isomorphic to an open set of the formal completion of $\cM^r_N \times_{\F_p} \toric_{\F_p}(\sig_n)$ along a closed subscheme and is endowed with the inverse image of the canonical log structure of $\toric_{\F_p}(\sig_n)$.  
\end{sbpara}

\begin{sbpara} \label{shapethm3} We prove that  the log scheme
 $\overline{\cM}^d_{N, \Sig}$  has the properties
 stated in Theorem \ref{shape}(3).

 Roughly speaking, the proof is as follows. From \ref{shape0}, we have that after we take some formal completions, the log schemes 
 $\overline{\cM}^d_{N, \Sig}$ and $\toric_{\F_p}(\sig)\times \mathbb{G}_m$ for $\sig\in \Sig$ have the same shape, and  the log $A[\frac{1}{N}]$-schemes 
  $A[\frac{1}{N}]\otimes_A \overline{\cM}^d_{N, \Sig}$ and $\toric_{A[\frac{1}{N}]}(\sig)$ have the same shape. This implies that they have the same shape \'etale locally.

The precise proof is given as follows. Since $\overline{\cM}^d_{N, \Sig}$ is the union over $\sig \in \Sig$ and $g\in \GL_d(A/NA)$ of the 
log schemes $g\overline{\cM}^d_{N, +,\sig}$ by \ref{moduliremarks}(1), it is sufficient to prove the following properties for each $\sig\in \Sig$ and  $x\in \overline{\cM}^d_{N, +,\sig}$. 
\begin{enumerate}
\item[(1)] There exist an \'etale neighborhood $U\to \overline{\cM}^d_{N, +,\sig}$ of $x$ and a smooth morphism $U\to \toric_{\F_p}(\sig)$.
\item[(2)] If $N$ is invertible at $x$, there exist an  \'etale neighborhood $U\to A[\frac{1}{N}]\otimes_A \overline{\cM}^d_{N, +,\sig}$ of $x$ and an \'etale morphism  $U\to \toric_{A[\frac{1}{N}]}(\sig)$.
\end{enumerate}
Let $r$ be the rank of the fiber at $x$ of the universal generalized Drinfeld module on $\overline{\cM}^d_{N, +, \sig}$. Let $n=d-r$.

 We prove (1) (resp., (2)). Consider the morphism 
$$
	X := \overline{\cM}_{N, +, \sig} \to Y := \toric_{\F_p}(\sig_n)\quad (\text{resp., } X := A[\tfrac{1}{N}]\otimes_A \overline{\cM}_{N, +, \sig} \to Y := \toric_{A[\frac{1}{N}]}(\sig_n))
$$
given by
$$((\cL, \phi), \iota)\mapsto  \left(\frac{\iota(e_0)}{\iota(e_{i+r-1})}\right)_{1\leq i\leq n}.$$
Let $y$ be the image of $x$ in $Y$. 
By \ref{shape0} and the fact that 
 $\cM^r_N$ is smooth over $\F_p$,  we have an $\hat \cO_{Y,y}$-isomorphism between  $\hat \cO_{X,x}$ and a finite \'etale ring over $\hat \cO_{Y,y}\ps{t_1, \dots, t_r}$ (resp., $\hat \cO_{Y,y}\ps{t_1, \dots, t_{r-1}}$), 
 where $t_i$ are indeterminates and ${\hat {\;}}$ indicates the completion. This proves that the morphism $X\to Y$ is smooth at $x$ with relative dimension $r$ (resp., $r-1$). Since $\toric_{\F_p}(\sig) \cong \mathbb{G}_m^{r-1}\times \toric_{\F_p}(\sig_n)$, we have that there are an \'etale neighborhood $U\to X$ of $x$ and a smooth morphism $U\to \toric_{\F_p}(\sig)$ (resp., an \'etale morphism $U\to \toric_{A[\frac{1}{N}]}(\sig)$)
 for which the log structure of $U$ is the inverse image of that of the latter toric variety.

\end{sbpara}

\begin{sbpara}\label{lrls} By \ref{shapethm3} (so conditional upon the assumption in \ref{Massump}), the log scheme $\overline{\cM}^d_{N, \Sig}$ is log regular, and $A[\frac{1}{N}]\otimes_A \overline{\cM}^d_{N, \Sig}$ is log smooth over $A[\frac{1}{N}]$.
\end{sbpara}

We show that $\overline{\cM}_{N,\Sig}^d$ represents the moduli functor $\overline{\fM}_{N,\Sig}^d$ on the category $\cC_{\log}$ of schemes with saturated log structures over $A$ (over $A[\frac{1}{N}]$ in the case $N$ has only one prime divisor). First, we prove the following.
 
\begin{sbprop}\label{M=M} Let $F_{\nl}$ be the functor $\overline{\fM}^d_N$ (resp.,  $\overline{\fM}^d_{N, \Sig}$, resp.,  $\overline{\fM}^d_{N, +, \sig}$) on $\cC_{\nl}$ and let $F_{\log}$ be the  functor 
$\overline{\fM}^d_N$ (resp.,  $\overline{\fM}^d_{N, \Sig}$, resp.,  $\overline{\fM}^d_{N, +, \sig}$) on $\cC_{\log}$. Then $F_{\nl}$ is the composition of $F_{\log}$  and the functor $\cC_{\nl} \to \cC_{\log}$ of \ref{log22}. 

\end{sbprop}

\begin{proof} 
Let $(S,U)$ be an object of $\cC_{\nl}$ and consider the associated log structure \ref{log2} on $S$. We have the map $F_{\log}(S) \to F_{\nl}(S,U)$ taking a log Drinfeld module $((\cL,\phi),\iota)$ on $S$ to the pair $((\cL,\phi),\iota|_U)$. The latter object is a generalized Drinfeld module over $(S,U)$ with level $N$ structure as the generalized Drinfeld module $(\cL,\phi)$ on $S$ restricts to a Drinfeld module on $U$ (see \ref{logD8} and its proof in \ref{Deflog1}). Moreover, it is an element of $F_{\nl}(S,U)$ since a log Drinfeld module satisfies (div) by \ref{logD8} (see \ref{regdiv}). 

On the other hand, we have seen that $F_{\nl}$ is represented by an object $(P, W)$ of $\cC_{\nl}$ with $P$ log regular (\ref{lrls}). We therefore have a map  $F_{\nl}(S, U) \to F_{\log}(S)$ given by pulling back the universal object on $(P, W)$. These maps are inverse to each other by construction.
\end{proof}
 
\begin{sbpara}\label{replog}

We prove that 
the functor $\overline{\fM}^d_{N, \Sig} \colon \cC_{\log}\to (\text{Sets})$ is represented by $\overline{\cM}^d_{N, \Sig}$ (conditional upon our assumption in \ref{Massump} of representability on $\cC_{\nl}$). We may assume that $\Sig$ is a subdivision of $\Sig_k$, since \ref{Sigk8} tells us that $\overline{\fM}^d_{N, \Sig}$ and $\overline{\fM}^d_{N, \Sig*\Sig_k}$ are equal on both $\cC_{\log}$ and $\cC_{\nl}$.

The universal generalized Drinfeld module on $(\overline{\cM}^d_{N,\Sig},\cM^d_N)$ gives rise to a log Drinfeld module on the log regular
scheme $\overline{\cM}^d_{N,\Sig}$. (This is the universal log Drinfeld module of rank $d$ on $\overline{\cM}^d_{N,\Sig}$.) Taking the pullback of 
this log Drinfeld module defines a natural transformation $\overline{\cM}^d_{N,\Sig} \to \overline{\fM}^d_{N,\Sig}$ of functors on $\cC_{\log}$.  

The functor $ \overline{\fM}^d_{N,\Sig}$ on $\cC_{\log}$ is a sheaf for the \'etale topology, and by the definition \ref{logD6} of log Drinfeld modules of rank $d$ with level $N$ structure, this morphism is a surjection of sheaves on $\cC_{\log}$. It remains to prove the injectivity. 

Let $a,b \colon S\to \overline{\cM}^d_{N, \Sig}$ be morphisms in $\cC_{\log}$ which induce the same element of $\overline{\fM}^d_{N, \Sig}(S)$. We claim that $a=b$. For this, we may assume that $S$ is of finite type over $A$. Let $s\in S$, let $\hat \cO_{S,s}$ be the completion of the local ring $\cO_{S,s}$,  and consider the morphisms 
$$
	\hat a_s, \hat b_s \colon\Spec(\hat \cO_{S,s})\to \overline{\cM}^d_{N, \Sig}
$$ 
of $\cC_{\log}$ induced by $a$ and $b$, respectively. Here, $\Spec(\hat \cO_{S,s})$ has the inverse image of the log structure of $S$. It is sufficient to prove that $\hat a_s=\hat b_s$. It follows from \ref{opencover} that by changing the basis of $(\frac{1}{N}A/A)^d$ if necessary, we may assume that $\hat a_s$ and $\hat b_s$ are morphisms $\Spec(\hat \cO_{S,s})\to \overline{\cM}^d_{N, +,\sig}$ for some $\sig\in \Sig$.
By \ref{MandM},  these morphisms correspond respectively to morphisms 
$$
	\hat a_s', \hat b'_s \colon\Spf(\hat \cO_{S,s})\to \overline{\cM}^{r,n}_{N, +,\sig}
$$ 
in $\hat \cC_{\log}$, for some $r$, $n$ such that $r+n=d$, where $\Spf(\hat \cO_{S,s})$ is given the inverse image of the log structure on $S$. Because they give the same element of $\overline{\fM}^{r,n}_{N, +, \sig}(\Spf(\hat \cO_{S,s}))$, we have   $\hat a'_s=\hat b'_s$ by Theorem \ref{fthm} for $\hat \cC_{\log}$. 

\end{sbpara}

\subsection{Toroidal compactifications} \label{ss:tornl}

In this subsection, for a finite rational subdivision $\Sig$ of $C_d$, we construct the toroidal compactification $\overline{\cM}_{N,\Sig}^d$ of $\cM_N^d$ and prove that it has the properties stated in the main theorems stated in Section \ref{main_result}. 

We proceed as follows. First, we construct the special toroidal  compactification $\overline{\cM}^d_{N, \Sig_k}$ by blowing-up the Satake compactification $\overline{\cM}^d_{N,\Sa}$ in order that the $N$-torsion points of the universal generalized Drinfeld module on $\overline{\cM}^d_{N, \Sa}$ satisfy the condition (div). We prove that $(\overline{\cM}^d_{N, \Sig_k}, \cM^d_N)$ represents the moduli functor $\overline{\fM}^d_{N, \Sig_k}$ on $\cC_{\nl}$. By \ref{lrls}, we then have that $\overline{\cM}^d_{N, \Sig_k}$ is log regular. 

Next, we consider the toroidal compactification for a general $\Sig$. Using toric geometry (formulated in the style of log geometry), we  construct $\overline{\cM}^d_{N, \Sig}$ out of $\overline{\cM}^d_{N,\Sig_k}$ as its modification associated to subdivisions of a fan. From the log regularity of $\overline{\cM}^d_{N,\Sig_k}$, we obtain the log regularity of  $\overline{\cM}^d_{N,\Sig}$. Using this log regularity, we can show that 
$(\overline{\cM}^d_{N, \Sig}, \cM^d_N)$ represents the moduli functor $\overline{\fM}^d_{N, \Sig}$ on $\cC_{\nl}$. By the results of Section \ref{ss:AMFM}, we may then conclude that $\overline{\cM}^d_{N, \Sig}$ has all properties stated in the main theorems in Introduction.

In the setting of abelian varieties, there is no special toroidal compactification of a given moduli space among the collection of all of them. While we do have a special toroidal compactification of each moduli space of Drinfeld modules, it is better to consider all toroidal compactifications since we can find regular ones (and smooth ones after inverting $N$) among them.

\begin{sbpara}\label{logSa} We  construct $\overline{\cM}^d_{N, \Sig_k}$ for $k = \deg N$ as follows. Let $S=\overline{\cM}^d_{N, \Sa}\supset U= \cM^d_N$, and let $M_S$ be the associated log structure on $S$ of \ref{log2}. Let 
$((\cL, \phi), \iota)$ be the universal generalized Drinfeld module over $(S,U)$ of rank $d$ with level $N$ structure. 

For each $a\in (\frac{1}{N}A/A)^d$, let $I_a$ be the ideal of $\cO_S$ which is locally generated by $f\in M_S\subset \cO_S$ such that the class of $f$ in $M_S/\cO^\times_S$ coincides with $\pole(\iota(a))$ (see \ref{logD53}). Let $I \subset \cO_S$ be the product of the ideals $I_a+I_b$ for all pairs $(a,b)$ of elements of $(\frac{1}{N}A/A)^d$. Let $\overline{\cM}^d_{N, \Sig_k}$ be the normalization of the blow-up of $S$ along $I$,
and note that it is proper as $S$ is.

Recall that  $\overline{\fM}^d_N=\overline{\fM}^d_{N, \Sig_k}$ by Proposition \ref{Sigk8}. 
\end{sbpara}

\begin{sbprop}\label{repnl1}
	The pair $(\overline{\cM}^d_{N, \Sig_k}, \cM^d_N)$ represents the functor $\overline{\fM}^d_{N, \Sig_k}$ on the category $\cC_{\nl}$. 
\end{sbprop}

\begin{pf} 
  Let $(S, U)$ be an object of $\cC_{\nl}$. We apply Lemma \ref{nllem}(2) for the pair $(P,W) = (\overline{\cM}^d_{N, \Sa},\cM^d_N)$ 
  and the ideals $I_a$ as in \ref{logSa} for $a \in (\frac{1}{N}A/A)^d$, in which case $P'= \overline{\cM}^d_{N, \Sig_k}$. 
Let $f$ be an element of $\overline{\fM}^d_{N,\Sa}(S,U)$. Then $f$ belongs to $\overline{\fM}^d_N(S,U)=\overline{\fM}^d_{N,\Sig_k}(S,U)$ if and only if $f$ satisfies the condition (div), that is, if and only if the pullbacks $I_{a,S}$ of $I_a$ under $f \colon S\to P$ satisfies either $I_{a,S}\subset I_{b,S}$ or $I_{a,S}\subset I_{b,S}$ for every pair $(a,b)$ of elements of $(\frac{1}{N}A/A)^d$. By \ref{nllem}(2), this condition is equivalent to the condition that there is a morphism $(S,U) \to (P', W)$ over $(P, W)$, and furthermore, such a morphism is unique if it exists. 
\end{pf}

\begin{sbpara}\label{lrls2}
By \ref{lrls} and \ref{repnl1}, the log scheme $\overline{\cM}^d_{N, \Sig_k}$ is log regular, and $A[\frac{1}{N}]\otimes_A \overline{\cM}^d_{N, \Sig_k}$ is log smooth over $A[\frac{1}{N}]$. This is the special case $\Sig = \Sig_k$ of a part of Theorem \ref{main}(1). Moreover, Theorem \ref{M=M0} follows from the case $\Sig = \Sig_k$ of Proposition \ref{M=M}, which is unconditional in this case by \ref{repnl1}. 
\end{sbpara}

\begin{sbpara}\label{repnl2} Let $((\cL,\phi), \iota)$ be the universal generalized Drinfeld module over $(\overline{\cM}^d_{N, \Sig_k}, \cM^d_N)$ of rank $d$ with level $N$ structure. For a basis $(e_i)_{0\leq i\leq d-1}$ as in \ref{logD9}(1),  the family $(\pole(\iota(e_i)))_{1\leq i\leq d-1}$, which is independent of the choice of basis,
provides an element of $[\Sig_k](\overline{\cM}^d_{N, \Sig_k})$, similarly to \ref{modulirem}. That is, we have a morphism  $\overline{\cM}^d_{N, \Sig_k}\to [\Sig_k]$ of functors on $\cC_{\log}$.

Consider the fs log scheme 
$$\overline{\cM}^d_{N, \Sig} := \overline{\cM}^d_{N, \Sig_k} \times_{[\Sig_k]} [\Sig*\Sig_k]
=\overline{\cM}^d_{N, \Sig_k} \times_{[C_d]} [\Sig]$$ 
(see \ref{toSig0}), where $*$ denotes the join.
This is proper and log \'etale over $\overline{\cM}^d_{N,\Sig_k}$, again using \ref{toSig0}. 
 In the case that $\Sig$ is a subdivision of $\Sig_k$, it is $\overline{\cM}^d_{N, \Sig_k} \times_{[\Sig_k]} [\Sig]$. 
 
Since a log \'etale scheme over a log regular scheme is log regular, we have that $\overline{\cM}^d_{N, \Sig}$ is log regular. From the log smoothness of  $A[\frac{1}{N}]\otimes_A\overline{\cM}^d_{N, \Sig_k}$, we have the log smoothness of
 $A[\frac{1}{N}]\otimes_A \overline{\cM}^d_{N, \Sig}$. 
 From the properness of  $\overline{\cM}^d_{N, \Sig_k}$, we have the properness of $\overline{\cM}^d_{N, \Sig}$.
The morphism $\overline{\cM}^d_{N,\Sig}\to \overline{\cM}^d_{N, \Sa}$  and the morphisms $\overline{\cM}^d_{N,\Sig'} \to \overline{\cM}^d_{N, \Sig}$ for finite rational subdivisions $\Sig'$ of $\Sig$ are birational because these morphisms induce the identity morphism on their dense open subscheme $\cM^d_N$. 
 
\end{sbpara}

\begin{sblem}\label{Mtriv} In $\overline{\cM}^d_{N, \Sig}$, the open set $\cM^d_N$ coincides with the set of all points at which the log structure is trivial.

\end{sblem}

\begin{proof} We prove this first for $\Sig= \Sig_k$. Let $W$ be the open set of $\overline{\cM}^d_{N, \Sig_k}$ consisting of all points at which the log structure is trivial. Then $\cM^d_N\subset W$. Since the restriction of the universal object $((\cL, \phi), \iota)$ on $(\overline{\cM}^d_{N, \Sig_k},\cM^d_N)$ to $W$ is a Drinfeld module of rank $d$ with a level $N$ structure,  we have by \ref{repnl1} a morphism $(\overline{\cM}^d_{N, \Sig_k}, W) \to (\overline{\cM}^d_{N, \Sig_k}, \cM^d_N)$ such that the composition $(\overline{\cM}^d_{N, \Sig_k}, \cM^d_N)\to (\overline{\cM}^d_{N, \Sig_k}, W) \to (\overline{\cM}^d_{N, \Sig_k}, \cM^d_N)$ is the identity morphism. This proves $W=\cM^d_N$.

We consider the general case. Replacing $\Sig$ by the join $\Sig *\Sig_k$, we may assume that $\Sig$ is a finite rational subdivision of $\Sig_k$. 
By \ref{shapethm3}, the morphism $\overline{\cM}^d_{N, \Sig_k} \to [\Sig_k]$ is strict, that is, on an \'etale neighborhood $U$ of each $x \in \overline{\cM}^d_{N, \Sig_k}$, it factors as $U \to \toric_{\F_p}(\sig) \twoheadrightarrow [\sig] \hookrightarrow [\Sig_k]$ for some $\sig\in \Sig$, where the log structure of $U$ is the inverse image of that of $\toric_{\F_p}(\sig)$. It follows that the morphism $\overline{\cM}^d_{N, \Sig}\to [\Sig]$ is also strict. The strictness of these morphisms show that the open subscheme $V$of $\overline{\cM}^d_{N, \Sig}$  consisting of points at which the log structure is trivial is the fiber product $\overline{\cM}^d_{N, \Sig}\times_{[\Sig]} [\{0\}]$. Similarly, we have $W = \overline{\cM}^d_{N, \Sig_k}\times_{[\Sig_k]} [\{0\}]$. It follows that $V = W$ and therefore $V= \cM^d_N$. 
\end{proof}

\begin{sbpara}

Since $\overline{\cM}^d_{N, \Sig}$ is log regular, it is normal by \ref{logD20}(1), and hence we have an object $(\overline{\cM}^d_{N, \Sig}, \cM^d_N)$ of $ \cC_{\nl}$. Moreover, this log regularity implies that the log structure on $\overline{\cM}^d_{N, \Sig}$ coincides with the log structure of \ref{log2} associated to the pair $(\overline{\cM}^d_{N, \Sig}, \cM^d_N)$ by \ref{logD20}(2). 
\end{sbpara}

\begin{sbprop}\label{repnl3} The pair $(\overline{\cM}^d_{N, \Sig}, \cM^d_N)\in \cC_{\nl}$  represents the functor $\overline{\fM}^d_{N, \Sig*\Sig_k}= \overline{\fM}^d_{N, \Sig}$ on $\cC_{\nl}$.

 \end{sbprop}
 
 \begin{proof} Let $(S,U)$ be an object of $\cC_{\nl}$. Endow $S$ with the log structure of \ref{log2}. 
 Consider the following collections:
 \begin{enumerate}
	\item[(a)] morphisms $(S, U) \to (\overline{\cM}^d_{N, \Sig}, \cM^d_N)$ in $\cC_{\nl}$,
 	\item[(b)] morphisms $S\to \overline{\cM}^d_{N, \Sig}$ in $\cC_{\log}$,
 	\item[(c)] pairs of a morphism $S\to \overline{\cM}^d_{N, \Sig_k}$ in $\cC_{\log}$ and a morphism $S\to [\Sig]$ which induce the same morphism $S\to [\Sig_k]$, and
 	\item[(d)] pairs of morphisms $(S,U)\to (\overline{\cM}^d_{N, \Sig_k},\cM^d_N)$ in $\cC_{\nl}$ and $S\to [\Sig]$ which induce the same morphism $S\to [\Sig_k]$.
\end{enumerate}
 By \ref{log22}(3) and \ref{Mtriv}, we have that (a) and (b) are in bijection, as are (c) and (d). The collections (b) and (c) are in bijection by the definition of $\overline{\cM}^d_{N, \Sig}$ in \ref{repnl2}. By \ref{repnl1} and the definition of $\overline{\cM}^d_{N, \Sig}$, the set of pairs in (d) is identified with $\overline{\fM}^d_{N, \Sig}(S,U)$. 
  \end{proof}

\begin{sbpara} \label{opencover} For $\sig\in \Sig$,  by \ref{repnl3} and \ref{moduliremarks}(3), we have an open set $\overline{\cM}^d_{N, +, \sig}$ of $\overline{\cM}^d_{N, \Sig}$ containing $\cM^d_N$ such that $(\overline{\cM}^d_{N, +, \sig}, \cM^d_N)$ represents the functor $\overline{\fM}^d_{N, +, \sig}$ on $\cC_{\nl}$.  
We then have an open covering
$$
	\overline{\cM}^d_{N,\Sig}= \bigcup_{\sig \in \Sig} \bigcup_{g \in \GL_d(A/NA)} g(\overline{\cM}^d_{N, +,\sig})
$$
 by Lemma \ref{moduliremarks}(1). 
\end{sbpara}

\begin{sbpara}\label{forSig} 
By \ref{repnl3}, all results in \ref{fMcM}--\ref{replog} become unconditional. These together with \ref{repnl2}, prove Theorem \ref{shape} and parts (1) and (2) of Theorem \ref{main}. Theorem \ref{main}(3) is proven in \ref{mainthm3} below. 
\end{sbpara}

\begin{sbthm}\label{NN'} Let $N'\in A$ be a nonzero multiple  of $N$. Let $k$ (resp., $k'$) be the degree of the polynomial $N$ (resp., $N'$). Let $\Sig$ be a finite rational subdivision of $\Sig_{k'}$, and let  $\Sig'$ be the corresponding finite rational subdivision of $\Sig_{k,k'}$ (see \ref{Sigk5}). 
Unless $N'$ has at least two prime divisors and $N$ has only one prime divisor, ${\overline \cM}^d_{N',\Sig}$ is the integral closure of ${\overline \cM}^d_{N, \Sig'}$ in the function field of $\cM^d_{N'}$. Otherwise, the integral closure is $A[\frac{1}{N}]\otimes_A {\overline \cM}^d_{N',\Sig}$.

\end{sbthm}

\begin{proof} We assume that either $N'$ has one prime divisor or $N$ has at least two, the remaining case being similar. We apply \ref{nllem}(2) for $P=\overline{\cM}^d_{N, \Sig'}$, $W=\cM^d_N$, and $W'=\cM^d_{N'}$. By the arguments as in \ref{Sapf2}, the functor on $\cC_{\nl}$ represented by the integral closure of  $\overline{\cM}^d_{N, \Sig'}$ in $\cM^d_{N'}$ is the subfunctor of $\overline{\fM}^d_{N'}$ which is the inverse image of the subfunctor $\overline{\fM}^d_{N,\Sig'}$ of 
$\overline{\fM}^d_N$. By \ref{Sigk5} and \ref{todvr}, this coincides with the subfunctor $\overline{\fM}^d_{N',\Sig}$ of $\overline{\fM}^d_{N'}$. Hence this integral closure coincides with $\overline{\cM}^d_{N', \Sig}$. 
\end{proof}

Recall that $\sig \cap \Z^{d-1}$ is regular for a cone $\sig$ in $\R^{d-1}$ if it is generated by $m$ elements as a monoid, where $m$ is the dimension of $\sig$.

\begin{sbprop} \label{regsm} Let $\Sig$ be a subdivision of $\Sig_k$, and suppose that $\sig\in \Sig$ is such that $\sig\cap \Z^{d-1}$ is regular. Then the underlying scheme of $\overline{\cM}^d_{N,+, \sig}$ is regular, the complement of $\cM^d_N$ in  $\overline{\cM}^d_{N, +,\sig}$  is a  normal crossing divisor, the underlying scheme of $A[\frac{1}{N}]\otimes_A \overline{\cM}^d_{N,+, \sig}$ is smooth over $A[\frac{1}{N}]$,  and the complement of $A[\frac{1}{N}]\otimes_A \cM^d_N$ in  $A[\frac{1}{N}]\otimes_A \overline{\cM}^d_{N, +, \sig}$  is a  relative normal crossing divisor over $A[\frac{1}{N}]$.

\end{sbprop} 

\begin{proof} For $\sig$ such that $\sig\cap \Z^{d-1}$ is regular,  the toric variety $\toric_{\F_p}(\sig)$ is smooth over $\F_p$ with normal crossing boundary (i.e., the complement of its torus part is a divisor with normal crossings). The morphisms of \ref{shapethm3}  constructed in property (1) (resp., (2)) yield the properties of $\overline{\cM}^d_{N, +,\sig}$ (resp., $A[\frac{1}{N}]\otimes_A \overline{\cM}^d_{N, +, \sig}$) stated in the proposition. For this, note for instance that the $U$ constructed in (1) (resp., (2)) is regular, being smooth over $\F_p$ 
(resp., smooth over $A[\frac{1}{N}]$, in that $\toric_{A[\frac{1}{N}]}(\sig)$ is smooth over $A[\frac{1}{N}]$).
\end{proof}

\begin{sbpara} \label{mainthm3}
As each finite rational cone decomposition of $C_d$  has a finite rational subdivision consisting of cones $\sig$ such that $\sig\cap \Z^{d-1}$ are regular (cf. \cite[Ch.~1, Thm.~11]{KKMS}), we obtain Theorem \ref{main}(3) on $\overline{\cM}^d_{N,\Sig}$ as a consequence
of Proposition \ref{regsm}.
\end{sbpara}

\subsection{An example}\label{ss:example}

In this final subsection, we provide an explicit description of the toroidal compactification $\overline{\cM}^3_{T,\Sig_1}$ and its relation to the Satake compactification $\overline{\cM}^3_{T,\Sa}$. We also describe the local monodromy of the universal log Drinfeld module on the former space (see \ref{Gal2}), and  of the universal generalized Drinfeld module on the latter space (see \ref{Gal1}).

\begin{sbpara} We make the identification
$$
	\mathbb{P}^2_{A[\frac{1}{T}]}=A[\tfrac{1}{T}]\otimes_{\F_q} \text{Proj}(\Sym_{\F_q}(V))
$$ 
for $V=\sum_{i=0}^2 \F_qu_i$ as in \ref{Sapf1a} in the case $d=3$. We prove that $\overline{\cM}^3_{T, \Sig_1}$ coincides with the blow-up of  $\mathbb{P}^2_{A[\frac{1}{T}]}$ along all $a \otimes A[\frac{1}{T}]\subset \mathbb{P}^2_{A[\frac{1}{T}]}$ with $a \in \mathbb{P}^2(\F_q)$. 

In fact, let $X$ be this blow-up. It is sufficient to prove that $(X, \cM^3_T)$ represents $\overline{\fM}^3_{T,\Sig_1}=\overline{\fM}^3_{T}$ on $\cC_{\nl}$. 
Note that $X$ is the blow-up of $\mathbb{P}^2_{A[\frac{1}{T}]}$ along the intersections of all pairs of 
different $\F_q$-rational hyperplanes in $\mathbb{P}^2$. We identify $V$ with $(\frac{1}{T}A/A)^3$.

Let $(S,U)$ be an object of $\cC_{\nl}$, and let
 $((\cL,\phi),\iota)\in \overline{\fM}^3_T(S,U)$. 
 Define a line bundle $\cL'$ on $S$ as follows. Because the condition (div) of \ref{strong} is satisfied, locally on $S$, there are nonzero elements $a,b$ of $V$ such that $\iota(b) \mid \iota(v) \mid \iota(a)$ for all nonzero elements $v$ of $V$. (That is, the pole of $a$ is the minimum, the pole of $b$ is the maximum, and actually, $\iota(a)$ is a basis of $\cL$ by \ref{unit}(2).) The line bundle $\cL'$ on $S$ is then uniquely defined such that $\cL'= \iota(b)\iota(a)^{-1}\cL= \cO_S \cdot\iota(b)$ locally. The canonical surjection $\cO_S \otimes_{\F_q} V \to \cL'$ induced by $\iota$ gives rise to a morphism $S\to \mathbb{P}^2_{A[\frac{1}{T}]}$ of $A[\frac{1}{T}]$-schemes that sends $U$ to 
$\cM^3_T$ by the description of $\cM^3_T$ in \ref{Sapf1b}. Since the condition (div) is satisfied, \ref{nllem}(2) tells us that
this morphism extends to a morphism $(S, U)\to (X, \cM^3_T)$ of $\cC_{\nl}$. 

Conversely, we have an element 
$((\cL, \phi), \iota)\in \overline{\fM}^3_T(X,\cM^3_T)$ 
as follows. Let $\cL'$ be the pullback of the line bundle $\cO(1)$ on $\mathbb{P}^2_{A[\frac{1}{T}]}$ to $X$, and let 
$p \colon \cO_X \otimes_{\F_q} V \to \cL'$ be the canonical surjection. By the construction of $X$ as the blow-up along the intersections of the $\F_q$-rational hyperplanes, for each pair of nonzero elements $(a, b)$ of $V$, we have either $p(a) \mid p(b)$ or $p(b) \mid p(a)$ locally on $X$. Hence locally on $X$, there are nonzero elements $a,b$ of $V$ such that $p(b)\mid p(v)\mid p(a)$ for all nonzero elements $v$ of $V$. Define a line bundle $\cL$ on $X$ locally by $\cL = p(a)p(b)^{-1}\cL' = \cO_X\cdot p(a)$.
We then define $\phi$ on $z \in \cL$ by 
$$
	\phi(T)(z)= Tz \prod_{v \in V \setminus \{0\}} (1-p(v)^{-1}z).
$$ 
Here, we have $p(v)^{-1}\cL=p(v)^{-1}\cO_X\cdot p(a)\subset \cO_X$. The level structure $\iota$ is then induced by $p$. For an object $(S,U)$ of $\cC_{\nl}$, a morphism $(S,U)\to (X, \cM^3_T)$ of $\cC_{\nl}$ then gives rise to an element of $\overline{\fM}^3_T(S,U)$ by pullback. 

As these maps are the inverse to each other by construction, we conclude that $(X,\cM^3_T)$ represents $\overline{\cM}^3_T$ on $\cC_{\nl}$,
as required.
\end{sbpara}

\begin{sbpara} For $a\in \mathbb{P}^2(\F_q)$, let $D(a)\subset \overline{\cM}^3_{T, \Sig_1}$ be  the inverse image of 
$$
	a \otimes A[\tfrac{1}{T}]\subset \mathbb{P}^2_{A[\frac{1}{T}]}.
$$ 
For a line $\ell$ in $\mathbb{P}^2(\F_q)$, let $D(\ell)$ be the proper transform in $\overline{\cM}^3_{T,\Sig_1}$ of $\ell\otimes A[\frac{1}{T}]\subset \mathbb{P}^2_{A[\frac{1}{T}]}$.

The complement $D$ of $\cM^3_T$ in $\overline{\cM}^3_{T, \Sig_1}$ is the union of all $D(a)$ and $D(\ell)$. 
The scheme $\overline{\cM}^3_{T,\Sig_1}$ is smooth over $A[\frac{1}{T}]$, and $D$ is a relative normal crossing divisor over $A[\frac{1}{T}]$. 

\end{sbpara}

\begin{sbpara} Recall that $\Sig_1$ is the set of all faces of $\sig :=  \{(s_1,s_2)\in \R^2_{\geq 0} \mid s_1\leq s_2\}$. We have the following statements.
\begin{enumerate}
\item[(1)]  The open set $\overline{\cM}^3_{T,+, \sig}$ of $\overline{\cM}^3_{T,\Sig_1}$ coincides with the complement of the union of $D(a)$ and $D(\ell)$, where $a$ ranges over all points of $\mathbb{P}^2(\F_q)$ which are not $(0:0:1)$, and $\ell$ ranges over all lines in $\mathbb{P}^2(\F_q)$ which are not the line $u_0=0$. 
\item[(2)] We can identify $\overline{\cM}^3_{T,+, \sig}$ with the open set of $\Spec(A[\frac{1}{T}][\frac{u_0}{u_1}, \frac{u_1}{u_2}])=\toric_{A[\frac{1}{T}]}(\sig)$ defined as the complement of the union of zeros of 
$$1+ a\frac{u_0}{u_1}, \quad 1+b \frac{u_0}{u_2}+ c\frac{u_1}{u_2}\quad (a,b,c\in \F_q).$$
The intersection of  $\overline{\cM}^3_{T,+,\sig}$ with  $D(u_0=0)$ is the part $\frac{u_0}{u_1}=0$, and the intersection with  $D(0:0:1)$  is the part $\frac{u_1}{u_2}=0$.
\item[(3)] Let $x\in\overline{\cM}^3_{T, \Sig_1}$. Then the fiber at $x$ of the universal generalized Drinfeld module of $\overline{\cM}^3_{T, \Sig_1}$ has rank $1$ if and only if $x$ belongs to $D(\ell)$ for some line $\ell$ in $\mathbb{P}^2(\F_q)$, rank $2$ if and only if $x$ does not belong to $D(\ell)$ for any $\ell$ and $x$ belongs to $D(a)$ for some $a\in \mathbb{P}^2(\F_q)$, and rank $3$ if and only if $x\in \cM^3_T$. 
\item[(4)] The formal scheme $\overline{\cM}^{1,2}_{T, +, \sig}$ is isomorphic to the $\frac{u_0}{u_1}$-adic formal completion 
 of $\overline{\cM}^3_{T,+,\sig}$. That is, $\overline{\cM}^{1,2}_{T, +, \sig}=\Spf(H)$, where 
 $$H=\varprojlim_n A[\tfrac{1}{T}]\left[\frac{u_0}{u_1}, \frac{u_1}{u_2}\right]\left[\left(1+c\frac{u_1}{u_2}\right)^{-1} \mid c\in \F_q\right]/\left(\frac{u_0}{u_1}\right)^n.$$
 \end{enumerate}

The proof of (1) is straightforward by the modular interpretation of $(\overline{\cM}^3_{T,\Sig_1}, \cM^3_T)$. Part (2) follows from part (1). Part (3) follows from part (2) because the fiber has rank $1$ if and only if $\frac{u_0}{u_1}=0$, rank $2$ if and only if $\frac{u_0}{u_1}\neq 0$ and $\frac{u_1}{u_2}=0$, and rank $3$ if and only if  $\frac{u_0}{u_2}\neq 0$. Part (4) follows from parts (2) and (3).

\end{sbpara}

\begin{sbpara}\label{exrank} The Satake compactification $\overline{\cM}^3_{T,\Sa}$ is the quotient of $\overline{\cM}^3_{T, \Sig_1}$ obtained by contracting $D(\ell)$ for all lines $\ell$ in $\mathbb{P}^2(\F_q)$: For the morphism $\overline{\cM}^3_{T, \Sig_1}\to \overline{\cM}^3_{T, \Sa}$ of \ref{logSa}, the image of $D(\ell)$ is isomorphic to $\Spec(A[\frac{1}{T}])$ for each $\ell$, and if we denote the image of $\bigcup_{\ell} D(\ell)$ by $Z$, then 
$$
	\overline{\cM}^3_{T, \Sig_1}\setminus \bigcup_{\ell} D(\ell)\to \overline{\cM}^3_{T,\Sa}\setminus Z
$$ 
is an isomorphism. 

Concerning the universal generalized Drinfeld module over $\overline{\cM}^3_{T, \Sa}$, the rank of the fiber at $s\in \overline{\cM}^3_{T, \Sa}$ is $3$ if $s\in \cM^3_T$,  is $1$ if $s\in Z$, and is $2$ otherwise.

\end{sbpara}

\begin{sbpara}\label{STF} For the remainder of this subsection, let us assume that $q = 2$. We describe some open sets of the Satake compactification $X := \overline{\cM}^3_{T, \Sa}$ and its relation with $\overline{\cM}^3_{T, \Sig_1}$ explicitly.

Take the affine open set  
$$
	B = \{x\in X\mid u_0f^{-1} \;\in \cO_{X.x} \textsp{for all} f\in V\setminus \{0\}\}
$$ 
of $X$, 
denote $u_0u_1^{-1}, u_0u_2^{-1}, u_0(u_1+u_2)^{-1}\in \cO_B(B)$ by $t_1, t_2, t_3$, respectively, and let 
$S$ be the open set $\{x\in B\mid 1+t_i \; \text{is invertible at}\; x\; \text{for}\; 1\leq i\leq 3\}$ of $B$. We have 
$$
	S= \Spec(A[\tfrac{1}{T}][t_1, t_2, t_3, (1+t_1)^{-1}, (1+t_2)^{-1}, (1+t_3)^{-1}]/(t_1t_2+t_2t_3+t_3t_1)).
$$
Then $\cM^3_T$ is identified with the open set of $S$ consisting of points at which $t_1t_2t_3$ is invertible. The universal generalized Drinfeld module $((\cL, \phi), \iota)$ over $S$ has
$\cL=\cO$,  
$$\phi(T)(z)= Tz\prod_{f\in V\setminus \{0\}}\; (1-u_0f^{-1}z)
= Tz(1-z) \prod_{i=1}^3 ((1-t_iz)(1-t_i(1+t_i)^{-1}z))$$
for $z \in \cO$, and
$\iota(e_0)=1$ and $\iota(e_i)=t_i^{-1}$ for $i=1,2$. The part $t_1=t_2=0$ of $S$ coincides with the set of all points of $S$ at which the fiber of the universal generalized Drinfeld module has rank $1$. 

The morphism $\overline{\cM}^3_{T, \Sig_1}\to \overline{\cM}^3_{T, \Sa}$ induces $\overline{\cM}^3_{T,+,\sig}\to S$ which comes from the ring homomorphism 
$$\frac{A[\tfrac{1}{T}][t_1, t_2, t_3]}{(t_1t_2+t_2t_3+t_3t_1)} \to A[\tfrac{1}{T}]\left[\frac{u_0}{u_1}, \frac{u_1}{u_2}, \left(1+ \frac{u_1}{u_2}\right)^{-1}\right].$$

\end{sbpara}

\begin{sbpara}\label{Sdiv} Let $s\in S$ of \ref{STF} be a point of the Satake compactification  such that $t_i(s)=0$ for $i=1,2,3$. Then by the explicit presentation of the affine ring of $S$ given in \ref{STF}, we have neither $t_1\mid t_2$ nor $t_2\mid t_1$ at  $s$. Hence the universal generalized Drinfeld module over $(S, \cM^3_T)$ does not satisfy the condition (div) in \ref{strong}. 

\end{sbpara}
In the following \ref{Gal1} and \ref{Gal2}, we 
consider the local monodromy of the universal generalized Drinfeld module at a point of Satake compactification and at a point of a toroidal compactification, respectively.

\begin{sbpara}\label{Gal1} Let $s$ be a point of $S$ such that $t_i(s)=0$ for $i=1,2,3$. Let 
$R$ be the completion of the strict henselization of the local ring $\cO_{S,s}$, and let $Q$ be the field of fractions of $R$.
Let
$v$ be  the maximal ideal $(T)$ of $A$. 
We consider the action of $\Gal(Q^{\sep}/Q)$ on the three dimensional $F_v$-vector space $V_v\phi=F_v\otimes_{A_v} T_v\phi$ where $T_v\phi$ is the $v$-adic Tate module of the universal generalized Drinfeld module $\phi$ over $S$ viewed as a Drinfeld module over $Q$.
Let $\psi$ be the Drinfeld module over $R$ of rank $1$ associated to $\phi$ as in \ref{prop51}. We can regard $T_v\psi\subset T_v\phi$. 

\medskip

{\bf Claim}. Consider the action of $\Gal(Q^{\sep}/Q)$ on $V_v\phi$ by taking a basis $(e_i)_{0\leq i\leq 2}$ such that $e_0$ is a basis of $V_v\psi$. Then the image of $\Gal(Q^{\sep}/Q)$ in $\GL_3(F_v)$ is an open subgroup of
  $$\left\{ \Pmatrix{1&e&f\\0&a&b\\ 0&c& d} \mid \Pmatrix{ a&b\\ c&d } \in \SL(2, F_v), e, f\in F_v\right\}.$$

\begin{proof}[Proof of Claim.] Let $\frak p_1$ be the prime ideal of $R$ generated by $t_2, t_3$, and let $\frak p_2$ be the prime ideal of $R$ generated by $t_1, t_3$ (so $t_i\notin \frak p_i$ for $i=1,2$). 
For $i=1,2$, let $\cV_i$ be the completion of  the local ring of $R$ at $\frak p_i$ (so $\cV_i$ is a complete discrete valuation ring), let $K_i$ be the field of fractions of the strict henselization of $\cV_i$, and fix embeddings $Q^{\sep}\to K_i^{\sep}$ over $Q$. Let $\psi_i$ be the Drinfeld module over $\cV_i$ of rank $2$ associated to $\phi$ as in \ref{prop51}. We can regard $T_v\psi\subset T_v\psi_i \subset T_v \phi$. 
Take an $A_v$-basis $(e_i)_{0\leq i\leq 2}$ of $T_v\phi$ such that $e_0$ is a basis of $T_v\psi$, the image of $e_i$ in the $T$-torsion  $\phi[T]$ is $1$, $t_1^{-1}$, $t_2^{-1}$ for $i=0,1,2$, respectively, and $e_0$ and $e_i$ form an $A_v$-basis of $T_v\psi_i$.
For $\sig\in \Gal(K_i^{\sep}/K_i)$, we have that $\sig-1$ kills $T_v\psi_i$, and by the proof of \ref{thmlm}, the elements $(\sig-1)A_ve_{3-i}$ for $\sig \in \Gal(K_i^{\sep}/K_i)$ generate an $A_v$-submodule of $T_v\psi_i$ of finite index. The action of $\Gal(K_i^{\sep}/K_i)$ on $T_v\phi$ factors through $\Gal(Q^{\sep}/Q)$. The action of $\Gal(Q^{\sep}/Q)$ on $V_v\phi$ is of determinant $1$ by the theory of determinants of Drinfeld modules (see \cite{An}, \cite[2.6.3]{Go2}). These prove the claim. Here we use the following fact:
\begin{itemize}
	\item If $U$ is an open subgroup of $\smatrix{ 1 & F_v\\ 0& 1 }$ and $V$ is an open subgroup of $\smatrix{ 1 & 0\\F_v& 1}$, then $U$ and $V$ 
	generate an open subgroup of $\SL_2(F_v)$. 
\end{itemize}
\end{proof}

\end{sbpara}

\begin{sbpara}\label{Gal2}  Let $x\in \cM^3_{T, +, \sig}$ (with $q = 2$) be a point at which $\frac{u_0}{u_1}(x)=\frac{u_1}{u_2}(x)=0$.  
Let 
$R$ be the completion of the strict henselization of the local ring at $x$, and let $Q$ be the field of fractions of $R$. Let $((\cL, \phi), \iota)$ be the universal log Drinfeld module over $\overline{\cM}^3_{T,\Sig_1}$ of rank $3$ with level $T$-structure, and consider the corresponding element of  $\overline{\fM}^{1,2}_{T, +, \sig}(\Spf(R))$. Let $\phi_i$ for $i\in\{0,1,2\}$ with $\phi_2=\phi$ and $\phi_0=\psi$ be as in \ref{itTate}. Let $v$ be as in \ref{Gal1}. We can then view $T_v\phi_1$ as satisfying $T_v\psi \subset T_v\phi_1 \subset T_v\phi$.

\medskip

{\bf Claim.} Consider the action of $\Gal(Q^{\sep}/Q)$ on the $3$-dimensional $F_v$-vector space $V_v\phi$ by using a basis $(e_i)_{0\leq i\leq 2}$ such that $e_0$ is a basis of $T_v\psi$ and $(e_0, e_1)$ is a  basis of $T_v\phi_1$. Then the image of $\Gal(Q^{\sep}/Q)$ in $\GL_3(F_v)$ is an open subgroup of $$\left\{\Pmatrix{ 1&a&b\\0&1&c\\ 0&0& 1 } \mid a,b,c \in F_v\right\}.$$

\begin{proof}[Proof of Claim] 
 Let $\frak p_1$ be the prime ideal of $R$ generated by $t_2t_1^{-1}$, and let $\frak p_2$ be the prime ideal of $R$ generated by $t_1$ and $t_2t_1^{-1}$. 
 Let $\cV_1$ be the valuation ring consisting of all elements of the local ring $R_{\frak p_1}$ whose residue classes are in the image of the local ring $R_{\frak p_2}$. Let $K_1$ be the field of fractions of the strict henselization of $\cV_1$ and embed $Q^{\sep}$ into $K_1^{\sep}$ over $Q$. We show that the  image of  $\Gal(K_1^{\sep}/K_1)\subset \Gal(Q^{\sep}/Q)$ in $\GL_3(F_v)$ already has the property in the claim above. 
  Let $\frak p_1'$ be the prime ideal of $\cV_1$ generated by $\frak p_1$, and let $\cV_2$ be the completion of the local ring of $\cV_1$ at $\frak p_2'$. Then $\cV_2$ is a complete discrete valuation ring. Let $K_2$ be the field of  fractions of the strict henselization of $\cV_2$ and embed $K_1^{\sep}$ into $K^{\sep}_2$ over $K_1$. Then over $\cV_2$, we may identify $\phi_1$ with the Drinfeld module of rank $2$ associated to $\phi$ (see \ref{prop51}).
 The action of $\Gal(Q^{\sep}/Q)$ preserves  $V_v\psi$ and $V_v(\phi_1)$ 
and is trivial on $V_v\psi$, on $V_v\phi_1/V_v\psi$, and on $V_v\phi/V_v\phi_1$. From the proof of \ref{thmlm}, we see that $(\sig-1)A_ve_3$ for $\sig\in \Gal(K_1^{\sep}/K_1)$ generates an $A_v$-submodule  of $T_v\phi_1$ of finite index,  the images of  $(\sig-1)T_v\phi$ in $ T_v(\phi_1)/T_v\psi$ for $\sig\in \Gal(K_2^{\sep}/K_2)$ 
generate an $A_v$-submodule of $T_v\phi_1/T_v\psi$  of finite index, and the $(\sig-1) T_v\phi_1$ for $\sig\in \Gal(K_2^{\sep}/K_2)$ generate an $A_v$-submodule  of $T_v\psi$ of finite index.
\end{proof}

\end{sbpara}


\begin{thebibliography}{99}

\bibitem{An}
{\sc Anderson, G.},
{\em $t$-motives}, Duke Math. J. {\bf 53} (1986), 457--502. 



\bibitem{AMRT}
{\sc Ash,~A., Mumford,~D., Rapoport,~M., Tai,~Y.-S.}, 
 {\em Smooth compactification of locally symmetric varieties}, Math. Sci. Press, 1975.


\bibitem{D}
{\sc Drinfeld, V. G.}, 
{\em Elliptic modules}, 
Mat. Sb. (N.S.) {\bf 94(136)} (1974), 594--627, 656 (Russian). English translation:
 Math. USSR-Sb. {\bf 23} (1974), 561--592 (1976). 
  
  
  \bibitem{FC}
  {\sc Faltings, G.,  Chai, C.-L.},
  {\em Degeneration of abelian varieties}, Ergeb. Math. Grenzgeb. (3), {\bf 22}, Springer-Verlag, 1990.

  
  \bibitem{Fu}
  {\sc Fujiwara, K.}, 
  {\em Theory of tubular neighborhood in etale topology},
  Duke Math. J. {\bf 80} (1995), 15--57.


\bibitem{GaRo}
{\sc Gabber, O., Romero, L.},
{\em Foundations for almost ring theory}, preprint,
arXiv:0409584v13.



\bibitem{GI}
{\sc Goldman, O., Iwahori, N.},
{\em The space of $p$-adic norms},
Acta Math. {\bf 109} (1963), 137--177.
 


\bibitem{Go}
{\sc Goss,~D.},
Basic structures of function field arithmetic, Springer-Verlag, 1998.

\bibitem{Go2}
{\sc Goss, ~D.}, 
{\em Drinfeld modules: Cohomology and special functions}. In: Motives, Proc. Sympos. Pure Math. \textbf{55}, Part 2, 1994, 309--362. 


\bibitem{GM}
{\sc Grothendieck, A.,  Murre, J. P.}, 
The tame fundamental group of a formal neighbourhood of a divisor with normal crossings on a scheme. Lecture Notes in Math. {\bf 208}, Springer-Verlag, 1971.


\bibitem{Hu}
{\sc Huber, R.}, 
A generalization of formal schemes and rigid analytic varieties, Math. Z. {\bf 217} (1994), 513--551.

\bibitem{Hu2}
{\sc Huber, R.},
\'Etale cohomology of rigid analytic varieties and adic spaces,  Aspects Math. {\bf 30}, Springer Fachmedien, 1996.

\bibitem{IL}
{\sc  Illusie, L.},
{\em  An overview of the work of K. Fujiwara, K. Kato, and C. Nakayama on logarithmic \'etale cohomology}, 
Ast\'erisque, {\bf 279} (2002), 271--322.

\bibitem{ILO}
{\sc Illusie, L., Laszio, Y., Orgogozo, F.}, 
{\em Travaux de Gabber sur l'uniformisation locale et la cohomologie \'etale des sch\'emas quasi-excellents}, 
Ast\'erisque, {\bf 363--364}, 2014. 

\bibitem{KKN}
{\sc Kajiwara, ~T., Kato, ~K.,  Nakayama, ~C.},
{\em Logarithmic abelian varieties}, I., J. Math. Sci. Univ. Tokyo {\bf 15} (2008), 69-198, II., Nagoya Math. J. {\bf 189} (2008), 63--138, 
III., Nagoya Math. J. {\bf 210} (2013), 59--81,
IV., Nagoya Math. J. {\bf 219} (2015), 9--63, V., Yokohama Math. J. {\bf 64} (2018), 21--82, VI., Yokohama Math. J. {\bf 65} (2019), 53--75, VII., 
Yokohama Math. J. {\bf 67} (2021), 9--48.


\bibitem{K}
{\sc Kapranov,~M. M.},
{\em Cuspidal divisors on the modular varieties of elliptic modules},  Izv. Akad. Nauk SSSR Ser. Mat. {\bf 51} (1987), 568--583; translation in
Math. USSR-Izv. {\bf 30} (1988), 533--547.


\bibitem{KK1}
{\sc Kato,~K.},
{\em Logarithmic structures of Fontaine-Illusie},
 In: Algebraic analysis, geometry, and number theory, Johns Hopkins Univ. Press, 1989, 191--224.

\bibitem{KK}
{\sc Kato,~K.},
{\em  Toric singularities},
Amer. J. Math. {\bf 116} (1994), 1073--1099. 

\bibitem{KNU}
{\sc Kato,  ~K., Nakayama, ~C.,  Usui, ~S.}, 
{\em Classifying spaces of degenerating mixed Hodge structures},
I., Adv. Stud. Pure Math. {\bf 54} (2009), 187--222, II., Kyoto J. Math. {\bf 51} (2011), 
149--261, III., J. Algebraic Geometry {\bf 22} (2013), 671--772. IV. Kyoto J. Math. {\bf 58} (2018). 




\bibitem{KU}
{\sc Kato,  ~K., Usui, ~S.},
{\em Classifying spaces of degenerating polarized 
Hodge structures}, 
Ann. of Math. Stud. {\bf 169}, Princeton Univ. Press, 2009.

\bibitem{KKMS}
{\sc Kempf, G., Knudsen, F., Mumford D.,  Saint-Donat, B.}, Toroidal embeddings I, Lecture Notes in
Math. {\bf 339}, Springer-Verlag, 1973.


\bibitem{ML}
{\sc Lazard, M.}, 
{\em Les z\'eros d'une fonction analytique d'une variable sur un corps valu\'e complet}, 
Publ. Math. Inst. Hautes \'Etudes Sci. {\bf 14} (1962), 47--75. 

\bibitem{TL}
{\sc  Lehmkuhl,  T.}, 
{\em Compactification of the Drinfeld Modular Surfaces}, 
Mem. Amer. Math. Soc. {\bf 921}, 2009. 


\bibitem{O}
{\sc Ogus, A.}, 
{\em Lectures on Logarithmic Algebraic Geometry}, 
Cambridge Stud. Adv. Math. {\bf 178}, Cambridge University Press, 2018.



\bibitem{P1}
{\sc Pink,  ~R.},
{\em On compactification of Drinfeld moduli schemes}. In: Moduli spaces, Galois representations, and $L$-functions,
 Surikaisekikenkyusho Kokyuroku {\bf 884} (1994), 178--183.


\bibitem{P2}
{\sc Pink,  ~R.},
{\em Compactification of Drinfeld modular varieties and Drinfeld modular forms of arbitrary rank}, Manuscripta Math. {\bf 140} (2013), 333--361.

\bibitem{PS}
{\sc Pink, ~R., Schieder, ~S.}, 
{\em Compactification of a Drinfeld period domain over a finite field},
J. Alg. Geom. {\bf 23} (2014), 201--243.

\bibitem{PA}
{\sc Puttick, A.}, 
{\em Compactification of the finite Drinfeld period domain as a moduli space of ferns}, 
Thesis, ETH Zurich (2018),  preprint, arXiv:2004.14742. 


\bibitem{taguchi}
{\sc Taguchi, Y.}, 
{\em Semi-simplicity of the Galois representations attached to Drinfeld modules over fields of ``infinite characteristics''}, 
J. Number Theory {\bf 44} (1993), 292--314.

\end{thebibliography}
\end{document}